\documentclass[10pt]{article}

\usepackage[english,plain]{fancyref}
\usepackage{hyperref}
\usepackage{amssymb}
\usepackage{amsmath}
\usepackage{mathtools}
\usepackage{paralist}
\usepackage[symbols,nogroupskip,toc=false]{glossaries-extra}

\input{liemacs10.sty}
\usepackage{tikz,tikz-cd}
\usepackage[all]{xy} 
\usepackage{tikz}
\usetikzlibrary{calc}
\usetikzlibrary{decorations.markings}
\usetikzlibrary{positioning}
\usepackage{caption}
\usepackage{float}
\usepackage{pgfplots}

\usepackage{scalerel,stackengine}
\stackMath
\newcommand\reallywidehat[1]{%
\savestack{\tmpbox}{\stretchto{%
  \scaleto{%
    \scalerel*[\widthof{\ensuremath{#1}}]{\kern.1pt\mathchar"0362\kern.1pt}%
    {\rule{0ex}{\textheight}}
  }{\textheight}%
}{2.4ex}}%
\stackon[-6.9pt]{#1}{\tmpbox}%
}

\newcommand{\sV}{{\tt V}}

\newcommand{\Id}{\mathrm{Id}}

\newcommand{\D}{{\mathbb D}}

\newcommand{\Out}{\mathop{{\rm Out}}\nolimits}

\renewcommand{\phi}{\varphi}

\newcommand\fM {\mathfrak M}
\newcommand\Han {\mathrm{Han}(\C_+)}
\newcommand\HanP {\mathrm{Han}(\C_+)_+}
\newcommand\CM {\mathcal{CM}\left(\R_+\right)}

\renewcommand{\phi}{\varphi}
\renewcommand{\epsilon}{\varepsilon}

\newcommand*{\fancyrefproplabelprefix}{prop}
\newcommand*{\frefpropname}{Proposition}
\frefformat{plain}{\fancyrefproplabelprefix}{\frefpropname\fancyrefdefaultspacing#1}
\newcommand*{\fancyrefthmlabelprefix}{thm}
\newcommand*{\frefthmname}{Theorem}
\frefformat{plain}{\fancyrefthmlabelprefix}{\frefthmname\fancyrefdefaultspacing#1}
\newcommand*{\fancyreflemlabelprefix}{lem}
\newcommand*{\freflemname}{Lemma}
\frefformat{plain}{\fancyreflemlabelprefix}{\freflemname\fancyrefdefaultspacing#1}
\newcommand*{\fancyrefcorlabelprefix}{cor}
\newcommand*{\frefcorname}{Corollary}
\frefformat{plain}{\fancyrefcorlabelprefix}{\frefcorname\fancyrefdefaultspacing#1}
\renewcommand*{\fancyrefeqlabelprefix}{eq}

\frefformat{plain}{\fancyrefeqlabelprefix}{(#1)}
\newcommand*{\fancyrefproblabelprefix}{prob}
\newcommand*{\frefprobname}{Problem}
\frefformat{plain}{\fancyrefproblabelprefix}{\frefprobname\fancyrefdefaultspacing#1}
\newcommand*{\fancyrefdeflabelprefix}{def}
\newcommand*{\frefdefname}{Definition}
\frefformat{plain}{\fancyrefdeflabelprefix}{\frefdefname\fancyrefdefaultspacing#1}
\newcommand*{\fancyrefapplabelprefix}{app}
\newcommand*{\frefappname}{Appendix}
\frefformat{plain}{\fancyrefapplabelprefix}{\frefappname\fancyrefdefaultspacing#1}
\renewcommand*{\fancyrefseclabelprefix}{sec}
\renewcommand*{\frefsecname}{Section}
\frefformat{plain}{\fancyrefseclabelprefix}{\frefsecname\fancyrefdefaultspacing#1}
\newcommand*{\fancyrefconlabelprefix}{con}
\newcommand*{\frefconname}{Conjecture}
\frefformat{plain}{\fancyrefconlabelprefix}{\frefconname\fancyrefdefaultspacing#1}
\newcommand*{\fancyrefremlabelprefix}{rem}
\newcommand*{\frefremname}{Remark}
\frefformat{plain}{\fancyrefremlabelprefix}{\frefremname\fancyrefdefaultspacing#1}
\newcommand*{\fancyrefexlabelprefix}{ex}
\newcommand*{\frefexname}{Example}
\frefformat{plain}{\fancyrefexlabelprefix}{\frefexname\fancyrefdefaultspacing#1}

\DeclarePairedDelimiterX\braket[2]{\langle}{\rangle}{#1 \delimsize\vert #2}

\makenoidxglossaries

\glsxtrnewsymbol[description={positive reals}]{R+}{\ensuremath{\R_+}}
\glsxtrnewsymbol[description={upper half-plane}]{C+}{\ensuremath{\C_+}}
\glsxtrnewsymbol[description={torus}]{Torus}{\ensuremath{\T}}
\glsxtrnewsymbol[description={unit disc}]{D}{\ensuremath{\D}}
\glsxtrnewsymbol[description={Riemann sphere}]{C8}{\ensuremath{\C_\infty}}
\glsxtrnewsymbol[description={holomorphic functions}]{O}{\ensuremath{\cO}}
\glsxtrnewsymbol[description={commutant}]{E'}{\ensuremath{E'}}
\glsxtrnewsymbol[description={scalar product}]{ZZZbraket}{\ensuremath{\braket*{\cdot}{\cdot}}}
\glsxtrnewsymbol[description={bounded operators}]{BH}{\ensuremath{B(\cH)}}
\glsxtrnewsymbol[description={self-adjoint operators}]{SH}{\ensuremath{S(\cH)}}
\glsxtrnewsymbol[description={orthogonal projection}]{PK}{\ensuremath{P_\cK}}
\glsxtrnewsymbol[description={partial isometry}]{pK}{\ensuremath{p_\cK}}
\glsxtrnewsymbol[description={infinitesimal generator}]{DU}{\ensuremath{\partial U}}
\glsxtrnewsymbol[description={\(p\)-integrable functions}]{Lp}{\ensuremath{L^p}}
\glsxtrnewsymbol[description={Lebesgue measure}]{ln}{\ensuremath{\lambda_n}}
\glsxtrnewsymbol[description={multiplication operator}]{Mg}{\ensuremath{M_g}}
\glsxtrnewsymbol[description={complex conjugation}]{CH}{\ensuremath{\cC_\cH}}
\glsxtrnewsymbol[description={complexified operator}]{AC}{\ensuremath{A_\C}}
\glsxtrnewsymbol[description={tensor product}]{ZZZtp}{\ensuremath{\otimes}}
\glsxtrnewsymbol[description={semi-regular one-parameter group}]{SROPG}{\ensuremath{SROPG}}
\glsxtrnewsymbol[description={regular one-parameter group}]{ROPG}{\ensuremath{ROPG}}
\glsxtrnewsymbol[description={translation action}]{tau}{\ensuremath{\tau}}
\glsxtrnewsymbol[description={\(\sharp\)-involution on \(L^2(\R,\cK_\C)\)}]{fsharp}{\ensuremath{f^\sharp}}
\glsxtrnewsymbol[description={Fourier transformation on \(L^2\)}]{fourier}{\ensuremath{\cF}}
\glsxtrnewsymbol[description={standard one-parameter group}]{S}{\ensuremath{S}}
\glsxtrnewsymbol[description={unitary one-parameter group generated by \(F\)}]{UF}{\ensuremath{U^F}}
\glsxtrnewsymbol[description={spectral preimage of \(M\) under \(F\)}]{DFM}{\ensuremath{D(F,M)}}
\glsxtrnewsymbol[description={orthogonal complement in \(H^2(\C_+,\cH)\)}]{ZZZO+}{\ensuremath{\perp_+}}
\glsxtrnewsymbol[description={orthogonal projection onto \(H^2(\C_+,\cH)\)}]{P+}{\ensuremath{P_+}}
\glsxtrnewsymbol[description={phase functions}]{et}{\ensuremath{e_t}}
\glsxtrnewsymbol[description={Blaschke--Potapov factor}]{phiov}{\ensuremath{\phi_{\omega,v}}}
\glsxtrnewsymbol[description={degree of Blaschke-Potapov products}]{degBlaschke}{\ensuremath{\deg(\phi)}}
\glsxtrnewsymbol[description={order of holomorphic functions}]{Ord}{\ensuremath{\mathrm{Ord}}}
\glsxtrnewsymbol[description={Cayley transform}]{h}{\ensuremath{h}}
\glsxtrnewsymbol[description={minimal pole set}]{Elmin}{\ensuremath{E_\lambda^\mathrm{min}}}
\glsxtrnewsymbol[description={holomorphic extension of \(\phi_\lambda \circ F\)}]{Flmin}{\ensuremath{F_\lambda^\mathrm{min}}}
\glsxtrnewsymbol[description={degree of Pick functions}]{degPick}{\ensuremath{\deg(F)}}
\glsxtrnewsymbol[description={rational Pick functions}]{Rat}{\ensuremath{\mathrm{Rat}(\cH)}}
\glsxtrnewsymbol[description={isomorphism from \(L^2(\R,\C) \hat{\otimes} \cH\) to \(L^2(\R,\cH)\)}]{TH}{\ensuremath{T_\cH}}
\glsxtrnewsymbol[description={isomorphism from \(L^2(\R,\cH) \hat{\otimes} \cK\) to \(L^2(\R,\cH \hat{\otimes} \cK)\)}]{IHK}{\ensuremath{I_\cH^\cK}}
\glsxtrnewsymbol[description={isomorphism from \(\cH \hat{\otimes} \cK\) to \(\cK \hat{\otimes} \cH\)}]{sHK}{\ensuremath{\sigma_\cH^\cK}}
\glsxtrnewsymbol[description={isomorphism from \(L^2(\R,\cH) \hat{\otimes} \cK\) to \(L^2(\R,\cK) \hat{\otimes} \cH\)}]{EHK}{\ensuremath{E_\cH^\cK}}
\glsxtrnewsymbol[description={Möbius group}]{Mob}{\ensuremath{\mathrm{Mob}}}
\glsxtrnewsymbol[description={reflection positive Hilbert space}]{RPHS}{\ensuremath{RPHS}}
\glsxtrnewsymbol[description={reflection}]{R}{\ensuremath{R}}
\glsxtrnewsymbol[description={unitary involution associated to \(h\)}]{thetah}{\ensuremath{\theta_h}}
\glsxtrnewsymbol[description={\(\flat\)-involution}]{hflat}{\ensuremath{h^\flat}}
\glsxtrnewsymbol[description={\(\sharp\)-involution on \(L^\infty(\R,\U(\cK_\C))\)}]{hsharp}{\ensuremath{h^\sharp}}
\glsxtrnewsymbol[description={Fourier transformation on \(L^1\)}]{fourier1}{\ensuremath{\cF_1}}
\glsxtrnewsymbol[description={Borel measures on \(X\)}]{BMX}{\ensuremath{\mathcal{BM}\left(X\right)}}
\glsxtrnewsymbol[description={finite Borel measures on \(X\)}]{BMXfin}{\ensuremath{\mathcal{BM}^{\mathrm{fin}}\left(X\right)}}
\glsxtrnewsymbol[description={reflection positive function}]{phimu}{\ensuremath{\phi_\mu}}
\glsxtrnewsymbol[description={fourier transform of reflection positive function}]{psimu}{\ensuremath{\psi_\mu}}
\glsxtrnewsymbol[description={variant of \(\psi_\mu\)}]{Psinu}{\ensuremath{\Psi_\nu}}
\glsxtrnewsymbol[description={map between Borel measures}]{W}{\ensuremath{W}}
\glsxtrnewsymbol[description={outer function associated to \(\sqrt{\Psi_\nu}\)}]{Fnu}{\ensuremath{F_\nu}}
\glsxtrnewsymbol[description={\(L^\infty\)-function defined as \(\frac{F_\nu}{RF_\nu}\)}]{hnu}{\ensuremath{h_\nu}}
\glsxtrnewsymbol[description={holomorphic function on \(\C_+\)}]{Lambda}{\ensuremath{\Lambda}}
\glsxtrnewsymbol[description={preimage of the Hardy space under multiplication by \(\Lambda\)}]{H2Lambda}{\ensuremath{H^2_\Lambda(\C_+)}}
\glsxtrnewsymbol[description={outer functions in \(H^2_\Lambda(\C_+)\)}]{Out2Lambda}{\ensuremath{\mathrm{Out}^2_\Lambda(\C_+)}}
\glsxtrnewsymbol[description={measure on \(\R\)}]{rho}{\ensuremath{\rho}}
\glsxtrnewsymbol[description={extension of \(\theta_h\) to \(L^2(\R,\C,\rho)\)}]{Thetah}{\ensuremath{\Theta_h}}
\glsxtrnewsymbol[description={map between Borel measures}]{T}{\ensuremath{\cT}}
\glsxtrnewsymbol[description={complex Borel measures on \(X\)}]{BMXC}{\ensuremath{\mathcal{BM}_\C\left(X\right)}}
\glsxtrnewsymbol[description={equivalence relation between finite Borel measures}]{ZZZtilde}{\ensuremath{\sim}}
\glsxtrnewsymbol[description={isotropic subspace for \(\braket*{\cdot}{\theta \cdot}\)}]{Ntheta}{\ensuremath{\cN_\theta}}
\glsxtrnewsymbol[description={quotient map}]{qtheta}{\ensuremath{q_\theta}}
\glsxtrnewsymbol[description={Osterwalder--Schrader transform}]{Ehat}{\ensuremath{\hat{\cE}}}
\glsxtrnewsymbol[description={Hankel operator associated to the symbol \(h\)}]{Hh}{\ensuremath{H_h}}
\glsxtrnewsymbol[description={projection from \(L^2(\R,\C)\) to \(H^2(\C_+)\)}]{p+}{\ensuremath{p_+}}
\glsxtrnewsymbol[description={Hankel operators on \(\C_+\)}]{Han}{\ensuremath{\Han}}
\glsxtrnewsymbol[description={positive Hankel operators on \(\C_+\)}]{HanP}{\ensuremath{\HanP}}
\glsxtrnewsymbol[description={Hankel operator associated to the Carleson measure \(\mu\)}]{Hmu}{\ensuremath{H_\mu}}
\glsxtrnewsymbol[description={Carleson measure associated to the Hankel operator \(H\)}]{muH}{\ensuremath{\mu_H}}
\glsxtrnewsymbol[description={weak operator topology}]{WOT}{\ensuremath{\mathrm{WOT}}}
\glsxtrnewsymbol[description={Radon--Nikodym derivative of the measure \(\cT \nu\)}]{tnu}{\ensuremath{t_\nu}}
\glsxtrnewsymbol[description={standard subspace}]{V}{\ensuremath{\sV}}
\glsxtrnewsymbol[description={Tomita operator}]{TV}{\ensuremath{T_\sV}}
\glsxtrnewsymbol[description={anti-unitary modular operator}]{JV}{\ensuremath{J_\sV}}
\glsxtrnewsymbol[description={positive modular operator}]{DeltaV}{\ensuremath{\Delta_\sV}}
\glsxtrnewsymbol[description={symplectic complement}]{Vprime}{\ensuremath{\sV'}}
\glsxtrnewsymbol[description={affine group}]{Aff}{\ensuremath{\mathrm{Aff}(\R)}}
\glsxtrnewsymbol[description={\(ax+b\)-group}]{Aff+}{\ensuremath{\mathrm{Aff}(\R)_+}}
\glsxtrnewsymbol[description={unitary and anti-unitary operators}]{AUH}{\ensuremath{\mathrm{AU}(\cH)}}
\glsxtrnewsymbol[description={Hilbert space of \(\sharp\)-invariant \(L^2\)-functions}]{HK}{\ensuremath{\cH_\cK}}
\glsxtrnewsymbol[description={representation of the affine group}]{rhoK}{\ensuremath{\rho_\cK}}
\glsxtrnewsymbol[description={standard subspace in \(\cH_\cK\)}]{VK}{\ensuremath{\sV_\cK}}
\glsxtrnewsymbol[description={semigroup of endomorphisms of \(\sV\)}]{EVU}{\ensuremath{\cE(\sV,U)}}
\glsxtrnewsymbol[description={Hardy space}]{Hp}{\ensuremath{H^p(\C_+)}}
\glsxtrnewsymbol[description={reproducing kernel functions}]{Qw}{\ensuremath{Q_w}}
\glsxtrnewsymbol[description={\(\cH\)-valued Hardy space}]{H2H}{\ensuremath{H^2(\C_+,\cH)}}
\glsxtrnewsymbol[description={\(B(\cH)\)-valued Hardy space}]{H8BH}{\ensuremath{H^\infty(\C_+,B(\cH))}}
\glsxtrnewsymbol[description={Hardy space on the disc}]{HpD}{\ensuremath{H^p(\D)}}
\glsxtrnewsymbol[description={integrability condition}]{IK}{\ensuremath{I(K)}}
\glsxtrnewsymbol[description={outer function associated to \(K\)}]{OutK}{\ensuremath{\mathrm{Out}(K)}}
\glsxtrnewsymbol[description={outer functions on \(\C_+\)}]{Out}{\ensuremath{\mathrm{Out}(\C_+)}}
\glsxtrnewsymbol[description={inner functions on \(\C_+\)}]{Inn}{\ensuremath{\mathrm{Inn}(\C_+,B(\cH))}}
\glsxtrnewsymbol[description={Blaschke factor}]{phio}{\ensuremath{\phi_\omega}}
\glsxtrnewsymbol[description={Pick functions}]{Pick}{\ensuremath{\mathrm{Pick}(\C_+,B(\cH))}}
\glsxtrnewsymbol[description={Pick functions with self-adjoint boundary values}]{PickR}{\ensuremath{\mathrm{Pick}(\C_+,B(\cH))_\R}}
\glsxtrnewsymbol[description={boundary values}]{Fstar}{\ensuremath{F_*}}
\glsxtrnewsymbol[description={minimal support for the measure \(\mu\)}]{Smu}{\ensuremath{S_\mu}}
\glsxtrnewsymbol[description={minimal support for the Pick function \(F\)}]{SF}{\ensuremath{S_F}}
\glsxtrnewsymbol[description={function on \(\R\)}]{gn}{\ensuremath{g_n}}
\glsxtrnewsymbol[description={function on \(\R\)}]{dpn}{\ensuremath{d_{p,n}}}
\glsxtrnewsymbol[description={function on \(\R\)}]{fpn}{\ensuremath{f_{p,n}}}
\glsxtrnewsymbol[description={function on \(\R\)}]{dpntilde}{\ensuremath{\tilde d_{p,n}}}

\begin{document} 
\renewcommand\thepage{}


\title{\textbf{Regular one-parameter groups,\\reflection positivity and their application to
\\Hankel operators and standard subspaces
\\~
\\~
\\Reguläre Einparameter-Gruppen,\\Reflexionspositivität und deren Anwendung auf
\\Hankel-Operatoren und Standard-Unterräume
}}
\date{}
\author{}

\maketitle

{\center\Large{~
\\~
\\~
\\Der Naturwissenschaftlichen Fakultät
\\der Friedrich-Alexander-Universität
\\Erlangen-Nürnberg
\\zur
\\Erlangung des Doktorgrades Dr. rer. nat.
\\~
\\~
\\~
\\~
\\~
\\vorgelegt von
\\\textbf{Jonas Schober}
\\~}}

\newpage
{\Large{~
\\Als Dissertation genehmigt von der Naturwissenschaftlichen Fakultät der Friedrich-Alexander-Universität Erlangen-Nürnberg.
\vspace{15cm}
\\Tag der mündlichen Prüfung: 2. August 2023
\vspace{1cm}
\\Gutachter: Prof. Dr. Karl-Hermann Neeb
\\\hspace*{2.53cm}Prof. Dr. Gandalf Lechner
\\\hspace*{2.53cm}Prof. Dr. Yoh Tanimoto
}}

\newpage
\setcounter{page}{1}
\renewcommand\thepage{\roman{page}}
\section*{Zusammenfassung}
Standard-Unterräume sind viel studierte Objekte in der algebraischen Quantenfeldtheorie (AQFT). Gegeben einen Standard-Unterraum \(\sV\) in einem Hilbertraum \(\cH\), ist man interessiert an unitären Einparameter-Gruppen auf \(\cH\) für welche \(U_t \sV \subeq \sV\) für alle \(t \in \R_+\) gilt. Falls \((\sV,U)\) ein nicht-degeneriertes Standard-Paar auf \(\cH\) ist, d.h. der selbst\-adjungierte infinitesimale Erzeuger von \(U\) ist ein positiver Operator mit trivialem Kern, haben wir zwei klassische Resultate mit dem Satz von Borchers, welcher nicht-degenerierte Standard-Paare mit Darstellungen der affinen Gruppe \(\mathrm{Aff}(\R)\) mit positiver Energie in Verbindung setzt und dem Satz von Longo--Witten, laut welchem die Halbgruppe der unitären Endomorphismen von \(\sV\) mit der Halbgruppe der symmetrischen operatorwertigen inneren Funktionen auf der oberen Halbebene identifiziert werden kann.

In dieser Arbeit beweisen wir Sätze ähnlich denen von Borchers und Longo--Witten auch für allgemeinere Fälle von unitären Einparameter-Gruppen ohne die Annahme, dass deren infinitesimaler Erzeuger positiv ist. Wir wollen diese Annahme durch die schwächere Annahme ersetzen, dass das Tripel \((\cH,\sV,U)\) eine sogenannte reelle reguläre Einparameter-Gruppe ist.

Nach Bereitstellung der grundlegenden Theorie von regulären Einparameter-Gruppen, werden wir reguläre Einparameter-Gruppen untersuchen, die infinitesimal erzeugt werden von operator\-wertigen Pick-Funktionen. Wir werden Kriterien bereit\-stellen, um zu prüfen, wann eine Pick-Funktion eine reguläre Einparameter-Gruppe generiert und in diesem Fall werden wir Formeln für den Multiplizi\-täten\-raum der regulären Einparameter-Gruppe bereitstellen und zeigen, dass die Dimension des Multiplizitätenraumes mit der Komposition von Pick-Funktionen verträglich ist.

Wir werden auch reflexionspositive reguläre Einparameter-Gruppen untersuchen und eine Normalform für diese angeben. Wir tun dies, indem wir sie in Verbindung setzen mit reflexionspositiven Hilberträumen der Form \((L^2(\R,\cK),H^2(\C_+,\cK),\theta_h)\) mit einem komplexen Hilbertraum \(\cK\) und einer Funktion \(h \in L^\infty(\R,\U(\cK))\), wobei die Involution \(\theta_h\) auf \(L^2(\R,\cK)\) durch \({(\theta_h f)(x) = h(x) \cdot f(-x)}\), \(x \in \R\), gegeben ist. In dem multiplizitätsfreien Fall \(\cK = \R\) geben wir eine vollständige Klassifikation aller Funktionen \({h \in L^\infty(\R,\T)}\), für welche das Tripel \((L^2(\R,\C),H^2(\C_+),\theta_h)\) ein maxi\-maler reflexions\-positiver Hilbertraum ist. Außerdem geben wir eine explizite Beschreibung der Osterwalder--Schrader--Transformation dieser reflexionspositiven Hilberträume.

Wir werden diese Resultate anwenden, um zu beweisen, dass jeder positive kontraktive Hankel-Operator auf \(H^2(\C_+)\) zu einer Involution \(\theta_h\) auf \(L^2(\R,\C)\) erweitert werden kann, für welche das Tripel \((L^2(\R,\C),H^2(\C_+),\theta_h)\) ein maximaler reflexionspositiver Hilbertraum ist.

Abschließend wenden wir unsere Resultate über (reflexionspositive) reguläre Einparameter-Gruppen auf den Kontext von Standard-Unterräumen an und benutzen sie, um Analoga für den Satz von Borchers und den Satz von Longo--Witten auch für solche Paare \((\sV,U)\) zu beweisen, für welche das Tripel \((\cH,\sV,U)\) eine reelle reguläre Einparameter-Gruppe ist bzw. für welche das Quadrupel \((\cH,\sV,U,J_\sV)\) eine reelle reflexionspositive reguläre Einparameter-Gruppe ist.

\newpage
\section*{Abstract}
Standard subspaces are a well-studied object in algebraic quantum field theory (AQFT). Given a standard subspace \(\sV\) of a Hilbert space \(\cH\), one is interested in unitary one-parameter groups on \(\cH\) with \(U_t \sV \subeq \sV\) for every \(t \in \R_+\). If \((\sV,U)\) is a non-degenerate standard pair on \(\cH\), i.e. the self-adjoint infinitesimal generator of \(U\) is a positive operator with trivial kernel, two classical results are given by Borchers' Theorem, relating non-degenerate standard pairs to positive energy representations of the affine group \(\mathrm{Aff}(\R)\) and the Longo--Witten Theorem, stating that the semigroup of unitary endomorphisms of \(\sV\) can be identified with the semigroup of symmetric operator-valued inner functions on the upper half-plane.

In this thesis, we prove results similar to the theorems of Borchers and of Longo--Witten for a more general framework of unitary one-parameter groups without the assumption that their infinitesimal generator is positive. We replace this assumption by the weaker assumption that the triple \((\cH,\sV,U)\) is a so-called real regular one-parameter group.

After providing some basic theory for regular one-parameter groups, we investigate regular one-parameter groups that are infinitesimally generated by operator-valued Pick functions. We provide criteria for a Pick function to generate a regular one-parameter group, and in this case, we provide a formula for the multiplicity space of the regular one-parameter group and show that the dimension of the multiplicity space is compatible with composition of Pick functions.

We also investigate reflection positive regular one-parameter groups and provide a normal form for them. We do this by linking them to reflection positive Hilbert spaces of the form \((L^2(\R,\cK),H^2(\C_+,\cK),\theta_h)\) with some complex Hilbert space \(\cK\) and some function \(h \in L^\infty(\R,\U(\cK))\), where the involution \(\theta_h\) on \(L^2(\R,\cK)\) is given by \({(\theta_h f)(x) = h(x) \cdot f(-x)}\), \(x \in \R\). In the multi\-plicity free case \(\cK = \R\) we give a full classification of all functions \({h \in L^\infty(\R,\T)}\) for which the triple \((L^2(\R,\C),H^2(\C_+),\theta_h)\) is a maximal reflection positive Hilbert space. Also, we give an explicit description of the Osterwalder--Schrader transform of these reflection positive Hilbert spaces.

We apply these results to prove that every positive contractive Hankel operator on \(H^2(\C_+)\) can be extended to an involution \(\theta_h\) on \(L^2(\R,\C)\) for which the triple \((L^2(\R,\C),H^2(\C_+),\theta_h)\) is a maximal reflection positive Hilbert space.

Finally, we apply our results about (reflection positive) regular one-parameter groups in the context of standard subspaces and use them to provide analogs of Borchers' Theorem and the Longo--Witten Theorem to pairs \((\sV,U)\) for which the triple \((\cH,\sV,U)\) is a real regular one-parameter group or for which the quadruple \((\cH,\sV,U,J_\sV)\) is a real reflection positive regular one-parameter group respectively.

\newpage
\section*{Acknowledgment}
First of all, I want to thank my supervisor Prof. Dr. Karl-Hermann Neeb for the great support he gave me throughout the years. He was always there to answer my questions, point out references, or discuss new ideas. I want to thank him for accompanying me from the beginning of my studies of mathematics at the university until the end of this PhD project.

~

I also want to thank Tobias Simon for always being there to discuss and explain mathematical concepts and ideas that we encountered in courses or seminaries we visited together. Also a big thanks to him for proofreading this thesis.

~

Further, I want to thank my mother Dagmar Barbara Schober for always being there for me in the last 25 years and supporting me in everything I did, both in my personal and professional life. I am also deeply grateful to her for always spell-checking my mathematical texts -- including this one -- even though she understands very little about mathematics.

~

Another big thank you goes to my father Alexander Schober, who always encouraged me to proceed a scientific path. Without him, I would probably not have studied math and without him, I would not be at this point in life.

~

Finally, I want to thank my wife Amy Gómez Mederos for being an immense support to me. I know that especially this last period of time of me writing this thesis demanded many sacrifices from her and I am very grateful that anyhow she supported me in every way possible.

\newpage
~
\newpage
\tableofcontents

\newpage
\section*{Notation}

\subsection*{Sets}
\begin{itemize}
\item We set \(\gls*{R+} \coloneqq \left(0,\infty\right)\) and \(\R_{\geq 0} \coloneqq \left[0,\infty\right)\) and define \(\R_- \coloneqq -\R_+\) and \(\R_{\leq 0} \coloneqq -\R_{\geq 0}\).
\item We set \(\R^\times \coloneqq \R \setminus \{0\}\) and \(\C^\times \coloneqq \C \setminus \{0\}\).
\item We define the complex half-planes
\begin{equation*}
\gls*{C+} \coloneqq \left\lbrace z \in \C : \mathrm{Im}\left(z\right) > 0\right\rbrace, \quad \C_r \coloneqq \left\lbrace z \in \C : \mathrm{Re}\left(z\right) > 0\right\rbrace
\end{equation*}
and
\begin{equation*}
\C_- \coloneqq -\C_+, \quad \C_l \coloneqq -\C_r.
\end{equation*}
\item We set
\[\gls*{Torus} \coloneqq \{z \in \C: |z| = 1\} \quad \text{and} \quad \gls*{D} \coloneqq \{z \in \C: |z| < 1\}.\]
\item By \(\gls*{C8}\) we denote the Riemann sphere.
\item For any open subset \(U \subseteq \C_\infty\) we denote by \(\gls*{O}\left(U\right)\) the set of holomorphic functions on \(U\).
\item For a set \(X\) and an involution \(\theta: X \to X\), we define
\[X^\theta \coloneqq \{x \in X: \theta x = x\}.\]
\item For an algebra \(\cA\) and a subset \(E \subeq \cA\) we define the \textit{commutant} \(\gls*{E'}\) by
\[E' \coloneqq \{a \in \cA: (\forall b \in E) \,ab=ba\}.\]
\end{itemize}

\subsection*{Hilbert spaces}
\begin{itemize}
\item We denote scalar products in a Hilbert space by \(\gls*{ZZZbraket}\), being linear in the second argument.
\item If not mentioned otherwise, we assume that all occuring Hilbert spaces are separable.
\item For a Hilbert space \(\cH\) we denote by \(\gls*{BH}\) the set of bounded linear operators on \(\cH\), by \(\gls*{SH}\) the set of self-adjoint linear operators on \(\cH\) and by \(S(\cH)_+\) the set of positive linear operators on~\(\cH\).
\item For a Hilbert space \(\cH\) and a closed subspace \(\cK \subeq \cH\) we denote by \(\gls*{PK} \in B(\cH)\) the orthogonal projection onto \(\cK\) and by \(\gls*{pK}: \cH \to \cK\) the partial isometry with
\[p_\cK\big|_\cK = \id_\cK \quad \text{and} \quad p_\cK\big|_{\cK^\perp} = 0.\]
\item For a complex Hilbert space \(\cH\) and a one-parameter (semi)group \(U: \R_{(\geq 0)} \to B(\cH)\), we write \(\gls*{DU}\) for the unbounded operator und \(\cH\) defined on
\[\cD(\partial U) \coloneqq \left\{v \in \cH: \lim_{t \to 0} \frac{1}{it}(U_t-\textbf{1})v \text{\,\,exists}\right\}\]
by
\[\partial U v = \lim_{t \to 0} \frac{1}{it}(U_t-\textbf{1})v\]
(cf. \cite[Def. 5.8]{Ne17}).
\end{itemize}

\subsection*{\(\gls*{Lp}\)-spaces}
\begin{itemize}
\item For \(1 \leq p \leq \infty\) by \(\cL^p\left(\Omega,\mathbb{K},\mu\right)\) we denote the vector space of \mbox{\(p\)-integrable} functions with values in the field \(\mathbb{K}=\R,\C\) on the measurable space \(\Omega\) with respect to the measure \(\mu\). Further, we write \(L^p(\Omega,\K,\mu)\) for the corresponding Banach space of equivalence classes of such functions with respect to the equivalence relation given by
\[f \sim g \qquad :\Leftrightarrow \qquad \left(\exists E \subeq \Omega\right) : \left(\mu(E) = 0 \quad \text{and} \quad \forall x \in \Omega \setminus E : f(x) = g(x)\right).\]
\item For \(n \in \N\), we denote the Lebesgue-measure on \(\R^n\) by \(\gls*{ln}\).
\item For \(\Omega \subseteq \R^n\) and \(1 \leq p \leq \infty\) we set
\[L^p(\Omega,\K) \coloneqq L^p\left(\Omega,\K,\lambda_n\big|_\Omega\right).\]
\item For a Hilbert space \(\cH\) over a field \(\mathbb{K}=\R,\C\), we define the Hilbert space
\[L^2(\Omega,\cH,\mu) \coloneqq L^2(\Omega,\mathbb{K},\mu) \,\hat{\otimes}\, \cH.\]
Interpreting \(L^2(\Omega,\cH,\mu)\) as \(\cH\)-valued, almost everywhere defined functions on \(\R\) with
\[(f \otimes v)(x) = f(x) \cdot v, \quad f \in L^2(\Omega,\mathbb{K},\mu), v \in \cH, x \in \R,\]
the scalar product on \(L^2(\Omega,\cH,\mu)\) is given by
\[\braket*{f}{g} \coloneqq \int_\Omega \braket*{f(x)}{g(x)}_\cH\,dx.\]
\item For a Hilbert space \(\cH\) over a field \(\mathbb{K} = \R,\C\) and \(g \in L^\infty(\Omega,\mathbb{K},\mu)\), we define the multiplication operator \(\gls*{Mg} \in B(L^2(\Omega,\cH,\mu))\) by
\[(M_g f)(x) = g(x) \cdot f(x), \quad x \in \R.\]
\item Let \(\cH\) be a Hilbert space over a field \(\mathbb{K}=\R,\C\) and let \((\Omega,\mu)\) be a \(\sigma\)-finite measure space. We call a function \(f: \Omega \to B(\cH)\) \textit{weak-\(*\)-measurable} if, for every trace-class operator \(T \in B_1(\cH)\), the function
\[f_T: \Omega \to \K, \quad x \mapsto \mathrm{Tr}(Tf(x))\]
is measurable. We write \(\cL^\infty(\Omega,B(\cH),\mu)\) for the set of bounded weak-\(*\)-measurable functions \(f: \Omega \to B(\cH)\). Further, we write \(L^\infty(\Omega,B(\cH),\mu)\) for the corresponding Banach space of equivalence classes of such functions with respect to the equivalence relation given by
\[f \sim g \qquad :\Leftrightarrow \qquad \left(\exists E \subeq \Omega\right) : \left(\mu(E) = 0 \quad \text{and} \quad \forall x \in \Omega \setminus E : f(x) = g(x)\right).\]
For \(g \in L^\infty(\Omega,B(\cH),\mu)\) we define the multiplication operator \(M_g \in B(L^2(\Omega,\cH,\mu))\) by
\[(M_g f)(x) = g(x) \cdot f(x), \quad x \in \R.\]
\end{itemize}

\subsection*{Complexification}
Let \(\cH\) be a real Hilbert space.
\begin{itemize}
\item We denote the \textit{complexification} of \(\cH\) by
\[\cH_\C \coloneqq \C \otimes_\R \cH.\]
\item By \(\cC_\cH\) we denote the complex conjugation
\[\gls*{CH}: \cH_\C \to \cH_\C, \quad z \otimes v \mapsto \overline{z} \otimes v.\]
\item We will identify \(\cH\) with the real subspace \(\cH_\C^{\cC_\cH} = \R \otimes_\R \cH\). Under this identification, we have \(\cH_\C = \cH + i\cH\).
\item For an operator \(A \in B(\cH)\), by \(\gls*{AC} \in B(\cH_\C)\), we denote the operator with
\[A_\C(x+iy) = Ax + iAy \quad \forall x,y \in \cH.\]
\end{itemize}

\subsection*{Tensor products}
\begin{itemize}
\item For two Hilbert spaces \(\cH,\cK\) we define the \textit{tensor product Hilbert space}
\[\cH \,\hat{\gls*{ZZZtp}}\, \cK\]
as the completion of the vector space \(\cH \otimes \cK\) with respect to the sesquilinear map on \(\cH \otimes \cK\) defined by
\[\braket*{a \otimes b}{v \otimes w} \coloneqq \braket*{a}{v}_\cH \cdot \braket*{b}{w}_\cK,\]
where \(\braket*{\cdot}{\cdot}_\cH\) denotes the scalar product on \(\cH\) and \(\braket*{\cdot}{\cdot}_\cK\) the scalar product on \(\cK\) respectively (cf. \cite[Sec. 1.20]{Sa71}, \cite[Def. 4.1]{Ne17}).
\item For two Hilbert spaces \(\cH,\cK\) we consider the inclusion
\[\iota: B(\cH) \otimes B(\cK) \to B(\cH \,\hat{\otimes}\, \cK)\]
with
\[\iota(A \otimes B)(v \otimes w) = (Av) \otimes (Bw)\]
(cf. \cite[Sec. 1.20]{Sa71}, \cite[Lem. 4.3]{Ne17}). Now, given two von Neumann algebras \(\cM \subeq B(\cH)\) and \(\cN \subeq B(\cK)\), we define their tensor product by
\[\cM \overline{\otimes} \cN \coloneqq \iota(\cM \otimes \cN)''.\]
\end{itemize}

\newpage
\setcounter{page}{1}
\renewcommand\thepage{\arabic{page}}
\section{Introduction}
In algebraic quantum field theory (AQFT), objects of current interest are standard subspaces, i.e. real subspaces \(\sV \subeq \cH\) of a complex Hilbert space \(\cH\) that satisfy
\[\sV \cap i\sV = \{0\} \qquad \text{and} \qquad \oline{\sV + i\sV} = \cH\]
(cf. \cite{Ne21}, \cite{MN21}, \cite{NO21}, \cite{NOO21}). Given a standard subspace \(\sV \subeq \cH\), one has the densely defined Tomita operator
\[T_\sV : \sV+i\sV \to \cH, \quad x+iy \mapsto x-iy\]
as well as the anti-unitary involution \(J_\sV\) and the positive densely defined operator \(\Delta_\sV\) that one obtains by the polar decomposition
\[T_\sV = J_\sV \Delta_\sV^{\frac 12}\]
of the Tomita operator. This pair \((\Delta_\sV,J_\sV)\) is referred to as the pair of modular objects of \(\sV\). In the context of standard subspaces, one often studies unitary one-parameter groups on \(\cH\) whose semigroup in positive direction consists of unitary operators that map the standard subspace \(\sV\) into itself, i.e. unitary one-parameter groups \(U: \R \to \U(\cH)\) with
\[U_t \sV \subeq \sV \quad \forall t \in \R_+.\]
If the self-adjoint infinitesimal generator
\[\partial U \coloneqq \lim_{t \to 0} \frac{1}{it}(U_t-\textbf{1})\]
of such a unitary one-parameter group \(U\) additionally is a positive operator with trivial kernel, i.e.
\[\partial U \geq 0 \qquad \text{and} \qquad \ker (\partial U) = \{0\},\]
one calls \((\sV,U)\) a non-degenerate standard pair on \(\cH\). Two classical results about these non-degenerate standard pairs are given by Borchers' Theorem (cf. \fref{thm:Borchers}, \cite[Thm. II.9]{Bo92}, \cite[Thm.~2.2.1]{Lo08}) and the Longo--Witten Theorem (cf. \cite[Thm. 2.6]{LW11}). The theorem by Borchers states that, given a standard pair \((\sV,U)\) on a Hilbert space \(\cH\), one can, using \(U\) and the modular data \((\Delta_\sV,J_\sV)\), construct a so-called strictly positive energy representation (cf. \fref{def:PosEnergyRep}) of the affine group \(\mathrm{Aff}(\R)\) on the Hilbert space \(\cH\). By the representation theory of \(\mathrm{Aff}(\R)\), one then knows that there exists a real Hilbert space \(\cK\) and a unitary operator
\[\psi: \cH \to L^2(\R,\cK_\C)^\sharp\]
with
\[\psi(\sV) = H^2(\C_+,\cK_\C)^\sharp \qquad \text{and} \qquad \psi \circ U_t \circ \psi^{-1} = S_t \quad \forall t \in \R.\]
Here \(L^2(\R,\cK_\C)^\sharp\) denotes the Hilbert space of square-integrable \(\cK_\C\)-valued functions one \(\R\) that are invariant under the involution \(\sharp\) defined by
\[f^\sharp(x) = \overline{f(-x)}, \quad x \in \R\]
and by \(H^2(\C_+,\cK_\C)^\sharp\) we denote the corresponding Hardy space, i.e. the subspace of \(L^2(\R,\cK_\C)^\sharp\) whose Fourier transform vanishes on the negative half-line. Further, \(S_t\) denotes the unitary operator on \(L^2(\R,\cK_\C)^\sharp\) defined by
\[(S_t f)(x) \coloneqq e^{itx} f(x), \quad x \in \R.\]
Under the identification
\[\cH \cong L^2(\R,\cK_\C)^\sharp \qquad \text{and} \qquad \sV \cong H^2(\C_+,\cK_\C)^\sharp\]
the theorem by Longo and Witten then states that the semigroup
\[\cE(\sV,U) \coloneqq \{A \in \U(\cH): A\sV \subeq \sV, \,(\forall t \in \R)\, U_t A = A U_t\}\]
is precisely given by the semigroup \(\Inn(\C_+,B(\cK_\C))^\sharp\) of symmetric inner functions on the upper half-plane (cf. \fref{def:Inner}).

Since also unitary representations with non-positive infinitesimal generator are a subject of investigation in AQFT (cf. \cite{Ni22}), in this thesis, we want to provide theorems similar to the ones of Borchers and of Longo--Witten for a more general framework of unitary one-parameter groups \(U: \R \to \U(\cH)\) with
\[U_t \sV \subeq \sV \quad \forall t \in \R_+,\]
without the additional assumption that the infinitesimal generator of \(U\) is positive. We want to replace this assumption by the weaker assumption that -- as it was the case for non-degenerate standard pairs -- there exists an isomorphism between \((\cH,\sV,U)\) and \((L^2(\R,\cK_\C)^\sharp,H^2(\C_+,\cK_\C)^\sharp,S)\) for some real Hilbert space \(\cK\). By the Lax--Phillips Theorem (\fref{thm:LaxPhillipsReal},\cite[Thm. 1]{LP64}), this condition is equivalent to \((\cH,\sV,U)\) being a so-called real regular one-parameter group.

~

In Chapter 2, we will introduce regular one-parameter groups, i.e. triples \((\cE,\cE_+,U)\) of a complex Hilbert space \(\cE\), a closed subspace \(\cE_+ \subeq \cE\) and a unitary one-parameter group \(U: \R \to \U(\cE)\) that satisfy
\[U_t \cE_+ \subeq \cE_+ \quad \forall t \geq 0, \qquad \bigcap_{t \in \R} U_t \cE_+ = \{0\} \qquad \text{and} \qquad \overline{\bigcup_{t \in \R} U_t \cE_+} = \cE.\]
By the Lax--Phillips Theorem (\fref{thm:LaxPhillipsReal}, \fref{thm:LaxPhillipsComplex}) these regular one-parameter groups are always unitarily equivalent to \((L^2(\R,\cK_\C)^\sharp,H^2(\C_+,\cK_\C)^\sharp,S)\) for some real Hilbert space \(\cK\) if \(\cE\) is a real Hilbert space or to \((L^2(\R,\cK),H^2(\C_+,\cK),S)\) for some complex Hilbert space \(\cK\) if \(\cE\) is a complex Hilbert space. In this chapter we show that the so obtained multiplicity space \(\cK\) is unique up to isomorphism (cf. \fref{thm:kernelAdjointGeneral}, \fref{thm:kernelAdjointGeneralReal}) and that -- up to unitary equivalence -- the space \(\cE_+\) is unique, i.e. if \((\cE,\cE_+,U)\) and \((\cE,\cE_+',U)\) are both regular one-parameter groups, then
\[(\cE,\cE_+,U) \cong (\cE,\cE_+',U)\]
via some isomorphism of regular one-parameter groups (cf. \fref{thm:IndependentFromSubspace}).

~

In Chapter 3, we study the semigroup \(\Inn(\C_+,B(\cH))\) of inner functions appearing in the context of the Longo--Witten Theorem. One-parameter semigroups in \(\Inn(\C_+,B(\cH))\) are infinitesi\-mally generated by operator-valued Pick functions (cf. \fref{def:Pick}), i.e. they are of the form
\[U^F_t \coloneqq e^{itF}, \quad t \in \R,\]
for some Pick function \(F\) with values in \(B(\cH)\). For \(\dim \cH < \infty\) we give criteria on the function \(F\) that are equivalent to \((L^2(\R,\cH),H^2(\C_+,\cH),U^F)\) being a regular one-parameter group (cf. \fref{thm:SpecNonDeg}, \fref{thm:SpecNonDegZeroSet}). In these cases, we are also interested in calculating the multiplicity space \(\cK_F\) of the regular one-parameter group \((L^2(\R,\cH),H^2(\C_+,\cH),U^F)\). A formula for this multiplicity space \(\cK_F\) is provided in \fref{thm:BigLWTheorem}. An interesting open problem in this context is whether similar results also hold in the case of an infinite-dimensional Hilbert space \(\cH\).

Finally, we will prove that the dimension of the multiplicity space is compatible with composition of Pick functions: Given Pick functions \(f\) and \(g\) with values in \(\C\) and a Pick function \(F\) with values in \(B(\cH)\) such that \((L^2(\R,\C),H^2(\C_+),U^f)\), \((L^2(\R,\C),H^2(\C_+),U^g)\) and \((L^2(\R,\cH),H^2(\C_+,\cH),U^F)\) are regular one-parameter groups, the function \(f \circ F \circ g\) is also a Pick function such that the triple \((L^2(\R,\cH),H^2(\C_+,\cH),U^{f \circ F \circ g})\) is a regular one-parameter group and for its multiplicity space \(\cK_{f \circ F \circ g}\) one has
\[\dim \cK_{f \circ F \circ g} = \dim \cK_f \cdot \dim \cK_F \cdot \dim \cK_g\]
(cf. \fref{cor:CompNonDeg}, \fref{thm:InvariantHomo}).

~

In Chapter 4, we add another structure to regular one-parameter groups that corresponds to the involution \(J_\sV\) in the context of standard subspaces. Concretely, we consider so-called reflection positive regular one-parameter groups. These are quadruples \((\cE,\cE_+,U,\theta)\), where \((\cE,\cE_+,U)\) is a regular one-parameter group and \(\theta\) is an involution on \(\cE\) with
\[\theta U_t = U_{-t}\theta \quad \forall t \in \R\]
such that \((\cE,\cE_+,\theta)\) is a reflection positive Hilbert space, i.e.
\begin{equation*}
\braket*{\xi}{\theta\xi} \geq 0 \qquad \forall \xi \in \mathcal{E}_+
\end{equation*}
(cf. \fref{def:RPHS}). This will lead us to the question for which functions \(h \in L^\infty(\R,\U(\cK))^\flat\) the triple \((L^2(\R,\cK),H^2(\C_+,\cK),\theta_h)\) is a reflection positive Hilbert space, where \(L^\infty(\R,\U(\cK))^\flat\) denotes the set of essentially bounded measurable functions on \(\R\) with values in \(\U(\cK)\) that are invariant under the involution \(\flat\) defined by
\[f^\flat(x) \coloneqq f(-x)^*, \quad x \in \R\]
and by \(\theta_h\) we denote the involution on \(L^2(\R,\cK)\) defined by
\[(\theta_h f)(x) = h(x) \cdot f(-x), \quad x \in \R.\]
This problem has already been described in \cite[Prob. 4.4.5]{NO18} and is in full generality still open and should be subject of further investigation. However we were able to classify all such functions \(h\) under the two additional assumptions that \(\cK = \R\) and that \((L^2(\R,\R),H^2(\C_+),\theta_h)\) is a maximal reflection positive Hilbert space. The maximality here is to be understood in the way that, for every subspace \(\cE_+ \supeq H^2(\C_+)\) for which \((L^2(\R,\R),\cE_+,\theta_h)\) is a reflection positive Hilbert space, one has \(\cE_+ = H^2(\C_+)\) (cf. \fref{def:maxRPHS}). A full characterization of the functions \(h \in L^\infty(\R,\T)^\flat\) for which the triple \((L^2(\R,\R),H^2(\C_+),\theta_h)\) is a maximal reflection positive Hilbert space is given in \fref{thm:GeneralizedMaxPosForm}.

Finally, given a function \(h \in L^\infty(\R,\T)^\flat\) for which the triple \((L^2(\R,\R),H^2(\C_+),\theta_h)\) is a maximal reflection positive Hilbert space, in \fref{cor:OSTrafoFromMeasures} we give an explicit description of the Osterwalder--Schrader transform (cf. \cite[Def. 3.1.1]{NO18}, \fref{def:OSTrafo}) of this reflection positive Hilbert space. Here, given a reflection positive Hilbert space \((\cE,\cE_+,\theta)\) and defining
\[\cN_\theta \coloneqq \{\eta \in \cE_+: \braket*{\eta}{\theta \eta} = 0\},\]
the Osterwalder--Schrader transform \(\hat \cE\) of \(\cE_+\) is defined as the Hilbert space completion of \(\cE_+ \slash \cN_\theta\) with respect to the scalar product
\[\braket*{\eta + \cN_\theta}{\xi+\cN_\theta}_{\hat \cE} \coloneqq \braket*{\eta}{\theta \xi}, \quad \eta,\xi \in \cE_+.\]

~

In Chapter 5, we apply our results from Chapter 4 to Hankel operators, i.e. bounded linear operators \({H \in B(H^2(\C_+))}\) that satisfy
\[S_t^* H = H S_t \quad \forall t \in \R_+\]
(cf. \fref{def:Hankel}), where
\[(S_t f)(x) = e^{itx} f(x), \quad f \in H^2(\C_+), x \in \R, t \in \R_+.\]
 We will see that, given a function \(h \in L^\infty(\R,\T)^\flat\) for which the triple \(\left(L^2\left(\R,\C\right),H^2\left(\C_+\right),\theta_h\right)\) is a maximal reflection positive Hilbert space, denoting by \(p_+\) the partial isometry from \(L^2(\R,\C)\) onto its subspace \(H^2(\C_+)\), the operator
 \[H_h \coloneqq p_+ \theta_h p_+^*\]
 is a positive Hankel operator with \({\left\lVert H_h\right\rVert \leq 1}\). The question we investigate in Chapter 5, which has already been under investigation in \cite{ANS22}, is if every positive Hankel operator \(H\) with \(\left\lVert H\right\rVert \leq 1\) is of this form. In \cite[Sec. 4]{ANS22} it has already been shown that this is true if one allows to change the scalar product on \(H^2(\C_+)\). We prove that in fact, the statement is true without changing the scalar product (cf. \fref{thm:HankelHasUnimodMeasure}).

~

In Chapter 6 we apply the results of the previous chapters in the context of standard subspaces. Given a complex Hilbert space \(\cH\), a standard subspace \(\sV \subeq \cH\), and a strongly continuous one-parameter group \(U: \R \to \U(\cH)\), we consider regular pairs, i.e. pairs \((\sV,U)\), for which \((\cH,\sV,U)\) is a real regular one-parameter group and reflection positive regular pairs, i.e. pairs \((\sV,U)\), for which \((\cH,\sV,U,J_\sV)\) is a real reflection positive regular one-parameter group (cf. \fref{def:RegPair}). First, we prove that both of these conditions are fulfilled if \((\sV,U)\) is a non-degenerate standard pair, to show that we are, in fact, dealing with a generalization of non-degenerate standard pairs (cf. \fref{cor:standardPairIsRPRegPair}). We will also show that in both cases, the triple \((\cH,\sV,J_\sV)\) is a maximal reflection positive Hilbert space (RPHS) (cf. \fref{prop:StandardSubspaceMaxPos}).

By the Lax--Phillips Theorem we know that, for every regular pair \((\sV,U)\) on a Hilbert space \(\cH\), the regular one-parameter group \((\cH,\sV,U)\) is equivalent to \((L^2(\R,\cK_\C)^\sharp,H^2(\C_+,\cK_\C)^\sharp,S)\) for some real Hilbert space \(\cK\). Here the multiplicity free case \(\cK = \R\) corresponds to the representation \(U\) being \(\R\)-cyclic, i.e.
\[\overline{\spann_\R U(\R)v} = \cH\]
(cf. \fref{prop:MultiplicityRCyclic}).

In \fref{thm:LWExtended}, we show that, for all regular pairs \((\sV,U)\), identifying
\[\cH \cong \cH_\cK \coloneqq L^2(\R,\cK_\C)^\sharp \qquad \text{and} \quad \sV \cong \sV_\cK \coloneqq H^2(\C_+,\cK_\C)^\sharp,\]
the semigroup
\[\cE(\sV,U) \coloneqq \{A \in \U(\cH): A\sV \subeq \sV, \,(\forall t \in \R)\, U_t A = A U_t\}\]
is given by a subsemigroup of the semigroup \(\Inn(\C_+,B(\cK_\C))^\sharp\) of symmetric inner functions and thereby provide a generalization of the Longo--Witten Theorem (cf. \cite[Thm. 2.6]{LW11}). In this theorem, we also provide some conditions under which semigroups in \(\cE(\sV,U)\) extend to regular one-parameter groups by using the results from Chapter 3.
\newpage
Many of our main results are summarized in the following table:
\vspace{0.2cm}
\begin{center}
\addtolength\tabcolsep{1.3pt}
\begin{tabular}{|c||c|c|c|} 
\hline
\raisebox{30pt}{\phantom{M}}\boldmath\((\sV,U)\)\raisebox{-30pt}{\phantom{M}} & \bf regular pair & \(\substack{\text{\normalsize \bf reflection positive} \\ \text{\normalsize \bf regular pair}}\) & \(\substack{\text{\normalsize \bf non-degenerate} \\ \text{\normalsize \bf standard pair}}\) \\
\hline\hline
\raisebox{30pt}{\phantom{M}} \bf Group \boldmath\(G\)\raisebox{-30pt}{\phantom{M}} & \(\R\) & \(\R \rtimes \{-1,1\}\) & \(\mathrm{Aff}(\R) = \R \rtimes \R^\times\) \\
\hline
\raisebox{30pt}{\phantom{M}}\(\substack{\text{\normalsize \bf Representation} \\ \text{\normalsize \boldmath\(\rho: G \to \mathrm{AU}(\cH)\)}}\)\raisebox{-30pt}{\phantom{M}}  & \(\rho(t) = U_t\) & \(\substack{\text{\normalsize \(\rho(t,1) = U_t\)} \\ \text{\normalsize \(\rho(0,-1) = J_\sV\)} \\ \text{\normalsize (\fref{def:RPHS})}}\) & \(\substack{\text{\normalsize \(\rho(t,1) = U_t\)} \\ \text{\normalsize \(\rho(0,e^s) = \Delta^{-\frac{is}{2\pi}}\)} \\ \text{\normalsize \(\rho(0,-1) = J_\sV\)} \\ \text{\normalsize (\fref{thm:Borchers})}}\) \\
\hline
\raisebox{30pt}{\phantom{M}}\(\substack{\text{\normalsize \bf Normal form} \\ \text{\normalsize \bf for \boldmath\((\cH,\sV,U,J_\sV)\)}}\) \raisebox{-30pt}{\phantom{M}} & \(\substack{\text{\normalsize \((\cH_\cK,\sV_\cK,S,\theta)\) with} \\ \text{\normalsize \((\cH_\cK,\sV_\cK,\theta)\) being a} \\ \text{\normalsize maximal RPHS} \\ \text{\normalsize (\fref{thm:StandardSubspaceROPG})}}\) & \(\substack{\text{\normalsize \((\cH_\cK,\sV_\cK,S,\theta_h)\) with} \\ \text{\normalsize \(h = h^\sharp = h^\flat\) and}\\ \text{\normalsize \((\cH_\cK,\sV_\cK,\theta_h)\) being a} \\ \text{\normalsize maximal RPHS} \\ \text{\normalsize (\fref{thm:StandardSubspaceRPROPG})}}\) & \(\substack{\text{\normalsize \((\cH_\cK,\sV_\cK,S,\theta_{i \cdot \sgn})\)} \\ \text{\normalsize (\fref{thm:StandardSubspaceSP})}}\) \\
\hline
\raisebox{30pt}{\phantom{M}}\(\substack{\text{\normalsize \bf Normal form} \\ \text{\normalsize \bf for \boldmath\((\cH,\sV,U,J_\sV)\)}\\ \text{\normalsize \bf if \boldmath\(U\) is \boldmath\(\R\)-cyclic}}\) \raisebox{-30pt}{\phantom{M}} & \(\substack{\text{\normalsize \((\cH_\R,\sV_\R,S,\theta)\) with} \\ \text{\normalsize \((\cH_\R,\sV_\R,\theta)\) being a} \\ \text{\normalsize maximal RPHS} \\ \text{\normalsize (\fref{thm:StandardSubspaceROPG})}}\) & \(\substack{\text{\normalsize \((\cH_\R,\sV_\R,S,\theta_{h_\nu})\) with} \\ \text{\normalsize \(\nu \in \mathcal{BM}^{\mathrm{fin}}([0,\infty]) \setminus \{0\}\)} \\ \text{\normalsize (\fref{thm:StandardSubspaceRPROPG})}}\) & \(\substack{\text{\normalsize \((\cH_\R,\sV_\R,S,\theta_{i \cdot \sgn}) =\)} \\ \text{\normalsize \((\cH_\R,\sV_\R,S,\theta_\nu)\) with} \\ \text{\normalsize \(d\nu(\lambda) = \frac 2{1+\lambda^2}\),} \\ \text{\normalsize (\fref{thm:StandardSubspaceSP})}}\) \\
\hline
\raisebox{30pt}{\phantom{M}}\(\substack{\text{\normalsize \bf OS transform} \\ \text{\normalsize \bf of \boldmath\((\cH,\sV,J_\sV)\)}\\ \text{\normalsize \bf if \boldmath\(U\) is \boldmath\(\R\)-cyclic}}\)\raisebox{-30pt}{\phantom{M}} & \(\substack{\text{\normalsize\(\overline{\sV_\R \slash \cN_\theta}\)} \\ \text{\normalsize (\fref{def:OSTrafo})}}\) & \(\substack{\text{\normalsize\(L^2(\R_+,\C,\cT \nu)\)} \\ \text{\normalsize (\fref{cor:OSTrafoFromMeasures})}}\) & \(\substack{\text{\normalsize\(L^2(\R_+,\C,2\lambda_1)\)} \\ \text{\normalsize (\fref{cor:OSTrafoFromMeasures},} \\ \text{\normalsize \fref{ex:Lebesgue})}}\) \\
\hline
\end{tabular}
\end{center}

\newpage
\section{Regular one-parameter groups}
In this chapter, we will introduce the concept of a regular one-parameter group (ROPG). Using the Lax--Phillips Theorem, we will introduce a standard form for these ROPGs and later, we will provide some basic theory about ROPGs.

\subsection{ROPGs and the Lax--Phillips Theorem}
\begin{definition}
Consider a triple \((\cE,\cE_+,U)\) of a complex Hilbert space \(\cE\), a closed subspace \(\cE_+ \subeq \cE\) and a unitary one-parameter group \(U: \R \to \U(\cE)\).
\begin{enumerate}[\rm (a)]
\item We call \((\cE,\cE_+,U)\) a \textit{semi-regular one-parameter group (\gls*{SROPG})}, if
\[U_t \cE_+ \subeq \cE_+ \quad \forall t \geq 0 \qquad \text{and} \qquad \bigcap_{t \in \R} U_t \cE_+ = \{0\}.\]
\item We define the unitary one-parameter group \(\check U: \R \to \U(\cE)\) by \(\check U_t = U_{-t}\), \(t \in \R\).
\item We call \((\cE,\cE_+,U)\) a \textit{regular one-parameter group (\gls*{ROPG})}, if \((\cE,\cE_+,U)\) and \((\cE,\cE_+^\perp,\check U)\) are both a SROPG.
\item We call a (S)ROPG \((\cE,\cE_+,U)\) real/complex, if \(\cE\) is real/complex.
\item We say that two ROPGs \((\cE,\cE_+,U)\) and \((\cE',\cE_+',U')\) are equivalent, if there exists a unitary operator \(\psi: \cE \to \cE'\) such that
\[\psi(\cE_+) = \cE_+' \quad \text{and} \quad \psi \circ U_t = U_t' \circ \psi \quad \forall t \in \R.\]
We call such an operator \(\psi\) an \textit{isomorphism between \((\cE,\cE_+,U)\) and \((\cE',\cE_+',U')\)}.
\end{enumerate}
\end{definition}
\begin{remark}\label{rem:ROPGAlternative}
For \(t \geq 0\) one has \(U_t \cE_+ \subeq \cE_+\), if and only if
\[\cE_+^\perp \subeq (U_t \cE_+)^\perp = U_t \cE_+^\perp = \check U_{-t} \cE_+^\perp,\]
which is equivalent to \(\check U_t \cE_+^\perp \subeq \cE_+^\perp\). Further
\[\left(\bigcap_{t \in \R} \check U_t \cE_+^\perp\right)^\perp = \overline{\bigcup_{t \in \R} \left(\check U_t \cE_+^\perp\right)^\perp} = \overline{\bigcup_{t \in \R} \check U_t \cE_+} = \overline{\bigcup_{t \in \R} U_t \cE_+}.\]
Therefore \((\cE,\cE_+,U)\) is a ROPG, if and only if
\[U_t \cE_+ \subeq \cE_+ \quad \forall t \geq 0\]
and
\[\bigcap_{t \in \R} U_t \cE_+ = \{0\} \qquad \text{and} \qquad \overline{\bigcup_{t \in \R} U_t \cE_+} = \cE.\]
\end{remark}
The following theorem gives us a normal form for ROPGs:
\begin{thm}\label{thm:LaxPhillips}{\rm\textbf{(Lax--Phillips)}(cf. \cite[Thm. 1]{LP64}\cite[Thm. 4.4.1]{NO18}, \cite[Thm. 3.3]{Fr20})}
For every real/com\-plex Hilbert space \(\cK\), the triple \((L^2(\R,\cK),L^2(\R_+,\cK),\tau)\) is a real/complex ROPG, where the one-parameter group \(\gls*{tau}: \R \to U(L^2(\R,\cK))\) is the translation action defined by
\[(\tau_tf)(p) = f(p-t), \qquad t,p \in \R.\]
Further, every real/complex ROPG \((\cE,\cE_+,U)\) is equivalent to \((L^2(\R,\cK),L^2(\R_+,\cK),\tau)\) for some real/complex Hilbert space \(\cK\).
\end{thm}
\begin{remark}
In the notation used in the theorem we identify \(L^2(\R_+,\cK)\) with the subspace
\[\left\{f \in L^2(\R,\cK): f\big|_{(-\infty,0]}\equiv 0\right\} \subeq L^2(\R,\cK).\]
\end{remark}
The version of the Lax--Phillips Theorem that we will use in the following is slightly different, since we will change to the Fourier transformed picture. For this, we need some definitions:
\begin{definition}
Let \(\cK\) be a real Hilbert space and \(\cC_\cK\) be the complex conjugation on \(\cK_\C\). 
\begin{enumerate}[\rm (a)]
\item In the following we identify \(L^2(\R,\cK)\) with the subspace
\[\left\{f \in L^2(\R,\cK_\C): \left(\forall x \in \R\right) \, f(x) \in \cK_\C^{\cC_\cK} \right\} \subeq L^2(\R,\cK_\C).\]
\item By \(\sharp: L^2(\R,\cK_\C) \to L^2(\R,\cK_\C)\), we denote the involution defined by
\[\gls*{fsharp}(x) = \cC_\cK f(-x).\]
\item For a complex Hilbert space \(\cH\) we consider the Fourier transform \(\gls*{fourier}: L^2(\R,\cH) \to L^2(\R,\cH)\) given by
\[(\cF f)(p) \coloneqq \frac 1{\sqrt{2\pi}} \int_\R e^{-ipx}f(x) \,dx\]
for every \(f \in L^1(\R,\C) \cap L^2(\R,\C)\).
\end{enumerate}
\end{definition}
To change to the Fourier-transformed picture, we first notice that
\[(\cF^{-1} \tau_t f)(x) = \frac 1{\sqrt{2\pi}} \int_\R e^{ipx}f(p-t) \,dp = \frac 1{\sqrt{2\pi}} \int_\R e^{i(p+t)x}f(x) \,dp = e^{itx} (\cF^{-1} f)(x).\]
Further, for a real Hilbert space \(\cK\), we have
\[(\cF^{-1} \cC_\cK f)(x) = \frac 1{\sqrt{2\pi}} \int_\R e^{ipx}\cC_\cK f(p) \,dp = \cC_\cK \frac 1{\sqrt{2\pi}} \int_\R e^{-ipx}f(p) \,dp = (\cF^{-1} f)^\sharp(x),\]
so
\[\cF^{-1}L^2(\R,\cK) = L^2(\R,\cK_\C)^\sharp.\]
Further, for a complex Hilbert space \(\cH\), by \cite[Thm. 5.28]{RR94}, we have 
\[\cF^{-1}L^2(\R_+,\cH) = H^2(\C_+,\cH),\]
where \(H^2(\C_+,\cH) \subeq L^2(\R,\cH)\) denotes the \(\cH\)-valued Hardy space (cf. \fref{sec:HardySpaces}).

With this information, one immediately gets a version of the Lax--Phillips Theorem in the Fourier transformed picture:
\begin{thm}\label{thm:LaxPhillipsComplex}{\rm\textbf{(Lax--Phillips, complex Fourier version)}}
For every complex Hilbert space \(\cK\) the triple \((L^2(\R,\cK),H^2(\C_+,\cK),S)\) with the one-parameter group
\[\gls*{S}: \R \to U(L^2(\R,\cK)), \quad (S_tf)(x) = e^{itx}f(x)\]
is a complex ROPG and every complex ROPG \((\cE,\cE_+,U)\) is equivalent to \((L^2(\R,\cK),H^2(\C_+,\cK),S)\) for some complex Hilbert space \(\cK\).
\end{thm}
\begin{thm}\label{thm:LaxPhillipsReal}{\rm\textbf{(Lax--Phillips, real Fourier version)}}
For every real Hilbert space \(\cK\) the triple \((L^2(\R,\cK_\C)^\sharp,H^2(\C_+,\cK_\C)^\sharp,S)\) with the one-parameter group
\[S: \R \to U(L^2(\R,\cK_\C)^\sharp), \quad (S_tf)(x) = e^{itx}f(x)\]
is a real ROPG and every real ROPG \((\cE,\cE_+,U)\) is equivalent to \((L^2(\R,\cK_\C)^\sharp,H^2(\C_+,\cK_\C)^\sharp,S)\) for some real Hilbert space \(\cK\).
\end{thm}
The one-parameter group \(S\) appearing in the last two theorems will appear often in the following sections. Many times, an object of interest will be its commutant \(S(\R)'\). To compute this commutant, we need the following theorem:
\begin{thm}\label{thm:CommutantTensor}{\rm (cf. \cite[Thm. 1]{RD75})}
Let \(\cH,\cK\) be Hilbert spaces and let \(M \subeq B(\cH)\) and \(N \subeq B(\cK)\) be von Neumann algebras. Then
\[(M \,\overline{\otimes}\, N)' = M' \,\overline{\otimes}\, N'.\]
\end{thm}
\begin{prop}\label{prop:SCommutant}
Let \(\cK\) be a complex Hilbert space and \(S: \R \to U(L^2(\R,\cK))\) be the one-parameter group defined by
\[(S_tf)(x) = e^{itx}f(x).\]
Then
\[S(\R)' = M\left(L^\infty(\R,B(\cK))\right).\]
\end{prop}
\begin{proof}
Considering the functions \(e_t \in L^\infty(\R,\C)\) with \(e_t(x) = e^{itx}\) and identifying
\[L^2(\R,\cK) \cong L^2(\R,\C) \,\hat{\otimes}\, \cK,\]
we have that
\[S_t = M_{e_t} \otimes \textbf{1}_\cK, \quad t \in \R.\]
Then, by \fref{prop:etDense}, the set \(S(\R)\) spans a weakly dense subspace in \(M(L^\infty(\R,\C)) \,\overline{\otimes}\, \C\textbf{1}_\cK\). This, by \fref{thm:CommutantTensor}, yields
\begin{align*}
S(\R)' &= \left(M(L^\infty(\R,\C))  \,\overline{\otimes}\, \C\textbf{1}_\cK\right)' = M(L^\infty(\R,\C))'  \,\overline{\otimes}\, (\C\textbf{1}_\cK)' = M(L^\infty(\R,\C))  \,\overline{\otimes}\, B(\cK),
\end{align*}
using \cite[Cor. 2.9.3]{Sa71} in the last step. The statement now follows since, by \cite[Thm.~1.22.13]{Sa71}, we have
\[M(L^\infty(\R,\C)) \,\overline{\otimes}\, B(\cK) = M\left(L^\infty(\R,B(\cK))\right). \qedhere\]
\end{proof}

\subsection{Formulas for the multiplicity space}
In this section, given a ROPG \((\cE,\cE_+,U)\), we will provide some ways to calculate the space \(\cK\) appearing in the Lax--Phillips Theorem, thereby showing that this space \(\cK\) is unique up to isomorphism.
\begin{theorem}\label{thm:SchurDimComplex}
Let \((\cE,\cE_+,U)\) be a complex ROPG and \(\cK\) a complex Hilbert space such that \((\cE,\cE_+,U)\) is equivalent to \((L^2(\R,\cK),H^2(\C_+,\cK),S)\). Then
\[B(\cK) \cong (U(\R) \cup \{P_{\cE_+}\})'.\]
\end{theorem}
\begin{proof}
Without loss of generality we can assume that \((\cE,\cE_+,U) = (L^2(\R,\cK),H^2(\C_+,\cK),S)\). Let \(A \in (S(\R) \cup \{P_{H^2(\C_+,\cK)}\})'\). Then \(A \in S(\R)'\), by \fref{prop:SCommutant}, implies that \(A=M_f\) with \(f \in L^\infty(\R,B(\cK))\). Further, \({A \in \{P_{H^2(\C_+,\cK)}\}'}\) implies
\[M_f H^2(\C_+,\cK) \subeq H^2(\C_+,\cK) \quad \text{and} \quad M_f H^2(\C_-,\cK) \subeq H^2(\C_-,\cK),\]
which, by \fref{prop:H2InclusionFunctions}, yields
\[f \in H^\infty(\C_+,B(\cK)) \quad \text{and} \quad f \in H^\infty(\C_-,B(\cK)),\]
so \(f = C \cdot \textbf{1}\) for some \(C \in B(\cK)\) by \fref{cor:HInftyIntersection}.

Conversely, for every \(C \in B(\cK)\) one has \(M_{C \cdot \textbf{1}} \in (S(\R) \cup \{P_{H^2(\C_+,\cK)}\})'\), so
\[(S(\R) \cup \{P_{H^2(\C_+,\cK)}\})' = B(\cK) \cdot M_\textbf{1}. \qedhere\]
\end{proof}
\begin{theorem}\label{thm:SchurDimReal}
Let \((\cE,\cE_+,U)\) be a real ROPG and \(\cK\) a real Hilbert space such that \((\cE,\cE_+,U)\) is equivalent to \((L^2(\R,\cK_\C)^\sharp,H^2(\C_+,\cK_\C)^\sharp,S)\). Then
\[B(\cK)^2 \cong (U(\R) \cup \{P_{\cE_+}\})'.\]
\end{theorem}
\begin{proof}
Without loss of generality we can assume that \((\cE,\cE_+,U) = (L^2(\R,\cK_\C)^\sharp,H^2(\C_+,\cK_\C)^\sharp,S)\). Let \(A \in (S(\R) \cup \{P_{H^2(\C_+,\cK_\C)^\sharp}\})'\). Then \(A \in S(\R)'\), by \fref{prop:SCommutant}, implies that \(A=M_f\) with \({f \in L^\infty(\R,B(\cK_\C))}\). Further, \({A \in \{P_{H^2(\C_+,\cK_\C)^\sharp}\}'}\) implies
\[M_f H^2(\C_+,\cK_\C)^\sharp \subeq H^2(\C_+,\cK_\C)^\sharp \quad \text{and} \quad M_f H^2(\C_-,\cK_\C)^\sharp \subeq H^2(\C_-,\cK_\C)^\sharp.\]
Since \(H^2(\C_+,\cK_\C) = H^2(\C_+,\cK_\C)^\sharp + i H^2(\C_+,\cK_\C)^\sharp\), this implies
\[M_f H^2(\C_+,\cK_\C) \subeq H^2(\C_+,\cK_\C) \quad \text{and} \quad M_f H^2(\C_-,\cK_\C) \subeq H^2(\C_-,\cK_\C),\]
which, by \fref{prop:H2InclusionFunctions} and \fref{cor:HInftyIntersection}, yields \(f = C \cdot \textbf{1}\) for some \(C \in B(\cK_\C)\). Further, \(AL^2(\R,\cK_\C)^\sharp \subeq L^2(\R,\cK_\C)^\sharp\) implies that \([\cC_\cK,C] = 0\), so \(C \in \{\cC_\cK\}'\). Conversely, for every \(C \in \{\cC_\cK\}'\) one has \(M_{C \cdot \textbf{1}} \in (S(\R) \cup \{P_{H^2(\C_+,\cK_\C)^\sharp}\})'\), so
\[(S(\R) \cup \{P_{H^2(\C_+,\cK_\C)^\sharp}\})' = \{\cC_\cK\}' \cdot M_\textbf{1}.\]
The statement then follows by the fact that
\[\Phi: B(\cK)^2 \to \{\cC_\cK\}', \quad \Phi(A,B)(x+iy) = Ax+iBy, \quad A,B \in B(\cK), \quad x,y \in \cK\]
defines an isomorphism.
\end{proof}
\begin{lem}\label{lem:QSpaceReal}
Let \(\cK\) be a real Hilbert space and \(\lambda \in \C_+\). Then
\[\left(Q_\lambda \cdot \cK_\C\right) \cap H^2(\C_+,\cK_\C)^\sharp = \begin{cases} Q_\lambda \cdot \cK & \text{if } \lambda \in i\R_+ \\ \{0\} & \text{if } \lambda \in \C_+ \setminus i\R_+, \end{cases}\]
with \(Q_\lambda \in H^2(\C_+)\) defined by
\[Q_\lambda(z) \coloneqq \frac 1{2\pi i} \frac 1{z-\overline{\lambda}}\]
(cf. \fref{prop:H2RKHS}).
\end{lem}
\begin{proof}
For \(v \in \cK_\C\), one has
\begin{align*}
(Q_\lambda \cdot v)^\sharp(z) &= \cC_\cK(Q_\lambda(-\overline{z}) \cdot v) = \overline{Q_\lambda(-\overline{z})} \cdot \cC_\cK v
\\&= \overline{\frac 1{2\pi i} \frac 1{-\overline{z}-\overline{\lambda}}} \cdot \cC_\cK v = \frac 1{2\pi i} \frac 1{z+\lambda} \cdot \cC_\cK v = Q_{-\overline{\lambda}}(z) \cdot \cC_\cK v.
\end{align*}
Therefore \((Q_\lambda \cdot v)^\sharp = Q_\lambda \cdot v\), if and only if
\[Q_{-\overline{\lambda}}(z) \cdot \cC_\cK v = Q_{\lambda}(z) \cdot v \quad \forall z \in \C_+.\]
For \(v \neq 0\) this implies that the map
\[\C_+ \ni z \mapsto \frac{Q_{\lambda}(z)}{Q_{-\overline{\lambda}}(z)} = \frac{z+\lambda}{z-\overline{\lambda}}\]
is constant, which yields \(\lambda = -\overline{\lambda}\) and therefore \(\lambda \in i\R_+\). For \(\lambda \in i\R_+\) we then have \((Q_\lambda \cdot v)^\sharp = Q_\lambda \cdot v\), if and only if \(\cC_\cK v = v\). We conclude that \(Q_\lambda \cdot v \in H^2(\C_+,\cK_\C)^\sharp\) is equivalent to \(\lambda \in i\R_+\) and \(v \in \cK_\C^{\cC_\cK} \cong \cK\).
\end{proof}
\begin{lem}\label{lem:kernelAdjoint}
Let \(\cK\) be a Hilbert space and set
\[\cE_+ \coloneqq \begin{cases} H^2(\C_+,\cK) & \text{if } \cK \text{ complex} \\ H^2(\C_+,\cK_\C)^\sharp & \text{if } \cK \text{ real}. \end{cases}\]
Further, let
\[A \coloneqq p_{\cE_+} (\partial S) p_{\cE_+}^* \in B(\cE_+).\]
Then, if \(\cK\) is a complex Hilbert space, for every \(\lambda \in \C_+\), one has
\[\ker\left(\overline{\lambda}\textbf{1}-A^*\right) = Q_\lambda \cdot \cK,\]
 and if \(\cK\) is a real Hilbert space, one has
\[\ker\left(\overline{\lambda}\textbf{1}-A^*\right) = \begin{cases} Q_\lambda \cdot \cK & \text{if } \lambda \in i\R_+ \\ \{0\} & \text{if } \lambda \in \C_+ \setminus i\R_+. \end{cases}\]
\end{lem}
\begin{proof}
We first consider the case that \(\cK\) is a complex Hilbert space. Since
\[\partial S = \lim_{t \to 0} \frac{S_t - \textbf{1}}{it} = \lim_{t \to 0} \frac{M_{e^{it \Id}} - \textbf{1}}{it} = M_{\Id},\]
for \(f \in D(A)\), we have \((Af)(z) = z f(z)\). Then, for \(\lambda \in \C_+\), \(v \in \cK\) and \(f \in D(A)\), by \fref{prop:H2RKHS}, we get
\[\braket*{Q_\lambda \cdot v}{Af} = \braket*{v}{\lambda f(\lambda)}_\cK = \lambda \braket*{v}{f(\lambda)}_\cK = \lambda \braket*{Q_\lambda \cdot v}{f} = \braket*{\overline{\lambda} Q_\lambda \cdot v}{f},\]
so \(Q_\lambda \cdot v \in D(A^*)\) with \(A^*Q_\lambda \cdot v = \overline{\lambda}Q_\lambda \cdot v\) and therefore \(Q_\lambda \cdot v \in \ker\left(\overline{\lambda}\textbf{1}-A^*\right)\).

Now, conversely, let \(f \in \ker\left(\overline{\lambda}\textbf{1}-A^*\right)\). For every \(v \in \cK\) and \(a,b \in \C_+\), we have
\[(Q_a \cdot v-Q_b \cdot v)(z) = \frac{1}{2 \pi i} \left(\frac{1}{z-\overline{a}}-\frac{1}{z-\overline{b}}\right) \cdot v = \frac{1}{2 \pi i} \frac{\overline{a}-\overline{b}}{(z-\overline{a})(z-\overline{b})} \cdot v,\]
so \(Q_a \cdot v-Q_b \cdot v \in D(A)\) with
\begin{align*}
(A(Q_a \cdot v-Q_b \cdot v))(z) &= z \cdot \frac{1}{2 \pi i} \frac{\overline{a}-\overline{b}}{(z-\overline{a})(z-\overline{b})} \cdot v 
\\&= \frac{1}{2 \pi i} \left(\frac{\overline{a}}{z-\overline{a}}-\frac{\overline{b}}{z-\overline{b}}\right) \cdot v = (\overline{a}Q_a \cdot v-\overline{b}Q_b \cdot v)(z).
\end{align*}
This leads to
\begin{align*}
\braket*{v}{\overline{\lambda}f(a)-\overline{\lambda}f(b)}_\cK &= \braket*{Q_a \cdot v-Q_b \cdot v}{\overline{\lambda}f} = \braket*{Q_a \cdot v-Q_b \cdot v}{A^*f} 
\\&= \braket*{A(Q_a \cdot v-Q_b \cdot v)}{f} = \braket*{\overline{a}Q_a \cdot v-\overline{b}Q_b \cdot v}{f}
\\&= a\braket*{v}{f(a)}_\cK-b \braket*{v}{f(b)}_\cK = \braket*{v}{a f(a)-b f(b)}_\cK.
\end{align*}
Since this is true for every \(v \in \cK\), we get
\[(\overline{\lambda}-a)f(a) = (\overline{\lambda}-b)f(b)\]
and therefore, since \(Q_\lambda(b) = \frac 1{2\pi i} \frac 1{b-\overline{\lambda}}\), we have
\[f(b) = \frac{a-\overline{\lambda}}{b-\overline{\lambda}}f(a) = Q_\lambda(b) \cdot 2\pi i(a - \overline{\lambda})f(a) = (Q_\lambda \cdot v)(b)\]
with \(v \coloneqq 2\pi i(a - \overline{\lambda})f(a) \in \cK\). This shows that
\[\ker\left(\overline{\lambda}\textbf{1}-A^*\right) = Q_\lambda \cdot \cK.\]
Now, in the case of a real Hilbert space \(\cK\), analogously to the complex case, one shows that \(Q_\lambda \cdot v \in \ker\left(\overline{\lambda}\textbf{1}-A^*\right)\) if \(Q_\lambda \cdot v \in H^2(\C_+,\cK_\C)^\sharp\), so if \(\lambda \in i\R_+\) and \(v \in \cK\) (cf. \fref{lem:QSpaceReal}). Also, as in the complex case, one shows that, for every \(f \in \ker\left(\overline{\lambda}\textbf{1}-A^*\right)\), one has
\[\braket*{v}{\overline{\lambda}f(a)-\overline{\lambda}f(b)}_{\cK_\C} = \braket*{v}{a f(a)-b f(b)}_{\cK_\C}\]
for every \(a,b \in i\R_+\) and \(v \in \cK\). But since \(\cK = \C \otimes_\R \cK\), this also implies
\[\braket*{v}{\overline{\lambda}f(a)-\overline{\lambda}f(b)}_{\cK_\C} = \braket*{v}{a f(a)-b f(b)}_{\cK_\C}\]
for every \(v \in \cK_\C\) and therefore, as in the complex case, one concludes that
\[f(b) = (Q_\lambda \cdot v)(b)\]
with \(v \coloneqq 2\pi i(a - \overline{\lambda})f(a)\) for every \(b \in i\R_+\). Since both functions \(f\) and \(Q_\lambda \cdot v\) are holomorphic and coincide on \(i\R_+\), they are equal, so \(f \in Q_\lambda \cdot \cK_\C\). Using \fref{lem:QSpaceReal}, this shows that
\[\ker\left(\overline{\lambda}\textbf{1}-A^*\right) = (Q_\lambda \cdot \cK_\C) \cap H^2(\C_+,\cK_\C)^\sharp = \begin{cases} Q_\lambda \cdot \cK & \text{if } \lambda \in i\R_+ \\ \{0\} & \text{if } \lambda \in \C_+ \setminus i\R_+. \end{cases} \qedhere\]
\end{proof}
\begin{thm}\label{thm:kernelAdjointGeneral}
Let \((\cE,\cE_+,U)\) be a complex ROPG and \(\cK\) a complex Hilbert space such that \((\cE,\cE_+,U)\) is equivalent to \((L^2(\R,\cK),H^2(\C_+,\cK),S)\). Further, let
\[A \coloneqq p_{\cE_+} (\partial U) p_{\cE_+}^* \in B(\cE_+).\]
Then, for every \(\lambda \in \C_+\), one has
\[\cK \cong \ker\left(\overline{\lambda}\textbf{1}-A^*\right).\]
\end{thm}
\begin{proof}
This follows immediately from \fref{lem:kernelAdjoint}.
\end{proof}
\begin{thm}\label{thm:kernelAdjointGeneralReal}
Let \((\cE,\cE_+,U)\) be a real ROPG and \(\cK\) a real Hilbert space such that \((\cE,\cE_+,U)\) is equivalent to \((L^2(\R,\cK_\C)^\sharp,H^2(\C_+,\cK_\C)^\sharp,S)\). Further, let
\[A \coloneqq p_{\cE_+} (\partial U) p_{\cE_+}^* \in B(\cE_+).\]
Then, for every \(\lambda \in i\R_+\), one has
\[\cK \cong \ker\left(\overline{\lambda}\textbf{1}-A^*\right).\]
\end{thm}
\begin{proof}
This follows immediately from \fref{lem:kernelAdjoint}.
\end{proof}
Since the kernel \(\ker\left(\overline{\lambda}\textbf{1}-A^*\right)\) appearing in the last two theorems just depends on \((\cE,\cE_+,U)\) and not on the choice of \(\cK\), this theorem shows that \(\cK\) is unique up to isomorphism. We therefore can formulate the following definition:
\begin{definition}
For every complex/real ROPG \((\cE,\cE_+,U)\), by the Lax--Phillips Theorem (\fref{thm:LaxPhillipsComplex}/\fref{thm:LaxPhillipsReal}) and \fref{thm:kernelAdjointGeneral}/\fref{thm:kernelAdjointGeneralReal}, there exists an --- up to isomorphism --- unique complex/real Hilbert space \(\cK\) such that the ROPG \((\cE,\cE_+,U)\) is equivalent to \((L^2(\R,\cK),H^2(\C_+,\cK),S)\) or \((L^2(\R,\cK_\C)^\sharp,H^2(\C_+,\cK_\C)^\sharp,S)\) respectively. We call this Hilbert space \(\cK\) the \textit{multiplicity space of \((\cE,\cE_+,U)\)}.
\end{definition}
In order to not having to prove everything separately for real and complex ROPGs, we now show how they are linked to each other:
\begin{prop}\label{prop:ROPGRealToComplex}
Let \(\cE\) be a real Hilbert space, \(\cE_+ \subeq \cE\) be a closed subspace and \(U: \R \to \U(\cE)\) be a unitary one-parameter group. Then \((\cE,\cE_+,U)\) is a real ROPG, if and only if \((\cE_\C,(\cE_+)_\C,U_\C)\) is a complex ROPG. In this case, if the multiplicity space of \((\cE,\cE_+,U)\) is \(\cK\), then the multiplicity space of \((\cE_\C,(\cE_+)_\C,U_\C)\) is given by \(\cK_\C\).
\end{prop}
\begin{proof}
For \(t \in \R\), we have
\[(U_t)_\C (\cE_+)_\C = (U_t)_\C (\cE_+ + i \cE_+) = U_t \cE_+ + i U_t\cE_+,\]
which immediately implies that \((U_t)_\C (\cE_+)_\C \subeq (\cE_+)_\C\), if and only if \(U_t \cE_+ \subeq \cE_+\). Further, we get
\[\bigcap_{t \in \R} (U_t)_\C (\cE_+)_\C = \left(\bigcap_{t \in \R} U_t \cE_+\right) + i \left(\bigcap_{t \in \R} U_t \cE_+\right)\]
and
\[\bigcup_{t \in \R} (U_t)_\C (\cE_+)_\C = \left(\bigcup_{t \in \R} U_t \cE_+\right) + i \left(\bigcup_{t \in \R} U_t \cE_+\right).\]
This proves that \((\cE,\cE_+,U)\) is a real ROPG, if and only if \((\cE_\C,(\cE_+)_\C,U_\C)\) is a complex ROPG.

Now let \((\cE,\cE_+,U)\) be a real ROPG with multiplicity space \(\cK\). Then there exists a isomorphism \(\psi\) of the real ROPG \((\cE,\cE_+,U)\) and \((L^2(\R,\cK_\C)^\sharp,H^2(\C_+,\cK_\C)^\sharp,S)\). Then, the map \(\psi_\C\) is an isomorphism of complex ROPG. This, together with the fact that
\[L^2(\R,\cK_\C) = L^2(\R,\cK_\C)^\sharp + i L^2(\R,\cK_\C)^\sharp \cong L^2(\R,\cK_\C)^\sharp_\C\]
and
\[H^2(\C_+,\cK_\C) = H^2(\C_+,\cK_\C)^\sharp + i H^2(\C_+,\cK_\C)^\sharp \cong H^2(\C_+,\cK_\C)^\sharp_\C\]
implies that the complex ROPG \((\cE_\C,(\cE_+)_\C,U_\C)\) is equivalent to \((L^2(\R,\cK_\C),H^2(\C_+,\cK_\C),S)\) and therefore has multiplicity space \(\cK_\C\).
\end{proof}
In the following sections, we will always first prove our results just for complex ROPGs and will later translate our final theorems also to the real case using \fref{prop:ROPGRealToComplex}.

\subsection{Independence of the multiplicity space from \(\cE_+\)}
In this section, we want to answer the question of whether the multiplicity space of a ROPG \((\cE,\cE_+,U)\) depends on the subspace \(\cE_+\) when one fixes the Hilbert space \(\cE\) and the unitary one-parameter group \(U\). For this we consider our standard example \((L^2(\R,\cK),H^2(\C_+,\cK),S)\) and want to vary the space \(\cE_+\), so we consider ROPGs of the form \((L^2(\R,\cK),\cE_+,S)\) for a closed subspace \(\cE_+ \subeq L^2(\R,\cK)\) and want to calculate their multiplicity space.

The theorem by Halmos (\cite[Thm. 3]{Ha61}) that we will use to solve this problem has been formulated for the Hardy space on the unit disc \(\D\) instead of the Hardy space on the upper half-plane \(\C_+\), but via Cayley transform we will be able to switch between these two pictures. To formulate this theorem by Halmos, we need some definitions:
\begin{definition}
Let \(\cK\) be a complex Hilbert space.
\begin{enumerate}[\rm (a)]
\item We write
\[L^2(\T,\cK) \coloneqq L^2(\T,\cK,\mu) \quad \text{and} \quad L^\infty(\T,B(\cK)) \coloneqq L^\infty(\T,B(\cK),\mu),\]
where \(\mu\) is the Haar measure on the group \(\T\) normalized by \(\mu(\T) = 1\). Further, we write \(H^2(\D,\cK) \subeq L^2(\T,\cK)\) for the corresponding Hardy space on the disc (cf. \fref{def:HardySpaceDisc}).
\item For \(\Omega = \T,\R\), we call \(F \in L^\infty(\Omega,\mathrm{B}(\cK))\) \textit{rigid}, if there exists a subspace \(\fU \subeq \cK\) such that for almost every \(z \in \Omega\) the operator \(F(z)\) is a partial isometry with initial subspace \(\fU\).
\end{enumerate}
\end{definition}
\begin{definition}
Let \(\cH\) be a Hilbert space. We then say a closed subspace \(\fM \subeq \cH\) \textit{reduces an operator \(A \in \mathrm{B}(\cH)\)} if \([A,P_\fM] = 0\). We say that a closed subspace \(\fM \subeq \cH\) is \textit{irreducible for \(A\)}, if no non-trivial subspace of \(\fM\) reduces \(A\).
\end{definition}
\begin{theorem}\label{thm:Halmos}{\rm(cf. \cite[Thm. 3]{Ha61})}
Let \(\cK\) be a complex Hilbert space and let \(\fM \subeq L^2(\T,\cK)\) be a closed subspace. Then \(\fM\) is irreducible for the operator \(M_{\Id_\T} \in \mathrm{B}\left(L^2(\T,\cK)\right)\) and satisfies \(M_{\Id_\T} \fM \subeq \fM\), if and only if there exists a rigid function \(F \in L^\infty(\T,\mathrm{B}(\cK))\) such that
\[\fM = M_F H^2(\D,\cK).\]
\end{theorem}
To use the previous theorem, we have to switch from the unit disc \(\D\) to the upper half-plane~\(\C_+\). We will do this by using the biholomorphic Cayley transform
\[h(z) \coloneqq \frac{z-i}{z+i}\]
that has the property that
\[h(\R) = \T \setminus \{1\} \quad \text{and} \quad h(\C_+) = \D.\]
Then \(h\) induces an isometric isomorphism of the Hilbert spaces \(L^2(\T,\cK)\) and \(L^2(\R,\cK)\) given by
\[\Gamma: L^2(\T,\cK) \to L^2(\R,\cK), \quad (\Gamma f)(x) = \frac {1}{\sqrt{\pi}(x+i)}(f \circ h)\left(x\right),\]
which which preserves the corresponding Hardy spaces, i.e.
\[\Gamma(H^2(\D,\cK)) = H^2(\C_+,\cK)\]
(cf. \cite[Cor. 5.24, Sec. 5.15]{RR94}). Using this map \(\Gamma\), we get the following analogue to \fref{thm:Halmos}:
\begin{thm}\label{thm:HalmosUpper}
Let \(\cK\) be a complex Hilbert space and \(\fN \subeq L^2(\R,\cK)\) be a closed subspace such that \((L^2(\R,\cK),\fN,S)\) is a SROPG. Then there exists a rigid function \(F \in L^\infty(\R,\mathrm{B}(\cK))\) such that
\[\fN = M_F H^2(\C_+,\cK).\]
\end{thm}
\begin{proof}
We first show that the subspace \(\fM \coloneqq \Gamma^{-1}\fN\) is irreducible for \(M_{\Id_\T}\). So let \(V \subeq \fM\) be a reducing subspace for \(M_{\Id_\T}\). Then \([M_{\Id_\T},P_V] = 0\) and therefore, since \(M_{\Id_\T}\) is unitary, we also get
\[[M_{\Id_\T}^*,P_V] = [M_{\Id_\T}^*,P_V^*] = -[M_{\Id_\T},P_V]^* = -0^* = 0.\]
This implies that
\[[M_{\Id_\T}^n,P_V] = [M_{\Id_\T^n},P_V] = 0 \quad \forall n \in \Z.\]
By \fref{prop:PolynomialsDense} we have
\[\overline{\spann \{\Id_\T^n: n \in \Z\}} = L^\infty(\T,\C),\]
where the closure is taken with respect to the weak-\(*\)-topology on \(L^\infty(\T,\C)\). Therefore
\[[M_f,P_V] = 0 \quad \forall f \in L^\infty(\T,\C).\]
Using Cayley transform to switch to the upper half-plane, setting \(\tilde V \coloneqq \Gamma V\), we get
\[[M_f,P_{\tilde V}] = 0 \quad \forall f \in L^\infty(\R,\C).\]
Since \(S_t = M_{e_t}\) for the functions \(e_t \in L^\infty(\R,\C)\) with \(e_t(x) \coloneqq e^{itx}\), we get
\[[S_t,P_{\tilde V}] = 0 \quad \forall t \in \R.\]
Therefore, for every \(t \in \R\), we get 
\[\tilde V = P_{\tilde V} \tilde V= S_t S_{-t} P_{\tilde V} \tilde V = S_t P_{\tilde V} S_{-t} \tilde V \subeq S_t \tilde V = S_t \Gamma V \subeq S_t \Gamma \fM = S_t \fN,\]
so
\[\tilde V \subeq \bigcap_{t \in \R}S_t \fN = \{0\}\]
which yields \(\tilde V = \{0\}\) and therefore also
\[V = \Gamma^{-1} \tilde V = \{0\}.\]
This shows that no non-trivial subspace of \(\fM\) reduces \(U\), and therefore \(\fM\) is irreducible. Further, by \fref{prop:etDense}, we have
\begin{align*}
M_{\Id_\T} \fM &\subeq H^\infty(\D,\C) \cdot \fM = H^\infty(\D,\C) \cdot \Gamma^{-1}\fN = \Gamma^{-1} \left(H^\infty(\C_+,\C) \cdot \fN\right)
\\&= \Gamma^{-1} \left(\overline{\bigcup_{t \geq 0} M_{e_t} \fN}\right) = \Gamma^{-1} \left(\overline{\bigcup_{t \geq 0} S_t \fN}\right) \subeq \Gamma^{-1} \fN = \fM.
\end{align*}
Therefore we can apply \fref{thm:Halmos} and get a rigid function \(\tilde F \in L^\infty(\T,\mathrm{B}(\cK))\) such that
\[\fM = M_{\tilde F} H^2(\D,\cK).\]
Transforming to the upper half-plane, for the rigid function
\[F \coloneqq \tilde F \circ h \in L^\infty(\R,\mathrm{B}(\cK)),\]
we get
\[\fN = \Gamma^{-1}\fM = \Gamma^{-1}M_{\tilde F} H^2(\D,\cK) = M_F \Gamma^{-1} H^2(\D,\cK) = M_F H^2(\C_+,\cK). \qedhere\]
\end{proof}
Given a complex Hilbert space \(\cK\) and a closed subspace \(\cE_+\subeq L^2(\R,\cK)\) such that the triple \((L^2(\R,\cK),\cE_+,S)\) is a ROPG, by the Lax--Phillips Theorem (\fref{thm:LaxPhillipsComplex}), there exists another complex Hilbert space \(\cK'\) such that \((L^2(\R,\cK),\cE_+,S)\) is equivalent to \((L^2(\R,\cK'),H^2(\C_+,\cK'),S)\). The following result tells us that in fact \(\cK' = \cK\):
\begin{cor}\label{cor:StandardIndependentFromSubspace}
Let \(\cK\) be a complex Hilbert space and \(\cE_+ \subeq L^2(\R,\cK)\) be a closed subspace such that \((L^2(\R,\cK),\cE_+,S)\) is a ROPG. Then \((L^2(\R,\cK),\cE_+,S)\) is equivalent to \((L^2(\R,\cK),H^2(\C_+,\cK),S)\).
\end{cor}
\begin{proof}
By \fref{thm:HalmosUpper} there exists a rigid function \(F \in L^\infty(\R,\mathrm{B}(\cK))\) such that
\[\cE_+ = M_F H^2(\C_+,\cK).\]
Now, we define
\[P: L^2(\R,\cK) \to L^2(\R,\cK), \quad (Pf)(x) = P_{F(x)\cK} f(x).\]
Then \(P\) is a projection with
\[PM_F = M_F \qquad \text{and} \qquad P \circ S_t = S_t \circ P, \quad t \in \R.\]
Now, for \(t \in \R\), we have
\[S_t \cE_+ = S_t M_F H^2(\C_+,\cK) = S_t P M_F H^2(\C_+,\cK) = P S_t M_F H^2(\C_+,\cK) \subeq P L^2(\R,\cK),\]
so
\[L^2(\R,\cK) = \overline{\bigcup_{t \in \R} S_t \cE_+} \subeq P L^2(\R,\cK).\]
This yields \(P = \textbf{1}\) and therefore \(P_x = \textbf{1}\) for almost every \(x \in \R\), which implies that \(F(x)\cK = \cK\) for almost every \(x \in \R\). Choosing such an element \(x_0 \in \R\), the operator \(F(x_0)\) then is a partial isometry with initial subspace \(\fU\) and final subspace \(\cK\) and therefore \(I \coloneqq F(x_0)^*\) is an isometry with initial subspace \(\cK\) and final subspace \(\fU\). We now define \(G \in L^\infty(\R,\mathrm{B}(\cK))\) by
\[G(x) \coloneqq F(x)I.\]
Then
\[M_G H^2(\C_+,\cK) = M_F I H^2(\C_+,\cK) = M_F H^2(\C_+,\fU) = M_F H^2(\C_+,\cK) = \cE_+,\]
identifying \(H^2(\C_+,\fU)\) with the subspace
\[\{f \in H^2(\C_+,\cK): (\forall z \in \C_+) \,f(z) \in \fU\} \subeq H^2(\C_+,\cK).\]
Further, since \(I\) is an isometry with initial subspace \(\cK\) and final subspace \(\fU\) and for almost every \(x \in \R\) the operator \(F(x)\) is a partial isometry with initial subspace \(\fU\) and final subspace \(\cK\), for almost every \(x \in \R\) the operator \(G(x)\) is a partial isometry with initial subspace \(\cK\) and final subspace \(\cK\) and therefore \(G(x)\) is unitary. This implies that also \(M_G\) is unitary. We now set
\[\psi \coloneqq M_G^{-1} \in \U(L^2(\R,\cK)).\]
Then
\[\psi\left(\cE_+\right) = H^2(\C_+,\cK)\]
and
\[\psi \circ S_t = S_t \circ \psi \quad \forall t \in \R. \qedhere\]
\end{proof}

\begin{thm}\label{thm:IndependentFromSubspace}
Let \(\cE\) be a Hilbert space and \(U\) be a unitary one-parameter group on \(\cE\). Further, let \(\cE_+,\cE_+' \subeq \cE\) be two closed subspaces such that \((\cE,\cE_+,U)\) and \((\cE,\cE_+',U)\) are ROPGs. Then \((\cE,\cE_+,U)\) is equivalent to \((\cE,\cE_+',U)\).
\end{thm}
\begin{proof}
We first consider the case of complex ROPGs. Then, by the Lax--Phillips Theorem (\fref{thm:LaxPhillipsComplex}), there exists a complex Hilbert space \(\cK\) such that the ROPG \((\cE,\cE_+,U)\) is equivalent to \((L^2(\R,\cK),H^2(\C_+,\cK),S)\). Let \(\psi\) be the corresponding isomorphism of ROPGs. Then \((\cE,\cE_+',U)\), using the same isomorphism \(\psi\), is equivalent to \((L^2(\R,\cK),\psi \cE_+',S)\). By \fref{cor:StandardIndependentFromSubspace} we know that \((L^2(\R,\cK),\psi \cE_+',S)\) is equivalent to \((L^2(\R,\cK),H^2(\C_+,\cK),S)\) and therefore also \((\cE,\cE_+,U)\) and \((\cE,\cE_+',U)\) are equivalent.

In the case of real ROPGs, we notice that, by \fref{prop:ROPGRealToComplex}, \((\cE,\cE_+,U)\) and \((\cE,\cE_+',U)\) are equivalent, if and only if \((\cE_\C,(\cE_+)_\C,U_\C)\) and \((\cE_\C,(\cE_+')_\C,U_\C)\) are equivalent, which, as just seen, is indeed the case.
\end{proof}
\begin{remark}
This theorem shows that, fixing a Hilbert space \(\cE\) and a unitary one-parameter group \(U\) on \(\cE\), there is --- up to equivalence --- only one possible choice of a subspace \(\cE_+\) such that \((\cE,\cE_+,U)\) is a ROPG.
\end{remark}

\newpage
\section{ROPGs generated by Pick functions}
In the last section, we answered the question, given a ROPG \((\cE,\cE_+,U)\), what happens when one fixes the Hilbert space \(\cE\) and the unitary one-parameter group \(U\) and varies the subspace \(\cE_+\). In this chapter, we want to answer the question, what happens when one fixes the Hilbert space \(\cE\) and the subspace \(\cE_+\) and varies the unitary one-parameter group \(U\), so we will have two ROPGs of the form \((\cE,\cE_+,U)\) and \((\cE,\cE_+,U')\). In the complex case, by the Lax--Phillips Theorem, we then find a complex Hilbert space \(\cH\) such that \((\cE,\cE_+,U)\) is equivalent to \((L^2(\R,\cH),H^2(\C_+,\cH),S)\) via some isomorphism \(\psi\) of ROPGs. Using that same isomorphism \(\psi\), the ROPG \((\cE,\cE_+,U')\) is equivalent to \((L^2(\R,\cH),H^2(\C_+,\cH),\psi U'\psi^{-1})\). Therefore, our question of interest is, for which unitary one-parameter groups \(U\) the triple \((L^2(\R,\cH),H^2(\C_+,\cH),U)\) is a complex ROPG and if so, how to calculate the corresponding multiplicity space.

\subsection{Regular Pick functions}
We will not give a complete characterization of all unitary one-parameter groups \(U\) such that the triple \((L^2(\R,\cH),H^2(\C_+,\cH),U)\) is a complex ROPG, but we will do so under two more assumptions. The first assumption is that our multiplicity space \(\cH\) is finite-dimensional. Our second assumption is, that the unitary one-parameter group \(U\) and our standard unitary one-parameter group \(S\) commute, i.e.
\[U_t S_s = S_s U_t \quad \forall s,t \in \R.\]
This last condition, by \fref{prop:SCommutant}, is equivalent to \(U_t \in M \left(L^\infty(\R,B(\cH))\right)\) for every \(t \in \R\). Then the condition
\[U_t H^2(\C_+,\cH) \subeq H^2(\C_+,\cH) \quad \forall t \geq 0,\]
by \fref{prop:H2InclusionFunctions}, implies that
\[U_t \in M\left(H^\infty(\C_+,B(\cH))\right) \quad \forall t \geq 0.\]
The unitarity of \(U_t\) then implies that
\[U_t \in M\left(\Inn(\C_+,B(\cH))\right) \quad \forall t \geq 0\]
(cf. \fref{def:Inner}), which, by \fref{thm:OPGOfInnerFunctions}, implies that \(U\) is infinitesimally generated by a Pick function, defined as follows:
\begin{definition}(cf. \fref{def:Pick})
Let \(\cH\) be a finite-dimensional complex Hilbert space.
\begin{enumerate}[\rm (a)]
\item We call a function \({F \in \cO(\C_+,B(\cH))}\) an \textit{operator-valued Pick function} if
\[\Im(F(z)) \coloneqq \frac 1{2i} (F(z)-F(z)^*) \geq 0, \quad \forall z \in \C_+.\]
\item We write \(\mathrm{Pick}(\C_+,B(\cH))\) for the set of operator-valued Pick functions.
\item We write \(\mathrm{Pick}(\C_+,B(\cH))_\R\) for the set of Pick functions \(F \in \mathrm{Pick}(\C_+,B(\cH))\), for which the limit
\[F_*(x) \coloneqq \lim_{\epsilon \downarrow 0} F(x+i\epsilon)\]
exists for almost every \(x \in \R\) and is self-adjoint. We say such a Pick function has self-adjoint boundary values.
\item We set
\[\mathrm{Pick}(\C_+) \coloneqq \mathrm{Pick}(\C_+,B(\C)) \quad \text{and} \quad \mathrm{Pick}(\C_+)_\R \coloneqq \mathrm{Pick}(\C_+,B(\C))_\R\]
and, in the latter case, we say these functions have real boundary values.
\end{enumerate}
\end{definition}
Since we want to analyze ROPGs, we are interested in Pick functions, that are infinitesimal generators for ROPGs. Here, we make the following definition:
\begin{definition}
Let \(\cH\) be a finite-dimensional complex Hilbert space. We call a Pick function \(F \in \mathrm{Pick}(\C_+,B(\cH))_\R\) \textit{regular} if
\[\bigcap_{t \in \R}e^{itF_*} \cdot H^2(\C_+,\cH) = \{0\} \quad \text{and} \quad \overline{\bigcup_{t \in \R}e^{itF_*} \cdot H^2(\C_+,\cH)} = L^2(\R,\cH).\]
We write \(F \in \mathrm{Pick}(\C_+,B(\cH))_\R^\mathrm{reg}\) for the set of all regular Pick functions.
\end{definition}
\begin{remark}
Given this definition, a Pick function \(F \in \mathrm{Pick}(\C_+,B(\cH))_\R\) is regular, if and only if the triple \((L^2(\R,\cH),H^2(\C_+,\cH),U^F)\) with
\[\gls*{UF}_t \coloneqq M_{e^{itF_*}}, \quad t \in \R,\]
is a complex ROPG. Here, we use that the condition
\[e^{itF_*} \cdot H^2(\C_+,\cH) \subeq H^2(\C_+,\cH) \quad \forall t \geq 0\]
is fulfilled for every Pick function \(F \in \mathrm{Pick}(\C_+,B(\cH))_\R\), since \(e^{itF} \in H^\infty(\C_+,B(\cH))\) for all \(t \geq 0\).
\end{remark}
The goal of this section now is to give a classification of all regular Pick functions. For this, we will use the fact that the boundary values of a Pick function are a measurable function (cf. \fref{lem:BoundaryMeasurable}) and, therefore, can be measurably diagonalized using the following results:
\begin{thm}\label{thm:measurableDiagonalizationPos}{\rm(cf. \cite[Thm. 2.1]{QR14})}
Let \(X\) be a measurable space and \(n \in \N\). Further, let \({F: X \to \mathrm{Mat}(n \times n)}\) be measurable such that \(F(x)\) is a positive definite matrix for every \(x \in X\). Then there exist measurable functions \(\Lambda,U: X \to \mathrm{Mat}(n \times n)\) such that \(\Lambda(x)\) is diagonal, \(U(x)\) is unitary and \(F(x) = U(x) \Lambda(x) U(x)^*\) for all \(x \in X\).
\end{thm}
\begin{cor}\label{cor:measurableDiagonalization}
Let \(X\) be a measurable space and \(n \in \N\). Further, let \({F: X \to \mathrm{Mat}(n \times n)}\) be measurable such that \(F(x)\) is a self-adjoint matrix for every \(x \in X\). Then there exist measurable functions \(\Lambda,U: X \to \mathrm{Mat}(n \times n)\) such that \(\Lambda(x)\) is diagonal, \(U(x)\) is unitary and \(F(x) = U(x) \Lambda(x) U(x)^*\) for all \(x \in X\).
\end{cor}
\begin{proof}
By applying \fref{thm:measurableDiagonalizationPos} to the function \(\exp \circ F\), we obtain measurable functions \(\tilde \Lambda, U\) such that \(\tilde \Lambda(x)\) is diagonal, \(U(x)\) is unitary and \((\exp \circ F)(x) = U(x) \tilde \Lambda(x) U(x)^*\) for all \(x \in X\). The statement then follows by choosing \(\Lambda \coloneqq \log \circ \tilde \Lambda\).
\end{proof}
We now can formulate our first classification theorem of regular Pick functions using the following definition:
\begin{definition}\label{def:DFM}
For a Pick function \(F \in \mathrm{Pick}(\C_+,B(\cH))_\R\) and \(M \subeq \R\), we define
\[\gls*{DFM} \coloneqq \{y \in \R: F_*(y) \text{ exists and } M \cap \Spec(F_*(y)) \neq \eset\}.\]
Further, for \(x \in \R\), we set
\[D(F,x) \coloneqq D(F,\{x\}) = \{y \in \R: F_*(y) \text{ exists and } x \in \Spec(F_*(y))\}.\]
\end{definition}
\newpage
\begin{thm}\label{thm:SpecNonDeg}
Let \(\cH\) be a finite-dimensional complex Hilbert space and \(F \in \mathrm{Pick}(\C_+,B(\cH))_\R\). Then the following are equivalent:
\begin{enumerate}[\rm (a)]
\item \(F\) is regular.
\item \(\Spec(F(z)) \subeq \C_+\) for all \(z \in \C_+\).
\item \(\ker(x\textbf{1}-M_{F_*}) \cap H^2(\C_+,\cH) = \{0\}\) for all \(x \in \R\).
\item \(\R \setminus D(F,x)\) has non-zero Lebesgue measure for all \(x \in \R\).
\item \(D(F,x)\) is a Lebesgue zero-set for all \(x \in \R\).
\end{enumerate}
\end{thm}
\begin{proof}
\begin{itemize}
\item[(a) \(\Rightarrow\) (c):] Let \(F\) be regular and \(x \in \R\). Further, let \(v \in \ker(x\textbf{1}-M_{F_*}) \cap H^2(\C_+,\cH)\). Then
\[\{0\} = \bigcap_{t \in \R}e^{itF_*}H^2(\C_+,\cH) \supeq \bigcap_{t \in \R}e^{itF_*}\C v = \bigcap_{t \in \R}e^{itx}\C v = \C v,\]
which implies \(v=0\).
\item[(b) \(\Rightarrow\) (a):] Assume that \(\Spec(F(z)) \subeq \C_+\) for all \(z \in \C_+\). Let \(h \in \bigcap_{t \in \R}e^{itF_*}H^2(\C_+,\cH)\). Now, for \(z \in \C_+\), let \(F(z)^*= D(z) + N(z)\) be the Jordan--Chevalley Decomposition of \(F(z)^*\) into a diagonalizable matrix \(D(z)\) and a nilpotent matrix \(N(z)\). Further, setting \(n \coloneqq \dim \cH\), we choose a basis \(v_1(z),\dots,v_n(z)\) of \(\cH\) of eigenvectors of \(D(z)\) and let \(\lambda_1(z),\dots,\lambda_n(z)\) be the corresponding eigenvalues, so
\[D(z)v_k(z) = \lambda_k(z)v_k(z)\]
for all \(k=1,\dots,n\). Then, for every \(t \in \R\) and \(k=1,\dots,n\), we have
\begin{align*}
\braket*{Q_z \cdot v_k(z)}{e^{itF_*}h} &= \braket*{v_k(z)}{e^{itF(z)}h(z)} = \braket*{e^{-itF(z)^*}v_k(z)}{h(z)}
\\&= \braket*{e^{-it\lambda_k(z)}\left(\sum_{j=0}^n \frac 1{j!}\left(-itN(z)\right)^j\right)v_k(z)}{h(z)}.
\end{align*}
with \(Q_z\) as in \fref{prop:H2RKHS}.
This yields
\begin{align*}
\frac{e^{-it\overline{\lambda_k(z)}}}{t^n} \braket*{Q_z \cdot v_k(z)}{e^{itF_*}h} = \braket*{\left(\sum_{j=0}^n \frac 1{j!t^{n-j}}\left(-iN(z)\right)^j\right)v_k(z)}{h(z)}.
\end{align*}
We further have
\[\left|\braket*{Q_z \cdot v_k(z)}{e^{itF_*}h}\right| \leq \left\lVert Q_z\right\rVert \cdot \left\lVert v_k(z)\right\rVert \cdot \left\lVert h\right\rVert \quad \text{and} \quad \left|\frac{e^{-it\overline{\lambda_k(z)}}}{t^n}\right| = \frac{e^{-t\Im(\lambda_k(z))}}{t^n} \xrightarrow{t \to \infty} 0,\]
using that by assumption \(\lambda_k(z) \in \C_+\) and therefore \(\Im(\lambda_k(z)) > 0\). Then, for \(t \to \infty\), we get
\[0 = \braket*{\frac 1{n!}\left(-iN(z)\right)^n v_k(z)}{h(z)}.\]
Therefore
\[\braket*{Q_z \cdot v_k(z)}{e^{itF_*}h} = \braket*{e^{-it\lambda_k(z)}\left(\sum_{j=0}^{n-1} \frac 1{j!}\left(-itN(z)\right)^j\right)v_k(z)}{h(z)}.\]
Now, repeating the same argument, one shows that
\[0 = \braket*{\frac 1{j!}\left(-iN(z)\right)^j v_k(z)}{h(z)}\]
for all \(j=0,\dots,n\). Especially, for \(j=0\), we get
\[0 = \braket*{v_k(z)}{h(z)}.\]
Since this is true for all \(k=1,\dots,n\) and the vectors \(v_1(z),\dots,v_n(z)\) form a basis of \(\cH\), we have \(h(z) = 0\) for every \(z \in \C_+\) and therefore \(h = 0\), so
\[\bigcap_{t \in \R}e^{itF_*}H^2(\C_+,\cH) = \{0\}.\]
On the other hand, using the same argument for the operator-valued Pick function \(F^\flat\) on the lower half-plane \(\C_-\) defined by
\[F^\flat(z) \coloneqq -F(\overline{z})^*, \quad z \in \C_+,\]
 one gets
\[\bigcap_{t \in \R} e^{it F^\flat_*}H^2(\C_-,\cH) = \{0\}.\]
Then, using that
\[F^\flat_* = -F^*_* = -F_* \quad \text{and} \quad H^2(\C_-,\cH) = H^2(\C_+,\cH)^\perp,\]
as seen in \fref{rem:ROPGAlternative}, we get
\begin{align*}
\overline{\bigcup_{t \in \R}e^{itF_*}H^2(\C_+,\cH)} &= \left(\bigcap_{t \in \R} e^{it F^\flat_*}H^2(\C_-,\cH)\right)^\perp  = \{0\}^\perp = L^2(\R,\cH),
\end{align*}
so \(F\) is regular.
\item[(c) \(\Rightarrow\) (b):] Assume that \(\Spec(F(z)) \not\subeq \C_+\) for some \(z \in \C_+\). Since, by \fref{prop:SpecImPos}(a), we have \(\Spec(F(z)) \subeq \overline{\C_+}\), there exists \(x \in \R\) and a \(v \in \cH \setminus \{0\}\) with \(F(z)v = xv\). Then, by \fref{prop:SpecImPos}(b), we have \(\Im(F(z))v = 0\). We now define the Pick function
\[f: \C_+ \to \C, \quad w \mapsto \braket*{v}{F(w)v}.\]
Then
\[\Im(f(z)) = \braket*{v}{\Im(F(z))v} = 0\]
and therefore, by \fref{lem:PickMaxModPrin}, we get that \(f\) is constant. So
\[\braket*{v}{\Im(F(w))v} = \Im(f(w)) = \Im(f(z)) = 0\]
for all \(w \in \C_+\). Since \(\Im(F(w)) \geq 0\), this implies that
\[0 = \Im(F(w))v = \frac 1{2i}(F(w)-F(w)^*)v\]
and therefore \(F(w)^*v = F(w)v\) for all \(w \in \C_+\).
We now define the function
\[g: \C_+ \to \C, \quad w \mapsto \left\lVert(F(w)-x)v\right\rVert^2.\]
Since
\begin{align*}
g(w) &= \braket*{F(w)v}{F(w)v} - x \braket*{F(w)v}{v} - x \braket*{v}{F(w)v} + x^2 \braket*{v}{v}
\\&= \braket*{F(w)^*v}{F(w)v} - x \braket*{F(w)^*v}{v} - x \braket*{v}{F(w)v} + x^2 \left\lVert v\right\rVert^2
\\&= \braket*{v}{F(w)^2v} - 2x \braket*{v}{F(w)v} + x^2 \left\lVert v\right\rVert^2,
\end{align*}
the function \(g\) is holomorphic with \(\Im(g) \equiv 0\), which implies that \(g\) is constant. Therefore, for every \(w \in \C_+\), we have
\[\left\lVert(F(w)-x)v\right\rVert^2 = g(w) = g(z) = 0,\]
which implies
\[F(w)v = xv.\]
Taking limits, this also implies
\[F_*(y)v = xv \qquad \forall y \in \R.\]
Now, for every function \(h \in H^2(\C_+) \setminus \{0\}\), we have have \(h \otimes v \in H^2(\C_+,\cH) \setminus \{0\}\) with
\[(x\textbf{1}-M_{F_*})(h \otimes v) = 0,\]
so especially \(\ker(x\textbf{1}-M_{F_*}) \cap H^2(\C_+,\cH) \neq \{0\}\).
\item[(d) \(\Rightarrow\) (c):] Assume that there exists \(v \in \ker(x\textbf{1}-M_{F_*}) \cap H^2(\C_+,\cH)\) with \(v \neq 0\) for some \(x \in \R\). Since \(v \in H^2(\C_+,\cH)\), by \fref{thm:HardyAlmostEverywhereNonZero}, we have \(v(y) \neq 0\) for almost every \(y \in \R\) and therefore
\[\ker(x\textbf{1}-F_*(y)) \supeq \C v(y) \neq \{0\}\]
for almost every \(y \in \R\). This implies that the Lebesgue measure of \(\R \setminus D(F,x)\) is zero.
\item[(e) \(\Rightarrow\) (d):] This follows immediately from the fact that \(\R\) has non-zero Lebesgue measure.
\item[(c) \(\Rightarrow\) (e):] Assume that \(D(F,x)\) is no Lebesgue zero-set for some \(x \in \R\). Then there is a \(R>0\) such that \(\lambda_1(D(F,x) \cap [-R,R])>0\), denoting by \(\lambda_1\) the Lebesgue measure on \(\R\). Setting \(n \coloneqq \dim \cH\), by \fref{cor:measurableDiagonalization} there exists a measurable function \(U: \R \to \mathrm{Mat}(n \times n)\) such that \(U(x)\) is unitary for all \(x \in \R\) and functions \(d_1,\dots,d_n: \R \to \R\) such that, after the identification \(\cH \cong \C^n\), we have
\[F_*(y) = U(y) \diag(d_1(y),\dots,d_n(y))U(y)^*\]
for almost every \(y \in \R\). For \(k = 1,\dots,n\), we define the measurable sets
\[M_k \coloneqq d_k^{-1}(x) \setminus \bigcup_{j=1}^{k-1}d_j^{-1}(x).\]
Denoting by \(\{e_1,\dots,e_n\}\) the canonical ONB of \(\cH \cong \C^n\), we now consider the measurable function
\[g \coloneqq \sum_{k=1}^n \chi_{M_k \cap [-R,R]}Ue_k.\]
One has
\[\left\lVert g\right\rVert^2 = \sum_{k=1}^n \lambda_1(M_k \cap [-R,R]) = \lambda_1(D(F,x) \cap [-R,R]) \in ]0,2R],\]
so \(g \in L^2(\R,\cH) \setminus \{0\}\).
Further, for almost every \(y \in \R\), we have
\begin{align*}
F_*(y)g(y) &= \sum_{k=1}^n \chi_{M_k \cap [-R,R]}(y)F_*(y)U(y)e_k = \sum_{k=1}^n \chi_{M_k \cap [-R,R]}(y)d_k(y)U(y)e_k
\\&= \sum_{k=1}^n \chi_{M_k \cap [-R,R]}(y)xU(y)e_k = xg(y),
\end{align*}
so \(g \in \ker(x\textbf{1}-M_{F_*})\). Since \(g\big|_{\R \setminus [-R,R]} \equiv 0\), we know that \(g \notin H^2(\C_-,\cH)\) by \fref{thm:HardyAlmostEverywhereNonZero}. Denoting by \(P_\pm\) the projections onto \(H^2(\C_\pm,\cH)\) and setting \(g_\pm \coloneqq P_\pm g\), we therefore have \(g_+ \neq 0\). Now, for every \(t \in \R\), we have
\[e^{itx}g_+ + e^{itx}g_- = e^{itx}g = e^{itF_*}g = e^{itF_*}g_+ + e^{itF_*}g_-.\]
For \(t \geq 0\) we have \(e^{itF_*}g_+ \in H^2(\C_+,\cH)\), so applying \(P_\pm\) to both sides, we get
\[e^{itx}g_+ = e^{itF_*}g_+ + P_+ e^{itF_*}g_- \quad \text{and} \quad e^{itx}g_- = P_- e^{itF_*}g_-.\]
This implies
\[\left\lVert P_+ e^{itF_*}g_-\right\rVert^2 = \left\lVert e^{itF_*}g_-\right\rVert^2 - \left\lVert P_- e^{itF_*}g_-\right\rVert^2 = \left\lVert g_-\right\rVert^2 - \left\lVert e^{itx}g_-\right\rVert^2 = 0,\]
so \(P_+ e^{itF_*}g_- = 0\) and therefore
\[e^{itx}g_+ = e^{itF_*}g_+ \quad \text{and} \quad e^{itx}g_- = e^{itF_*}g_-\]
for all \(t \geq 0\). This implies that
\[g_+,g_- \in \ker(x\textbf{1}-M_{F_*}),\]
so since \(g_+ \neq 0\), we especially get
\[\ker(x\textbf{1}-M_{F_*}) \cap H^2(\C_+,\cH) \neq \{0\}. \qedhere\]
\end{itemize}
\end{proof}
We now will see some consequences of this theorem:
\begin{cor}\label{cor:ConstNonDeg}
A function \(f \in \mathrm{Pick}(\C_+)_\R\) is regular, if and only if \(f\) is non-constant.
\end{cor}
\begin{proof}
If \(f \in \mathrm{Pick}(\C_+)_\R\) is constant, than, since \(f\) has real boundary values, there is a \(\lambda \in \R\) with \(f \equiv \lambda\), so \(\Im f(z) = 0\) for every \(z \in \C_+\). If on the other hand \(f \in \mathrm{Pick}(\C_+)_\R\) is non-constant, then, by \fref{lem:PickMaxModPrin}, we have \(\Im f(z) > 0\) for every \(z \in \C_+\). Therefore \(f\) is non-constant, if and only if
\[\Spec(f(z)) = \{f(z)\} \subeq \C_+\]
for every \(z \in \C_+\) which, by \fref{thm:SpecNonDeg}, is equivalent to the function \(f\) being regular.
\end{proof}
Since, by \fref{thm:SpecNonDeg}, for \(f \in \mathrm{Pick}(\C_+)_\R^{\mathrm{reg}}\) and \(F \in \mathrm{Pick}(\C_+,B(\cH))_\R^{\mathrm{reg}}\), for \(z \in \C_+\), one has
\[f(z) \in \C_+ \quad \text{and} \quad \Spec(f(z)) \subeq \C_+,\]
the operators
\[F(f(z)) \in B(\cH) \quad \text{and} \quad f(F(z)) \in B(\cH)\]
are defined. Therefore, the maps
\[\circ: \mathrm{Pick}(\C_+,B(\cH))_\R^{\mathrm{reg}} \times \mathrm{Pick}(\C_+)_\R^{\mathrm{reg}} \to \mathrm{Pick}(\C_+,B(\cH)), \quad (F \circ f)(z) \coloneqq F(f(z))\]
and
\[\circ: \mathrm{Pick}(\C_+)_\R^{\mathrm{reg}} \times \mathrm{Pick}(\C_+,B(\cH))_\R^{\mathrm{reg}} \to \mathrm{Pick}(\C_+,B(\cH)), \quad (f \circ F)(z) \coloneqq f(F(z))\]
are also defined. The following theorem shows that their image, in fact, is again \(\mathrm{Pick}(\C_+,B(\cH))_\R^{\mathrm{reg}}\):
\begin{cor}\label{cor:CompNonDeg}
Let \(\cH\) be a finite-dimensional complex Hilbert space. Further, let \(f,g \in \mathrm{Pick}(\C_+)_\R\) and \(F \in \mathrm{Pick}(\C_+,B(\cH))_\R\) all be regular. Then \(f \circ F \circ g\) is also regular.
\end{cor}
\begin{proof}
By \fref{thm:SpecNonDeg}, for \(z \in \C_+\), we have
\[\{f(z)\} = \Spec(f(z)) \subeq \C_+, \quad \{g(z)\} = \Spec(g(z)) \subeq \C_+ \quad \text{and} \quad \Spec(F(z)) \subeq \C_+.\]
Therefore, for \(z \in \C_+\), we get
\[\Spec((f \circ F \circ g)(z)) = f(\Spec(F(g(z)))) \subeq f(\C_+) \subeq \C_+,\]
which by \fref{thm:SpecNonDeg} implies that \(f \circ F \circ g\) is regular.
\end{proof}
With help of this corollary, we can prove a theorem with a generalization of condition (e) in \fref{thm:SpecNonDeg}:
\begin{thm}\label{thm:SpecNonDegZeroSet}
Let \(\cH\) be a finite-dimensional complex Hilbert space and \(F \in \mathrm{Pick}(\C_+,B(\cH))_\R\). Then the following are equivalent:
\begin{enumerate}[\rm (a)]
\item \(F\) is regular.
\item \(D(F,M)\) is a Lebesgue zero-set for all Lebesgue zero-sets \(M \subeq \R\).
\end{enumerate}
\end{thm}
\begin{proof}
\begin{itemize}
\item[(b) \(\Rightarrow\) (a):] Follows immediately from \fref{thm:SpecNonDeg}, using the fact that \(\{x\}\) is a Lebesgue zero-set for every \(x \in \R\).
\item[(a) \(\Rightarrow\) (b):] We assume that there is a Lebesgue zero-set \(M \subeq \R\) such that \(D(F,M)\) has positive Lebesgue measure. By \fref{cor:measurableDiagonalization} there exists a measurable function \(U: \R \to \mathrm{Mat}(n \times n)\) such that \(U(x)\) is unitary for all \(x \in \R\) and measurable functions \(d_1,\dots,d_n: \R \to \R \cup \{\infty\}\) such that, after the identification \(\cH \cong \C^n\), we have
\[F_*(y) = U(y) \diag(d_1(y),\dots,d_n(y))U(y)^*\]
for every \(y \in \R\) for which \(F_*(y)\) exists and \(d_1(y),\dots,d_n(y) = \infty\) for every \(y \in \R\) for which \(F_*(y)\) doesn't exist. Then
\[D(F,M) = \bigcup_{k=1}^n d_k^{-1}(M).\]
So, by assumption, there exists \(K \in \{1,\dots,n\}\) such that \(d_K^{-1}(M)\) has positive Lebesgue measure. We now choose \(R>0\) such that
\[\lambda_1\left(d_K^{-1}(M) \cap [-R,R]\right)>0.\]
Then
\[\mu(E) \coloneqq \lambda_1\left(d_K^{-1}(M \cap E) \cap [-R,R]\right)\]
defines a measure \(\mu\) on \(\R\) with
\[\mu(\R) = \lambda_1\left(d_K^{-1}(M) \cap [-R,R]\right),\]
which implies
\[0 < \mu(\R) \leq 2R < \infty,\]
so \(\mu\) is non-zero and finite. Further, since \(M\) is a Lebesgue zero-set and \({\mu(\R \setminus M) = 0}\), it is also Lebesgue-singular. Therefore, by \fref{thm:GePick},
\[F_\mu(z) \coloneqq \int_\R \frac 1{\lambda-z} - \frac \lambda{1+\lambda^2} \,d\mu(\lambda)\]
defines a Pick function with real boundary-values. Further, by \fref{thm:GePick}, setting
\[S_\mu \coloneqq \left\{x \in \R: \lim_{\epsilon \downarrow 0} \Im(F_\mu(x+i\epsilon)) = +\infty\right\},\]
we have \(\mu(E) = \mu(E \cap S_\mu)\) for every measurable subset \(E \subeq \R\). In particular, this implies that
\[0 \neq \mu(\R) = \mu(S_\mu) = \lambda_1\left(d_K^{-1}(M \cap S_\mu) \cap [-R,R]\right).\]
Further, defining the Pick function
\[\eta: \C_+ \to \C_+, \quad z \mapsto -\frac 1z,\]
we get
\[(\eta \circ F_\mu)_*(x) = 0\]
for every \(x \in S_\mu\). For \(x \in d_K^{-1}(M \cap S_\mu) \cap [-R,R]\), by definition, we have that
\[\eset \neq (M \cap S_\mu) \cap \Spec(F_*(x)),\]
so
\[\eset \neq (\eta \circ F_\mu)_*(M \cap S_\mu) \cap \Spec((\eta \circ F_\mu \circ F)_*(x)) = \{0\} \cap \Spec((\eta \circ F_\mu \circ F)_*(x))\]
and therefore
\[0 \in \Spec((\eta \circ F_\mu \circ F)_*(x)).\]
We then obtain
\[D(\eta \circ F_\mu \circ F,0) \supeq d_K^{-1}(M \cap S_\mu) \cap [-R,R]\]
and therefore \(D(\eta \circ F_\mu \circ F,0)\) is no Lebesgue zero-set, so, by \fref{thm:SpecNonDeg}, we have that \(\eta \circ F_\mu \circ F\) is non-regular. Since \(\eta\) and \(F_\mu\) are non-constant and therefore, by \fref{cor:ConstNonDeg}, regular, we get, by \fref{cor:CompNonDeg}, that \(F\) is non-regular. \qedhere
\end{itemize}
\end{proof}
After now having some characterizations for regular Pick functions, i.e. the Pick functions \(F \in \mathrm{Pick}(\C_+,B(\cH))_\R\) such that the triple \((L^2(\R,\cH),H^2(\C_+,\cH),U^F)\) is a complex ROPG, we now want to show how to use these characterizations for real ROPGs. This question is answered in the following theorem:
\begin{theorem}\label{thm:RealToComplexPick}
Let \(\cH\) be a real Hilbert space and \(F \in \mathrm{Pick}(\C_+,B(\cH_\C))_\R\). Then \(U^F\) defines a unitary one-parameter group on \(L^2(\R,\cH_\C)^\sharp\), if and only if \(F = F^\sharp\) with
\[F^\sharp(z) \coloneqq -\cC_\cH F(-\overline{z})\cC_\cH, \quad z \in \C_+.\]
In this case the triple \((L^2(\R,\cH_\C)^\sharp,H^2(\C_+,\cH_\C)^\sharp,U^F)\) is a real ROPG with multiplicity space \(\cK\), if and only if the triple \((L^2(\R,\cH),H^2(\C_+,\cH_\C),U^F)\) is a complex ROPG with multiplicity space \(\cK_\C\).
\end{theorem}
\begin{proof}
One has
\[e^{itF_*} \cdot L^2(\R,\cH_\C)^\sharp \subeq L^2(\R,\cH_\C)^\sharp \quad \forall t \in \R,\]
if and only if
\begin{align*}
\left(e^{itF_*} \cdot f\right)(x) &= \left(e^{itF_*} \cdot f\right)^\sharp(x) = \cC_\cH e^{itF_*(-x)} f(-x)
\\&= e^{-it\cC_\cH F_*(-x)\cC_\cH} \cC_\cH f(-x) = e^{it F^\sharp_*(x)} f^\sharp(x) = e^{it F^\sharp_*(x)} f(x)
\end{align*}
for all \(t \in \R\), \(f \in L^2(\R,\cH_\C)^\sharp\) and almost all \(x \in \R\). This is the case, if and only if
\[F_*(x) = F^\sharp_*(x)\]
for almost all \(x \in \R\), which is equivalent to \(F = F^\sharp\). The second statement follows by \fref{prop:ROPGRealToComplex}.
\end{proof}

\subsection{The multiplicity free case}
In the last section, given a finite-dimensional complex Hilbert space \(\cH\), we have seen for which Pick functions \({F \in \mathrm{Pick}(\C_+,B(\cH))_\R}\) the triple \((L^2(\R,\cH),H^2(\C_+,\cH),U^F)\) is a complex ROPG and likewise, given a finite-dimensional real Hilbert space \(\cH\), for which \(F \in \mathrm{Pick}(\C_+,B(\cH_\C))_\R\) the triple \((L^2(\R,\cH_\C)^\sharp,H^2(\C_+,\cH_\C)^\sharp,U^F)\) is a real ROPG.

The goal of the following sections will be, given a regular Pick function \({F \in \mathrm{Pick}(\C_+,B(\cH))_\R}\) or \(F \in \mathrm{Pick}(\C_+,B(\cH_\C))_\R\) respectively, to calculate the multiplicity space for the complex ROPG \((L^2(\R,\cH),H^2(\C_+,\cH),U^F)\) or the real ROPG \((L^2(\R,\cH_\C)^\sharp,H^2(\C_+,\cH_\C)^\sharp,U^F)\) respectively.

We will later provide an answer in the general case of any finite-dimensional Hilbert space \(\cH\). In this section, we will discuss the multiplicity-free case \(\cH = \C\). Although not generalizable to a case with multiplicities, this proof gives a good idea of the behavior of Pick functions and how they produce multiplicities.

By \fref{thm:GePick}, every function \(F \in \mathrm{Pick}(\C_+)_\R\) is of the form
\[F(z) = C + Dz + \int_\R \frac 1{\lambda - z} - \frac{\lambda}{1+\lambda^2} \,d\mu(\lambda)\]
with \(C \in \R\), \(D \geq 0\) and some singular Borel measure \(\mu\) on \(\R\) supported on the set
\[S_\mu \coloneqq \left\{x \in \R: \lim_{\epsilon \downarrow 0} \Im(F(x+i\epsilon)) = +\infty\right\}.\]
We set
\[S_F \coloneqq \begin{cases} S_\mu & \text{if }D = 0 \\ S_\mu \cup \{\infty\} & \text{if }D > 0\end{cases}\]
(cf. \fref{rem:GePick}). With this definition, we can formulate our theorem:
\begin{theorem}\label{thm:PickMultiplicity}
Let \(F \in \mathrm{Pick}(\C_+)_\R\) such that \(S_F\) is a non-empty closed subset of \(\overline{\R}\). Then \((L^2(\R,\C),H^2(\C_+),U^F)\) is a complex ROPG and its multiplicity space \(\cK\) is given by
\[\cK \coloneqq \begin{cases} \C^n & \text{if }|S_F| = n \in \N \\ \ell^2(\N,\C) & \text{if }|S_F| = \infty.\end{cases}\]
\end{theorem}
\begin{proof}
Since \(S_F\) is non-empty, the function \(F\) is non-constant and therefore, by \fref{cor:ConstNonDeg}, the triple \((L^2(\R,\C),H^2(\C_+),U^F)\) is a complex ROPG. We now calculate its multiplicity space.

By \fref{lem:countableUnion} there is a countable family \(\left(a_j,b_j\right)_{j \in J}\) in \(\overline{\R} \times \overline{\R}\) such that
\[\overline{\R} \setminus S_F = \bigsqcup_{j \in J} \left(a_j,b_j\right).\]
One has \(S_F = \{a_j: j \in J\}\) and therefore \(|S_F| = |J|\). Further, by \fref{prop:PickBoundaryDiffeo}, for every \(j \in J\), the map
\[F_j \coloneqq F_*\big|_{\left(a_j,b_j\right)}: \left(a_j,b_j\right) \to \R\]
is monotonically increasing, bijective, and continuously differentiable. Therefore, for every function \(g \in L^1(\R,\C) \cong L^1\left(\overline{\R},\C\right)\), one has
\begin{align*}
\int_\R g(x) \,dx &= \int_{\overline{\R} \setminus S_F} g(x) \,dx = \sum_{j \in J}\int_{(a_j,b_j)}g(x)\,dx = \sum_{j \in J}\int_\R g\left(F_j^{-1}(y)\right) \left(F_j^{-1}\right)'(y)\,dy.
\end{align*}
This implies that the map
\[V: L^2(\R,\C) \to \ell^2(J,L^2(\R,\C)), \quad f \mapsto \left(\sqrt{\left(F_j^{-1}\right)'} \cdot \left(f \circ F_j^{-1}\right)\right)_{j \in J}\]
is a unitary isomorphism of Hilbert spaces. Further, by the Fubini-Tonelli Theorem, the map
\[T: \ell^2(J,L^2(\R,\C)) \to L^2(\R,\ell^2(J,\C)), \quad T\left(\left(f_j\right)_{j\in J}\right)(x) = \left(f_j(x)\right)_{j \in J},\]
is a unitary isomorphism of Hilbert spaces. Then, for \(f \in L^2(\R,\C)\) and \(t \in \R\), one has
\begin{align*}
(T \circ V)(U^F_t f)(x) &= (T \circ V)\left(e^{itF_*} f\right)(x) = T\left(\left(\sqrt{\left(F_j^{-1}\right)'} \cdot e^{it\left(F_* \circ F_j^{-1}\right)}\left(f \circ F_j^{-1}\right)\right)_{j \in J}\right)(x)
\\&=\left(\sqrt{\left(F_j^{-1}\right)'(x)} \cdot e^{it\left(F_* \circ F_j^{-1}\right)(x)}\left(f \circ F_j^{-1}\right)(x)\right)_{j \in J}
\\&=\left(\sqrt{\left(F_j^{-1}\right)'(x)} \cdot e^{itx}\left(f \circ F_j^{-1}\right)(x)\right)_{j \in J}
\\&=e^{itx} \cdot \left(\sqrt{\left(F_j^{-1}\right)'(x)} \cdot \left(f \circ F_j^{-1}\right)(x)\right)_{j \in J}
\\&= e^{itx} \cdot \left(T \circ V\right)(f)(x) = \left(S_t \circ T \circ V\right)(f)(x),
\end{align*}
so
\[(T \circ V) \circ U^F_t = S_t \circ (T \circ V) \quad \forall t \in \R.\]
Now set
\[\cE_+ \coloneqq \left(T \circ V\right)\left(H^2(\C_+)\right) \subeq L^2(\R,\ell^2(J,\C)).\]
Then, for every \(t \in \R\), one has
\begin{align*}
S_t \cE_+ &= \left(S_t \circ \left(T \circ V\right)\right)\left(H^2(\C_+)\right) = \left(\left(T \circ V\right) \circ U^F_t\right)\left(H^2(\C_+)\right),
\end{align*}
so, using that \((L^2(\R,\C),H^2(\C_+),U^F)\) is a complex ROPG, we get
\[S_t \cE_+ = \left(T \circ V\right)\left(U^F_t H^2(\C_+)\right) \subeq \left(T \circ V\right)\left(H^2(\C_+)\right) = \cE_+ \quad \forall t \geq 0\]
and
\[\bigcap_{t \in \R} S_t \cE_+ = \left(T \circ V\right)\left(\bigcap_{t \in \R} U^F_t H^2(\C_+)\right) = \left(T \circ V\right)\left(\{0\}\right) = \{0\}\]
and
\[\overline{\bigcup_{t \in \R} S_t \cE_+} = \left(T \circ V\right)\left(\overline{\bigcup_{t \in \R} U^F_t H^2(\C_+)}\right) = \left(T \circ V\right)\left(L^2(\R,\C)\right) = L^2(\R,\ell^2(J,\C)).\]
This implies that the triple \((L^2(\R,\ell^2(J,\C)),\cE_+,S)\) is a complex ROPG, which is equivalent to \((L^2(\R,\C),H^2(\C_+),U^F)\) via the isomorphism \(T \circ V\). Further, by \fref{thm:IndependentFromSubspace} the complex ROPG \((L^2(\R,\ell^2(J,\C)),\cE_+,S)\) is equivalent to \((L^2(\R,\ell^2(J,\C)),H^2(\C_+,\ell^2(J,\C)),S)\). So, setting \({\cK \coloneqq \ell^2(J,\C)}\), we have that \((L^2(\R,\C),H^2(\C_+),U^F)\) is equivalent to \((L^2(\R,\cK),H^2(\C_+,\cK),S)\) and therefore \(\cK = \ell^2(J,\C)\) is the multiplicity space of \((L^2(\R,\C),H^2(\C_+),U^F)\). The statement then follows since
\[\ell^2(J,\C) \cong \C^{|J|} = \C^{|S_F|}\]
if \(|J| = |S_F| \in \N\) and
\[\ell^2(J,\C) \cong \ell^2(\N,\C)\]
if \(|J| = |S_F| = \infty\).
\end{proof}
\begin{remark}
One has \(|S_F| \in \N\), if and only if the measure in the Herglotz representation of \(F\) is supported on a finite set, i.e. if there are \(\lambda_1,\dots,\lambda_n \in \R\) and \(a_1,\dots,a_n \in \R_+\) such that
\[F(z) = C + Dz + \sum_{k=1}^n \frac{a_k}{\lambda_k-z}, \quad z \in \C_+,\]
for some \(C \in \R\) and \(D \in \R_{\geq 0}\). In this case, one has
\[|S_F| = \begin{cases} n & \text{if } D= 0 \\ n+1 & \text{if } D > 0.\end{cases}\]
\end{remark}
We will now also provide an analogue to \fref{thm:PickMultiplicity} in the real case:
\begin{thm}\label{thm:PickMultiplicityReal}
Let \(F \in \mathrm{Pick}(\C_+)_\R\) such that \(S_F\) is a non-empty closed subset of \(\overline{\R}\). Then \((L^2(\R,\C)^\sharp,H^2(\C_+)^\sharp,U^F)\) is a real ROPG, if and only if
\[F(z) = - \overline{F(-\overline{z})} \quad \forall z \in \C_+\]
and in this case its multiplicity space \(\cK\) is given by
\[\cK \coloneqq \begin{cases} \R^n & \text{if }|S_F| = n \in \N \\ \ell^2(\N,\R) & \text{if }|S_F| = \infty.\end{cases}\]
\end{thm}
\begin{proof}
This follows immediately by \fref{thm:RealToComplexPick} and \fref{thm:PickMultiplicity}.
\end{proof}
\begin{remark}
The condition \(F(z) = - \overline{F(-\overline{z})}\), for a Pick function
\[F(z) = C + Dz + \int_\R \frac 1{\lambda-z}-\frac \lambda{1+\lambda^2} \,d\mu(\lambda)\]
is equivalent to the measure \(\mu\) being symmetric, i.e. \(\mu(A) = \mu(-A)\) for all measurable subsets \(A \subeq \R\). In this case, the Pick function can be written as
\begin{align*}
F(z) &= C + Dz - \frac{\mu(\{0\})}z + \int_{\R_+} \frac 1{\lambda-z}+\frac 1{-\lambda-z} \,d\mu(\lambda)
\\&= C + Dz - \frac{\mu(\{0\})}z + \int_{\R_+} \frac {2z}{\lambda^2-z^2} \,d\mu(\lambda), \qquad \qquad z \in \C_+.
\end{align*}
\end{remark}

\subsection{The general case}
After considering the one-dimensional case, in this section, given any finite-dimensional complex Hilbert space \(\cH\) and a regular Pick function \({F \in \mathrm{Pick}(\C_+,B(\cH))_\R}\), we want to calculate the multiplicity space for the complex ROPG \((L^2(\R,\cH),H^2(\C_+,\cH),U^F)\), as well as provide an analogue in the real case. We will do this by using \fref{thm:kernelAdjointGeneral}, in which appears the operator
\[p\left(\partial U^F\right)p^* = p M_{F_*} p^* = M_F\]
with \(p \coloneqq p_{H^2(\C_+,\cH)}\).
Therefore, given a Pick function \({F \in \mathrm{Pick}(\C_+,B(\cH))_\R}\) and \(\lambda \in \C_+\), we have to better understand the kernel \(\ker\left(\overline{\lambda}\textbf{1}-M_F^*\right) \subeq H^2(\C_+,\cH)\). We will see that this kernel has a property that we will call \(M_{H^\infty}^*\)-stability. Therefore, the following subsection will be dedicated to the theory of \(M_{H^\infty}^*\)-stable subspaces.

\subsubsection{\(M_{H^\infty}^*\)-stable subspaces and finite Blaschke--Potapov products}
Throughout this subsection, by \(\perp\) we will denote the orthogonal complement in the Hilbert space \(L^2(\R,\cH)\) and by \(\gls*{ZZZO+}\) the orthogonal complement in the Hilbert space \(H^2(\C_+,\cH)\).
\begin{definition}
Let \(\cH\) be a complex Hilbert space. We set
\[\gls*{P+} \coloneqq P_{H^2(\C_+,\cH)} \in B(L^2(\R,\cH)).\]
A subspace \(V \subeq H^2(\C_+,\cH)\) is called \textit{\(M_{H^\infty}^*\)-stable}, if
\[P_+ M_\phi^* V \subeq V\]
for every \(\phi \in H^\infty(\C_+)\).
\end{definition}
It turns out that \(M_{H^\infty}^*\)-stable subspaces are closely related to SROPGs:
\begin{prop}\label{prop:stableCharakterizationSROPG}
Let \(\cH\) be a complex Hilbert space. A subspace \(V \subeq H^2(\C_+,\cH)\) is \(M_{H^\infty}^*\)-stable, if and only if \((L^2(\R,\cH),V^{\perp_+},S)\) is a SROPG.
\end{prop}
\begin{proof}
We first assume that \(V\) is a \(M_{H^\infty}^*\)-stable subspace, so especially \(P_+S_t^* V \subeq V\) for every \(t \geq 0\) since \(S_t = M_{e_t}\) for the functions \(e_t \in L^\infty(\R,\C)\) with
\[\gls*{et}(x) \coloneqq e^{itx}.\]
Then, for \(t \geq 0\), \(f \in V^{\perp_+}\) and \(g \in V\), one has
\[\braket*{S_t f}{g} = \braket*{f}{S_t^* g} = \braket*{P_+ f}{S_t^* g} = \braket*{f}{P_+ S_t^* g} = 0,\]
using that \(P_+ S_t^* g \in V\). This shows that \(S_t V^{\perp_+} \subeq V^{\perp_+}\) for \(t \geq 0\). Since also
\[\bigcap_{t \in \R} S_t V^{\perp_+} \subeq \bigcap_{t \in \R} S_t H^2(\C_+,\cH) = \{0\},\]
we get that \((L^2(\R,\cH),V^{\perp_+},S)\) is a SROPG.

Now, conversely assume that \((L^2(\R,\cH),V^{\perp_+},S)\) is a SROPG. Then, for \(t \geq 0\), \(f \in V^{\perp_+}\) and \(g \in V\), one has
\[\braket*{f}{P_+ S_t^* g} = \braket*{S_t f}{g} = 0,\]
using that \(S_t f \in V^{\perp_+}\). This shows that
\[P_+ M_{e_t}^* V = P_+ S_t^* V \subeq V \quad t \geq 0,\]
which, by \fref{prop:etDense} implies the statement.
\end{proof}
\begin{cor}\label{cor:stableCharakterizationRigid}
Let \(\cH\) be a complex Hilbert space. A subspace \(V \subeq H^2(\C_+,\cH)\) is \(M_{H^\infty}^*\)-stable, if and only if it is of the form
\[V = \left(M_\phi H^2(\C_+,\cH)\right)^{\perp_+}\]
for some function \(\phi \in H^\infty(\C_+,B(\cH))\).
\end{cor}
\begin{proof}
Let \(V \subeq H^2(\C_+,\cH)\) be \(M_{H^\infty}^*\)-stable. Then, by \fref{prop:stableCharakterizationSROPG} and \fref{thm:HalmosUpper} there exists a rigid function \(\phi \in L^\infty(\R,\mathrm{B}(\cH))\) such that
\[M_\phi H^2(\C_+,\cH) = V^{\perp_+} \subeq H^2(\C_+,\cH).\]
The last inclusion in particular implies \(\phi \in H^\infty(\C_+,\mathrm{B}(\cH))\) by \fref{prop:H2InclusionFunctions}.

Conversely, for every \(\phi \in H^\infty(\C_+,\mathrm{B}(\cH))\), one has
\[S_t \left(M_\phi H^2(\C_+,\cH)\right) = M_\phi \left(S_t H^2(\C_+,\cH)\right) \subeq M_\phi H^2(\C_+,\cH) \quad \forall t \geq 0\]
and
\[\bigcap_{t \in \R} S_t \left(M_\phi H^2(\C_+,\cH)\right) \subeq \bigcap_{t \in \R} S_t H^2(\C_+,\cH) = \{0\}.\]
Therefore \((L^2(\R,\cH),M_\phi H^2(\C_+,\cH),S)\) is a SROPG, which, by \fref{prop:stableCharakterizationSROPG}, is equivalent to \(\left(M_\phi H^2(\C_+,\cH)\right)^{\perp_+}\) being a \(M_{H^\infty}^*\)-stable subspace.
\end{proof}
When working with \(M_{H^\infty}^*\)-stable subspaces, the following result is often useful:
\begin{lem}\label{lem:ProjectionEquation}
Let \(\cH\) be a complex Hilbert space and \(\phi,\psi \in H^\infty(\C_+,B(\cH))\). Then
\[P_+ M_\phi^* P_+ M_\psi^* = P_+ M_{\psi \cdot \phi}^*.\]
\end{lem}
\begin{proof}
We have
\[P_+M_\phi^*(\textbf{1}-P_+)L^2(\R,\cH) = P_+M_\phi^*H^2(\C_-,\cH) \subeq P_+H^2(\C_-,\cH) = \{0\},\]
so \(P_+M_\phi^*(\textbf{1}-P_+) = 0\) and therefore
\[P_+M_\phi^*P_+ = P_+M_\phi^*,\]
which implies
\[P_+ M_\phi^* P_+ M_\psi^* = P_+ M_\phi^* M_\psi^* = P_+ \left(M_\psi M_\phi\right)^* = P_+ M_{\psi \cdot \phi}^*. \qedhere\]
\end{proof}
\begin{lem}\label{lem:SpaceLokalDecomp}
Let \(\cH\) be a complex Hilbert space and \(V \subeq H^2(\C_+,\cH)\) be a finite-dimensional \(M_{H^\infty}^*\)-stable subspace. Then, if \(V \neq \{0\}\), there exist \(\omega \in \C_+\) and \(v \in \cH \setminus \{0\}\) such that \(Q_\omega \cdot v \in V\).
\end{lem}
\begin{proof}
Since \(V\) is \(M_{H^\infty}^*\)-stable, we have
\[P_+S_t^*V \subeq V \quad \forall t \geq 0.\]
Further, for \(s,t \geq 0\) and \(f \in V\), by \fref{lem:ProjectionEquation},
one has
\[\left(P_+S_s^*\right)\left(P_+S_t^*\right) = P_+S_{(s+t)}^*.\]
Since \(\left\lVert P_+S_t^*\right\rVert \leq 1\), this implies that
\[M_t \coloneqq (P_+S_t^*) \big|_V^V\]
defines a one-parameter semigroup of contractions. Therefore, by the Lumer–Phillips Theorem (cf.~\cite[Thm. 3.15, Cor. 3.20]{EN00}), there exists an operator \(A \in B(V)\) such that
\[M_t = e^{tA} \quad \forall t \geq 0.\]
As an operator on a finite-dimensional vector space, the operator \(A\) has at least one eigenvalue with a corresponding eigenvector, so there exist \(\mu \in \C\) and \(f \in V \setminus \{0\}\) such that \(A f = \mu f\). This implies that
\[e^{t\mu} f = e^{tA} f = M_t f \quad \forall t \geq 0.\]
Applying Fourier transform, for all \(t \geq 0\), we get
\[e^{t\mu} (\cF f)(p) = (\cF M_t f)(p) = (\cF f)(p+t)\]
for almost every \(p \in \R_+\). This implies that there exists a \(\tilde v \in \cH \setminus \{0\}\) such that
\[(\cF f)(p) = \begin{cases} e^{p\mu} \tilde v & \text{for } p > 0 \\ 0 & \text{for } p \leq 0,\end{cases}\]
using that \(f \in V \subeq H^2(\C_+,\cH)\), which yields that \(\cF f\) is supported on \(\R_+\). Then the fact that \(\cF f \in L^2(\R,\cH)\) implies \(\Re \mu < 0\). Setting \(\omega = -i \overline{\mu}\) we therefore have \(\omega \in \C_+\). Also we set \(v \coloneqq \sqrt{2\pi} \cdot \tilde v \neq 0\). Then, using that \(f \in V \subeq H^2(\C_+,\cH)\), for \(z \in \C_+\), inverse Fourier transformation gives us
\begin{align*}
f(z) &= \frac 1{\sqrt{2\pi}}\int_0^\infty e^{ipz}e^{p\mu} \,dp \cdot \tilde v = \frac{1}{2\pi}\int_0^\infty e^{ip(z-\overline{\omega})} \,dp \cdot v = \frac i{2\pi} \cdot \frac 1{z-\overline{\omega}} \cdot v = Q_\omega(z) \cdot v. \qedhere
\end{align*}
\end{proof}
We will use this lemma to link \(M_{H^\infty}^*\)-stable subspaces to so-called Blaschke--Potapov products, which are defined as follows:
\begin{definition}{\textrm{(cf. \cite{CHL22})}}
Let \(\cH\) be a complex Hilbert space.
\begin{enumerate}[\rm (a)]
\item For \(\omega \in \C_+\) we define \(\phi_\omega \in \Inn(\C_+)\) (cf. \fref{def:Inner}) by
\[\phi_\omega(z) \coloneqq \frac{z-\omega}{z-\overline{\omega}}.\]
\item A \textit{Blaschke--Potapov factor} is a function of the form
\[\phi_\omega P + (1-P) \in \Inn(\C_+,B(\cH))\]
for some \(\omega \in \C_+\) and some finite rank projection \(P \in B(\cH)\).
\item For \(\omega \in \C_+\) and \(v \in \cH\) we define the corresponding Blaschke--Potapov factor
\[\gls*{phiov} \coloneqq (\textbf{1} - P_{\C v}) + \phi_\omega P_{\C v}.\]
\item A \textit{finite Blaschke--Potapov product} is a function of the form
\[u \prod_{n=1}^N (\phi_{\omega_n} P_n + (1-P_n)) \in \Inn(\C_+,B(\cH))\]
with \(\omega_1,\dots,\omega_N \in \C_+\), finite rank projections \(P_1,\dots,P_N \in B(\cH)\) and \(u \in \U(\cH)\).
\item The \textit{degree} of a finite Blaschke--Potapov product
\[\phi = u \prod_{n=1}^N (\phi_{\omega_n} P_n + (1-P_n))\]
is defined as
\[\gls*{degBlaschke} \coloneqq \sum_{n=1}^N \dim P_n \cH \in \N_0.\]
\end{enumerate}
\end{definition}
\begin{remark}
\begin{enumerate}[\rm (a)]
\item In the case \(\cH = \C\) Blaschke--Potapov factors and finite Blaschke--Potapov products are precisely Blaschke factors and finite Blaschke products (cf. \fref{def:Blaschke}).
\item For a finite-dimensional subspace \(V \subeq \cH\), \(\omega \in \C_+\) and \(u \in \U(\cH)\) we have
\[u \left(\phi_\omega P_V + (1-P_V)\right) = \left(\phi_\omega P_{uV} + (1-P_{uV})\right) u.\]
This identity shows that the product of two finite Blaschke--Potapov products is again a finite \mbox{Blaschke--Potapov} product.
\item We will see that the degree is well defined in \fref{prop:FiniteCharacterization}.
\item By definition, for two finite Blaschke--Potapov products \(\phi\) and \(\psi\), one has
\[\deg(\phi \cdot \psi) = \deg(\phi) + \deg(\psi).\]
\item If \(P\) is a finite rank projection and \((b_1,\dots,b_m)\) is an orthogonal basis of \(P\cH\), then
\[\phi_\omega P + (1-P) = \prod_{j=1}^m \phi_{\omega,b_j}. \]
Therefore every finite Blaschke--Potapov product \(\phi\) is of the form
\[\phi = u \prod_{n=1}^{\deg(\phi)} \phi_{\omega_n,v_n}\]
with \(\omega_1,\dots,\omega_{\deg(\phi)} \in \C_+\), \(v_1,\dots,v_{\deg(\phi)} \in \cH\) and \(u \in \U(\cH)\).
\end{enumerate}
\end{remark}
The following lemma will be useful later when analyzing finite Blaschke--Potapov products:
\begin{lem}\label{lem:BlaschkeOrthogonal}
Let \(\cH\) be a complex Hilbert space, \(v \in \cH\) and \(\omega \in \C_+\). Then
\[\left(Q_\omega \cdot v\right)^{\perp_+} = M_{\phi_{\omega,v}} H^2(\C_+,\cH).\]
\end{lem}
\begin{proof}
For \(v = 0\) both sides are equal to \(H^2(\C_+,\cH)\), so we can assume that \(v \neq 0\). One has \(f \in \left(Q_\omega \cdot v\right)^{\perp_+}\), if and only if
\[0 = \braket*{Q_\omega \cdot v}{f} = \braket*{v}{f(\omega)}.\]
Writing \(f = \tilde f \oplus f_v\) with \(\tilde f \in H^2(\C_+,\cH \ominus \C v)\) and \(f_v \in H^2(\C_+,\C v) \cong H^2(\C_+)\), this is equivalent to the function \(f_v\) being zero in \(\omega\), so, by \fref{prop:BlaschkeFactorization}, \(f_v = \phi_\omega \cdot g\) for some \(g \in H^2(\C_+)\). Therefore
\[\left(Q_\omega \cdot v\right)^{\perp_+} = (\textbf{1} - P_{\C v})H^2(\C_+,\cH) + \phi_\omega P_{\C v}H^2(\C_+,\cH) = M_{\phi_{\omega,v}} H^2(\C_+,\cH). \qedhere\]
\end{proof}
Using this lemma, we can connect \(M_{H^\infty}^*\)-stable subspaces to Blaschke--Potapov products:
\begin{prop}\label{prop:FiniteCharacterization}
Let \(\cH\) be a complex Hilbert space and \(V \subeq H^2(\C_+,\cH)\). Then \(V\) is a finite-dimensional \(M_{H^\infty}^*\)-stable subspace, if and only if there exists a finite Blaschke--Potapov product \(\phi\) such that
\[V = \left(M_\phi H^2(\C_+,\cH)\right)^{\perp_+}\]
and in this case
\[\deg(\phi) = \dim V.\]
\end{prop}
\begin{proof}
We first prove by induction over \(N \in \N_0\), that every \(M_{H^\infty}^*\)-stable subspace \(V \subeq H^2(\C_+,\cH)\) with \(\dim V = N\) is of the form \(V = \left(M_\phi H^2(\C_+,\cH)\right)^{\perp_+}\) for some finite Blaschke--Potapov product \(\phi\) with \(\deg(\phi) = N\). For \(N=0\) we have
\[V = \{0\} = \left(H^2(\C_+,\cH)\right)^{\perp_+} = \left(M_\textbf{1}H^2(\C_+,\cH)\right)^{\perp_+}\]
with \(\deg(\textbf{1}) = 0\).

We now assume that the statement is true for some \(N \in \N_0\) and let \(V \subeq H^2(\C_+,\cH)\) be a \mbox{\(M_{H^\infty}^*\)-stable} subspace with \(\dim V = N +1\). Then, by \fref{lem:SpaceLokalDecomp}, there exist \(\omega \in \C_+\) and \(v \in \cH \setminus \{0\}\) with \(Q_\omega \cdot v \in V\). Then, by \fref{lem:BlaschkeOrthogonal}, we get
\[V^{\perp_+} \subeq (\C Q_\omega \cdot v)^{\perp_+} = M_{\phi_{\omega,v}} H^2(\C_+,\cH).\]
Using that \(M_{\phi_{\omega,v}}^{-1} = M_{\phi_{\omega,v}}^*\), this yields
\[M_{\phi_{\omega,v}}^* V^{\perp_+} \subeq H^2(\C_+,\cH).\]
Now, we set
\[W \coloneqq \left(M_{\phi_{\omega,v}}^* V^{\perp_+}\right)^{\perp_+}.\]
Then
\begin{align*}
W &= P_+\left(M_{\phi_{\omega,v}}^* P_+V^\perp\right)^\perp = P_+M_{\phi_{\omega,v}}^* \left( P_+V^\perp\right)^\perp = P_+M_{\phi_{\omega,v}}^* \left(V + H^2(\C_-,\cH)\right) = P_+M_{\phi_{\omega,v}}^* V.
\end{align*}
So, for \(\phi \in H^\infty(\C_+)\), by \fref{lem:ProjectionEquation}, one has
\begin{align*}
P_+ M_\phi^* W &= P_+M_\phi^* P_+ M_{\phi_{\omega,v}}^* V = P_+ M_{\phi_{\omega,v} \cdot \phi}^* V
\\&= P_+ M_{\phi \cdot \phi_{\omega,v}}^* V = P_+ M_{\phi_{\omega,v}}^* P_+M_\phi^* V \subeq P_+ M_{\phi_{\omega,v}}^* V = W
\end{align*}
and therefore \(W\) is \(M_{H^\infty}^*\)-stable. Further, using that
\[M_{\phi_{\omega,v}}^* Q_\omega(z) \cdot v = \frac{z-\overline{\omega}}{z-\omega} \cdot \frac 1{2\pi} \cdot \frac i{z-\overline{\omega}} \cdot v = \frac 1{2\pi} \cdot \frac i{z-\omega} \cdot v = Q_{\overline{\omega}}(z) \cdot v,\]
we have
\begin{align*}
M_{\phi_{\omega,v}}^* V = M_{\phi_{\omega,v}}^* (V \ominus \C Q_\omega \cdot v) \oplus M_{\phi_{\omega,v}}^* (\C Q_\omega \cdot v) = M_{\phi_{\omega,v}}^* (V \ominus \C Q_\omega \cdot v) \oplus \C Q_{\overline{\omega}} \cdot v.
\end{align*}
Therefore, since \(Q_{\overline{\omega}} \in H^2(\C_-)\) and
\[V \ominus \C Q_\omega \cdot v \subeq (\C Q_\omega \cdot v)^{\perp_+} = M_{\phi_{\omega,v}} H^2(\C_+,\cH)\]
we get
\begin{align*}
W = P_+ M_{\phi_{\omega,v}}^* V = P_+ \left(M_{\phi_{\omega,v}}^* (V \ominus \C Q_\omega \cdot v) \oplus \C Q_{\overline{\omega}} \cdot v\right) = M_{\phi_{\omega,v}}^* (V \ominus \C Q_\omega \cdot v),
\end{align*}
which implies
\[\dim W = \dim V - \dim \left(\C Q_\omega \cdot v\right) = (N+1) - 1 = N.\] 
We, therefore, can apply the assumption of our induction and get a finite Blaschke--Potapov product \(\psi\) with \(\deg(\psi) = N\) such that
\(W = \left(M_\psi H^2(\C_+,\cH)\right)^{\perp_+}\). Setting \(\phi \coloneqq \phi_{\omega,v} \cdot \psi\), we get
\[\deg(\phi) = \deg(\phi_{\omega,v} \cdot \psi) = \deg(\psi) + \deg(\phi_{\omega,v}) = N + 1\]
and
\[V = \left(M_{\phi_{\omega,v}} W^{\perp_+}\right)^{\perp_+} = \left(M_{\phi_{\omega,v}} M_\psi H^2(\C_+,\cH)\right)^{\perp_+} = \left(M_\phi H^2(\C_+,\cH)\right)^{\perp_+}.\]

For the converse implication, we assume that \(\phi\) is a finite Blaschke--Potapov product, so
\[\phi = u \prod_{n=1}^N \phi_{\omega_n,v_n}\]
for some \(\omega_1,\dots,\omega_N \in \C_+\), \(v_1,\dots,v_N \in \cH\) and \(u \in \U(\cH)\). Setting
\[W_n \coloneqq u \prod_{j=1}^n \phi_{\omega_j,v_j} \cdot H^2(\C_+,\cH), \qquad n=0,\dots,N,\]
we have \(W_0 \supeq W_1 \supeq \dots \supeq W_N\) and
\begin{align*}
\dim(W_n^\perp \cap W_{n-1}) &= \dim \left(u \prod_{j=1}^{n-1} \phi_{\omega_j,v_j} \cdot \left(\left(\phi_{\omega_n,v_n} \cdot H^2(\C_+,\cH)\right)^\perp \cap H^2(\C_+,\cH)\right)\right)
\\&= \dim (M_{\phi_{\omega_n,v_n}} H^2(\C_+,\cH))^{\perp_+} = \dim \left(\C Q_{\omega_n} \cdot v_n\right) = 1,
\end{align*}
for \(n=1,\dots,N\), using \fref{lem:BlaschkeOrthogonal} in the penultimate step. Then, for \(V \coloneqq \left(M_\phi H^2(\C_+,\cH)\right)^{\perp_+}\), we have
\[V = W_N^\perp \cap W_0 = \bigoplus_{n=1}^N \left(W_n^\perp \cap W_{n-1}\right)\]
and therefore
\begin{align*}
\dim V &= \sum_{n=1}^N \dim \left(W_n^\perp \cap W_{n-1}\right) = \sum_{n=1}^N 1 = N = \deg(\phi) < \infty.
\end{align*}
It remains to show that \(V\) is a \(M_{H^\infty}^*\)-stable subspace, but this follows by \fref{cor:stableCharakterizationRigid}.
\end{proof}
Finally, we will provide a way to calculate the degree of a finite Blaschke--Potapov product. For this, we need the following definition:
\begin{definition}
For a non-empty open subset \(U \subeq \C\), \(\xi \in U\) and \(f \in \cO(U \setminus \{\xi\})\) with Laurent series
\[f(z) = \sum_{n \in \Z} a_n (z-\xi)^n\]
we define the \textit{order of \(f\) in \(\xi\)} by
\[\gls*{Ord}(f,\xi) \coloneqq \min\{n \in \Z: a_n \neq 0\}\]
if that minimum exists. Else we set
\[\mathrm{Ord}(f,\xi) \coloneqq -\infty\]
if \(a_n \neq 0\) for arbitrarily small \(n \in \Z\) and
\[\mathrm{Ord}(f,\xi) \coloneqq \infty\]
if \(a_n = 0\) for all \(n \in \Z\).

Further, for an open subset \(U \subeq \C_\infty\) with \(\infty \in U\) and \(f \in \cO(U \setminus \{\infty\})\), we set
\[\mathrm{Ord}(f,\infty) = \mathrm{Ord}\left(z \mapsto f\left(\frac 1z\right),0\right).\]
\end{definition}
\begin{remark}
One can easily see that
\[\mathrm{Ord}(f \cdot g,\xi) = \mathrm{Ord}(f,\xi) + \mathrm{Ord}(g,\xi)\]
if 
\[\mathrm{Ord}(f,\xi) \neq -\infty \neq \mathrm{Ord}(g,\xi).\]
This especially implies that
\[\mathrm{Ord}(c \cdot f,\xi) = \mathrm{Ord}(f,\xi)\]
for every constant \(c \in \C^\times\).
\end{remark}
\begin{prop}\label{prop:BlaschkeDegFormula}
Let \(\cH\) be a finite-dimensional complex Hilbert space and \(\phi\) be a finite Blaschke--Potapov product in \(H^\infty(\C_+,B(\cH))\). Then
\[\deg(\phi) = \sum_{\xi \in \C_+} \mathrm{Ord}(\det \circ \phi,\xi).\]
\end{prop}
\begin{proof}
For a Blaschke--Potapov factor
\[\phi_\omega P + (1-P)\]
we have
\[\det(\phi_\omega(z) P + (1-P)) = \left(\frac{z-\omega}{z-\overline{\omega}}\right)^{\dim P\cH},\]
so
\[\mathrm{Ord}\left(\det \circ (\phi_\omega P + (1-P)),\xi\right) = \begin{cases} \dim P\cH & \text{if } \xi = \omega \\ 0 & \text{if } \xi \in \C_+ \setminus \{\omega\}\end{cases}\]
and therefore
\[\sum_{\xi \in \C_+} \mathrm{Ord}(\det \circ (\phi_\omega P + (1-P)),\xi) = \dim P\cH.\]
Now, for a finite Blaschke--Potapov product
\[\phi = u \prod_{n=1}^N (\phi_{\omega_n} P_n + (1-P_n))\]
and \(\xi \in \C_+\), we get
\begin{align*}
\mathrm{Ord}(\det \circ \phi,\xi) &= \mathrm{Ord}\left(\det \circ \left(u \prod_{n=1}^N (\phi_{\omega_n} P_n + (1-P_n))\right),\xi\right)
\\&= \mathrm{Ord}\left(\det(u) \cdot \prod_{n=1}^N (\det \circ (\phi_{\omega_n} P_n + (1-P_n))),\xi\right)
\\&= \sum_{n=1}^N \mathrm{Ord}\left(\det \circ (\phi_{\omega_n} P_n + (1-P_n)),\xi\right),
\end{align*}
using that \(\det(u) \in \T\). Therefore
\begin{align*}
\sum_{\xi \in \C_+}\mathrm{Ord}(\det \circ \phi,\xi) &= \sum_{\xi \in \C_+}\sum_{n=1}^N \mathrm{Ord}\left(\det \circ (\phi_{\omega_n} P_n + (1-P_n)),\xi\right)
\\&=\sum_{n=1}^N \sum_{\xi \in \C_+}\mathrm{Ord}\left(\det \circ (\phi_{\omega_n} P_n + (1-P_n)),\xi\right)
\\&=\sum_{n=1}^N \dim P_n \cH = \deg(\phi). \qedhere
\end{align*}
\end{proof}

\subsubsection{The kernel space}
In this subsection, given a regular Pick function \(F \in \mathrm{Pick}(\C_+,B(\cH))_\R\), we will apply our results obtained in the last subsection to the kernel space \(\ker\left(\overline{\lambda}\textbf{1}-M_F^*\right)\) appearing in \fref{thm:kernelAdjointGeneral}. We first prove that, in fact, \(\ker\left(\overline{\lambda}\textbf{1}-M_F^*\right)\) is a \(M_{H^\infty}^*\)-stable subspace:
\begin{lemma}\label{lem:kernelShiftInvariant}
Let \(\cH\) be a Hilbert space and \(F \in \mathrm{Pick}(\C_+,B(\cH))_\R\). Then, for every \(\lambda \in \C_+\) the space \(\ker\left(\overline{\lambda}\textbf{1}-M_F^*\right)\) is \(M_{H^\infty}^*\)-stable.
\end{lemma}
\begin{proof}
Let \(\phi \in H^\infty(\C_+)\) and \(g \in \ker\left(\overline{\lambda}\textbf{1}-M_F^*\right)\). Then, for \(f \in D(M_F)\), we have
\[\braket*{P_+M_\phi^*g}{(\lambda\textbf{1}-M_F)f} = \braket*{g}{M_\phi(\lambda\textbf{1}-M_F)f} = \braket*{g}{(\lambda\textbf{1}-M_F)M_\phi f} = \braket*{(\overline{\lambda}\textbf{1}-M_F^*)g}{M_\phi f} = 0,\]
where we use that \(M_\phi f \in D(F)\) with \(M_F(M_\phi f) = M_\phi M_Ff\). So \(P_+M_\phi^*g \in \ker\left(\overline{\lambda}\textbf{1}-M_F^*\right)\).
\end{proof}
This lemma together with \fref{cor:stableCharakterizationRigid} implies that there exists a function \(\phi \in H^\infty(\C_+,B(\cH))\) such that
\[\ker\left(\overline{\lambda}\textbf{1}-M_F^*\right) = \left(M_\phi H^2(\C_+,\cH)\right)^{\perp_+}.\]
We want to give an explicit formula for \(\phi\). For this, we need the following lemmata:
\begin{lem}\label{lem:compPickBounded}
Let \(\cH\) be a finite-dimensional complex Hilbert space and \(F \in \mathrm{Pick}(\C_+,B(\cH))_\R\) be a regular Pick function. Then
\[\Phi_F: L^\infty(\R,\C) \to L^\infty(\R,B(\cH)), \quad f \mapsto f \circ F_*\]
is a well-defined linear contraction.
\end{lem}
\begin{proof}
For a function \(f \in \cL^\infty(\R,\C)\) the expression \(f \circ F_*\) makes sense as an element in \(L^\infty(\R,B(\cH))\), since for almost every \(x \in \R\) the boundary value \(F_*(x)\) is self-adjoint and therefore has real eigenvalues. Also it is clear that this map is linear and satisfies \(\left\lVert \Phi_F\right\rVert \leq 1\) since \(\left\lVert f \circ F_*\right\rVert_\infty \leq \left\lVert f\right\rVert_\infty\) for every \(f \in \cL^\infty(\R,\C)\). It remains to show that this map factorizes over the quotient map
\[q: \cL^\infty(\R,\C) \to \cL^\infty(\R,\C) \slash \cN = L^\infty(\R,B(\cH)),\]
where, by \(\cN\), we denote the subspace of almost everywhere null functions. So let \(f \in \cN\). Then there is a Lebesgue zero-set \(M\) such that \(f\big|_{\R \setminus M} \equiv 0\) and
\[\{x \in \R: f \circ F_* \neq 0\} \subeq \{x \in \R: \Spec(F_*(x)) \cap M \neq \eset\} = D(F,M),\]
which, by \fref{thm:SpecNonDegZeroSet} is a Lebesgue zero-set. Therefore \(f \circ F_* = 0 \in L^\infty(\R,B(\cH))\).
\end{proof}
\begin{lem}\label{lem:firstCompRelation}
Let \(\cH,\cK\) be complex Hilbert spaces with \(\dim \cH < \infty\) and \(F \in \mathrm{Pick}(\C_+,B(\cH))_\R\) be regular. Further, let
\[\psi: L^2(\R,\cH) \to L^2(\R,\cK)\]
be a unitary map such that
\[\psi \circ U^F_t = S_t \circ \psi \qquad \forall t \in \R,\]
where \(S_t \coloneqq M_{e^{it\mathrm{Id}}}\). Then, for every \(h \in L^\infty(\R,\C)\), one has
\[\psi \circ M_{h \circ F_*} = M_h \circ \psi.\]
\end{lem}
\begin{proof}
The linear maps
\[L^\infty(\R,\C) \to B(L^2(\R,\cH)), \quad h \mapsto \psi^{-1} \circ M_h \circ \psi\]
and
\[L^\infty(\R,\C) \to B(L^2(\R,\cH)), \quad h \mapsto M_{h \circ F_*}\]
are both bounded, where for the boundedness of the second one we use \fref{lem:compPickBounded}. Further, they coincide on the subset \(\{e^{it \Id} : t \in \R\} \subeq L^\infty(\R,\C)\). Since these functions span a weakly dense subspace by \fref{prop:etDense}, the two functions coincide, so, for every \(h \in L^\infty(\R,\C)\), we have
\[\psi^{-1} \circ M_h \circ \psi = M_{h \circ F_*}. \qedhere\]
\end{proof}
\begin{prop}\label{prop:PickKernelForm}
Let \(\cH\) be a finite-dimensional complex Hilbert space, \(F \in \mathrm{Pick}(\C_+,B(\cH))_\R\) be regular and \(\lambda \in \C_+\). Then
\[\ker\left(\overline{\lambda}\textbf{1}-M_F^*\right) = (M_{\phi_\lambda \circ F} H^2(\C_+,\cH))^{\perp_+}.\]
\end{prop}
\begin{proof}
Let \(\cK\) be the multiplicity space of \((L^2(\R,\cH),H^2(\C_+,\cH),U^F)\) and \(\psi: L^2(\R,\cH) \to L^2(\R,\cK)\) be an isomorphism between \((L^2(\R,\cH),H^2(\C_+,\cH),U^F)\)  and \((L^2(\R,\cK),H^2(\C_+,\cK),S)\). Then, by \fref{lem:kernelAdjoint}, we have
\[\psi \ker\left(\overline{\lambda}\textbf{1}-M_F^*\right) = Q_\lambda \cdot \cK,\]
so, by \fref{lem:BlaschkeOrthogonal} and \fref{lem:firstCompRelation}, we have
\begin{align*}
\ker\left(\overline{\lambda}\textbf{1}-M_F^*\right)^{\perp_+} &= \psi^{-1}(Q_\lambda \cdot \cK)^{\perp_+} = \psi^{-1} M_{\phi_\lambda} H^2(\C_+,\cK)
\\&= M_{\phi_\lambda \circ F} \psi^{-1} H^2(\C_+,\cK) = M_{\phi_\lambda \circ F} H^2(\C_+,\cH),
\end{align*}
which proves the equation since the kernel is a closed subspace.
\end{proof}
\begin{cor}\label{cor:PickFiniteCharacterizationBlaschke}
Let \(\cH\) be a finite-dimensional complex Hilbert space, \(F \in \mathrm{Pick}(\C_+,B(\cH))_\R\) be regular and \(\lambda \in \C_+\). Then, one has
\[\dim \ker\left(\overline{\lambda}\textbf{1}-M_F^*\right) < \infty,\]
if and only if \(\phi_\lambda \circ F\) is a finite Blaschke--Potapov product and in this case
\[\dim \ker\left(\overline{\lambda}\textbf{1}-M_F^*\right) = \deg(\phi_\lambda \circ F).\]
\end{cor}
\begin{proof}
By \fref{lem:kernelShiftInvariant} and \fref{prop:FiniteCharacterization} we have that
\[\dim \ker\left(\overline{\lambda}\textbf{1}-M_F^*\right) < \infty,\]
if and only if there exists a finite Blaschke--Potapov product \(\phi\) with \(\deg(\phi) = \dim \ker\left(\overline{\lambda}\textbf{1}-M_F^*\right)\) such that
\[\ker\left(\overline{\lambda}\textbf{1}-M_F^*\right) = \left(M_\phi H^2(\C_+,\cH)\right)^{\perp_+}.\]
By \fref{prop:PickKernelForm} we also have
\[\ker\left(\overline{\lambda}\textbf{1}-M_F^*\right) = \left(M_{\phi_\lambda \circ F} H^2(\C_+,\cH)\right)^{\perp_+},\]
so we get
\[M_\phi H^2(\C_+,\cH) = M_{\phi_\lambda \circ F} H^2(\C_+,\cH)\]
and therefore, by \fref{cor:H2EqualFunctions}, we have
\[\phi_\lambda \circ F = \phi \cdot u\]
for some unitary operator \(u \in \U(\cH)\). This shows that
\[\dim \ker\left(\overline{\lambda}\textbf{1}-M_F^*\right) < \infty,\]
if and only if \(\phi_\lambda \circ F\) is a finite Blaschke--Potapov product with
\[\deg(\phi_\lambda \circ F) = \deg(\phi) = \dim \ker\left(\overline{\lambda}\textbf{1}-M_F^*\right). \qedhere\]
\end{proof}
With these results, we get a characterization of the multiplicity space of a ROPG infinitesimally generated by a Pick function:
\begin{prop}\label{prop:DimByOrd}
Let \(\cH\) be a finite-dimensional complex Hilbert space, \(F \in \mathrm{Pick}(\C_+,B(\cH))_\R\) be regular and let \(\cK\) be the multiplicity space of \((L^2(\R,\cH),H^2(\C_+,\cH),U^F)\). Then \(\dim \cK < \infty\), if and only if for one -- or equivalently for all -- \(\lambda \in \C_+\) the function \(\phi_\lambda \circ F\) is a finite Blaschke--Potapov product. In this case
\[\dim \cK = \deg(\phi_\lambda \circ F) = \sum_{\xi \in \C_+} \mathrm{Ord}(\det \circ \phi_\lambda \circ F,\xi) \quad \forall \lambda \in \C_+.\]
\end{prop}
\begin{proof}
This follows immediately by \fref{cor:PickFiniteCharacterizationBlaschke} and \fref{prop:BlaschkeDegFormula}, since by \fref{thm:kernelAdjointGeneral} we have 
\[\cK \cong \ker\left(\overline{\lambda}\textbf{1}-M_F^*\right) \quad \forall \lambda \in \C_+. \qedhere\]
\end{proof}

\subsubsection{Holomorphic extensions}
To avoid dealing with infinities, in the following, we switch from the upper half-plane \(\C_+\) to the unit disc \(\D\) via Cayley transform, so we consider the holomorphic map:
\[\gls*{h}: \C_+ \to \D, \quad z \mapsto \frac{z-i}{z+i}.\]
Setting \(f_\lambda \coloneqq \det \circ \phi_\lambda \circ F \circ h^{-1}: \D \to \C\) for a regular Pick function \(F \in \mathrm{Pick}(\C_+,B(\cH))_\R\) and \(\lambda \in \C_+\), we would like to write the sum appearing in \fref{prop:DimByOrd}, using the argument principle, as
\[\sum_{\xi \in \C_+} \mathrm{Ord}(\det \circ \phi_\lambda \circ F,\xi) = \sum_{\omega \in \D} \mathrm{Ord}(f_\lambda,\omega) = \frac 1{2\pi i} \int_{f_\lambda \circ \gamma} \frac 1z \,dz\]
for some closed path \(\gamma\) in \(\C\) enclosing the closed unit disc \(\overline{\D}\). For this to be true, we would need that the function \(f_\lambda\) has a holomorphic extension to a simply connected domain containing the closed unit disc \(\overline{\D}\). We will later see that the right notion of extendability is the following:
\begin{definition}\label{def:LambdaExt}
Let \(\cH\) be a finite-dimensional complex Hilbert space and \(\lambda \in \C_+\). We call \(F \in \mathrm{Pick}(\C_+,B(\cH))_\R\) \mbox{\textit{\(\lambda\)-extendable}}, if there exists a finite set \(E_\lambda \subeq \C_-\) such that
\[\phi_\lambda \circ F: \C_+ \to B(\cH)\]
can be holomorphically extended to a function \(F_\lambda: \C_\infty \setminus E_\lambda \to B(\cH)\), such that, for every \(v,w \in \cH\), the matrix coefficient
\[\C_\infty \setminus E_\lambda \ni z \mapsto \braket*{v}{F_\lambda(z)w} \in \C\]
is meromorphic. For a \(\lambda\)-extendable Pick function \(F\), we write \(\gls*{Elmin}\) for the minimal choice of \(E_\lambda\) and \(\gls*{Flmin}\) for the corresponding extension.
\end{definition}
To classify the \(\lambda\)-extendable Pick functions, we need the following lemmata:
\begin{lem}\label{lem:ratBlaschkePotapov}{\rm{(cf. \cite[Cor. 3.7, Def. 3.3]{CHL22})}}
Let \(\cH\) be a complex Hilbert space and let \(\phi: \C_+ \to B(\cH)\) be an inner function. Then \(\phi\) is a finite Blaschke--Potapov product, if and only if there is a finite Blaschke product \(\theta \in \Inn(\C_+)\) (cf. \fref{def:Blaschke}) such that
\[M_\theta H^2(\C_+,\cH) \subeq M_\phi H^2(\C_+,\cH).\]
\end{lem}
\begin{lem}\label{lem:SpectrumPerturbation}
Let \(\cH\) be a finite-dimensional Hilbert space and \(A \in B(\cH)\) be normal. Then, for every \(S \in B(\cH)\), one has
\[\Spec(A+S) \subeq \Spec(A) + U_{\leq \left\lVert S\right\rVert}(0).\]
\end{lem}
\begin{proof}
Let
\[\mu \in \C \setminus (\Spec(A) + U_{\leq \left\lVert S\right\rVert}(0)).\]
Then
\[\dist(\mu,\Spec(A)) > \left\lVert S\right\rVert.\]
Since \(A\) and therefore \(A-\mu \textbf{1}\) is normal, we have
\begin{align*}
\left\lVert(A-\mu \textbf{1})^{-1}\right\rVert^{-1} &= \min\{|\lambda|: \lambda \in \Spec(A-\mu \textbf{1})\}
\\&= \min\{|\lambda-\mu|: \lambda \in \Spec(A)\} = \dist(\mu,\Spec(A)) > \left\lVert S\right\rVert,
\end{align*}
which implies that
\[(A -\mu \textbf{1}) + S = (A+S) - \mu \textbf{1}\]
is invertible and therefore \(\mu \in \C \setminus \Spec(A+S)\).
\end{proof}
We now can formulate our theorem characterizing \(\lambda\)-extendable Pick functions using the following definition:
\begin{definition}\label{def:RationalPick}
Let \(\cH\) be a finite-dimensional complex Hilbert space. We call a Pick function \(F \in \mathrm{Pick}(\C_+,B(\cH))_\R\) \textit{rational}, if
there are \(C \in S(\cH)\), \(D \in S(\cH)_+\) and \(A_1,\dots,A_m \in S(\cH)_+\)
such that
\[F(z) = C + zD + \sum_{j=1}^m \frac {1}{\lambda_j-z}A_j\]
and we call
\[\gls*{degPick} \coloneqq \rk(D) + \sum_{j=1}^m \rk(A_j)\]
the \textit{degree of \(F\)}. Further, we write \(\gls*{Rat}\) for the set of rational Pick functions and set
\[\deg(F) \coloneqq \infty\]
for all \(F \in \mathrm{Pick}(\C_+,B(\cH))_\R \setminus \mathrm{Rat}(\cH)\).
\end{definition}
\begin{thm}\label{thm:holomExt}
Let \(\cH\) be a finite-dimensional complex Hilbert space, \(F \in \mathrm{Pick}(\C_+,B(\cH))_\R\) be regular and \(\lambda \in \C_+\). Then the following are equivalent:
\begin{enumerate}[\rm (a)]
\item \(\phi_\lambda \circ F\) is a finite Blaschke--Potapov product.
\item \(F\) is \(\lambda\)-extendable.
\item \(F \in \mathrm{Rat}(\cH)\).
\end{enumerate}
\end{thm}
\newpage
\begin{proof}
\begin{itemize}
\item[(a) \(\Rightarrow\) (b):]
Let
\[\phi_\lambda \circ F = u \prod_{n=1}^N (\phi_{\omega_n} P_n + (1-P_n))\]
with \(\omega_1,\dots,\omega_N \in \C_+\), projections \(P_1,\dots,P_N \in B(\cH)\) and \(u \in \U(\cH)\). All the functions \(\phi_{\omega_n}\) can be holomorphically extended to \(\C \setminus \{\overline{\omega_n}\}\), so setting \(E \coloneqq \{\overline{\omega_1},\dots,\overline{\omega_N}\} \subeq \C_-\), we get that \(\phi_\lambda \circ F\) can be holomorphically extended to a function on \(\C \setminus E\). To get the holomorphic extension in \(\infty\), we notice that
\[(\phi_\lambda \circ F)\left(\frac 1z\right) = u \prod_{n=1}^N \left(\frac{\frac 1z - \omega_n}{\frac 1z - \overline{\omega_n}} P_n + (1-P_n)\right) = u \prod_{n=1}^N \left(\frac{1 - \omega_n z}{1 - \overline{\omega_n}z} P_n + (1-P_n)\right),\]
which shows that \(\phi_\lambda \circ F\) can be holomorphically extended to \(\infty\) with the value
\[F_\lambda^{\mathrm{min}}(\infty) = u \prod_{n=1}^N \left(\frac{1 - \omega_n \cdot 0}{1 - \overline{\omega_n} \cdot 0} P_n + (1-P_n)\right) = u \prod_{n=1}^N \left(P_n + (1-P_n)\right) = u.\]
That, for every \(v,w \in \cH\), the matrix coefficient
\[\C_\infty \setminus E_\lambda \ni z \mapsto \braket*{v}{F(z)w} \in \C\]
is meromorphic, follows from the fact that all the functions \(\phi_{\omega_n}\), \(n=1,\dots,N\), are meromorphic.
\item[(b) \(\Rightarrow\) (a):]
Since \(F\) is \(\lambda\)-extendable, all the matrix coefficients of \(F_\lambda^{\mathrm{min}}\) are meromorphic and therefore rational functions. Since their poles are located in \(E_\lambda^{\mathrm{min}} \subeq \C_-\), we can write
\[F_\lambda^{\mathrm{min}}(z) = \left(\prod_{\omega \in E_\lambda^{\mathrm{min}}}(z-\omega)^{-m_\omega}\right) P(z),\]
where \(P\) is a matrix-valued function whose entries are polynomials and \(m_\omega \in \N\) for every \(\omega \in E_\lambda^{\mathrm{min}}\). Since \(F_\lambda^{\mathrm{min}}\) can be holomorphically extended to \(\infty\), we know that, for all \(v,w \in \cH\), the degree of the polynomial \(\C \ni z \mapsto \braket*{v}{P(z)w}\) is at most \(\sum_{\omega \in E_\lambda^{\mathrm{min}}}^N m_\omega\). This shows that for the function
\[g(z) = \left(\prod_{\omega \in E_\lambda^{\mathrm{min}}}(z-\overline{\omega})^{-m_\omega}\right) P(z)\]
one has \(g \in H^\infty(\C_-,B(\cH))\). Further, defining the Blaschke product
\[\theta \coloneqq \prod_{\omega \in E_\lambda^{\mathrm{min}}}\phi_{\overline{\omega}}^{m_\omega}\]
we have \(F_\lambda^{\mathrm{min}} = \theta g\). We now define \(h \in H^\infty(\C_+,B(\cH))\) by \(h(z) \coloneqq g(\overline{z})^*\). Then, for \(x \in \R\), we have
\[F_\lambda^{\mathrm{min}}(x) = \theta(x) h(x)^*.\]
Since \(F_\lambda^{\mathrm{min}}\) and \(\theta\) have unitary boundary values, also \(h\) has unitary boundary values and therefore \(\theta = F_\lambda^{\mathrm{min}} h\), which yields
\[M_\theta H^2(\C_+,\cH) = M_{F_\lambda^{\mathrm{min}}} M_h H^2(\C_+,\cH) \subeq M_{F_\lambda^{\mathrm{min}}} H^2(\C_+,\cH) = M_{\phi_\lambda \circ F} H^2(\C_+,\cH).\]
By \fref{lem:ratBlaschkePotapov} this implies that \(\phi_\lambda \circ F\) is a finite Blaschke--Potapov product.
\item[(b) \(\Rightarrow\) (c):]
Since \(F\) is \(\lambda\)-extendable, all the matrix coefficients of \(F_\lambda^{\mathrm{min}}\) are meromorphic and therefore rational functions. This implies that also \(\det \circ \left(F_\lambda^{\mathrm{min}} - \textbf{1}\right)\) is a rational function and therefore the set
\[T \coloneqq \{x \in \R : 1 \in \Spec\left(F_\lambda^{\mathrm{min}}(x)\right)\}\]
is finite. For \(x \in \R \setminus T\), we have \(\phi_\lambda^{-1}\left(F_\lambda^{\mathrm{min}}(x)\right) \in B(\cH)\),
which implies that the limit
\[\lim_{\epsilon \downarrow 0} F(x+i \epsilon) = \phi_\lambda^{-1} \left(\lim_{\epsilon \downarrow 0} F_\lambda^{\mathrm{min}}(x+i \epsilon)\right) = \phi_\lambda^{-1}\left(F_\lambda^{\mathrm{min}}(x)\right)\]
exists. This implies that for every \(v \in \cH\) we have
\[\lim_{\epsilon \downarrow 0} \,\braket*{v}{F(x+i \epsilon)v} \neq \infty,\]
so the measure corresponding to the Pick function
\[\C_+ \ni z \mapsto \braket*{v}{F(z)v},\]
by \fref{thm:GePick} is supported on the finite set \(T\). Together with \cite[Thm. 4.3.4]{Sc20}, this implies that there are operators \((A_t)_{t \in T} \subeq S(\cH)_+\) and an operator \(C \in S(\cH)\) and an operator \(D \in S(\cH)_+\) such that
\[F(z) = C + zD + \sum_{t \in T}\frac 1{t-z}A_t,\]
which implies that \(F\) is a rational Pick function.
\item[(c) \(\Rightarrow\) (b):]
Let \(F \in \mathrm{Rat}(\cH)\), so
\[F(z) = C + zD + \sum_{j=1}^m \frac {1}{\lambda_j-z}A_j\]
for some \(C \in S(\cH)\), \(D \in S(\cH)_+\) and \(A_1,\dots,A_m \in S(\cH)_+\). By replacing \(F\) with \(F \circ m^{-1}\), where \(m\) is an automorphism of the Riemann sphere with \(m(\C_+) = \C_+\), chosen in a suitable way, for example as the Möbius transformation
\[m(z) = \frac 1{\xi-z} \quad \text{for } \xi \in \R \setminus \{\lambda_1,\dots,\lambda_n\},\]
we can, without loss of generality, assume that \(D=0\), so
\[F(z) = C + \sum_{j=1}^m \frac {1}{\lambda_j-z}A_j.\]
Then \(F\) can be holomorphically extended to a function
\[\tilde F: \C_\infty \setminus \{\lambda_1,\dots,\lambda_m\} \to B(\cH).\]
Now, for \(k \in \{1,\dots,m\}\), we define
\[F_k: \C_\infty \setminus \{\lambda_1,\dots,\lambda_{k-1},\lambda_{k+1},\lambda_m\} \to B(\cH), \quad z \mapsto C + \sum_{\substack{j \in \{1,\dots,m\} \\ j\neq k}} \frac {1}{\lambda_j-z}A_j\]
and set
\[K \coloneqq \max\left\{\left\lVert F_k(\lambda_k)\right\rVert : k=1,\dots,m\right\}.\]
Since the spectra \((\Spec(A_k))_{k=1,\dots,m}\) are all finite, there are \(a_1,\dots,a_m>0\) such that
\[\Spec(A_k) \subeq \{0\} \cup [a_k,\infty) \quad \forall k \in \{1,\dots,m\}.\]
We now choose \(\epsilon_1,\dots,\epsilon_m>0\) such that the sets
\[U_{\epsilon_1}(\lambda_1), \dots, U_{\epsilon_m}(\lambda_m)\]
are pairwise disjoint and
\[\left\lVert F_k(z) - F_k(\lambda_k)\right\rVert \leq 1 \quad \text{and} \quad \left|\frac 1{z-\lambda_k}\right|a_k \geq 2K+4\]
for every \(z \in U_{\epsilon_k}(\lambda_k) \setminus \{\lambda_k\}\). Then, for every \(z \in U_{\epsilon_k}(\lambda_k) \setminus \{\lambda_k\}\), one has
\[\left\lVert \tilde F(z) - \frac 1{z-\lambda_k}A_k\right\rVert = \left\lVert F_k(z)\right\rVert \leq \left\lVert F_k(\lambda_k)\right\rVert + \left\lVert F_k(z) - F_k(\lambda_k)\right\rVert \leq K + 1\]
and
\[\Spec\left(\frac 1{z-\lambda_k}A_k\right) \subeq \{0\} \cup \C \setminus U_{\leq \left|\frac 1{z-\lambda_k}\right|a_k}(0) \subeq \{0\} \cup \C \setminus U_{\leq 2K+4}(0).\]
Therefore, by \fref{lem:SpectrumPerturbation}, using that \(A_k\) is self-adjoint and therefore \(\frac 1{z-\lambda_k}A_k\) is normal, we get
\begin{align*}
\Spec(\tilde F(z)) &\subeq \Spec\left(\frac 1{z-\lambda_k}A_k\right) + U_{\leq K+1}(0) \subeq \left(\{0\} \cup \C \setminus U_{\leq 2K+4}(0)\right) + U_{\leq K+1}(0)
\\&= U_{\leq K+1}(0) \cup \C \setminus U_{\leq (2K+4)-(K+1)}(0) = U_{\leq K+1}(0) \cup \C \setminus U_{\leq K+3}(0)
\\&\subeq \C \setminus U_1(-(K+2)i).
\end{align*}
This shows, that, denoting for \(\omega \in \C_+\) by \(\tilde \phi_\omega\) the holomorphic extension of the function \(\phi_\omega\) to \(\C_\infty \setminus \{\overline{\omega}\}\), the function \(\tilde \phi_{(K+2)i} \circ \tilde F\) is bounded on the open set
\[U \coloneqq \bigcup_{k=1}^m U_{\epsilon_k}(\lambda_k) \setminus \{\lambda_k\}.\]
Further, for every \(k=1,\dots,m\), the set \(\{\lambda_k\} \cup U\) is an open neighbourhood of \(\lambda_k\). Then Riemann's Theorem on removable singularities implies that \(\tilde \phi_{(K+2)i} \circ \tilde F\) can be holomorphically extended to \(\{\lambda_1,\dots,\lambda_m\}\), which gives us a holomorphic extension
\[H: \C_\infty \setminus I \to B(\cH)\]
of \(\tilde \phi_{(K+2)i} \circ \tilde F\) and therefore of \(\phi_{(K+2)i} \circ F\), where
\[I \coloneqq \left\{z \in \C: -(K+2)i \in \Spec(\tilde F(z))\right\}.\]
We now notice that \(z \in I\), if and only if
\[\det\left(\tilde F(z) + (K+2)i\right) = 0,\]
which, since \(\det\) is a polynomial in its entries and \(\tilde F\) has meromorphic matrix coefficients, implies that \(I\) is finite. Since
\[\Spec(H(z)) = \Spec(\phi_{(K+2)i}(F(z))) = \phi_{(K+2)i} \left(\Spec(F(z))\right) \subeq \phi_{(K+2)i}(\C_+) = \D \quad \forall z \in \C_+\]
by \fref{thm:SpecNonDeg} and
\[\Spec(H(x)) \subeq \partial \D \quad \forall x \in \R \cup \{\infty\}\]
because \(F\) has self-adjoint boundary values, we have \(I \subeq \C_-\). 

Analogously, for \(\lambda \in \C_+\), one argues that
\[J_\lambda \coloneqq \left\{z \in \C: \overline{\lambda} \in \Spec(\tilde F(z))\right\}\]
is a finite subset of \(\C_-\). This implies that
\[H_\lambda: \C_\infty \setminus (I \cup J_\lambda) \to B(\cH), \quad z \mapsto (\tilde \phi_\lambda \circ \tilde \phi_{(K+2)i}^{-1} \circ H)(z)\]
defines an holomorphic extension of \(\phi_\lambda \circ F\) to \(\C_\infty \setminus (I \cup J_\lambda)\) with \(I \cup J_\lambda\) being finite.
Since \(\phi_\lambda\) and all matrix coefficients of \(F\) are rational functions, also all matrix coefficients of \(\phi_\lambda \circ F\) are rational functions, which implies that \(F\) is \(\lambda\)-extendable. \qedhere
\end{itemize}
\end{proof}
\begin{cor}\label{cor:FiniteCharacterizationIntegral}
Let \(\cH\) be a finite-dimensional complex Hilbert space, \(F \in \mathrm{Pick}(\C_+,B(\cH))_\R\) be regular and let \(\cK\) be the multiplicity space of \((L^2(\R,\cH),H^2(\C_+,\cH),U^F)\). Then \(\dim \cK < \infty\), if and only if \(F \in \mathrm{Rat}(\cH)\). In this case, for every \(\lambda \in \C_+\), there exists \(R_\lambda > 1\) such that the function
\[f_\lambda \coloneqq \det \circ \phi_\lambda \circ F \circ h^{-1}: \D \to \C\]
has a holomorphic extension \(\tilde f_\lambda: R_\lambda \D \to \C\) and then for every \(r \in (1,R_\lambda)\) one has
\[\dim \cK = \frac 1{2\pi i} \int_{\tilde f_\lambda \circ \gamma_r}\frac 1z\,dz\]
for the closed path
\[\gamma_r: [0,2\pi] \to \C^\times, \quad t \mapsto r e^{it}.\]
\end{cor}
\begin{proof}
By \fref{prop:DimByOrd} and \fref{thm:holomExt} we immediately get that \(\dim \cK < \infty\), if and only if \(F \in \mathrm{Rat}(\cH)\) and that in this case, for every \(\lambda \in \C_+\), one has
\[\dim \cK = \sum_{\xi \in \C_+} \mathrm{Ord}(\det \circ \phi_\lambda \circ F,\xi).\]
Further, by \fref{thm:holomExt}, the function \(F\) is \(\lambda\)-extendable, which immediately implies that there exists a finite subset \(E \subeq \C \setminus \overline{\D}\) such that the function \(f_\lambda\) can be holomorphically extended to \(\C \setminus E\). Setting
\[R_\lambda \coloneqq 1 + \dist(E,\overline{\D}) \in (1,\infty),\]
we have that \(R_\lambda\D \subeq \C \setminus E\) and therefore there exists a holomorphic extension \(\tilde f_\lambda\) of \(f_\lambda\) to \(R_\lambda\D\). For this extension, we get
\[\dim \cK = \sum_{\xi \in \C_+} \mathrm{Ord}(\det \circ \phi_\lambda \circ F,\xi) = \sum_{\omega \in \D} \mathrm{Ord}(\det \circ \phi_\lambda \circ F \circ h^{-1},\omega) = \sum_{\omega \in \D} \mathrm{Ord}(\tilde f_\lambda,\omega).\]
Further, by \fref{prop:DimByOrd} the function \(\phi_\lambda \circ F\) is a Blaschke--Potapov product and therefore no holomorphic extension of \(\det \circ \phi_\lambda \circ F\) has a zero on \(\C_\infty \setminus \D\), which implies that \(\tilde f_\lambda\) has no zeros on \(R_\lambda\D \setminus \D\). For \(r \in (1,R_\lambda)\) the argument principle then gives us
\begin{align*}
\dim \cK = \sum_{\omega \in \D} \mathrm{Ord}(\tilde f_\lambda,\omega) = \frac 1{2\pi i} \int_{\partial(r\D)}\frac{\tilde f_\lambda'(\xi)}{\tilde f_\lambda(\xi)}\,d\xi = \frac 1{2\pi i} \int_{\tilde f_\lambda \circ \gamma_r}\frac 1z \,dz
\end{align*}
for the closed path
\[\gamma_r: [0,2\pi] \to \C^\times, \quad t \mapsto r e^{it}. \qedhere\]
\end{proof}

\subsubsection{The multiplicity space and the degree of Pick functions}
In this subsection, we will provide a way to express the degree of a regular Pick function as an integral in a similar fashion as the integral
\[\frac 1{2\pi i} \int_{\tilde f_\lambda \circ \gamma_r}\frac 1z \,dz\]
appearing in \fref{cor:FiniteCharacterizationIntegral}, and finally, we will compare these two integrals to obtain an explicit formula for the multiplicity space of ROPGs that are infinitesimally generated by a Pick function. We start with the following lemma:
\begin{lemma}\label{lem:RatPoleOrder}
Let \(\cH\) be a finite-dimensional complex Hilbert space and \(F \in \mathrm{Rat}(\cH)\) be regular. Let \(C \in S(\cH)\), \(D \in S(\cH)_+\) and \(A_1,\dots,A_m \in S(\cH)_+\) and \(\lambda_1,\dots,\lambda_m \in \R\) such that
\[F(z) = C + zD + \sum_{j=1}^m \frac {1}{\lambda_j-z}A_j, \quad z \in \C_+\]
and let \(\tilde F\) be the holomorphic extension of \(F\) to \(\C \setminus \{\lambda_1,\dots,\lambda_m\}\). Then, for \(\lambda \in \C_+\), one has
\[\mathrm{Ord}\left(\det \circ (\tilde F-\overline{\lambda} \textbf{1}),x\right) = \begin{cases} -\mathrm{rank}(D) & \text{if } x=\infty \\ -\mathrm{rank}(A_k) & \text{if } x=\lambda_k \\ 0 & \text{if } x \in \R \setminus \{\lambda_1,\dots,\lambda_m\}.\end{cases}\]
\end{lemma}
\begin{proof}
For \(x \in \R \setminus \{\lambda_1,\dots,\lambda_m\}\), we have \(\tilde F(x) \in S(\cH)\) and therefore \(\det(\tilde F(x)-\overline{\lambda} \textbf{1}) \neq 0\), so \(\mathrm{Ord}(\det \circ (\tilde F-\overline{\lambda} \textbf{1}),x) = 0\).

For \(x = \infty\), let \(P\) be the projection onto \(\ker(D)\) and set
\[Q(z) \coloneqq P + \frac 1z(\textbf{1}-P)\]
Then
\[\det(Q(z)) = \left(\frac 1z\right)^{\mathrm{rank}(D)},\]
so
\[\mathrm{Ord}\left(\det \circ Q,\infty\right) = \mathrm{rank}(D).\]
Further, defining the partial isometries
\[p \coloneqq P\big|^{\ker(D)}: \cH \to \ker(D) \quad \text{and} \quad q \coloneqq (1-P)\big|^{\ker(D)^\perp}: \cH \to \ker(D)^\perp\]
and using that \(Q(z) \cdot zD = D = (1-P)D\), we have
\begin{align*}
\det\left(Q(z)\left(\tilde F(z)-\overline{\lambda} \textbf{1}\right)\right) &=\det\left(Q(z)\left(C + zD + \sum_{j=1}^m \frac {1}{\lambda_j-z}A_j-\overline{\lambda} \textbf{1}\right)\right)
\\&= \det\left(Q(z)\left(C + \sum_{j=1}^m \frac {1}{\lambda_j-z}A_j-\overline{\lambda} \textbf{1}\right)+(1-P)D\right)
\\&\xrightarrow{|z| \to \infty} \det\left(P(C-\overline{\lambda} \textbf{1}) + (1-P)D\right) = \det(p(C-\overline{\lambda} \textbf{1})p^*) \cdot \det(qDq^*)
\\&=\det(pCp^*-\overline{\lambda} \textbf{1}) \cdot \det(qDq^*) \neq 0,
\end{align*}
where the last inequality follows since \(\det(qDq^*) \neq 0\) by definition of \(q\) and \(\det(pCp^*-\overline{\lambda} \textbf{1}) \neq 0\) follows by the fact that \(pCp^*\) is self-adjoint and therefore just has real eigenvalues. This implies that
\[\mathrm{Ord}\left(\det \circ \left(Q \cdot \left(\tilde F-\overline{\lambda} \textbf{1}\right)\right),\infty\right) = 0\]
and therefore
\[\mathrm{Ord}\left(\det \circ \left(\tilde F-\overline{\lambda} \textbf{1}\right),\infty\right) = -\mathrm{Ord}\left(\det \circ Q,\infty\right) = -\mathrm{rank}(D).\]
Now, for \(x=\lambda_k\), we consider the Möbius transform
\[m(z) \coloneqq \lambda_k - \frac 1z.\]
Then, using that
\begin{align*}
\frac 1{\lambda_j-\left(\lambda_k - \frac 1z\right)} &= \frac 1{\left(\lambda_j-\lambda_k\right) + \frac 1z} = \frac 1{\lambda_j-\lambda_k} \cdot \frac z{z + \frac 1{\lambda_j- \lambda_k}}
\\&= \frac 1{\lambda_j-\lambda_k} \left(1 - \frac {\frac 1{\lambda_j- \lambda_k}}{z + \frac 1{\lambda_j- \lambda_k}}\right) = \frac 1{\lambda_j-\lambda_k} - \frac {\frac 1{(\lambda_j- \lambda_k)^2}}{z + \frac 1{\lambda_j- \lambda_k}} = \frac 1{\lambda_j-\lambda_k} + \frac {\frac 1{(\lambda_j- \lambda_k)^2}}{\frac 1{\lambda_k- \lambda_j}-z},
\end{align*}
we have
\begin{align*}
(\tilde F \circ m)(z) &= C + \left(\lambda_k - \frac 1z\right)D + \sum_{j=1}^m \frac 1{\lambda_j-\left(\lambda_k - \frac 1z\right)}A_j
\\&=\left(C + \lambda_k D + \sum_{\substack{j \in \{1,\dots,m\} \\ j \neq k}} \frac 1{\lambda_j- \lambda_k}A_j\right) + zA_k + \frac 1{0-z}D + \sum_{\substack{j \in \{1,\dots,m\} \\ j \neq k}} \frac {\frac 1{(\lambda_j- \lambda_k)^2}}{\frac 1{\lambda_k- \lambda_j}-z}A_j,
\end{align*}
so by the previous argument, using that \(F \circ m\) is again regular by \fref{cor:ConstNonDeg} and \fref{cor:CompNonDeg}, we get
\begin{align*}
-\mathrm{rank}(A_k) &= \mathrm{Ord}\left(\det \circ \left(\tilde F \circ m-\overline{\lambda} \textbf{1}\right),\infty\right)
\\&= \mathrm{Ord}\left(\det \circ \left(\tilde F-\overline{\lambda} \textbf{1}\right)\circ m,m^{-1}(\lambda_k)\right) = \mathrm{Ord}\left(\det \circ \left(\tilde F-\overline{\lambda} \textbf{1}\right),\lambda_k\right). \qedhere
\end{align*}
\end{proof}
\begin{prop}\label{prop:DegFIntegralFormula}
Let \(\cH\) be a finite-dimensional complex Hilbert space, \(F \in \mathrm{Rat}(\cH)\) be regular and \(\lambda \in \C_+\). Then there exists \(M_\lambda > 1\) and a meromorphic extension \(\tilde g_\lambda\) on \(M_\lambda\D\) of the function
\[g_\lambda \coloneqq \det \circ (F - \overline{\lambda}) \circ h^{-1}: \D \to \C\]
such that \(\tilde g_\lambda\) has no zeros or poles on \(M_\lambda\D \setminus \overline{\D}\). Then, for every \(r \in (1,M_\lambda)\), one has
\[\deg(F) = -\frac 1{2\pi i} \int_{\tilde g_\lambda \circ \gamma_r}\frac 1z\,dz\]
for the closed path
\[\gamma_r: [0,2\pi] \to \C^\times, \quad t \mapsto r e^{it}.\]
\end{prop}
\begin{proof}
Let \(C \in S(\cH)\), \(D \in S(\cH)_+\) and \(A_1,\dots,A_m \in S(\cH)_+\) and \(\lambda_1,\dots,\lambda_m \in \R\) such that
\[F(z) = C + zD + \sum_{j=1}^m \frac {1}{\lambda_j-z}A_j, \quad z \in \C_+\]
and let \(\tilde F\) be the holomorphic extension of \(F\) to \(\C \setminus \{\lambda_1,\dots,\lambda_m\}\). Then all poles of \(F\) are located on \(\R \cup \{\infty\}\) and by \fref{thm:SpecNonDeg} we have
\[\Spec(F(z)) \subeq \C_+ \quad \forall z \in \C_+.\]
This shows that defining the holomorphic extension
\[g \coloneqq \det \circ (\tilde F - \overline{\lambda}) \circ h^{-1}: \C \setminus \{1,h(\lambda_1),\dots,h(\lambda_m)\} \to \C\]
of \(g_\lambda\), all zeroes of \(g\) are located in \(\C \setminus \overline{\D}\) and since
\[\tilde F(\overline{z}) = \tilde F(z)^*, \quad z \in \C_+,\]
all poles of \(g\) are located in \(\partial \D\). Further, since \(F \in \mathrm{Rat}(\cH)\) the function \(g\) is rational and therefore the set \(g^{-1}(\{0\})\) is finite. The first statement then follows by setting
\[M_\lambda \coloneqq 1 + \dist\left(\D,g^{-1}(\{0\})\right)\]
and
\[\tilde g_\lambda \coloneqq g\big|_{M_\lambda \D}.\]
Now, using that \(\tilde g_\lambda\) has no zeros and that all its poles are located on \(\partial \D\), for every \(r \in (1,M_\lambda)\), the argument principle gives us
\begin{align*}
-\frac 1{2\pi i} \int_{\tilde g_\lambda \circ \gamma_r}\frac 1z\,dz &= -\sum_{\omega \in \partial \D} \mathrm{Ord}\left(\det \circ \left(\tilde F-\overline{\lambda} \textbf{1}\right) \circ h^{-1},\omega\right)
\\&= -\sum_{x \in \R \cup \{\infty\}} \mathrm{Ord}\left(\det \circ \left(\tilde F-\overline{\lambda} \textbf{1}\right),x\right) = \mathrm{rank}(D) + \sum_{j=1}^m \mathrm{rank}(A_j) = \deg(F),
\end{align*}
using \fref{lem:RatPoleOrder} in the penultimate step.
\end{proof}
Now, we can finally state our theorem:
\begin{thm}\label{thm:BigLWTheorem}
Let \(\cH\) be a finite-dimensional complex Hilbert space, \(F \in \mathrm{Pick}(\C_+,B(\cH))_\R\) be regular and let \(\cK\) be the multiplicity space of \((L^2(\R,\cH),H^2(\C_+,\cH),U^F)\). Then
\[\dim \cK = \deg F,\]
so
\[\cK \cong \begin{cases} \C^{\deg F} & \text{if } F \in \mathrm{Rat}(\cH) \\ \ell^2(\N,\C) & \text{if } F \notin \mathrm{Rat}(\cH).\end{cases}\]
\end{thm}
\begin{proof}
We first consider the case \(F \in \mathrm{Rat}(\cH)\). Let \(\lambda \in \C_+\). We choose \(R_\lambda\) and \(\tilde f_\lambda\) as in \fref{cor:FiniteCharacterizationIntegral} and \(M_\lambda\) and \(\tilde g_\lambda\) as in \fref {prop:DegFIntegralFormula}, so \(\tilde f_\lambda\) is an extension of 
\[f_\lambda \coloneqq \det \circ \phi_\lambda \circ F \circ h^{-1}\]
and 
\(\tilde g_\lambda\) is an extension of 
\[g_\lambda \coloneqq \det \circ (F - \overline{\lambda}) \circ h^{-1}\]
and for every \(r \in (1,R)\) with \(R \coloneqq \min\{R_\lambda,M_\lambda\}\) we have
\[\dim \cK = \frac 1{2\pi i} \int_{\tilde f_\lambda \circ \gamma_r}\frac 1z\,dz \quad \text{and} \quad \deg(F) = -\frac 1{2\pi i} \int_{\tilde g_\lambda \circ \gamma_r}\frac 1z\,dz\]
for the closed path
\[\gamma_r: [0,2\pi] \to \C^\times, \quad t \mapsto r e^{it}.\]
We now fix \(r \in (1,R)\) and consider the homotopy of paths
\[H: [0,1] \times [0,2\pi] \to \C, \quad (s,t) \mapsto \det \left(\left(F_\lambda^{\mathrm{min}} \circ h^{-1}\right)(re^{it})-s\textbf{1}\right),\]
where \(F_\lambda^{\mathrm{min}}\) is defined as in \fref{def:LambdaExt}. For \(z \in \D\) we have
\[\det \left(\left(F_\lambda^{\mathrm{min}} \circ h^{-1}\right)(z)-0 \cdot \textbf{1}\right) = \det \left(\left(\phi_\lambda \circ F \circ h^{-1}\right)(z)\right) = f_\lambda(z)\]
and using that
\[\phi_\lambda(\omega)-1 = \frac{\omega-\lambda}{\omega-\overline{\lambda}} - 1 = \frac{\overline{\lambda}-\lambda}{\omega-\overline{\lambda}}\]
we get
\begin{align*}
\det \left(\left(F_\lambda^{\mathrm{min}} \circ h^{-1}\right)(z)-1 \cdot \textbf{1}\right) &= \det \left(\left(\phi_\lambda \circ F \circ h^{-1}\right)(z)-\textbf{1}\right) = \det \left(\left(\frac{\overline{\lambda}-\lambda}{F-\overline{\lambda} \textbf{1}} \circ h^{-1}\right)(z)\right)
\\&= (\overline{\lambda}-\lambda)^{\dim \cH} \cdot \det \left(\left(\left(F-\overline{\lambda} \textbf{1}\right) \circ h^{-1}\right)(z)\right)^{-1} = (\overline{\lambda}-\lambda)^{\dim \cH} \cdot g_\lambda^{-1}.
\end{align*}
Therefore
\[H(0,\cdot) = \tilde f_\lambda \circ \gamma_r \qquad \text{and} \qquad H(1,\cdot) = (\overline{\lambda}-\lambda)^{\dim \cH} \cdot \tilde g_\lambda^{-1} \circ \gamma_r,\]
which especially implies that
\[H(\{1\} \times [0,2\pi]) \subeq \C^\times.\]
We now show that also \(H([0,1) \times [0,2\pi]) \subeq \C^\times\). For this we notice that, for every \(B(\cH)\)-valued Blaschke-Potapov factor \(\phi_\omega P + (1-P)\) and every \(v \in \cH\) and \(z \in \C_- \setminus \{\overline{\omega}\}\), one has
\[|\phi_\omega(z)| = \left|\frac{z-\omega}{z-\overline{\omega}}\right| \geq 1\]
and therefore
\[\left\lVert(\phi_\omega(z) P + (1-P))v\right\rVert^2 = |\phi_\omega(z)|^2 \cdot \left\lVert Pv\right\rVert^2 + \left\lVert (1-P)v\right\rVert^2 \geq \left\lVert Pv\right\rVert^2 + \left\lVert (1-P)v\right\rVert^2 = \left\lVert v\right\rVert^2.\]
This implies that also for the function \(F_\lambda^{\mathrm{min}}\), which is a finite Blaschke-Potapov product by \fref{thm:holomExt}, one has
\[\left\lVert F_\lambda^{\mathrm{min}}(z)v\right\rVert \geq \left\lVert v\right\rVert\]
and therefore
\[\Spec \left(F_\lambda^{\mathrm{min}}(z)\right) \subeq \C \setminus \D\]
for all \(v \in \cH\) and \(z \in \C_- \setminus E_\lambda^{\mathrm{min}}\).
This yields
\[\Spec \left(\left(F_\lambda^{\mathrm{min}} \circ h^{-1}\right)(z)\right) \subeq \C \setminus \D \quad \forall z \in \C \setminus \left(\overline{\D} \cup h(E_\lambda^{\mathrm{min}})\right) \subeq R\D \setminus \overline{\D},\]
so especially
\[[0,1) \cap \Spec \left(\left(F_\lambda^{\mathrm{min}} \circ h^{-1}\right)(z)\right) = \eset \quad \forall z \in R\D \setminus \overline{\D}.\]
This implies
\[H([0,1) \times [0,2\pi]) \subeq \C^\times,\]
so we have
\[H([0,1] \times [0,2\pi]) \subeq \C^\times.\]
Then, by the homotopy invariance of the integral, we get
\begin{align*}
\int_{\tilde f_\lambda \circ \gamma_r}\frac 1z\,dz &= \int_{H(0,\cdot)}\frac 1z\,dz = \int_{H(1,\cdot)}\frac 1z\,dz
\\&= \int_{(\overline{\lambda}-\lambda)^{\dim \cH} \cdot \tilde g_\lambda^{-1} \circ \gamma_r}\frac 1z\,dz = \int_{\tilde g_\lambda^{-1} \circ \gamma_r}\frac 1z\,dz = -\int_{\tilde g_\lambda \circ \gamma_r}\frac 1z\,dz,
\end{align*}
which yields
\[\dim \cK = \frac 1{2\pi i} \int_{\tilde f_\lambda \circ \gamma_r}\frac 1z\,dz = -\frac 1{2\pi i} \int_{\tilde g_\lambda \circ \gamma_r}\frac 1z\,dz = \deg(F).\]
Now, in the case \(F \notin \mathrm{Rat}(\cH)\), by definition we have \(\deg(F) = \infty\) and by \fref{prop:DimByOrd} we also have \(\dim \cK = \infty\). Further, by assumption, one has \(L^2(\R,\cH) \cong L^2(\R,\cK)\) and since \(L^2(\R,\cH)\) is separable, also \(L^2(\R,\cK)\) has to be separable, which implies that also \(\cK\) is separable and therefore \(\cK \cong \ell^2(\N,\C)\).
\end{proof}
After giving a formula for the multiplicity space for complex ROPGs that are infinitesimally generated by Pick functions, we want to do the same for real ROPGs:
\begin{thm}\label{thm:BigLWTheoremReal}
Let \(\cH\) be a finite-dimensional real Hilbert space, \(F \in \mathrm{Pick}(\C_+,B(\cH_\C))_\R\) be regular with \(F = F^\sharp\) and let \(\cK\) be the multiplicity space of \((L^2(\R,\cH_\C)^\sharp,H^2(\C_+,\cH_\C)^\sharp,U^F)\). Then
\[\dim \cK = \deg F,\]
so
\[\cK \cong \begin{cases} \R^{\deg F} & \text{if } F \in \mathrm{Rat}(\cH_\C) \\ \ell^2(\N,\R) & \text{if } F \notin \mathrm{Rat}(\cH_\C).\end{cases}\]
\end{thm}
\begin{proof}
This follows immediately by \fref{thm:RealToComplexPick} and \fref{thm:BigLWTheorem}.
\end{proof}
\begin{remark}
For a Pick function \(F \in \mathrm{Rat}(\cH_\C)\) of the form
\[F(z) = C + zD + \sum_{j=1}^m \frac {1}{\lambda_j-z}A_j, \quad z \in \C_+\]
one has
\[F^\sharp(z) = \left(-\cC_\cH C\cC_\cH\right) +z \left(\cC_\cH D\cC_\cH\right) + \sum_{j=1}^m \frac {1}{-\lambda_j-z} \left(\cC_\cH A_j\cC_\cH\right), \quad z \in \C_+.\]
Therefore the condition \(F = F^\sharp\) is fulfilled, if and only if \(F\) is of the form
\[F(z) = C + zD - \frac 1z A_0 + \sum_{j=1}^m \frac {1}{\lambda_j-z}A_j + \sum_{j=1}^m \frac {1}{-\lambda_j-z}\left(\cC_\cH A_j\cC_\cH\right), \quad z \in \C_+\]
with \(C \in S(\cH)\), \(D \in S(\cH)_+\), \(A_0 \in S(\cH)_+\), \(A_1,\dots,A_m \in S(\cH)_+\) and \(\lambda_1,\dots,\lambda_m \in \R_+\)
such that
\[-\cC_\cH C\cC_\cH = C, \quad \cC_\cH D\cC_\cH = D \quad \text{and} \quad \cC_\cH A_0\cC_\cH = A_0.\]
\end{remark}
\begin{example}
We consider the function
\[F: \C_+ \to \mathrm{Mat}(2 \times 2,\R), \quad z \mapsto \begin{pmatrix}1 & 1 \\ 1 & z\end{pmatrix},\]
which is a function in \(\mathrm{Pick}(\C_+,B(\cH))_\R\), since for \(z \in \C_+\), one has
\[\Im(F(z)) = \begin{pmatrix}0 & 0 \\ 0 & \Im(z)\end{pmatrix} \geq 0\]
and for every \(x \in \R\) the matrix
\[\lim_{\epsilon \to 0}\begin{pmatrix}1 & 1 \\ 1 & x+i\epsilon\end{pmatrix} = \begin{pmatrix}1 & 1 \\ 1 & x\end{pmatrix}\]
is self-adjoint. Further, for \(z \in \C_+\) and \(t \in \R\), we have
\[\det(t\textbf{1}-F(z)) = (t-1)(t-z)-1 = (t^2 - t - 1) + z(1-t),\]
which implies
\[\Im(\det(t\textbf{1}-F(z))) = (1-t) \Im(z),\]
so \(\Im(\det(t\textbf{1}-F(z))) = 0\), if and only if \(t=1\), but for \(t=1\), we have
\[\det(\textbf{1}-F(z)) = -1 \neq 0,\]
so, for every \(z \in \C_+\), we have
\[\Spec(F(z)) \cap \R = \eset.\]
Since by \fref{prop:SpecImPos}, we have \(\Spec(F(z)) \subeq \overline{\C_+}\), we get \(\Spec(F(z)) \subeq \C_+\) and therefore \(F\) is regular by \fref{thm:SpecNonDeg}. Since
\[F(z) = \begin{pmatrix}1 & 1 \\ 1 & 0\end{pmatrix} + z \begin{pmatrix}0 & 0 \\ 0 & 1\end{pmatrix},\]
by \fref{thm:BigLWTheorem}, for the multiplicity space \(\cK\) of \((L^2(\R,\C^2),H^2(\C_+,\C^2),U^F)\) we get
\[\dim \cK = \deg(F) = \mathrm{rank}\left(\begin{pmatrix}0 & 0 \\ 0 & 1\end{pmatrix}\right) = 1\]
and therefore \(\cK \cong \C\).

One can also see this by considering the function \(\phi_\lambda \circ F\) for \(\lambda \in \C_+\). By \fref{prop:DimByOrd}, the function \(\phi_\lambda \circ F\) should be a Blaschke--Potapov product with
\[\deg(\phi_\lambda \circ F) = \dim \cK = 1.\]
In fact, one can calculate that
\[\phi_\lambda \circ F = u \cdot \left(\phi_\omega \cdot P + (\textbf{1}-P)\right)\]
for
\[\omega \coloneqq \lambda + \frac 1{1-\lambda} \in \C_+, \qquad u \coloneqq \begin{pmatrix} \frac{1-\lambda}{1-\overline{\lambda}} & 0 \\ 0 & 1 \end{pmatrix} \in U(\C^2)\]
and the projection
\[P \coloneqq \frac 1{1 + \frac 1{|1-\lambda|^2}}\begin{pmatrix}\frac 1{|1- \lambda|^2} & \frac{-1}{1- \lambda} \\ \frac{-1}{1- \overline{\lambda}} & 1 \end{pmatrix} \in B(\C^2)\]
with \(\mathrm{rank}(P) = 1\).

In \fref{thm:PickMultiplicity}, we have seen that in the case of a scalar-valued Pick function, the dimension of the multiplicity space is exactly the number of copies of \(\R\) appearing in the boundary values. This also reflects in this matrix-valued example since every real number appears exactly one time as an eigenvalue on the boundary as the following graph shows:
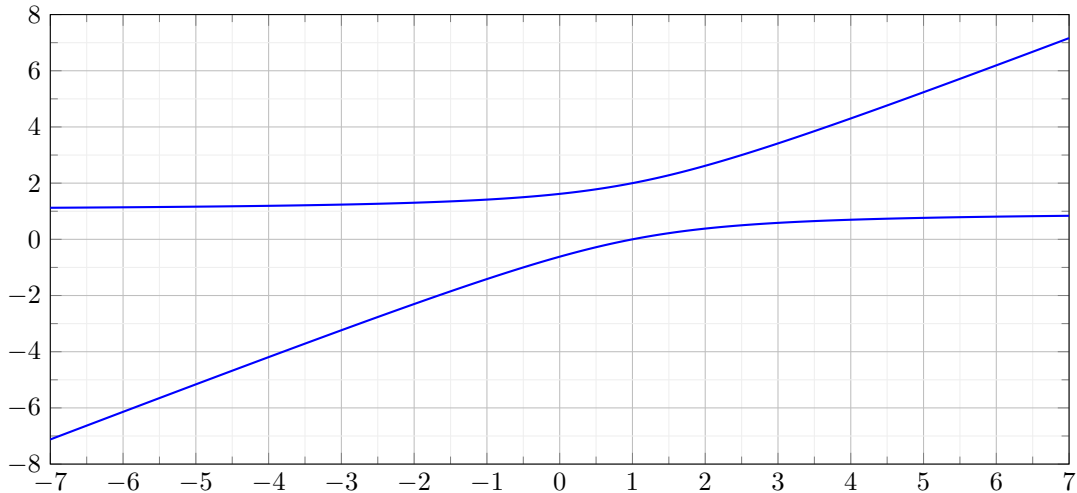
\begin{figure}[H]
\begin{tikzpicture}
\begin{axis}[
    xmin = -7, xmax = 7,
    ymin = -8, ymax = 8,
    xtick distance = 1,
    ytick distance = 2,
    grid = both,
    minor tick num = 1,
    major grid style = {lightgray},
    minor grid style = {lightgray!25},
    width = \textwidth,
    height = 0.5\textwidth]
    \addplot[
        domain = -7:7,
        samples = 200,
        smooth,
        thick,
        blue,
    ] {(x+1+sqrt(x*x-2*x+5))/2};
        \addplot[
        domain = -7:7,
        samples = 200,
        smooth,
        thick,
        blue,
    ] {(x+1-sqrt(x*x-2*x+5))/2};
\end{axis}
\end{tikzpicture}
\caption*{Plot of the eigenvalues of \(x \mapsto \begin{pmatrix}1 & 1 \\ 1 & x\end{pmatrix}\)}
\end{figure}
\end{example}

\subsection{The degree of Pick functions and composition}
In \fref{cor:CompNonDeg} we have seen that for regular Pick functions \(f,g \in \mathrm{Pick}(\C_+)_\R\) and a regular Pick function \(F \in \mathrm{Pick}(\C_+,B(\cH))_\R\) the composition \(f \circ F \circ g\) is again regular. By \fref{thm:BigLWTheorem} we know that the dimension of the multiplicity space of the complex ROPG \((L^2(\R,\cH),H^2(\C_+,\cH),U^{f \circ F \circ g})\) is given by \(\deg(f \circ F \circ g)\). In this section, we want to answer the question of whether there is an easy way to calculate the degree \(\deg(f \circ F \circ g)\) if one knows the degrees of the Pick functions \(f,g\) and \(F\). In other words: We will investigate how the degree behaves under composition. For this, we make some definitions:
\begin{definition}
Let \(\cH,\cK\) be Hilbert spaces. Then, we define the canonical unitary maps
\[\gls*{TH}: L^2(\R,\C)  \,\hat{\otimes}\, \cH \to L^2(\R,\cH), \quad f \otimes v \mapsto f \cdot v\]
and
\[\gls*{IHK} \coloneqq T_{\cH  \,\hat{\otimes}\, \cK} \circ \left(T_\cH^{-1} \otimes \textbf{1}\right): L^2(\R,\cH)  \,\hat{\otimes}\, \cK \to L^2(\R,\cH  \,\hat{\otimes}\, \cK).\]
Finally, considering the unitary map
\[\gls*{sHK}: \cH  \,\hat{\otimes}\, \cK \to \cK  \,\hat{\otimes}\, \cH, \quad v \otimes w \mapsto w \otimes v,\]
we define
\[\gls*{EHK} \coloneqq \left(T_\cK \otimes \textbf{1}\right) \circ \left(\textbf{1} \otimes \sigma_\cH^\cK\right) \circ \left(T_\cH^{-1} \otimes \textbf{1}\right): L^2(\R,\cH)  \,\hat{\otimes}\, \cK \to L^2(\R,\cK)  \,\hat{\otimes}\, \cH.\]
\end{definition}
\begin{lemma}\label{lem:secondCompRelation}
Let \(\cH,\cK\) be complex Hilbert spaces and let \(g \in \mathrm{Pick}(\C_+)_\R\). Further, let
\[\psi: L^2(\R,\C) \to L^2(\R,\cK)\]
be a unitary map such that
\[\psi \circ U^g_t = S_t \circ \psi\]
for every \(t \in \R\). Then, for every \(F \in L^\infty(\R,B(\cH))\), one has
\[\left(M_F \otimes \textbf{1}\right) \circ \left(E_\cK^\cH \circ \left(\psi \otimes \textbf{1}\right) \circ T_\cH^{-1}\right) = \left(E_\cK^\cH \circ \left(\psi \otimes \textbf{1}\right) \circ T_\cH^{-1}\right) \circ M_{F \circ g_*}.\]
\end{lemma}
\begin{proof}
For \(f \in L^2(\R,\cK)\) and \(v \in \cH\), choosing a basis \((w_j)_{j \in J}\) of \(\cK\), we have
\begin{align*}
(\left(M_F \otimes \textbf{1}\right) \circ E_\cK^\cH)(f \otimes v) &= \left(\left(M_F \otimes \textbf{1}\right) \circ \left(T_\cK \otimes \textbf{1}\right) \circ \left(\textbf{1} \otimes \sigma_\cH^\cK\right) \circ \left(T_\cH^{-1} \otimes \textbf{1}\right)\right)\left(f \otimes v\right)
\\&= \left(\left(M_F \otimes \textbf{1}\right) \circ \left(T_\cK \otimes \textbf{1}\right) \circ \left(\textbf{1} \otimes \sigma_\cH^\cK\right)\right)\left(\sum_{j \in J} \braket*{w_j}{f} \otimes w_j \otimes v\right)
\\&= \left(\left(M_F \otimes \textbf{1}\right) \circ \left(T_\cK \otimes \textbf{1}\right)\right)\left(\sum_{j \in J} \braket*{w_j}{f} \otimes v \otimes w_j\right)
\\&= \left(M_F \otimes \textbf{1}\right)\left(\sum_{j \in J} \left(\braket*{w_j}{f} \cdot v\right) \otimes w_j\right)
\\&= \sum_{j \in J} \left(\braket*{w_j}{f} \cdot Fv\right) \otimes w_j.
\end{align*}
Then, choosing a basis \((v_i)_{i \in I}\) of \(\cH\), we get
\begin{align*}
(\left(M_F \otimes \textbf{1}\right) \circ E_\cK^\cH)(f \otimes v) &= \sum_{j \in J} \left(\braket*{w_j}{f} \cdot Fv\right) \otimes w_j
\\&=\sum_{j \in J} \left(\braket*{w_j}{f} \cdot \sum_{i \in I} \braket*{v_i}{Fv} \cdot v_i\right) \otimes w_j
\\&= \sum_{i \in I} \sum_{j \in J} \left(\braket*{v_i}{Fv} \cdot  \braket*{w_j}{f} \cdot v_i\right) \otimes w_j
\\&= E_\cK^\cH \left(\sum_{i \in I} \sum_{j \in J} \left(\braket*{v_i}{Fv} \cdot  \braket*{w_j}{f} \cdot w_j\right) \otimes v_i\right)
\\&= E_\cK^\cH \left(\sum_{i \in I} \left(\braket*{v_i}{Fv} \cdot  f\right) \otimes v_i\right).
\end{align*}
Therefore, for \(f \in L^2(\R,\C)\) and \(v \in \cH\), we have
\begin{align*}
(\left(M_F \otimes \textbf{1}\right) \circ E_\cK^\cH \circ \left(\psi \otimes \textbf{1}\right))(f \otimes v) &= (\left(M_F \otimes \textbf{1}\right) \circ E_\cK^\cH)((\psi f) \otimes v)
\\&= E_\cK^\cH \left(\sum_{i \in I} \left(\braket*{v_i}{Fv} \cdot  \psi f\right) \otimes v_i\right).
\end{align*}
Then, using that \(F\) and therefore also \(\braket*{v_i}{Fv}\) is a bounded function and applying \fref{lem:firstCompRelation}, yields
\begin{align*}
(\left(M_F \otimes \textbf{1}\right) \circ E_\cK^\cH \circ \left(\psi \otimes \textbf{1}\right))(f \otimes v) &= E_\cK^\cH \left(\sum_{i \in I} \left(\braket*{v_i}{Fv} \cdot  \psi f\right) \otimes v_i\right)
\\&= E_\cK^\cH \left(\sum_{i \in I} \left(\psi(\braket*{v_i}{(F \circ g_*)v} \cdot f)\right) \otimes v_i\right)
\\&= \left(E_\cK^\cH \circ \left(\psi \otimes \textbf{1}\right)\right) \left(\sum_{i \in I} \left(\braket*{v_i}{(F \circ g_*)v} \cdot f\right) \otimes v_i\right)
\\&= \left(E_\cK^\cH \circ \left(\psi \otimes \textbf{1}\right) \circ T_\cH^{-1}\right) \left((F \circ g_*) \cdot f \cdot v\right)
\\&= \left(E_\cK^\cH \circ \left(\psi \otimes \textbf{1}\right) \circ T_\cH^{-1} \circ M_{F \circ g_*} \circ T_\cH\right) \left(f \otimes v\right),
\end{align*}
so
\[\left(M_F \otimes \textbf{1}\right) \circ E_\cK^\cH \circ \left(\psi \otimes \textbf{1}\right) = E_\cK^\cH \circ \left(\psi \otimes \textbf{1}\right) \circ T_\cH^{-1} \circ M_{F \circ g_*} \circ T_\cH. \qedhere\]
\end{proof}
With help of this lemma, we can now prove our theorem:
\begin{theorem}\label{thm:InvariantHomo}
Let \(\cH\) be a finite-dimensional complex Hilbert space and let \(f,g \in \mathrm{Pick}(\C_+)_\R\) and \(F \in \mathrm{Pick}(\C_+,B(\cH))_\R\) be regular Pick functions. Then
\[\deg(f \circ F \circ g) = \deg(f) \cdot \deg(F) \cdot \deg(g).\]
\end{theorem}
\begin{proof}
By \fref{thm:BigLWTheorem}, there exist Hilbert spaces \(\cK_f,\cK_F,\cK_g\) and unitary operators
\begin{align*}
\psi_f&: L^2(\R,\C) \to L^2(\R,\cK_f)
\\\psi_F&: L^2(\R,\cH) \to L^2(\R,\cK_F)
\\\psi_g&: L^2(\R,\C) \to L^2(\R,\cK_g)
\end{align*}
such that
\begin{align*}
\psi_f\left(H^2(\C_+)\right) &= H^2(\C_+,\cK_f) \quad \text{and} \quad \psi_f \circ U^f_t = S_t \circ \psi_f
\\\psi_F\left(H^2(\C_+,\cH)\right) &= H^2(\C_+,\cK_F) \quad \text{and} \quad \psi_F \circ U^F_t = S_t \circ \psi_F
\\\psi_g\left(H^2(\C_+)\right) &= H^2(\C_+,\cK_g) \quad \text{and} \quad \psi_g \circ U^g_t = S_t \circ \psi_g
\end{align*}
and
\begin{align*}
\dim \cK_f &= \deg(f)
\\\dim \cK_F &= \deg(F)
\\\dim \cK_g &= \deg(g).
\end{align*}
The idea of the proof is now to have a sequence of complex ROPGs \((\cE,\cE_+,U)\) such that each one of them is equivalent to the next one and the first one is given by \((L^2(\R,\cH),H^2(\C_+,\cH),U^{f \circ F \circ g})\) and the last one is given by \((L^2(\R,\cK_f  \,\hat{\otimes}\, \cK_F  \,\hat{\otimes}\, \cK_g),H^2(\C_+,\cK_f  \,\hat{\otimes}\, \cK_F  \,\hat{\otimes}\, \cK_g),S_t)\). The Hilbert spaces and isomorphisms of ROPGs as well as the unitary one-parameter groups are indicated in the following diagram (the spaces \(\cE_+\) are omitted since they are given by the corresponding Hardy spaces):
\[
\begin{tikzcd}[column sep=2.5pc,row sep=5pc]
\color{blue}e^{it(f \circ F \circ g)_*}&L^2(\R,\cH) \arrow{r}{T_\cH^{-1}} & L^2(\R,\C)  \,\hat{\otimes}\, \cH \arrow{dl}{\psi_g \otimes \textbf{1}} &
\\&L^2(\R,\cK_g)  \,\hat{\otimes}\, \cH \arrow{r}{E_{\cK_g}^\cH} & L^2(\R,\cH)  \,\hat{\otimes}\, \cK_g \arrow{dl}{\psi_F \otimes \textbf{1}} & \color{blue}e^{it(f \circ F)_*} \otimes \textbf{1}
\\\color{blue}e^{itf_*} \otimes \textbf{1} &L^2(\R,\cK_F)  \,\hat{\otimes}\, \cK_g \arrow{r}{T_{\cK_F}^{-1} \otimes \textbf{1}} & L^2(\R,\C)  \,\hat{\otimes}\, \cK_F  \,\hat{\otimes}\, \cK_g \arrow{dl}{\psi_f \otimes \textbf{1} \otimes \textbf{1}} & \color{blue}e^{itf_*} \otimes \textbf{1} \otimes \textbf{1}
\\\color{blue}S_t \otimes \textbf{1} \otimes \textbf{1} &L^2(\R,\cK_f)  \,\hat{\otimes}\, \cK_F  \,\hat{\otimes}\, \cK_g \arrow{r}{I_{\cK_f}^{\cK_F  \,\hat{\otimes}\, \cK_g}} & L^2(\R,\cK_f  \,\hat{\otimes}\, \cK_F  \,\hat{\otimes}\, \cK_g) & \color{blue}S_t
\end{tikzcd}
\]
We will now carry out the details. First, we define the map
\[\psi_0 \coloneqq I_{\cK_f}^{\cK_F} \circ \left(\psi_f \otimes \textbf{1}\right) \circ T_{\cK_F}^{-1} \circ \psi_F: L^2(\R,\cH) \to L^2(\R,\cK_f  \,\hat{\otimes}\, \cK_F).\]
Since all occurring maps are unitary and preserve the corresponding Hardy spaces, \(\psi\) is unitary and we have
\[\psi_0\left(H^2(\C_+,\cH)\right) = H^2(\C_+,\cK_f  \,\hat{\otimes}\, \cK_F).\]
Further
\begin{align*}
S_t \circ \psi_0 &= S_t \circ I_{\cK_f}^{\cK_F} \circ \left(\psi_f \otimes \textbf{1}\right) \circ T_{\cK_F}^{-1} \circ \psi_F = I_{\cK_f}^{\cK_F} \circ \left(S_t \circ \psi_f \otimes \textbf{1}\right) \circ T_{\cK_F}^{-1} \circ \psi_F
\\&= I_{\cK_f}^{\cK_F} \circ \left(\psi_f \circ U^f_t \otimes \textbf{1}\right) \circ T_{\cK_F}^{-1} \circ \psi_F = I_{\cK_f}^{\cK_F} \circ \left(\psi_f \otimes \textbf{1}\right) \circ T_{\cK_F}^{-1} \circ U^f_t \circ \psi_F
\\&= I_{\cK_f}^{\cK_F} \circ \left(\psi_f \otimes \textbf{1}\right) \circ T_{\cK_F}^{-1} \circ \psi_F \circ U^{f \circ F}_t = \psi_0 \circ U^{f \circ F}_t,
\end{align*}
where the penultimate equality follows by \fref{lem:firstCompRelation}. We now define the map
\[\psi \coloneqq I_{\cK_f  \,\hat{\otimes}\, \cK_F}^{\cK_g} \circ \left(\psi_0 \otimes \textbf{1}\right) \circ \left(E_{\cK_g}^\cH \circ \left(\psi_g \otimes \textbf{1}\right) \circ T_\cH^{-1}\right): L^2(\R,\cH) \to L^2(\R,\cK_f  \,\hat{\otimes}\, \cK_F  \,\hat{\otimes}\, \cK_g)\]
Since all occurring maps are unitary and preserve the corresponding Hardy spaces, \(\psi\) is unitary and we have
\[\psi\left(H^2(\C_+,\cH)\right) = H^2(\C_+,\cK_f  \,\hat{\otimes}\, \cK_F  \,\hat{\otimes}\, \cK_g).\]
We now show that \(\psi\) is also an intertwining operator. We have
\begin{align*}
S_t \circ \psi &= S_t \circ I_{\cK_f  \,\hat{\otimes}\, \cK_F}^{\cK_g} \circ \left(\psi_0 \otimes \textbf{1}\right) \circ \left(E_{\cK_g}^\cH \circ \left(\psi_g \otimes \textbf{1}\right) \circ T_\cH^{-1}\right)
\\&= I_{\cK_f \otimes \cK_F}^{\cK_g} \circ \left(S_t \circ \psi_0 \otimes \textbf{1}\right) \circ \left(E_{\cK_g}^\cH \circ \left(\psi_g \otimes \textbf{1}\right) \circ T_\cH^{-1}\right)
\\&= I_{\cK_f \otimes \cK_F}^{\cK_g} \circ \left(\psi_0 \circ U^{f \circ F}_t \otimes \textbf{1}\right) \circ \left(E_{\cK_g}^\cH \circ \left(\psi_g \otimes \textbf{1}\right) \circ T_\cH^{-1}\right)
\\&= I_{\cK_f \otimes \cK_F}^{\cK_g} \circ \left(\psi_0 \otimes \textbf{1}\right) \circ \left(U^{f \circ F}_t \otimes \textbf{1}\right) \circ \left(E_{\cK_g}^\cH \circ \left(\psi_g \otimes \textbf{1}\right) \circ T_\cH^{-1}\right)
\\&= I_{\cK_f \otimes \cK_F}^{\cK_g} \circ \left(\psi_0 \otimes \textbf{1}\right) \circ \left(E_{\cK_g}^\cH \circ \left(\psi_g \otimes \textbf{1}\right) \circ T_\cH^{-1}\right) \circ U^{f \circ F \circ g}_t = \psi \circ U^{f \circ F \circ g}_t,
\end{align*}
where the penultimate equality follows by \fref{lem:secondCompRelation}. The statement now follows with \fref{thm:BigLWTheorem} by
\[\deg(f \circ F \circ g) = \dim (\cK_f  \,\hat{\otimes}\, \cK_F  \,\hat{\otimes}\, \cK_g) = \dim \cK_f \cdot \dim \cK_F \cdot \dim \cK_g = \deg(f) \cdot \deg(F) \cdot \deg(g). \qedhere\]
\end{proof}
Now, the question is, composing with which functions from left and right leaves the degree of a Pick function unchanged. By \fref{thm:InvariantHomo} these are exactly the Pick functions \(f \in \mathrm{Pick}(\C_+)_\R\) with \(\deg(f) = 1\) (such a function is non-constant and therefore by \fref{cor:ConstNonDeg} always regular). We will see that these are exactly the Möbius transformations:
\begin{definition}
We write \(\gls*{Mob}\left(\C_\infty\right)\) for the group of Möbius transformations of the Riemann sphere \(\C_\infty\), i.e. for the set of functions \(m: \C_\infty \to \C_\infty\) such that there are \(a,b,c,d \in \C\) with \(ad-bc = 1\) and
\[m(z) = \frac{az+b}{cz+d}.\]
Further, we set
\begin{align*}
\mathrm{Mob}(\C_+) &\coloneqq \{m \in \mathrm{Mob}\left(\C_\infty\right): m(\C_+) = \C_+\}
\\&= \left\{m \in \mathrm{Mob}\left(\C_\infty\right): (\exists a,b,c,d \in \R) : ad-bc = 1 \text{ and } m(z) = \frac{az+b}{cz+d}\right\}.
\end{align*}
\end{definition}
\begin{prop}\label{prop:InvariantMob}
One has
\[\mathrm{Mob}(\C_+) = \{f \in \mathrm{Pick}(\C_+)_\R : \deg(f) = 1\}.\]
\end{prop}
\begin{proof}
By \fref{thm:BigLWTheorem}, for \(f \in \mathrm{Pick}(\C_+)_\R\), one has that \(\deg(f) = 1\), if and only if \(f\) is of the form
\[f(z) = c + dz\]
with \(c \in \R\) and \(d > 0\) or of the form
\[f(z) = c + \frac a{\lambda-z} = \frac{(c\lambda + a) - cz}{\lambda - z}\]
with \(c,\lambda \in \R\) and \(a>0\). In both cases \(f \in \mathrm{Mob}(\C_+)\). Conversely, for \(f \in \mathrm{Mob}(\C_+)\), there exist \(a,b,c,d \in \R\) with \(ad-bc = 1\) and
\[f(z) = \frac{az+b}{cz+d}.\]
For \(c \neq 0\), we have
\[f(z) = \frac{az+b}{cz+d} = \frac 1c \cdot \frac{acz+bc}{cz+d} = \frac 1c \cdot \left(a + \frac{bc - ad}{cz+d}\right) = \frac 1c \cdot \left(a - \frac{1}{cz+d}\right) = \frac ac + \frac{\frac 1{c^2}}{-\frac dc-z}\]
and for \(c = 0\), we have \(ad = ad - bc = 1\) and therefore
\[f(z) = \frac{az+b}{d} = \frac bd + \frac ad z = \frac bd + a^2 z,\]
so in every case \(f \in \mathrm{Pick}(\C_+)_\R\) with \(\deg(f) = 1\).
\end{proof}
\begin{cor}\label{cor:InvariantMobConj}
Let \(\cH\) be a finite-dimensional complex Hilbert space and let \(F \in \mathrm{Pick}(\C_+,B(\cH))_\R\) be regular.
Then, for \(f,g \in \mathrm{Pick}(\C_+)_\R\), one has
\[\deg(f \circ F \circ g) = \deg(F),\]
if and only if \(f,g \in \mathrm{Mob}(\C_+)\).
\end{cor}
\begin{proof}
This follows immediately by \fref{prop:InvariantMob} and  \fref{thm:InvariantHomo}.
\end{proof}
\begin{cor}
Let \(\cH\) be a finite-dimensional complex Hilbert space and let \(F \in \mathrm{Pick}(\C_+,B(\cH))_\R\) be regular.
Then, for \(m,n \in \mathrm{Mob}(\C_+)\), one has
\[m \circ F \circ n \in \mathrm{Rat}(\cH) \Leftrightarrow F \in \mathrm{Rat}(\cH).\]
\end{cor}
\begin{proof}
By \fref{cor:InvariantMobConj} and \fref{thm:BigLWTheorem}, we have
\[m \circ F \circ n \in \mathrm{Rat}(\cH) \Leftrightarrow \deg(m \circ F \circ n) < \infty \Leftrightarrow \deg(F) < \infty \Leftrightarrow F \in \mathrm{Rat}(\cH). \qedhere\]
\end{proof}

\newpage
\section{Reflection positivity on the Hardy space}
In the last two chapters, we were analyzing ROPGs. In this chapter, we will add another structure to them. Concretely we will consider quadruples \((\cE,\cE_+,U,\theta)\), where \((\cE,\cE_+,U)\) is a ROPG and \(\theta\) is an involution on \(\cE\) with certain compatibility conditions. At the beginning of this chapter, we will show that, to classify these reflection positive ROPGs, one needs to better understand so-called reflection positive Hilbert spaces, which we will introduce in the following section:

\subsection{Reflection positive Hilbert spaces and ROPGs}
\begin{definition}\label{def:RPHS}{(cf. \cite[Def. 2.1.1, Def. 4.2.1]{NO18})}
Consider a triple \((\cE,\cE_+,\theta)\) of a Hilbert space \(\cE\), a closed subspace \(\cE_+ \subeq \cE\) and a unitary involution \(\theta \in \U(\cE)\).
\begin{enumerate}[\rm (a)]
\item We call \((\cE,\cE_+,\theta)\) a \textit{reflection positive Hilbert space (\gls*{RPHS})}, if
\begin{equation*}
\braket*{\xi}{\theta\xi} \geq 0 \qquad \forall \xi \in \mathcal{E}_+.
\end{equation*}
\item We call a quadruple \((\cE,\cE_+,U,\theta)\) of a RPHS \((\cE,\cE_+,\theta)\) and a strongly continuous unitary one-parameter group \(U: \R \to \U(\cE)\) a \textit{reflection positive unitary one-parameter group}, if
\[U_t \cE_+ \subeq \cE_+ \qquad \forall t \geq 0\]
and
\[\theta \circ U_t = U_{-t} \circ \theta \qquad \forall t \in \R.\]
\item We call a quadruple \((\cE,\cE_+,U,\theta)\) a \textit{reflection positive ROPG}, if \((\cE,\cE_+,U,\theta)\) is a reflection positive unitary one-parameter group and \((\cE,\cE_+,U)\) is a ROPG.
\item We say that a reflection positive Hilbert space, unitary one-parameter group, or ROPG is real/complex, if \(\cE\) is real/complex.
\item We say that two reflection positive ROPGs \((\cE,\cE_+,U,\theta)\) and \((\cE',\cE_+',U',\theta')\) are equivalent, if there exists a unitary operator \(\psi: \cE \to \cE'\) such that \(\psi\) is an isomorphism between the ROPGs \((\cE,\cE_+,U)\) and \((\cE',\cE_+',U')\) that additionally satisfies
\[\psi \circ \theta = \theta' \circ \psi.\]
We call such an operator \(\psi\) an \textit{isomorphism between \((\cE,\cE_+,U,\theta)\) and \((\cE',\cE_+',U',\theta')\)}.
\end{enumerate}
\end{definition}
We first want to investigate the relation between real and complex reflection positive ROPGs. Here we have the following propositions:
\begin{prop}\label{prop:RPHSRealComplex}
Let \(\cE\) be a real Hilbert space, \(\cE_+ \subeq \cE\) be a closed subspace and \(\theta \in \U(\cH)\) be an involution. Then \((\cE,\cE_+,\theta)\) is a real RPHS, if and only if \((\cE_\C,(\cE_+)_\C,\theta_\C)\) is a complex RPHS.
\end{prop}
\begin{proof}
Using that the scalar product on \(\cE\) is real-valued and that \(\theta\) is a unitary involution, we get
\[\braket*{v}{\theta w} = \braket*{\theta w}{v} = \braket*{w}{\theta v} \quad \forall v,w \in \cE.\]
This yields
\[\braket*{(v+iw)}{\theta_\C(v+iw)} = \braket*{v}{\theta v} + \braket*{w}{\theta w} + i \braket*{v}{\theta w} - i \braket*{w}{\theta v} = \braket*{v}{\theta v} + \braket*{w}{\theta w} \quad \forall v,w \in \cE,\]
which immediately implies that
\[\left(\braket*{\xi}{\theta \xi} \geq 0 \quad \forall \xi \in \cE_+\right) \quad \Leftrightarrow \quad \left(\braket*{\xi}{\theta_\C \xi} \geq 0 \quad \forall \xi \in (\cE_+)_\C\right). \qedhere\]
\end{proof}
\begin{prop}\label{prop:RPROPGRealToComplex}
Let \(\cE\) be a real Hilbert space, \(\cE_+ \subeq \cE\) be a closed subspace, \(U: \R \to \U(\cE)\) be a unitary one-parameter group and \(\theta \in \U(\cH)\) be an involution. Then \((\cE,\cE_+,U,\theta)\) is a real reflection positive ROPG, if and only if \((\cE_\C,(\cE_+)_\C,U_\C,\theta_\C)\) is a complex reflection positive ROPG.
\end{prop}
\begin{proof}
By \fref{prop:ROPGRealToComplex} we have that \((\cE,\cE_+,U)\) is a real ROPG, if and only if \((\cE_\C,(\cE_+)_\C,U_\C)\) is a complex ROPG. Further, by \fref{prop:RPHSRealComplex}, we have that \((\cE,\cE_+,\theta)\) is a real RPHS, if and only if \((\cE_\C,(\cE_+)_\C,\theta_\C)\) is a complex RPHS. The statement then follows since
\[\theta_\C (U_t)_\C \theta_\C = (\theta U_t \theta)_\C \quad \forall t \in \R\]
and therefore
\[\theta_\C (U_t)_\C \theta_\C = (U_{-t})_\C \quad \Leftrightarrow \quad \theta U_t \theta = U_{-t}. \qedhere\]
\end{proof}
This theorem shows that we can reduce our analysis of reflection positive ROPGs to the complex case. Given a complex reflection positive ROPG \((\cE,\cE_+,U,\theta)\), by the Lax--Phillips Theorem (\fref{thm:LaxPhillipsComplex}), there exists a complex Hilbert space \(\cK\) such that \((\cE,\cE_+,U,\theta)\) is equivalent to \((L^2(\R,\cK),H^2(\C_+,\cK),S,\tilde \theta)\) for some involution \(\tilde \theta\) on \(L^2(\R,\cK)\). To classify the involutions \(\tilde \theta\) for which \((L^2(\R,\cK),H^2(\C_+,\cK),S,\tilde \theta)\) is a reflection positive ROPG, we need the following definitions:
\begin{definition}\label{def:defTheta}
Let \(\cK\) be a complex Hilbert space and \(X\) be a set.
\begin{enumerate}[\rm (a)]
\item For function \(f: \R \to X\) we define the function
\[\gls*{R}f: \R \to X \quad x \mapsto f(-x).\]
\item For \(h \in L^\infty(\R,B(\cK))\) we set
\[\gls*{thetah} \coloneqq M_h \circ R \in B(L^2(\R,\cK)).\]
\item We define the involution
\begin{equation*}
\flat:  L^\infty\left(\R,B(\cK)\right) \to L^\infty\left(\R,B(\cK)\right), \quad \gls*{hflat} \left(x\right) \coloneqq (Rh)(x)^* = h\left(-x\right)^*.
\end{equation*}
\item We write \(L^\infty(\R,\U(\cK))\) for the set of functions \(f \in L^\infty(\R,B(\cK))\) such that \(f(x) \in \U(\cK)\) for almost every \(x \in \R\). In the case of \(\cK = \C\) we write \(L^\infty(\R,\T)\) for the set of functions \(f \in L^\infty(\R,\C) \cong L^\infty(\R,B(\C))\) such that \(f(x) \in \T \cong \U(\C)\) for almost every \(x \in \R\).
\end{enumerate}
\end{definition}
The following results will show how these definitions relate to reflection positive ROPGs of the form \((L^2(\R,\cK),H^2(\C_+,\cK),S,\theta)\):
\begin{lemma}\label{lem:thetahForm}
Let \(\cK\) be a complex Hilbert space and \(\theta \in B(L^2(\R,\cK))\). Then
\[\theta \circ S_t = S_{-t} \circ \theta \quad \forall t \in \R,\]
if and only if there exists \(h \in L^\infty(\R,B(\cK))\) such that \(\theta = \theta_h\).
\end{lemma}
\begin{proof}
We have
\[S_t \circ \theta = S_t \circ \left(\theta \circ R\right) \circ R = \left(\theta \circ R\right) \circ S_t \circ R + [S_t,\theta \circ R] \circ R = \theta \circ S_{-t} + [S_t,\theta \circ R] \circ R.\]
This implies that 
\[\theta \circ S_t = S_{-t} \circ \theta \quad \forall t \in \R,\]
if and only if
\[\theta \circ R \in S(\R)' = M(L^\infty(\R,B(\cK))),\]
using \fref{prop:SCommutant} in the last step. This is equivalent to the existence of \(h \in M(L^\infty(\R,B(\cK)))\) such that \(\theta \circ R = M_h\), so \(\theta = M_h \circ R = \theta_h\).
\end{proof}
\begin{lemma}\label{lem:thetahUnitary}
Let \(\cK\) be a complex Hilbert space and \(h \in L^\infty(\R,B(\cK))\). Then \(\theta_h\) is a unitary involution, if and only if \(h \in L^\infty(\R,\U(\cK))^\flat\).
\end{lemma}
\begin{proof}
Since \(R\) is unitary, the operator \(\theta_h\) is unitary, if and only if \(h \in L^\infty(\R,\U(\cK))\). Further, we have
\[\theta_h \circ \theta_h = M_h \circ R \circ M_h \circ R = M_h \circ M_{R h} \circ R \circ R = M_{h \cdot R h},\]
which shows that \(\theta_h\) is an involution, if and only if one has \(h \cdot R h = \textbf{1}\), which is equivalent to
\[R h = h^{-1} = h^*\]
and therefore to \(h^\flat = h\).
\end{proof}
\begin{prop}\label{prop:RPROPGFormthetah}
Let \(\cK\) be a complex Hilbert space and \(\theta \in B(L^2(\R,\cK))\) be a unitary involution. Then \((L^2(\R,\cK),H^2(\C_+,\cK),S,\theta)\) is a reflection positive ROPG, if and only if there exists a function \(h \in L^\infty(\R,\U(\cK))^\flat\) such that
\begin{enumerate}
\item \(\theta = \theta_h\).
\item \((L^2(\R,\cK),H^2(\C_+,\cK),\theta_h)\) is a RPHS.
\end{enumerate}
\end{prop}
\begin{proof}
This follows immediately by \fref{lem:thetahForm} and \fref{lem:thetahUnitary}.
\end{proof}
\begin{prop}\label{prop:thetahUnique}
Let \(\cK\) be a complex Hilbert space and \(h,h' \in L^\infty(\R,\U(\cK))^\flat\) such that the quadruples \((L^2(\R,\cK),H^2(\C_+,\cK),S,\theta_h)\) and \((L^2(\R,\cK),H^2(\C_+,\cK),S,\theta_{h'})\) are reflection positive ROPGs. Then \((L^2(\R,\cK),H^2(\C_+,\cK),S,\theta_h)\) and \((L^2(\R,\cK),H^2(\C_+,\cK),S,\theta_{h'})\) are equivalent, if and only if there exists \(u \in \U(\cK)\) such that \(h' = u h u^*\).
\end{prop}
\begin{proof}
If \((L^2(\R,\cK),H^2(\C_+,\cK),S,\theta_h)\) and \((L^2(\R,\cK),H^2(\C_+,\cK),S,\theta_{h'})\) are equivalent there exists a unitary operator \(\psi \in \U(L^2(\R,\cK))\) such that
\[\psi H^2(\C_+,\cK) = H^2(\C_+,\cK), \qquad \psi \circ S_t = S_t \circ \psi \quad \forall t \in \R \qquad \text{and} \qquad \psi \circ \theta_h = \theta_{h'} \circ \psi.\]
The last of these conditions implies that \(\psi \in S(\R)'\), so by \fref{prop:SCommutant} there exists a function \(f \in M(L^\infty(\R,B(\cK)))\) such that \(\psi = M_f\). The condition
\[\psi H^2(\C_+,\cK) = H^2(\C_+,\cK)\]
together with the unitarity of \(\psi\) then implies
\[M_f H^2(\C_+,\cK) = H^2(\C_+,\cK) \quad \text{and} \quad M_f H^2(\C_-,\cK) = H^2(\C_-,\cK),\]
so, by \fref{prop:H2InclusionFunctions}, we have
\[f \in H^\infty(\C_+,B(\cK)) \cap H^\infty(\C_-,B(\cK)) = B(\cK) \cdot \textbf{1},\]
using \fref{cor:HInftyIntersection} in the last step. Therefore, since \(\psi = M_f\) is unitary, there exists \(u \in \U(\cK)\) such that \(f = u \cdot \textbf{1}\). This implies
\[\theta_{h'} = \psi \circ \theta_h \circ \psi^{-1} = M_{u \cdot \textbf{1}} \circ M_h \circ R \circ M_{u^* \cdot \textbf{1}} = M_{uhu^*} \circ R = \theta_{uhu^*}\]
and therefore \(h' = uhu^*\).

If conversely \(h' = uhu^*\), then the operator \(\psi \coloneqq M_{u \cdot \textbf{1}}\) is an isomorphism between the reflection positive ROPGs \((L^2(\R,\cK),H^2(\C_+,\cK),S,\theta_h)\) and \((L^2(\R,\cK),H^2(\C_+,\cK),S,\theta_{h'})\).
\end{proof}
With these results about our standard example, we get a normal form for complex reflection positive ROPGs:
\newpage
\begin{theorem}\label{thm:RFROPGNormalForm}{\rm \textbf{(Normal form for reflection positive ROPGs -- complex version)}}
Let \((\cE,\cE_+,U,\theta)\) be a complex reflection positive ROPG. Then there exists a --- up to isomorphism -- unique complex Hilbert space \(\cK\) and a --- up to conjugation with a unitary operator --- unique function \(h \in L^\infty(\R,\U(\cK))^\flat\) such that \((L^2(\R,\cK),H^2(\C_+,\cK),S,\theta_h)\) is a reflection positive ROPG which is equivalent to \((\cE,\cE_+,U,\theta)\).
\end{theorem}
\begin{proof}
By the Lax--Phillips Theorem (\fref{thm:LaxPhillipsComplex}), there exists a complex Hilbert space \(\cK\) such that \((\cE,\cE_+,U,\theta)\) is equivalent to \((L^2(\R,\cK),H^2(\C_+,\cK),S,\tilde \theta)\) for some involution \(\tilde \theta\) on \(L^2(\R,\cK)\). By \fref{thm:kernelAdjointGeneral}, this Hilbert space \(\cK\) is unique up to isomorphism. Further, by \fref{prop:RPROPGFormthetah}, we have \(\tilde \theta = \theta_h\) for some function \(h \in L^\infty(\R,\U(\cK))^\flat\). By \fref{prop:thetahUnique}, this function is unique up to conjugation with a unitary operator.
\end{proof}
We now want to use this normal form for complex reflection positive ROPGs to also get a normal form for real reflection positive ROPGs. For this, we need the following lemma:
\begin{lemma}\label{lem:thetahSymmetric}
Let \(\cK\) be a real Hilbert space and \(h \in L^\infty(\R,\U(\cK_\C))^\flat\). Then \(\theta_h\) defines a unitary involution on \(L^2(\R,\cK_\C)^\sharp\), if and only if \(h = h^\sharp\) with
\[\gls*{hsharp}(x) \coloneqq \cC_\cK h(-x)\cC_\cK, \quad z \in \C_+,\]
denoting by \(\cC_\cK\) the complex conjugation on \(\cK_\C\). In this case \((L^2(\R,\cK_\C)^\sharp,H^2(\C_+,\cK_\C)^\sharp,S,\theta_h)\) is a real reflection positive ROPG, if and only if \((L^2(\R,\cK),H^2(\C_+,\cK_\C),S,\theta_h)\) is a complex reflection positive ROPG.
\end{lemma}
\begin{proof}
One has
\[\theta_h L^2(\R,\cK_\C)^\sharp \subeq L^2(\R,\cK_\C)^\sharp,\]
if and only if
\[\theta_h f = (\theta_h f)^\sharp = \cC_\cK R \theta_h f = \cC_\cK \theta_{R h} R f = \theta_{h^\sharp} \cC_\cK R f = \theta_{h^\sharp} f^\sharp = \theta_{h^\sharp} f\]
for all \(f \in L^2(\R,\cK_\C)^\sharp\). This is the case, if and only if \(h = h^\sharp\). The second statement follows by \fref{prop:RPROPGRealToComplex}.
\end{proof}
\begin{theorem}\label{thm:NormalformRPROPG}{\rm \textbf{(Normal form for reflection positive ROPGs -- real version)}~}
Let \((\cE,\cE_+,U,\theta)\) be a real reflection positive ROPG. Then there exists a --- up to isomorphism --- unique real Hilbert space \(\cK\) and a --- up to conjugation with a unitary operator --- unique function \(h \in L^\infty(\R,\U(\cK_\C))^\flat\) with \(h = h^\sharp\) such that \((L^2(\R,\cK_\C)^\sharp,H^2(\C_+,\cK_\C)^\sharp,S,\theta_h)\) is a reflection positive ROPG which is equivalent to \((\cE,\cE_+,U,\theta)\).
\end{theorem}
\begin{proof}
By the Lax--Phillips Theorem (\fref{thm:LaxPhillipsReal}), there exists a real Hilbert space \(\cK\) such that \((\cE,\cE_+,U,\theta)\) is equivalent to \((L^2(\R,\cK_\C)^\sharp,H^2(\C_+,\cK_\C)^\sharp,S,\tilde \theta)\) for some involution \(\tilde \theta\) on \(L^2(\R,\cK_\C)^\sharp\). By \fref{thm:kernelAdjointGeneral}, this Hilbert space \(\cK\) is unique up to isomorphism. Then, by \fref{prop:RPROPGRealToComplex}, the quadruple \((L^2(\R,\cK_\C),H^2(\C_+,\cK_\C),S,\tilde \theta_\C)\) is a complex reflection positive ROPG. Therefore, by \fref{prop:RPROPGFormthetah} and \fref{lem:thetahSymmetric}, there exists \(h \in L^\infty(\R,\U(\cK_\C))^\flat\) with \(h = h^\sharp\) such that \(\tilde \theta = \theta_h\). By \fref{prop:thetahUnique}, this function is unique up to conjugation with an operator in \(\U(\cK_\C)\).
\end{proof}
In \fref{thm:RFROPGNormalForm} and \fref{prop:RPROPGFormthetah} we have seen that classifying all complex reflection positive ROPGs comes down to classifying all functions \(h \in L^\infty(\R,\U(\cK))^\flat\) such that the triple \((L^2(\R,\cK),H^2(\C_+,\cK),\theta_h)\) is a RPHS. To also classify the real reflection positive ROPGs, we additionally have to determine which of these functions satisfy the condition \(h = h^\sharp\).

We will not be able to answer this question in full generality. However, we will do so under two more conditions. The first one is that we deal with the multiplicity free case \(\cK = \C\), so we ask for which functions \(h \in L^\infty(\R,\T)^\flat\) the triple \((L^2(\R,\C),H^2(\C_+),\theta_h)\) is a RPHS. In this case, we have
\[h^\sharp(x) = \cC_\cK h(-x) = \overline{h(-x)} = h(-x)^* = h^\flat(x) = h(x),\]
so the extra condition for real reflection positive ROPGs is automatically fulfilled. The second condition is that the triple \((L^2(\R,\C),H^2(\C_+),\theta_h)\) is a maximal RPHS in the following sense:
\begin{definition}\label{def:maxRPHS}
Let \((\cE,\cE_+,\theta)\) be a RPHS. We call \((\cE,\cE_+,\theta)\) \textit{maximal} if, for every closed subspace \(\cE_+' \subeq \cE\) with \(\cE_+ \subeq \cE_+'\), the triple \((\cE,\cE_+',\theta)\) is a RPHS, if and only if \(\cE_+' = \cE_+\).
\end{definition}
So the rest of this chapter we will dedicate to answering the question for which functions \({h \in L^\infty(\R,\T)^\flat}\) the triple \((L^2(\R,\C),H^2(\C_+),\theta_h)\) is a maximal RPHS. We will answer this question in \fref{thm:GeneralizedMaxPosForm}. Checking if a RPHS is maximal can often be difficult. The following theorem gives us an equivalent condition that is sometimes easier to check:
\begin{thm}\label{thm:EPlusMaximal}
Let \((\cE,\cE_+,\theta)\) be a RPHS. Then the following are equivalent:
\begin{enumerate}[\rm (a)]
\item \((\cE,\cE_+,\theta)\) is a maximal RPHS.
\item There exists a dense subspace \(V \subeq \cE^\theta\) such that \(V \subeq \cE_+ + \cE^{-\theta}\).
\item One has \(\cE^\theta \subeq \cE_+ + \cE^{-\theta}\).
\end{enumerate}
\end{thm}
\begin{proof}
\item[(a) \(\Rightarrow\) (b):] Let \((\cE,\cE_+,\theta)\) be a maximal RPHS and set
\[V \coloneqq \cE^\theta \cap \left(\cE_+ + \cE^{-\theta}\right).\]
We now want to show that \(V\) is dense in \(\cE^\theta\). So let \(v \in V^\perp \cap \cE^\theta\). We want to show that \(v = 0\). For \(\xi \in \cE_+\), we have \(\xi = \xi^\theta + \xi^{-\theta}\), where
\[\xi^{\pm \theta} \coloneqq \frac 12 (\xi \pm \theta \xi) \in \cE^{\pm \theta}.\]
One also has
\[\xi^\theta = \xi - \xi^{-\theta} \in \cE_+ + \cE^{-\theta},\]
so \(\xi^\theta \in V\) and therefore, using that \(\cE^\theta \perp \cE^{-\theta}\) and \(v \in V^\perp \cap \cE^\theta\), we get
\[\braket*{\xi}{v} = \braket*{\xi^\theta}{v} + \braket*{\xi^{-\theta}}{v} = 0 + 0 = 0\]
for every \(\xi \in \cE_+\), so \(v \in \cE_+^\perp\). On the other hand, for every \(\xi \in \cE_+\) and every \(t \in \K\) (assuming that \(\cE\) is a Hilbert space over the field \(\K = \R,\C\)), using that \(v \in \cE^\theta\) and that \(\braket*{\xi}{v}=0\), we have
\begin{align*}
\braket*{(\xi+tv)}{\theta(\xi+tv)} &= \braket*{\xi}{\theta \xi} + \left|t\right|^2 \braket*{v}{\theta v} + t\braket*{\xi}{\theta v} + \overline{t}\braket*{v}{\theta \xi}
\\&= \braket*{\xi}{\theta \xi} + \left|t\right|^2 \braket*{v}{\theta v} + t\braket*{\xi}{\theta v} + \overline{t}\braket*{\theta v}{\xi}
\\&= \braket*{\xi}{\theta \xi} + \left|t\right|^2 \braket*{v}{v} + t\braket*{\xi}{v} + \overline{t}\braket*{v}{\xi}
\\&= \braket*{\xi}{\theta \xi} + \left|t\right|^2 \left\lVert v\right\rVert^2 \geq 0.
\end{align*}
This shows that the triple \((\cE,\cE_+ + \K v,\theta)\) is a RPHS and therefore \(\cE_+ + \K v = \cE_+\) due to the maximality of \((\cE,\cE_+,\theta)\). This yields \(v \in \cE_+\) and therefore \(v \in \cE_+ \cap \cE_+^\perp\), so \(v=0\). This implies that \(V^\perp \cap \cE^\theta = \{0\}\) and therefore \(V\) is dense in \(\cE^\theta\).
\item[(b) \(\Rightarrow\) (c):] Let \(V \subeq \cE^\theta\) be a dense subspace such that \(V \subeq \cE_+ + \cE^{-\theta}\). Then, for every \(v \in V\), there exist \(v_+ \in \cE_+\) and \(v^{-\theta} \in \cE^{-\theta}\) such that \(v = v_+ + v^{-\theta}\). We now show that this composition is unique, so assume that there is \(u_+ \in \cE_+\) and \(u^{-\theta} \in \cE^{-\theta}\) such that \(v = u_+ + u^{-\theta}\). Then
\[\cE_+ \ni u_+ - v_+ = \left(v - u^{-\theta}\right) - \left(v - v^{-\theta}\right) = v^{-\theta} - u^{-\theta} \in \cE^{-\theta},\]
so
\[u_+ - v_+ \in \cE_+ \cap \cE^{-\theta}.\]
Since for every \(\xi \in \cE^{-\theta} \setminus \{0\}\) one has
\[\braket*{\xi}{\theta \xi} = \braket*{\xi}{- \xi} = - \left\lVert \xi\right\rVert^2 < 0,\]
we have \(\cE_+ \cap \cE^{-\theta} = \{0\}\) and therefore \(u_+ - v_+ = 0\), which implies \(u_+ = v_+\) and therefore also \(u^{-\theta} = v^{-\theta}\). So, the composition is unique and therefore defines linear maps
\[C: V \to \cE_+, \quad v \mapsto v_+ \quad \text{and} \quad D: V \to \cE^{-\theta}, \quad v \mapsto v^{-\theta},\]
of which we now want to show that they are continuous. For \(v \in V \subeq \cE^\theta \subeq \left(\cE^{-\theta}\right)^\perp\) one has
\[0 \leq \braket*{v_+}{\theta v_+} = \braket*{(v-v^{-\theta})}{\theta(v-v^{-\theta})} = \left\lVert v\right\rVert^2 - \left\lVert v^{-\theta}\right\rVert^2,\]
so
\[\left\lVert v^{-\theta}\right\rVert \leq \left\lVert v\right\rVert\]
and therefore \(\left\lVert D\right\rVert \leq 1\). Also, we have
\[\left\lVert v^+\right\rVert = \left\lVert v-v^{-\theta}\right\rVert \leq \left\lVert v\right\rVert + \left\lVert v^{-\theta}\right\rVert \leq 2\left\lVert v\right\rVert,\]
so \(\left\lVert C\right\rVert \leq 2\). Since \(V\) is dense in \(\cE^\theta\) and both \(\cE_+\) and \(\cE^{-\theta}\) are closed, we can uniquely extend the two maps to continuous maps
\[\tilde C: \cE^\theta \to \cE_+ \quad \text{and} \quad \tilde D: \cE^\theta \to \cE^{-\theta}.\]
Now let \(w \in \cE^\theta\). Then there exists a sequence \((w_n)_{n \in \N} \subeq V\) such that \(w_n \xrightarrow{n \to \infty}w\). We get
\[w = \lim_{n \to \infty}w_n = \lim_{n \to \infty}Cw_n + Dw_n = \tilde Cw +\tilde Dw \in \cE_+ + \cE^{-\theta}.\]
Since \(w \in \cE^\theta\) was chosen arbitrary, we get \(\cE^\theta \subeq \cE_+ + \cE^{-\theta}\).
\item[(c) \(\Rightarrow\) (a):] Let \(\cE^\theta \subeq \cE_+ + \cE^{-\theta}\) and let \(v \in \cE\) such that \((\cE,\cE_+ + \K v,\theta)\) is a reflection positive Hilbert space. We want to show that \(v \in \cE_+\). Let \(v = v^\theta + v^{-\theta}\) with \(v^{\pm \theta} \in \cE^{\pm \theta}\). Since \(\cE^\theta \subeq \cE_+ + \cE^{-\theta}\) there are \(v_+ \in \cE_+\) and \(u \in \cE^{-\theta}\) such that \(v^\theta = v_+ + u\). We have
\[\cE_+ + \K v \ni v_+ - v = \left(v^\theta - u\right) - \left(v^\theta + v^{-\theta}\right) = -u - v^{-\theta} \in \cE^{-\theta},\]
which yields \(v_+-v = 0\), since
\[\left(\cE_+ + \K v\right) \cap \cE^{-\theta} = \{0\}\]
due to the reflection positivity of the triple \((\cE,\cE_+ + \K v,\theta)\). Therefore \(v = v_+ \in \cE_+\), which shows that \((\cE,\cE_+,\theta)\) is a maximal RPHS.
\end{proof}
\begin{cor}\label{cor:EPlusMaximalReal}
Let \((\cE,\cE_+,\theta)\) be a real RPHS. Then \((\cE,\cE_+,\theta)\) is a maximal RPHS, if and only if \((\cE_\C,(\cE_+)_\C,\theta_\C)\) is a maximal RPHS.
\end{cor}
\begin{proof}
We have
\[(\cE_\C)^{\theta_\C} = (\cE^\theta)_\C \quad \text{and} \quad (\cE_+)_\C + (\cE_\C)^{-\theta_\C} = (\cE_+ + \cE^{-\theta})_\C.\]
This shows that
\[\cE^\theta \subeq \cE_+ + \cE^{-\theta} \quad \Leftrightarrow \quad (\cE_\C)^{\theta_\C} \subeq (\cE_+)_\C + (\cE_\C)^{-\theta_\C}.\]
The statement then follows by \fref{thm:EPlusMaximal}.
\end{proof}

\subsection{Involutions with fixed points in the Hardy space}
Our first step in classifying all functions \({h \in L^\infty(\R,\T)^\flat}\) such that the triple \((L^2(\R,\C),H^2(\C_+),\theta_h)\) is a maximal RPHS is to classify all such functions under the additional assumption that \(\theta_h\) has a fixed point in \(H^2(\C_+)\), i.e. that \(\ker(\theta_h - \textbf{1}) \neq \{0\}\). We will later drop that assumption and consider the general case.

\subsubsection{Fixed points in the Hardy space and outer functions}
We start by proving the following result that, given a RPHS \((L^2(\R,\C),H^2(\C_+),\theta_h)\), one can split off inner functions (cf. \fref{def:Inner}) from the function \(h\):
\begin{lemma}\label{lem:InnerSplitOff}
Let \({h \in L^\infty(\R,\T)^\flat}\) and \(\phi \in \Inn(\C_+)\). Then the triple \((L^2(\R,\C),H^2(\C_+),\theta_{(\phi \phi^\flat)h})\) is a RPHS, if and only if \((L^2(\R,\C),M_{\phi}^{-1} H^2(\C_+),\theta_h)\) is a RPHS. In this case also the triple \((L^2(\R,\C),H^2(\C_+),\theta_h)\) is a RPHS.
\end{lemma}
\begin{proof}
The triple \((L^2(\R,\C),H^2(\C_+),\theta_{(\phi \phi^\flat)h})\) is a RPHS, if and only if
\[0 \leq \braket*{f}{\theta_{(\phi \phi^\flat)h} f} = \braket*{f}{\phi \phi^\flat h Rf} = \braket*{\phi^* f}{h R (\phi^* f)} \quad \forall f \in H^2(\C_+).\]
This is equivalent to
\[\braket*{f}{\theta_h f} \geq 0 \quad \forall f \in M_{\phi^*} H^2(\C_+) = M_\phi^{-1} H^2(\C_+)\]
and therefore to \((L^2(\R,\C),M_{\phi}^{-1} H^2(\C_+),\theta_h)\) being a RPHS. Since \(M_\phi H^2(\C_+) \subeq H^2(\C_+)\) and therefore
\[H^2(\C_+) \subeq M_\phi^{-1} H^2(\C_+),\]
this especially implies
\[\braket*{f}{\theta_h f} \geq 0 \quad \forall f \in H^2(\C_+),\]
so \((L^2(\R,\C),H^2(\C_+),\theta_h)\) is a RPHS.
\end{proof}
In this subsection, we want to use this lemma to show that, for a function \({h \in L^\infty(\R,\T)^\flat}\) such that the triple \((L^2(\R,\C),H^2(\C_+),\theta_h)\) is a RPHS, the kernel \(\ker(\theta_h - \textbf{1})\) just contains outer functions, which are defined as follows:
\begin{definition}{(cf. \fref{def:Outer})}
\begin{enumerate}[\rm (a)]
\item For \(C \in \T\) and an almost everywhere defined function \(K: \R \to \R_{\geq 0}\) with
\begin{equation*}
I(K) \coloneqq \int_\R \frac{\left|\log\left(K\left(p\right)\right)\right|}{1 + p^2} \,dp < \infty
\end{equation*}
we define the function \(\Out(C,K) \in \cO(\C_+)\) by
\begin{equation*}
\Out(C,K)\left(z\right) = C \cdot \exp\left(\frac{1}{\pi i} \int_\R \left[\frac{1}{p-z} - \frac{p}{1 + p^2}\right] \log\left(K\left(p\right)\right) dp\right), \quad z \in \C_+
\end{equation*}
and call such a function an {\it outer function}.
\item We set \(\Out(K) \coloneqq \Out(1,K)\).
\item We write \(\mathrm{Out}(\C_+)\) for the set of outer functions on \(\C_+\) and set
\[\Out^2(\C_+) = \Out(\C_+) \cap H^2(\C_+).\]
\end{enumerate} 
\end{definition}
\begin{remark}
Since
\[\Out(C,K)\left(z\right) \in \T \cdot \exp(\C) = \C^\times,\]
we have \(0 \notin \Out(\C_+)\) and therefore also \(0 \notin \Out^2(\C_+)\). This also shows that, for an outer function~\(F \in \Out(\C_+)\), expressions like \(\frac 1F\) are defined.
\end{remark}
\begin{prop}\label{prop:OuterMaximal}
Let \({h \in L^\infty(\R,\T)^\flat}\) such that the triple \((L^2(\R,\C),H^2(\C_+),\theta_h)\) is a RPHS and
\[\ker(\theta_h - \textbf{1}) \cap \Out^2(\C_+) \neq \eset.\]
Then \((L^2(\R,\C),H^2(\C_+),\theta_h)\) is a maximal RPHS.
\end{prop}
\begin{proof}
We choose
\[F \in \left(\ker(\theta_h - \textbf{1}) \cap \Out^2(\C_+)\right) \setminus \{0\}.\]
Then
\[F = \theta_h F = h \cdot RF\]
and therefore
\[h = \frac{F}{RF}.\]
For \(f \in L^\infty(\R,\C)\) we get
\[\theta_h(fF) = \frac {F}{RF} R(fF) = \frac {F}{RF} (R f)(RF) = (R f) F,\]
so, for \(f \in L^\infty(\R,\C)^{\pm R}\), we have \(\theta_h(fF) = \pm fF\) and therefore
\[\overline{L^\infty(\R,\C)^{\pm R}F} \subeq L^2(\R,\C)^{\pm \theta_h}.\]
Further, by \fref{cor:OuterDense}, we have
\[\overline{L^\infty(\R,\C)^{R}F + L^\infty(\R,\C)^{-R}F} = \overline{L^\infty(\R,\C)F} = L^2(\R,\C) = L^2(\R,\C)^{\theta_h} \oplus L^2(\R,\C)^{-\theta_h},\]
which yields
\[L^2(\R,\C)^{\pm \theta_h} = \overline{L^\infty(\R,\C)^{\pm R}F}.\]
Also by \fref{cor:OuterDense}, we have
\[\overline{\left(H^\infty(\C_+) + H^\infty(\C_-)\right)F} = L^2(\R,\C),\]
wherefore, for every \(f \in L^\infty(\R,\C)^{R}\), there are sequences
\[(f_n)_{n \in \N} \subeq H^\infty(\C_+) \quad \text{and} \quad (g_n)_{n \in \N} \subeq H^\infty(\C_-)\]
such that
\[\lim_{n \to \infty}(f_n+g_n)F = fF.\]
This also implies
\begin{align*}
\lim_{n \to \infty}(R(f_n+g_n))F &= \lim_{n \to \infty}(R(f_n+g_n)) \theta_h F = \theta_h \left(\lim_{n \to \infty}(f_n+g_n) F\right)
\\&= \theta_h(fF) = (R f) \theta_h F = (R f) F = fF.
\end{align*}
If we then set
\[s_n \coloneqq \frac{f_n+R g_n}2 \in H^\infty(\C_+), \quad n \in \N,\]
we get
\begin{align*}
\lim_{n \to \infty}\left(s_n+R s_n\right)F &= \lim_{n \to \infty}\left(\left(\frac{f_n+R g_n}2\right)+R \left(\frac{f_n+R g_n}2\right)\right)F
\\&= \frac 12 \left(\lim_{n \to \infty}\left(f_n+g_n\right)F + \lim_{n \to \infty}(R\left(f_n+g_n\right))F\right)
\\&= \frac 12 \left(fF + fF\right) = fF.
\end{align*}
This shows that the set
\[V \coloneqq \left\{(f+R f)F : f \in H^\infty(\C_+)\right\}\]
is dense in \(L^\infty(\R,\C)^{R}F\) and therefore also in \(\overline{L^\infty(\R,\C)^{R}F} = L^2(\R,\C)^{\theta_h}\). For every element \(v = (f+R f)F \in V\), we have
\begin{align*}
v &= (f+R f)F = 2fF - (f-R f)F
\\&\in H^\infty(\C_+)F + L^\infty(\R,\C)^{-R}F \subeq H^2(\C_+) + L^2(\R,\C)^{-\theta_h},
\end{align*}
so
\[V \subeq H^2(\C_+) + L^2(\R,\C)^{-\theta_h}.\]
This, by \fref{thm:EPlusMaximal}, implies that \((L^2(\R,\C),H^2(\C_+),\theta_h)\) is a maximal RPHS.
\end{proof}
\begin{cor}\label{cor:kernelOuterOrZero}
Let \({h \in L^\infty(\R,\T)^\flat}\) such that the triple \((L^2(\R,\C),H^2(\C_+),\theta_h)\) is a RPHS. Then
\[\ker(\theta_h - \textbf{1}) \cap H^2(\C_+) \subeq \{0\} \cup \Out^2(\C_+).\]
\end{cor}
\begin{proof}
Let \(f \in \ker(\theta_h - \textbf{1}) \cap H^2(\C_+)\). If \(f = 0\) then obviously \(f \in \{0\} \cup \Out^2(\C_+)\). If \(f \neq 0\), then, by \fref{thm:OuterInnerDecomp}, there exists an outer function \({F \in \Out^2(\C_+)}\) and an inner function \(\phi \in \Inn(\C_+)\) such that
\[f = \phi \cdot F.\]
We set
\[k \coloneqq (\phi \phi^\flat)^{-1} h \in L^\infty(\R,\C)^\flat.\]
Then
\[h = (\phi \phi^\flat) k\]
and therefore, by \fref{lem:InnerSplitOff} the triples \((L^2(\R,\C),M_\phi^{-1} H^2(\C_+),\theta_k)\) and \((L^2(\R,\C),H^2(\C_+),\theta_k)\) are RPHSs. Since
\[k = (\phi \phi^\flat)^{-1} h = (\phi \phi^\flat)^{-1} \frac{f}{Rf} = \frac{\phi^{-1} f}{R (\phi^{-1} f)} = \frac F{RF},\]
we have
\[F \in \ker(\theta_k - \textbf{1}) \cap \Out^2(\C_+).\]
By \fref{prop:OuterMaximal} this implies that \((L^2(\R,\C),H^2(\C_+),\theta_k)\) is a maximal RPHS. Since we have \(H^2(\C_+) \subeq M_\phi^{-1} H^2(\C_+)\), this implies that
\[H^2(\C_+) = M_\phi^{-1} H^2(\C_+) = M_{\phi^*} H^2(\C_+),\]
which yields \(\phi^* \in H^\infty(\C_+)\) by \fref{prop:H2InclusionFunctions} and therefore \(\phi \in H^\infty(\C_-)\). On the other hand \(\phi \in H^\infty(\C_+)\), so by \fref{cor:HInftyIntersection} we get that \(\phi = C \cdot \textbf{1}\) for some \(C \in \T\). This implies that
\[f = \phi \cdot F = C \cdot F \in \Out^2(\C_+). \qedhere\]
\end{proof}
\begin{prop}\label{prop:kerNonTrivOuterFix}
Let \({h \in L^\infty(\R,\T)^\flat}\) such that the triple \((L^2(\R,\C),H^2(\C_+),\theta_h)\) is a RPHS. If
\[\ker(\theta_h - \textbf{1}) \cap H^2(\C_+) \neq \{0\}\]
then there exists an outer function \(F \in \Out^2(\C_+)\) such that
\[h = \frac F{RF}\]
and the triple \((L^2(\R,\C),H^2(\C_+),\theta_h)\) is a maximal RPHS.
\end{prop}
\begin{proof}
By \fref{cor:kernelOuterOrZero} we immediately get an outer function \(F \in \Out^2(\C_+)\) such that
\[F = \theta_h F = h \cdot RF,\]
which is equivalent to \(h = \frac F{RF}\). This, by \fref{prop:OuterMaximal}, implies that \((L^2(\R,\C),H^2(\C_+),\theta_h)\) is a maximal RPHS.
\end{proof}

\subsubsection{Involutions with fixed points and reflection positive functions}
In \fref{prop:kerNonTrivOuterFix} we have seen that the functions \({h \in L^\infty(\R,\T)^\flat}\), we are looking for, are of the form \(h = \frac F{RF}\) for some outer function \(F \in \Out^2(\C_+)\). To answer the question for which outer functions \(F \in \Out^2(\C_+)\) the triple \((L^2(\R,\C),H^2(\C_+),\theta_{\frac F{RF}})\) is a maximal RPHS, we need the definition of a reflection positive function:
\begin{definition}(cf. \cite[Def. 4.1.3]{NO18})
\begin{enumerate}[\rm (a)]
\item Let \(X\) be a non-empty set. A function
\[K: X \times X \to \R\]
is called a \textit{positive definite kernel} if, for any \(n \in \N\) and \(x_1,\dots,x_n \in X\), the matrix
\[(K(x_k,x_j))_{k,j = 1,\dots,n}\]
is a positive semidefinite matrix.
\item A function \(\phi: \R \to \R\) is called \textit{reflection positive}, if both kernels
\[\R \times \R \ni (s,t) \mapsto \phi\left(s-t\right) \quad \text{and} \quad \R_+ \times \R_+ \ni (s,t) \mapsto \phi\left(s+t\right)\]
are positive definite.
\end{enumerate}
\end{definition}
In the following, we will use the Fourier transform
\[\gls*{fourier1}: L^1(\R,\C) \to L^\infty(\R,\C)\]
with
\[(\cF_1 f)(t) = \int_\R e^{-itx} f(x) dx.\]
Since this function is injective, we can define the map
\[\cF_1^{-1}: \cF_1(L^1(\R,\C)) \to L^1(\R,\C), \quad \cF_1 f \mapsto f.\]
\begin{prop}\label{prop:outerequivalence}
Let \({h \in L^\infty(\R,\T)^\flat}\) and let
\[F \in \ker(\theta_h - \textbf{1}) \cap \Out^2(\C_+).\]
Then the following are equivalent:
\begin{enumerate}[\rm (a)]
\item \(\left(L^2\left(\R,\C\right),H^2\left(\C_+\right),\theta_h\right)\) is a maximal RPHS.
\item \(\left(L^2\left(\R,\C\right),H^2\left(\C_+\right),\theta_h\right)\) is a RPHS.
\item The function \(\R \ni t \mapsto \phi(t) \coloneqq \braket*{S_t F}{F}\) is reflection positive.
\item \(|F|^2 = \cF_1^{-1} \phi\) for some reflection positive function \(\phi \in \cF_1(L^1(\R,\C))\).
\item \(F = \Out\left(C,\sqrt{\cF_1^{-1} \phi}\right)\) for some \(C \in \T\) and some reflection positive function \(\phi \in \cF_1(L^1(\R,\C))\) with \(I\left(\sqrt{\cF_1^{-1} \phi}\right) < \infty\).
\end{enumerate}
\end{prop}
\begin{proof}
\begin{itemize}
\item[(a) \(\Leftrightarrow\) (b):]
This follows immediately by \fref{prop:OuterMaximal}.
\item[(b) \(\Leftrightarrow\) (c):]
The kernel
\[\R \times \R \ni (s,t) \mapsto \phi(s-t)\]
is always positive definite, since, for \(n \in \N\), \(v_1,\dots,v_n \in \C\) and \(t_1,\dots,t_n \in \R\), we have
\begin{align*}
\sum_{k=1}^n \sum_{j=1}^n \oline{v_k}\phi\left(t_k-t_j\right)v_j &= \sum_{k=1}^n \sum_{j=1}^n \oline{v_k}\braket*{S_{t_k-t_j} F}{F}v_j = \sum_{k=1}^n \sum_{j=1}^n \oline{v_k}\braket*{S_{t_k} F}{S_{t_j} F}v_j
\\&= \braket*{\sum_{k=1}^n v_k S_{t_k} F}{\sum_{j=1}^n v_j S_{t_j} F} \geq 0.
\end{align*}
Therefore it remains to show that \(\left(L^2\left(\R,\C\right),H^2\left(\C_+\right),\theta_h\right)\) is a RPHS, if and only if the function
\[K: \R_+ \times \R_+ \to \C, \quad (s,t) \mapsto \phi(s+t)\]
is a positive definite kernel. For \(n \in \N\), \(v_1,\dots,v_n \in \C\) and \(t_1,\dots,t_n \in \R\), we have
\begin{align*}
\sum_{k=1}^n \sum_{j=1}^n \oline{v_k}\phi(t_k+t_j) v_j &= \sum_{k=1}^n \sum_{j=1}^n \oline{v_k}\braket*{S_{t_k+t_j} F}{F}v_j = \sum_{k=1}^n \sum_{j=1}^n \oline{v_k}\braket*{S_{t_k} F}{S_{-t_j} F}v_j
\\&=\sum_{k=1}^n \sum_{j=1}^n \oline{v_k}\braket*{S_{t_k} F}{S_{-t_j} \theta_h F}v_j = \sum_{k=1}^n \sum_{j=1}^n \oline{v_k}\braket*{S_{t_k} F}{\theta_h S_{t_j} F}v_j
\\&= \braket*{\left(\sum_{k=1}^n v_k S_{t_k} F\right)}{\theta_h \left(\sum_{j=1}^n v_j S_{t_j} F\right)}.
\end{align*}
This shows that \(K\) is a positive definite kernel, if and only if
\[\braket*{f}{\theta_h f} \geq 0 \quad \forall f \in \Spann S(\R_+)F.\]
Since, by \fref{thm:outerSpan}, we have	
\[H^2(\C_+) = \overline{\Spann S(\R_+)F},\]
this is equivalent to
\[\braket*{f}{\theta_h f} \geq 0 \quad \forall f \in H^2(\C_+)\]
and therefore to the triple \(\left(L^2\left(\R,\C\right),H^2\left(\C_+\right),\theta_h\right)\) being a RPHS.
\end{itemize}
\item[(c) \(\Leftrightarrow\) (d):]
This follows immediately by
\begin{equation*}
\phi(t) = \braket*{S_t F}{F} =  \int_{\R} e^{-itx}\left|F\left(x\right)\right|^2 \,dx = (\cF_1 |F|^2)(t), \quad t \in \R.
\end{equation*}
\item[(d) \(\Leftrightarrow\) (e):]
We have \(|F|^2 = \cF_1^{-1} \phi\) for some reflection positive function \(\phi\), if and only if \(|F| = \sqrt{\cF_1^{-1} \phi}\), which, by \fref{thm:OuterBetrag}, is equivalent to
\[F = \Out\left(C,\sqrt{\cF_1^{-1} \phi}\right)\]
for some \(C \in \T\).
\end{proof}
This proposition shows that any fixed point \(F \in \Out^2(\C_+)\) of an involution \(\theta_h\), such that \(\left(L^2\left(\R,\C\right),H^2\left(\C_+\right),\theta_h\right)\) is a maximal RPHS, is of the form
\[F = \Out\left(C,\sqrt{\cF_1^{-1} \phi}\right)\] for some reflection positive function \(\phi \in \cF_1(L^1(\R,\C))\) with
\[I\left(\sqrt{\cF_1^{-1} \phi}\right) < \infty.\]
The question we now have to answer is for which reflection positive function \(\phi\) these two conditions are fulfilled. For this, we have to better understand reflection positive functions, and for this, we make the following definitions:
\begin{definition}
\begin{enumerate}[\rm (a)]
\item Let \(X\) be a locally compact Hausdorff space. We denote the set of Borel measures on \(X\) by \(\gls*{BMX}\) and the subset of finite Borel measures on \(X\) by \(\gls*{BMXfin}\).
\item For a measure \(\mu \in \mathcal{BM}^{\mathrm{fin}}\left(\R_{\geq 0}\right)\) we set
\begin{equation*}
\gls*{phimu}\left(t\right) \coloneqq \int_{\R_{\geq 0}} e^{-\lambda\left|t\right|}d\mu\left(\lambda\right), \quad t \in \R.
\end{equation*}
\item For a measure \(\mu \in \mathcal{BM}^{\mathrm{fin}}\left(\R_+\right)\) we set
\begin{equation*}
\gls*{psimu}\left(p\right) \coloneqq \frac{1}{\pi}\int_{\R_+} \frac{\lambda}{\lambda^2+p^2} \,d\mu\left(\lambda\right), \quad p \in \R^\times.
\end{equation*}
\end{enumerate}
\end{definition}
The following results will motivate this definition. Here we identify \(\mathcal{BM}^{\mathrm{fin}}\left(\R_+\right)\) with the set
\[\{\mu \in \mathcal{BM}^{\mathrm{fin}}\left(\R_{\geq 0}\right): \mu(\{0\}) = 0\}.\]
\begin{theorem}\label{thm:Widder}{\rm (cf. \cite[Thm. 4.1.1, Ex. 4.1.4]{NO18})}
A function \(\phi: \R \to \R\) is reflection positive, if and only if there exists a measure \(\mu \in \mathcal{BM}^{\mathrm{fin}}\left(\R_{\geq 0}\right)\) such that \(\phi = \phi_\mu\).
\end{theorem}
\begin{prop}\label{prop:PhiPsiFourier}
Let \(\mu \in \mathcal{BM}^{\mathrm{fin}}\left(\R_+\right)\). Then \(\psi_\mu \in L^1(\R,\C)\) with \(\left\lVert \psi_\mu\right\rVert_1 = \mu(\R_+)\) and
\[\cF_1 \psi_\mu = \phi_\mu.\]
\end{prop}
\begin{proof}
Since, for \(\lambda \in \R_+\), we have
\[\int_\R \frac{\lambda}{\lambda^2+p^2} \,dp = \int_\R \frac{1}{1+\left(\frac p \lambda\right)^2} \,\frac 1 \lambda \,dp = \int_\R \frac{1}{1+x^2} \,dx = \left[\arctan(x)\right]_{-\infty}^\infty = \frac \pi 2 - \left(-\frac \pi 2\right) = \pi,\]
using the Fubini-Tonelli Theorem, we get
\begin{align*}
\mu(\R_+) &= \frac{1}{\pi}\int_{\R_+} \pi \,d\mu\left(\lambda\right) = \frac{1}{\pi}\int_{\R_+} \int_\R \frac{\lambda}{\lambda^2+p^2} \,dp \,d\mu\left(\lambda\right)
\\&= \int_\R \frac{1}{\pi}\int_{\R_+} \frac{\lambda}{\lambda^2+p^2} \,d\mu\left(\lambda\right) \,dp
= \int_\R \psi_\mu(p) \,dp = \left\lVert \psi_\mu\right\rVert_1
\end{align*}
and therefore \(\psi_\mu \in L^1(\R,\C)\).

To calculate the Fourier transform, for \(\lambda \in \R_+\), we define the function \(f_\lambda \in L^1(\R,\C) \cap L^2(\R,\C)\) by
\[f_\lambda(t) = e^{-\lambda|t|}.\]
Then, for \(\lambda \in \R_+\) and \(x \in \R\), we get
\[\int_{\R_+} e^{itx} f_\lambda(t) \,dt = \int_{\R_+} e^{itx} e^{-\lambda t} \,dt = \int_{\R_+} e^{t(-\lambda+ix)} \,dt = \left[\frac 1{-\lambda+ix} e^{t(-\lambda+ix)}\right]_0^\infty = \frac 1{\lambda-ix}\]
and therefore, due to the symmetry of \(f_\lambda\), we have
\[\int_{\R_-} e^{itx} f_\lambda(t) \,dt = \int_{\R_+} e^{-itx} f_\lambda(t) \,dt = \frac 1{\lambda+ix},\]
so
\[\sqrt{2\pi} \cdot (\cF^{-1}f_\lambda)(x) = \int_\R e^{itx} f_\lambda(t) \,dt = \left[\frac 1{\lambda-ix} + \frac 1{\lambda+ix}\right] = \frac {2\lambda}{\lambda^2 + x^2} \eqqcolon g_\lambda(x).\]
This implies
\[\cF_1 g_\lambda = \sqrt{2\pi} \cdot \cF g_\lambda = 2\pi \cdot f_\lambda\]
and therefore
\[\frac{1}{\pi} \int_\R e^{-itp} \frac{\lambda}{\lambda^2+p^2} \,dp = \frac 1{2\pi} \cdot (\cF_1 g_\lambda)(t) = f_\lambda(t) = e^{-\lambda|t|}, \quad t \in \R.\]
Therefore, using the Fubini-Tonelli Theorem, we have
\begin{align*}
(\cF_1 \psi)(t) &= \int_\R e^{-itp} \psi(p) \,dp = \int_\R e^{-itp} \frac{1}{\pi}\int_{\R_+} \frac{\lambda}{\lambda^2+p^2} \,d\mu\left(\lambda\right) \,dp
\\&= \int_{\R_+} \frac{1}{\pi} \int_\R e^{-itp} \frac{\lambda}{\lambda^2+p^2} \,dp \,d\mu\left(\lambda\right) = \int_{\R_+} e^{-\lambda|t|} \,d\mu\left(\lambda\right) = \phi_\mu(t). \qedhere
\end{align*}
\end{proof}
\begin{cor}\label{cor:muFourierImage}
Let \(\mu \in \mathcal{BM}^{\mathrm{fin}}\left(\R_{\geq 0}\right)\). Then \(\phi_\mu \in \cF_1(L^1(\R,\C))\), if and only if \(\mu(\{0\}) = 0\).
\end{cor}
\begin{proof}
We write
\[\mu = \mu(\{0\}) \cdot \delta_0 + \tilde \mu\]
with \(\tilde \mu \in \mathcal{BM}^{\mathrm{fin}}\left(\R_+\right)\). Then
\[\phi_\mu = \mu(\{0\}) \cdot \phi_{\delta_0} + \phi_{\tilde \mu} = \mu(\{0\}) \cdot \textbf{1} + \phi_{\tilde \mu}.\]
Since, by \fref{prop:PhiPsiFourier}, we have \(\phi_{\tilde \mu} \in \cF_1(L^1(\R,\C))\), this shows that \(\psi_\mu \in \cF_1(L^1(\R,\C))\), if and only if \(\mu(\{0\}) \cdot \textbf{1} \in \cF_1(L^1(\R,\C))\). The latter is equivalent to \(\mu(\{0\}) = 0\).
\end{proof}
\begin{cor}\label{cor:phiFourierImage}
Let \(\phi: \R \to \R\) be a reflection positive function. Then \(\phi \in \cF_1(L^1(\R,\C))\), if and only if \(\phi = \phi_\mu\) for some measure \(\mu \in \mathcal{BM}^{\mathrm{fin}}\left(\R_+\right)\).
\end{cor}
\begin{proof}
This follows immediately by \fref{thm:Widder} and \fref{cor:muFourierImage}.
\end{proof}
We now want to see, for which measures \(\mu \in \mathcal{BM}^{\mathrm{fin}}\left(\R_+\right)\) one has
\[I\left(\sqrt{\cF_1^{-1}\phi_\mu}\right) < \infty.\]
For this, we make some definitions:
\begin{definition}
\begin{enumerate}[\rm (a)]
\item We write \(\left[0,\infty\right]\) for the one-point compactification of \(\left[0,\infty\right)\).
\item Let \(\nu \in \mathcal{BM}^{\mathrm{fin}}\left(\left[0,\infty\right]\right)\). For \(p \in \R^\times\) we consider the bounded and continuous function \({\lambda \mapsto \frac{1+\lambda^2}{p^2 + \lambda^2}}\) with \(\infty \mapsto \frac{1+\infty^2}{p^2 + \infty^2} \eqqcolon 1\) and define
\[\gls*{Psinu}(p) \coloneqq \frac{1}{\pi}\int_{[0,\infty]} \frac{1+\lambda^2}{p^2 + \lambda^2} \,d\nu\left(\lambda\right).\]
\item We define the map
\[\gls*{W}: \mathcal{BM}^{\mathrm{fin}}\left(\R_+\right) \to \mathcal{BM}^{\mathrm{fin}}\left(\left[0,\infty\right]\right), \quad d(W(\mu))(\lambda) = \frac{\lambda}{1+\lambda^2} \,d\mu(\lambda),\]
\end{enumerate}
\end{definition}
\begin{remark}
That for \(\mu \in \mathcal{BM}^{\mathrm{fin}}\left(\R_+\right)\) one actually has \(W(\mu) \in \mathcal{BM}^{\mathrm{fin}}\left([0,\infty]\right)\) follows from the fact that, for \(\lambda \in \R_+\), one has
\[0 \leq (\lambda - 1)^2 = 1+ \lambda^2 - 2 \lambda\]
and therefore \(\frac{\lambda}{1+\lambda^2} \leq \frac 12\), so
\[(W(\mu))([0,\infty]) = \int_{[0,\infty]} d(W(\mu))(\lambda) = \int_{\R_+} \frac{\lambda}{1+\lambda^2} \,d\mu(\lambda) \leq \int_{\R_+} \frac 12\,d\mu(\lambda) = \frac 12 \cdot \mu(\R_+) < \infty.\]
\end{remark}
\begin{lemma}\label{lem:BigSmallPsi}
For \(\mu \in \mathcal{BM}^{\mathrm{fin}}\left(\R_+\right)\) one has
\[\Psi_{W(\mu)} = \psi_\mu.\]
\end{lemma}
\begin{proof}
For \(p \in \R^\times\) one has
\begin{align*}
\Psi_{W(\mu)}(p) &= \frac{1}{\pi}\int_{[0,\infty]} \frac{1+\lambda^2}{p^2 + \lambda^2} \,d(W\mu)\left(\lambda\right) = \frac{1}{\pi}\int_{\R_+} \frac{1+\lambda^2}{p^2 + \lambda^2} \cdot \frac{\lambda}{1+\lambda^2} \,d\mu\left(\lambda\right)
\\&= \frac{1}{\pi}\int_{\R_+} \frac{\lambda}{\lambda^2+p^2} \,d\mu\left(\lambda\right) = \psi_\mu(p). \qedhere
\end{align*}
\end{proof}
We now want to see that, for the functions \(\Psi_\nu\) with \(\nu \in \mathcal{BM}^{\mathrm{fin}}\left([0,\infty]\right)\) one has \(I(\Psi_\nu) < \infty\). For this, we need the following lemma:
\begin{lem}\label{lem:PsiEstimate}
Let \(\nu \in \mathcal{BM}^{\mathrm{fin}}\left([0,\infty]\right)\). Then, one has the estimate
\[\nu([0,\infty]) \cdot \vartheta(p) \leq \Psi_\nu(p) \leq \nu([0,\infty]) \cdot \eta(p), \quad p \in \R^\times,\]
with the functions
\[\vartheta: \R \to \R, \quad p \mapsto \begin{cases}1 & \text{if} \,\left|p\right| < 1 \\ \frac{1}{p^2} & \text{if} \,\left|p\right| \geq 1\end{cases}\]
and
\[\eta: \R^\times \to \R, \quad p \mapsto \begin{cases} \frac 1{p^2} &\text{if }\left|p\right| < 1 \\ 1 &\text{if }\left|p\right| \geq 1.\end{cases}\]
\end{lem}
\begin{proof}
Let \(\lambda \in [0,\infty]\). Then, for \(p \in \R^\times\) with \(\left|p\right|<1\), we have
\[1 = \frac{p^2+\lambda^2}{p^2+\lambda^2} \leq \frac{1+\lambda^2}{p^2+\lambda^2} \leq \frac{1+\lambda^2}{p^2+p^2\lambda^2} = \frac 1{p^2}\]
and for \(p \in \R^\times\) with \(\left|p\right|\geq 1\), we get
\[\frac 1{p^2} = \frac{1+\lambda^2}{p^2+p^2\lambda^2} \leq \frac{1+\lambda^2}{p^2+\lambda^2} \leq \frac{p^2+\lambda^2}{p^2+\lambda^2} = 1,\]
so, for every \(\lambda \in [0,\infty]\), we get
\[\vartheta(p) \leq \frac{1+\lambda^2}{p^2+\lambda^2} \leq \eta(p).\]
The statement then follows by integrating over \(\lambda \in [0,\infty]\) with respect to the measure \(\nu\).
\end{proof}
\begin{lem}\label{lem:PsiNuIntegral}
Let \(\nu \in \mathcal{BM}^{\mathrm{fin}}\left([0,\infty]\right) \setminus \{0\}\). Then
\begin{equation*}
I(\Psi_\nu) < \infty.
\end{equation*}
\end{lem}
\begin{proof}
Since \(\nu \neq 0\), we have \(\nu([0,\infty]) > 0\). Then, by \fref{lem:OuterHomo}, we have \(I(\Psi_\nu) < \infty\), if and only if \(I\left(\frac 1{\nu([0,\infty])} \cdot \Psi_\nu\right) < \infty\). By \fref{lem:PsiEstimate}, for \(p \in \R^\times\) with \(|p|<1\), we have
\[0 = \log(1) \leq \log\left(\frac 1{\nu([0,\infty])} \cdot \Psi_\nu\right) \leq \log\left(\frac 1{p^2}\right) = -2 \log(p) = 2|\log(p)|,\]
and for \(p \in \R^\times\) with \(|p|\geq 1\), we have
\[-2 |\log(p)| = -2 \log(p) = \log\left(\frac 1{p^2}\right) \leq \log\left(\frac 1{\nu([0,\infty])} \cdot \Psi_\nu\right) \leq \log(1) = 0.\]
This shows that, for any \(p \in \R^\times\), we have
\[\left|\log\left(\frac 1{\nu([0,\infty])} \cdot \Psi_\nu\right)\right| \leq 2|\log(p)|\]
and therefore, we have
\begin{align*}
\int_\R \frac{\left|\log\left(\frac 1{\nu([0,\infty])} \cdot \Psi_\nu\left(p\right)\right)\right|}{1 + p^2} \,dp &\leq \int_\R \frac{2 |\log p|}{1 + p^2} \,dp = 4 \int_{\R_+} \frac{|\log p|}{1 + p^2} \,dp
\\&= 4 \int_\R \frac{|x|}{1 + e^{2x}} \cdot e^x \,dx = 4 \int_\R \frac{|x|}{e^{-x} + e^x} \,dx < \infty. \qedhere
\end{align*}
\end{proof}
Using this lemma, we can make the following definition:
\begin{definition}\label{def:FNuHNu}
For \(\nu \in \mathcal{BM}^{\mathrm{fin}}\left([0,\infty]\right) \setminus \{0\}\) we define
\[\gls*{Fnu} \coloneqq \Out\left(\sqrt{\Psi_\nu}\right) \in \mathrm{Out}(\C_+)\]
and
\[\gls*{hnu} \coloneqq \frac{F_\nu}{R F_\nu} \in L^\infty(\R,\T)^\flat.\]
\end{definition}
\begin{remark}
\begin{enumerate}[\rm (a)]
\item That \(F_\nu\) is defined follows by \fref{lem:PsiNuIntegral}, since we have
\[I\left(\sqrt{\Psi_\nu}\right) = \frac 12 I\left(\Psi_\nu\right)< \infty.\]
\item To see that \(h_\nu\) actually defines an element in \(L^\infty(\R,\T)^\flat\) we first notice that, by \fref{thm:OuterBetrag}, we have \(|F_\nu| = \sqrt{\Psi_\nu}\) on \(\R\). Then, using that
\[\Psi_\nu(-p) = \Psi_\nu(p), \quad p \in \R^\times,\]
we get
\[|h_\nu(x)| = \frac{F_\nu(x)}{|F_\nu(-x)|} = \frac{\sqrt{\Psi_\nu(x)}}{\sqrt{\Psi_\nu(-x)}} = 1, \quad x \in \R^\times\]
and therefore \(h_\nu \in L^\infty(\R,\T)\). This then implies that
\[h_\nu^\flat(x) = \overline{h_\nu(-x)} = \frac 1{h_\nu(-x)} = \frac{RF_\nu(-x)}{F_\nu(-x)} = \frac{F_\nu(x)}{RF_\nu(x)} = h_\nu(x), \quad x \in \R^\times\]
and therefore \(h \in L^\infty(\R,\T)^\flat\). 
\end{enumerate}
\end{remark}
With this we get our theorem classifying involutions \(\theta_h\) with fixed points in the Hardy space such that \((L^2(\R,\C),H^2(\C_+),\theta_h)\) is a (maximal) RPHS:
\begin{thm}\label{thm:thetaFixedPointCharac}
Let \(h \in L^\infty(\R,\T)^\flat\).
Then the following are equivalent:
\begin{enumerate}[\rm (a)]
\item \(\left(L^2\left(\R,\C\right),H^2\left(\C_+\right),\theta_h\right)\) is a RPHS with \(\ker(\theta_h - \textbf{1}) \cap H^2(\C_+) \neq \{0\}\).
\item \(\left(L^2\left(\R,\C\right),H^2\left(\C_+\right),\theta_h\right)\) is a maximal RPHS with \(\ker(\theta_h - \textbf{1}) \cap H^2(\C_+) \neq \{0\}\).
\item \(h = h_{W(\mu)}\) for some measure \(\mu \in \mathcal{BM}^{\mathrm{fin}}\left(\R_+\right) \setminus \{0\}\).
\end{enumerate}
\end{thm}
\newpage
\begin{proof}
\begin{itemize}
\item[(a) \(\Leftrightarrow\) (b):] This follows immediately by \fref{prop:kerNonTrivOuterFix}.
\item[(b) \(\Rightarrow\) (c):] If \(\left(L^2\left(\R,\C\right),H^2\left(\C_+\right),\theta_h\right)\) is a maximal RPHS with
\[\ker(\theta_h - \textbf{1}) \cap H^2(\C_+) \neq \{0\},\]
then, by \fref{prop:kerNonTrivOuterFix}, there exists an outer function \(F \in \Out^2(\C_+)\) such that
\[h = \frac F{RF}.\]
For this function, we have
\[F \in \ker(\theta_h - \textbf{1}) \cap \Out^2(\C_+)\]
and therefore, by \fref{prop:outerequivalence}, we have
\[F = \Out\left(C,\sqrt{\cF_1^{-1} \phi}\right)\]
for some \(C \in \T\) and some reflection positive function \(\phi \in \cF_1(L^1(\R,\C))\) with
\[I\left(\sqrt{\cF_1^{-1} \phi}\right) < \infty.\]
By \fref{cor:phiFourierImage} we then have \(\phi = \phi_\mu\) for some measure \(\mu \in \mathcal{BM}^{\mathrm{fin}}\left(\R_+\right)\). This, by \fref{prop:PhiPsiFourier} and \fref{lem:BigSmallPsi}, yields
\[\cF_1^{-1} \phi = \cF_1^{-1} \phi_\mu = \psi_\mu = \Psi_{W(\mu)}\]
and therefore
\[F = \Out\left(C,\sqrt{\cF_1^{-1} \phi}\right) = C \cdot \Out\left(\sqrt{\Psi_{W(\mu)}}\right) = C \cdot F_{W(\mu)}.\]
This implies
\[h = \frac F{RF} = \frac{C \cdot F_{W(\mu)}}{C \cdot RF_{W(\mu)}} = \frac{F_{W(\mu)}}{RF_{W(\mu)}} = h_{W(\mu)}.\]
\item[(c) \(\Rightarrow\) (b):] If \(h = h_{W(\mu)}\) for some measure \(\mu \in \mathcal{BM}^{\mathrm{fin}}\left(\R_+\right)\), we have
\[\theta_h F_{W(\mu)} = \frac{F_{W(\mu)}}{RF_{W(\mu)}} \cdot RF_{W(\mu)} = F_{W(\mu)}\]
and therefore
\[F_{W(\mu)} \in \ker(\theta_h - \textbf{1}) \cap \Out^2(\C_+)\]
so especially
\[\ker(\theta_h - \textbf{1}) \cap H^2(\C_+) \neq \{0\}.\]
Further, by \fref{lem:BigSmallPsi} and \fref{prop:PhiPsiFourier}, we have
\[|F_{W(\mu)}|^2 = \Psi_{W(\mu)} = \psi_\mu = \cF_1^{-1} \phi_\mu\]
and therefore \(\left(L^2\left(\R,\C\right),H^2\left(\C_+\right),\theta_h\right)\) is a maximal RPHS by \fref{prop:outerequivalence}. \qedhere
\end{itemize}
\end{proof}

\subsection{General involutions}
In the last section, we have given a full characterization of the functions \(h \in L^\infty(\R,\T)^\flat\) such that \(\left(L^2\left(\R,\C\right),H^2\left(\C_+\right),\theta_h\right)\) is a maximal RPHS under the extra condition that
\[\ker(\theta_h - \textbf{1}) \cap H^2(\C_+) \neq \{0\}.\]
We have seen that these are precisely the functions of the form
\[h = h_{W(\mu)}\]
for some measure \(\mu \in \mathcal{BM}^{\mathrm{fin}}\left(\R_+\right) \setminus \{0\}\). In this section, we want to get rid of the extra assumption that \(\theta_h\) has a non-zero fixed point in the Hardy space.

\subsubsection{Fixed points outside the Hardy space and outer functions}
First we notice that the functions \(F_\nu\) and \(h_\nu\) are defined for general measures \(\nu \in \mathcal{BM}^{\mathrm{fin}}\left(\left[0,\infty\right]\right) \setminus \{0\}\) and not just for measures of the form \(W(\mu)\) for some measure \(\mu \in \mathcal{BM}^{\mathrm{fin}}\left(\R_+\right)\). Although, by \fref{thm:thetaFixedPointCharac}, these involution \(\theta_h\) have no non-zero fixed point in the Hardy space, they still formally satisfy the equation
\[\theta_h F_\nu = \frac{F_\nu}{R F_\nu} RF_\nu = F_\nu,\]
which we will interpret as them having a fixed point in a space slightly bigger than the Hardy space. We will now define this space:
\begin{definition}
\begin{enumerate}[\rm (a)]
\item We define the function \(\gls*{Lambda} \in \cO(\C_+)\) by
\[\Lambda(z) \coloneqq \frac{iz}{(z+i)^2}.\]
\item We set
\[\gls*{H2Lambda} \coloneqq \{f \in \cO(\C_+): \Lambda \cdot f \in H^2(\C_+)\}.\]
\item We define
\[\gls*{Out2Lambda} \coloneqq H^2_\Lambda(\C_+) \cap \Out(\C_+).\]
\end{enumerate}
\end{definition}
\begin{remark}\label{rem:DefLambda}
\begin{enumerate}[\rm (a)]
\item We have
\[\left|\frac{ix}{(x+i)^2}\right| = \frac{|x|}{1+x^2} \quad \forall x \in \R,\]
so by \fref{lem:OuterExamples} we have
\[\Lambda = \Out(|\Lambda|) \in \Out(\C_+).\]
Since \(|\Lambda| \in L^2(\R,\C) \cap L^\infty(\R,\C)\), by \fref{thm:OuterBetrag}, we get
\[\Lambda \in \Out^2(\C_+) \cap H^\infty(\C_+).\]
\item Since \(\Lambda \in H^\infty(\C_+)\), we get
\[M_\Lambda H^2(\C_+) \subeq H^2(\C_+)\]
and therefore
\[H^2(\C_+) \subeq H^2_\Lambda(\C_+).\]
\item Since \(0 \notin \Out(\C_+)\), we also have \(0 \notin \Out^2_\Lambda(\C_+)\).
\end{enumerate}
\end{remark}
\newpage
The following lemma about the function \(\Lambda\) will be useful later:
\begin{lemma}\label{lem:phiLambdaIdentity}
Let \(\phi_i \in \Inn(\C_+)\) be the inner function with
\[\phi_i(z) = \frac{z-i}{z+i}, \quad z \in \C_+.\]
Then
\[\frac{\Lambda(x)}{(R\Lambda)(x)} = -\frac{\phi_i(x)}{(R\phi_i)(x)} \quad \forall x \in \R.\]
\end{lemma}
\begin{proof}
For \(x \in \R\) we have
\[\frac{\Lambda(x)}{(R\Lambda)(x)} = \frac{\frac{ix}{(x+i)^2}}{\frac{-ix}{(-x+i)^2}} = -\frac{(-x+i)^2}{(x+i)^2} = -\frac{\frac{x-i}{x+i}}{\frac{-x-i}{-x+i}} = -\frac{\phi_i(x)}{(R\phi_i)(x)}. \qedhere\]
\end{proof}
Every function \(f \in H^2_\Lambda(\C_+)\) can be written as \(f = \frac{g}{\Lambda}\) for some function \(g \in H^2(\C_+)\). This, in particular, implies that the limit
\[f(x) \coloneqq \lim_{\epsilon \downarrow 0} f(x+i\epsilon) = \left(\lim_{\epsilon \downarrow 0} g(x+i\epsilon)\right) \cdot \left(\lim_{\epsilon \downarrow 0} \Lambda(x+i\epsilon)\right)^{-1}\]
exists for almost every \(x \in \R\) since the corresponding limits exist for outer functions and Hardy space functions. This, in particular, allows us to interpret functions \(f \in H^2_\Lambda(\C_+)\) as almost everywhere defined functions on \(\R\), similarly to outer functions and Hardy space functions. Further, we have
\[\infty > \left\lVert g\right\rVert_2^2 = \int_\R |g(x)|^2 dx = \int_\R |f(x)|^2 \cdot |\Lambda(x)|^2 dx = \int_\R |f(x)|^2 \frac{x^2}{(1+x^2)^2}dx,\]
so
\[f \in L^2\left(\R,\C,\frac{x^2}{(1+x^2)^2}dx\right).\]
This motivates the following definition:
\begin{definition}
\begin{enumerate}[\rm (a)]
\item We define a measure \(\gls*{rho}\) on \(\R\) by
\[d\rho(x) = \frac{x^2}{(1+x^2)^2}dx.\]
\item For a function \(h \in L^\infty(\R,\C)^\flat\) we define an involution \(\gls*{Thetah} \in B(L^2\left(\R,\C,\rho\right))\) by
\[\Theta_h f \coloneqq h \cdot Rf.\]
\end{enumerate}
\end{definition}
\begin{remark}
\begin{enumerate}[\rm (a)]
\item The boundedness of the function \(x \mapsto \frac{x^2}{(1+x^2)^2}\) implies that
\[L^2(\R,\C) \subeq L^2(\R,\C,\rho).\]
\item By definition, we have
\[\Theta_h \big|_{L^2(\R,\C)}^{L^2(\R,\C)} = \theta_h.\]
\item As seen before, we can identify \(H^2_\Lambda(\C_+)\) with a subspace of \(L^2(\R,\C,\rho)\) by taking boundary values.
\end{enumerate}
\end{remark}
\newpage
The following proposition shows how useful this extension is since it shows that, although not every involution \(\theta_h\) has a non-zero fixed point in the Hardy space, the slightly extended function \(\Theta_h\) always has such a fixed point in the extended Hardy space:
\begin{prop}\label{prop:GeneralizedFixedPoint}
Let \(h \in L^\infty(\R,\C)^\flat\) such that \((L^2(\R,\C),H^2(\C_+),\theta_h)\) is a maximal RPHS. Then
\[\ker(\Theta_h-\textbf{1}) \cap H^2_\Lambda(\C_+) \neq \{0\}.\]
\end{prop}
\begin{proof}
Since \((L^2(\R,\C),H^2(\C_+),\theta_h)\) is a maximal RPHS, we know that \((L^2(\R,\C),M_{\phi_i}^{-1}H^2(\C_+),\theta_h)\) is not a RPHS and therefore there exists a function \(f \in M_{\phi_i}^{-1}H^2(\C_+)\) such that
\[\braket*{f}{\theta_h f} < 0.\]
Now, let
\[f = f_+ + f_- \quad \text{with} \quad f_\pm \in L^2(\R,\C)^{\pm \theta_h}.\]
Since 
\((L^2(\R,\C),H^2(\C_+),\theta_h)\) is a maximal RPHS, by \fref{thm:EPlusMaximal}, there exists a function \({\tilde f \in H^2(\C_+)}\) and a function \(\tilde f_- \in L^2(\R,\C)^{-\theta_h}\) such that
\[f_+ = \tilde f + \tilde f_-.\]
We now set
\[g \coloneqq f-\tilde f \in M_{\phi_i}^{-1}H^2(\C_+) + H^2(\C_+) = M_{\phi_i}^{-1}H^2(\C_+).\]
Then \(\tilde g \coloneqq M_{\phi_i} g \in H^2(\C_+)\). Since
\[\braket*{f}{\theta_h f} < 0 \quad \text{and} \quad \braket*{\tilde f}{\theta_h \tilde f} \geq 0,\]
we have \(f \neq \tilde f\) and therefore \(g \neq 0\), so \(\tilde g \neq 0\). Further
\[g = f-\tilde f = (f_+ + f_-) - (f_+ - \tilde f_-) = f_- + \tilde f_- \in L^2(\R,\C)^{-\theta_h},\]
so \(\theta_h g = - g\). This yields
\[\theta_h \tilde g = \theta_h M_{\phi_i}g = M_{R\phi_i} \theta_h g = -M_{R\phi_i} g = -M_{\frac{R\phi_i}{\phi_i}} M_{\phi_i}g = -M_{\frac{R\phi_i}{\phi_i}} \tilde g = M_{\frac{R\Lambda}{\Lambda}} \tilde g, \tag{\(\bigstar\)}\]
using \fref{lem:phiLambdaIdentity} in the last step. We now define
\[F \coloneqq \frac{\tilde g}{\Lambda} \in H^2_\Lambda(\C_+) \setminus \{0\}.\]
Then
\[\Theta_h F = \Theta_h \frac{\tilde g}{\Lambda} = \frac 1{R\Lambda} \Theta_h \tilde g = \frac 1{R\Lambda} \theta_h \tilde g \overset{(\bigstar)}{=} \frac 1{R\Lambda} M_{\frac{R\Lambda}{\Lambda}} \tilde g = \frac{\tilde g}{\Lambda} = F,\]
so we have
\[F \in \left(\ker(\Theta_h-\textbf{1}) \cap H^2_\Lambda(\C_+)\right) \setminus \{0\}. \qedhere\]
\end{proof}
In \fref{cor:kernelOuterOrZero} we have seen that, for any \({h \in L^\infty(\R,\T)^\flat}\) for which \((L^2(\R,\C),H^2(\C_+),\theta_h)\) is a maximal RPHS, the kernel \(\ker(\theta_h - \textbf{1}) \cap H^2(\C_+)\) just consists of outer functions. We now want to see that the same is true for the extended involution \(\Theta_h\). For this we need the following analogue of \fref{prop:OuterMaximal} for the involution \(\Theta_h\):
\begin{prop}\label{prop:maxPos}
Let \({h \in L^\infty(\R,\T)^\flat}\) such that the triple \((L^2(\R,\C),H^2(\C_+),\theta_h)\) is a RPHS and
\[\ker(\Theta_h - \textbf{1}) \cap \Out^2_\Lambda(\C_+) \neq \eset.\]
Then \((L^2(\R,\C),H^2(\C_+),\theta_h)\) is a maximal RPHS.
\end{prop}
\begin{proof}
By assumption, we can choose a function
\[F \in \ker(\Theta_h - \textbf{1}) \cap \Out^2_\Lambda(\C_+).\]
We set
\[G \coloneqq \Lambda \cdot F \in \mathrm{Out}^2(\C_+).\]
We now set
\[V \coloneqq M_{\phi_i}^{-1}\left(M_{\phi_i}^{-1}H^\infty(\C_-) \oplus \C\textbf{1} \oplus M_{\phi_i} H^\infty(\C_+)\right) \cdot G \subeq L^\infty(\R,\C) \cdot G \subeq L^2(\R,\C).\]
Our first goal is to prove that \(V\) is dense in \(L^2(\R,\C)\). Since, by \fref{prop:BlaschkeFactorization}, we have
\[M_{\phi_i} H^2(\C_+) = \{f \in H^2(\C_+): f(i) = 0\}\]
and \(G(i) \neq 0\), we have
\[H^2(\C_+) = \C G + M_{\phi_i} H^2(\C_+) = \C G + M_{\phi_i} \overline{H^\infty(\C_+)G},\]
using \fref{thm:outerSpan} for the last equality. This yields
\[\overline{\C G + M_{\phi_i} H^\infty(\C_+)G} \supeq \C G + M_{\phi_i} \overline{H^\infty(\C_+)G} = H^2(\C_+) \supeq H^\infty(\C_+)G\]
and therefore, using that
\[(R\phi_i)(x) = \frac{-x-i}{-x+i} = \frac{x+i}{x-i} = \left(\phi_i(x)\right)^{-1},\]
we have
\begin{align*}
\overline{M_{\phi_i}^{-1} H^\infty(\C_-)G + \C G} &= \frac{G}{RG} \cdot \overline{M_{\phi_i}^{-1}H^\infty(\C_-)RG + \C RG} = \frac{G}{RG} \cdot R \left(\overline{\C G + M_{\phi_i} H^\infty(\C_+)G}\right)
\\&\supeq \frac{G}{RG} \cdot R \left(H^\infty(\C_+)G\right) = \frac{G}{RG} H^\infty(\C_-)RG = H^\infty(\C_-)G.
\end{align*}
Using these two inclusions, we get
\begin{align*}
M_{\phi_i} \overline{V} &= \overline{\left(M_{\phi_i}^{-1} H^\infty(\C_-) + \C\textbf{1} + M_{\phi_i} H^\infty(\C_+)\right)G}
\\&= \overline{M_{\phi_i}^{-1} H^\infty(\C_-)G + \C G + M_{\phi_i} H^\infty(\C_+)G}
\\&= \overline{\overline{M_{\phi_i}^{-1}H^\infty(\C_-)G + \C G}+\overline{\C G + M_{\phi_i} H^\infty(\C_+)G}}
\\&\supeq \overline{H^\infty(\C_-)G + H^\infty(\C_+)G}
\\&= \overline{\left(H^\infty(\C_-)+H^\infty(\C_+)\right)G} = L^2(\R,\C),
\end{align*}
using \fref{cor:OuterDense} in the last step. This yields
\[\overline{V} = M_{\phi_i}^{-1} L^2(\R,\C) = L^2(\R,\C). \tag{\(\bigstar\)}\]
Since \(\Theta_h F = F\), we have
\[h = \frac F{RF} = \frac{R\Lambda}{\Lambda} \cdot \frac{G}{RG} = -\frac{R\phi_i}{\phi_i} \cdot \frac{G}{RG},\]
using \fref{lem:phiLambdaIdentity} in the last step. For \(f \in L^\infty(\R,\C)\) this yields
\begin{align*}
\theta_h M_{\phi_i}^{-1} (f \cdot G) &= -\frac{R\phi_i}{\phi_i} \cdot \frac{G}{RG} \cdot R\left( \frac{f}{\phi_i} \cdot G\right)
\\&= -\frac{R\phi_i}{\phi_i} \cdot \frac{G}{RG} \cdot \frac{Rf}{R\phi_i} \cdot RG = -\frac{Rf}{\phi_i} \cdot G = M_{\phi_i}^{-1} ((-Rf) \cdot G).
\end{align*}
This implies that \(V\) is \(\theta_h\)-invariant, since, for \(f \in H^\infty(\C_-)\), \(g \in H^\infty(\C_+)\) and \(z \in \C\), we have
\begin{align*}
\theta_h M_{\phi_i}^{-1}\left(\left(M_{\phi_i}^{-1}f + z\textbf{1} + M_{\phi_i}g\right) G\right) &= M_{\phi_i}^{-1}\left(\left(-R\left(M_{\phi_i}^{-1}f + z\textbf{1} + M_{\phi_i}g\right)\right) G\right)
\\&= M_{\phi_i}^{-1}\left(\left(M_{\phi_i}^{-1}(-Rg) -z\textbf{1} + M_{\phi_i}(-Rf)\right) G\right) \in V.
\end{align*}
This also shows that an element
\[M_{\phi_i}^{-1}\left(\left(M_{\phi_i}^{-1}f + z\textbf{1} + M_{\phi_i}g\right) G\right) \in V\]
is \(\theta_h\)-invariant, if and only if \(f = -Rg\) and \(z=0\), so
\[V^{\theta_h} = \left\{M_{\phi_i}^{-1}\left(\left(M_{\phi_i}^{-1}(-Rg) + M_{\phi_i}g\right) G\right) : g \in H^\infty(\C_+)\right\}.\]
Analogously we get
\[V^{-\theta_h} = \left\{M_{\phi_i}^{-1}\left(\left(M_{\phi_i}^{-1}(Rg) + z\textbf{1} + M_{\phi_i}g\right) G\right) : g \in H^\infty(\C_+), z \in \C\right\}.\]
Then, for every \(v = M_{\phi_i}^{-1}\left(\left(M_{\phi_i}^{-1}(-Rg) + M_{\phi_i}g\right) G\right) \in V^{\theta_h}\), one has
\begin{align*}
v &= M_{\phi_i}^{-1}\left(\left(M_{\phi_i}^{-1}(-Rg) + M_{\phi_i}g\right) G\right) = gG + M_{\phi_i}^{-1}(M_{\phi_i}^{-1}(-Rg)G)
\\&= 2gG - M_{\phi_i}^{-1}\left(\left(M_{\phi_i}^{-1}(Rg) + M_{\phi_i}g\right) G\right)
\\&\in H^\infty(\C_+)G + V^{-\theta_h} \subeq H^2(\C_+) + L^2(\R,\C)^{-\theta_h},
\end{align*}
so
\[V^{\theta_h} \subeq H^2(\C_+) + L^2(\R,\C)^{-\theta_h}.\]
Writing \(P \coloneqq \frac 12 (1+\theta_h)\), we have
\[L^2(\R,\C)^{\theta_h} = \overline{L^2(\R,\C)^{\theta_h}} \supeq \overline{V^{\theta_h}} = \overline{PV} \supeq P\overline{V} \overset{(\bigstar)}{=} PL^2(\R,\C) = L^2(\R,\C)^{\theta_h},\]
so \(\overline{V^{\theta_h}} = L^2(\R,\C)^{\theta_h}\) and therefore \(V^{\theta_h}\) is a dense subspace of \(L^2(\R,\C)^{\theta_h}\) with
\[V^{\theta_h} \subeq H^2(\C_+) + L^2(\R,\C)^{-\theta_h}.\]
This, by \fref{thm:EPlusMaximal}, implies that \((L^2(\R,\C),H^2(\C_+),\theta_h)\) is a maximal RPHS.
\end{proof}
We can now prove analogues of \fref{cor:kernelOuterOrZero} and \fref{prop:kerNonTrivOuterFix} for the involution \(\Theta_h\):
\begin{cor}\label{cor:kernelOuterOrZeroExtended}
Let \({h \in L^\infty(\R,\T)^\flat}\) such that the triple \((L^2(\R,\C),H^2(\C_+),\theta_h)\) is a maximal RPHS. Then
\[\ker(\Theta_h - \textbf{1}) \cap H^2_\Lambda(\C_+) \subeq \{0\} \cup \Out^2_\Lambda(\C_+).\]
\end{cor}
\begin{proof}
Let \(f \in \ker(\Theta_h - \textbf{1}) \cap H^2_\Lambda(\C_+)\). If \(f = 0\) then obviously \(f \in \{0\} \cup \Out^2_\Lambda(\C_+)\). If \(f \neq 0\), then
\[f = \frac{g}{\Lambda}\]
for some function \(g \in H^2(\C_+)\). Then, by \fref{thm:OuterInnerDecomp}, there exist an inner function \(\phi \in \Inn(\C_+)\) and an outer function \({G \in \Out^2(\C_+)}\) such that
\[g = \phi \cdot G.\]
Setting
\[F \coloneqq \frac G\Lambda \in \Out^2_\Lambda(\C_+)\]
this yields
\[f = \frac{g}{\Lambda} = \frac{\phi \cdot G}{\Lambda} = \phi \cdot F.\]
We now set
\[k \coloneqq (\phi \phi^\flat)^{-1} h \in L^\infty(\R,\C)^\flat.\]
Then
\[h = (\phi \phi^\flat) k\]
and therefore, by \fref{lem:InnerSplitOff} the triples \((L^2(\R,\C),M_\phi^{-1} H^2(\C_+),\theta_k)\) and \((L^2(\R,\C),H^2(\C_+),\theta_k)\) are RPHSs. Since
\[k = (\phi \phi^\flat)^{-1} h = (\phi \phi^\flat)^{-1} \frac{f}{Rf} = \frac{\phi^{-1} f}{R (\phi^{-1} f)} = \frac F{RF},\]
we have
\[F \in \ker(\Theta_k - \textbf{1}) \cap \Out^2_\Lambda(\C_+).\]
By \fref{prop:maxPos} this implies that \((L^2(\R,\C),H^2(\C_+),\theta_k)\) is a maximal RPHS. Since we have \(H^2(\C_+) \subeq M_\phi^{-1} H^2(\C_+)\) and \((L^2(\R,\C),M_\phi^{-1} H^2(\C_+),\theta_k)\) is a RPHS, this implies that
\[H^2(\C_+) = M_\phi^{-1} H^2(\C_+) = M_{\phi^*} H^2(\C_+),\]
which yields \(\phi^* \in H^\infty(\C_+)\) by \fref{prop:H2InclusionFunctions} and therefore \(\phi \in H^\infty(\C_-)\). On the other hand \(\phi \in H^\infty(\C_+)\), so by \fref{cor:HInftyIntersection} we get that \(\phi = C \cdot \textbf{1}\) for some \(C \in \T\). This implies that
\[f = \phi \cdot F = C \cdot F \in \Out^2_\Lambda(\C_+). \qedhere\]
\end{proof}
\begin{cor}\label{cor:kerNonTrivOuterFixExtended}
Let \({h \in L^\infty(\R,\T)^\flat}\) such that the triple \((L^2(\R,\C),H^2(\C_+),\theta_h)\) is a maximal RPHS. Then there exists an outer function \(F \in \Out^2_\Lambda(\C_+)\) such that
\[h = \frac F{RF}.\]
\end{cor}
\begin{proof}
By \fref{prop:GeneralizedFixedPoint} and \fref{cor:kernelOuterOrZeroExtended} we have
\[\{0\} \neq \ker(\Theta_h - \textbf{1}) \cap H^2_\Lambda(\C_+) \subeq \{0\} \cup \Out^2_\Lambda(\C_+),\]
so there exists an outer function \(F \in \Out^2_\Lambda(\C_+)\) such that \(F \in \ker(\Theta_h - \textbf{1})\), so
\[F = \Theta_h F = h \cdot RF,\]
which is equivalent to \(h = \frac F{RF}\). 
\end{proof}

\subsubsection{General involutions and measures}
Now the question left to answer is for which functions \(F \in \Out^2_\Lambda(\C_+)\), defining \(h = \frac F{RF}\), the triple \((L^2(\R,\C),H^2(\C_+),\theta_h)\) is a maximal RPHS. In \fref{thm:thetaFixedPointCharac} we have seen that examples of such functions are given by the functions \(F_{W(\mu)} \in \Out^2(\C_+)\) for measures \(\mu \in \mathcal{BM}^{\mathrm{fin}}\left(\R_+\right) \setminus \{0\}\). In this subsection, we will first see that moreso, for every measure \(\nu \in \mathcal{BM}^{\mathrm{fin}}\left([0,\infty]\right) \setminus \{0\}\), considering the functions \(F_\nu\) and \(h_\nu = \frac {F_\nu}{RF_\nu}\), one has \(F_\nu \in \Out^2_\Lambda(\C_+)\) and the triple \((L^2(\R,\C),H^2(\C_+),\theta_{h_\nu})\) is a maximal RPHS. In the end it turns out that these are the only functions with this property. We start by showing that in fact \(F_\nu \in \Out^2_\Lambda(\C_+)\) for \(\nu \in \mathcal{BM}^{\mathrm{fin}}\left([0,\infty]\right) \setminus \{0\}\):
\begin{lem}\label{lem:FNuLambdaOuter}
Let \(\nu \in \mathcal{BM}^{\mathrm{fin}}\left([0,\infty]\right) \setminus \{0\}\). Then
\[|\Lambda(x)|^2 \cdot \Psi_\nu(x) \leq \nu([0,\infty]) \cdot \frac 1{1+x^2} \quad \forall x \in \R^\times\]
and \(F_\nu \in \Out^2_\Lambda(\C_+)\).
\end{lem}
\begin{proof}
By \fref{lem:OuterHomo} and \fref{rem:DefLambda} we have \(\Lambda \cdot F_\nu \in \Out(\C_+)\). Further, by \fref{lem:OuterHomo} and \fref{lem:PsiEstimate} we have
\[|F_\nu(x)|^2 = |\Out(\Psi_\nu)(x)| = \Psi_\nu(x) \leq \nu([0,\infty]) \cdot \left(\frac 1{x^2} + 1\right) = \nu([0,\infty]) \cdot \frac{1 + x^2}{x^2}, \quad \forall x \in \R^\times.\]
Since
\[|\Lambda(x)| = \frac{|x|}{1+x^2}\]
(cf. \fref{rem:DefLambda}), we get
\[|\Lambda(x) \cdot F_\nu(x)|^2 = |\Lambda(x)|^2 \cdot \Psi_\nu(x) \leq \nu([0,\infty]) \cdot \frac{1+x^2}{x^2} \cdot \left(\frac{|x|}{1+x^2}\right)^2 = \nu([0,\infty]) \cdot \frac 1{1+x^2}.\]
Since
\[\int_\R \frac 1{1+x^2} \,dx = \pi < \infty,\]
this shows that \(\Lambda \cdot F_\nu \in L^2(\R,\C)\). Then, by \fref{thm:OuterBetrag}, we get \(\Lambda \cdot F_\nu \in \Out^2(\C_+)\) and therefore
\(F_\nu \in \Out^2_\Lambda(\C_+)\).
\end{proof}
Now, given a measure \(\nu \in \mathcal{BM}^{\mathrm{fin}}\left([0,\infty]\right) \setminus \{0\}\), we want to show that \((L^2(\R,\C),H^2(\C_+),\theta_{h_\nu})\) is a RPHS, i.e. that for every \(f \in H^2(\C_+)\), we have
\[\braket*{f}{\theta_{h_\nu}f} \geq 0.\]
For this, we make the following definition:
\begin{definition}\label{def:defCT}
We define the map
\[\gls*{T}: \mathcal{BM}^{\mathrm{fin}}\left(\left[0,\infty\right]\right) \setminus \{0\} \to \mathcal{BM}\left(\R_+\right)\]
by
\[d(\cT\nu)(\lambda) \coloneqq \left|F_\nu(i\lambda)\right|^{-2} \frac{1+\lambda^2}{\lambda} \,d\nu(\lambda), \quad \lambda \in \R_+.\]
\end{definition}
\begin{remark}\label{rem:scaleable}
For every \(c \in \R_+\) and every measure \(\nu \in \mathcal{BM}^{\mathrm{fin}}\left([0,\infty]\right) \setminus \{0\}\), by \fref{lem:OuterHomo}, we have
\[\left|F_{\left(c \cdot \nu\right)}\left(i\lambda\right)\right|^2 = \left|\Out\left(\Psi_{\left(c \cdot \nu\right)}\right)\left(i\lambda\right)\right| = \left|\Out\left(c \cdot \Psi_\nu\right)\left(i\lambda\right)\right| = c \cdot \left|\Out\left(\Psi_\nu\right)\left(i\lambda\right)\right| = c \cdot \left|F_\nu\left(i\lambda\right)\right|^2\]
so
\begin{align*}
d(\cT\left(c \cdot \nu\right))\left(\lambda\right) &= \left|F_{\left(c \cdot \nu\right)}\left(i\lambda\right)\right|^{-2} \frac{1+\lambda^2}{\lambda} \,d\left(c \cdot \nu\right)\left(\lambda\right)
\\&= \frac{1}{c} \cdot \left|F_\nu\left(i\lambda\right)\right|^{-2} \frac{1+\lambda^2}{\lambda}  \cdot c \,d\nu\left(\lambda\right)
\\&= \left|F_\nu\left(i\lambda\right)\right|^{-2} \frac{1+\lambda^2}{\lambda} \,d\nu\left(\lambda\right) = d(\cT \nu)\left(\lambda\right), \quad \lambda \in \R_+
\end{align*}
and therefore
\begin{equation*}
\cT\left(c \cdot \nu\right) = \cT \nu.
\end{equation*}
\end{remark}
We now want to show that, for every measure \(\nu \in \mathcal{BM}^{\mathrm{fin}}\left([0,\infty]\right) \setminus \{0\}\) and every \(f,g \in H^2(\C_+)\), we have
\[\braket*{f}{\theta_{h_\nu}g} = \int_{\R_+} \overline{f(i\lambda)}\,g(i\lambda) \,d(\cT \nu)(\lambda).\]
For this, we need the following proposition:
\begin{prop}\label{prop:denseL2Measure}
Let \(A \in B(L^2(\R,\C))\) and \(\mu \in \mathcal{BM}\left(\R_+\right)\). Further, for a function \(f \in H^2(\C_+)\) we consider the function
\[f^\sV: \R_+ \to \C, \quad \lambda \mapsto f(i \lambda).\]
If there exists a dense subspace \(V \subeq H^2(\C_+)\) such that
\[\braket*{f}{Ag} = \int_{\R_+} \overline{f^\sV(\lambda)} \,g^\sV(\lambda)\,d\mu(\lambda) \quad \forall f,g \in V,\]
then \(f^\sV \in L^2(\R_+,\C,\mu)\) for every \(f \in H^2(\C_+)\) and
\[\braket*{f}{Ag} = \braket*{f^\sV}{g^\sV}_{L^2(\R_+,\C,\mu)} \quad \forall f,g \in H^2(\C_+).\]
\end{prop}
\begin{proof}
For every \(f \in V\) we have
\begin{equation*}
\int_{\R_+} |f^\sV(\lambda)|^2\,d\mu(\lambda) = \braket*{f}{A f} < \infty,
\end{equation*}
so \(f^\sV \in L^2\left(\R_+,\C,\mu\right)\). Now let \(f \in H^2\left(\C_+\right)\). Since \(V\) is dense in \(H^2(\C_+)\), there exists a sequence \(\left(f_n\right)_{n \in \N}\) in \(V\) such that \({\left\lVert f_n-f\right\rVert \xrightarrow{n \to \infty} 0}\). Then
\begin{equation*}
\left\lVert f_n^\sV\right\rVert_{L^2\left(\R_+,\C,\mu\right)}^2 = \braket*{f_n^\sV}{f_n^\sV}_{L^2(\R_+,\C,\mu)} = \braket*{f_n}{A f_n} \xrightarrow{n \to \infty} \braket*{f}{A f},
\end{equation*}
so the sequence \((f_n^\sV)_{n \in \N}\) is bounded in \(L^2\left(\R_+,\C,\mu\right)\), and therefore, by the Banach-Alaoglu Theorem, has a weakly-convergent subsequence. By replacing \((f_n^\sV)_{n \in \N}\) with that subsequence, we can assume that there exists \(F \in L^2\left(\R_+,\C,\mu\right)\) such that \(f_n^\sV \xrightarrow{n \to \infty} F\) weakly. We now want to see that \(F=f^\sV\). For \(0<a<b<\infty\) and \(A \subseteq \left[a,b\right]\), we have
\begin{equation*}
\chi_A f_n^\sV \xrightarrow{n \to \infty} \chi_A f^\sV
\end{equation*}
uniformly by \fref{prop:uniform}. Then, using that
\[\int_{\R_+} \chi_A^2 \,d\mu = \int_{\R_+} \chi_A \,d\mu = \mu(A) \leq \mu([a,b]) < \infty\]
and therefore \(\chi_A \in L^2(\R_+,\C,\mu) \cap L^1(\R_+,\C,\mu)\), we have
\begin{equation*}
\int_{\R_+} \chi_A F \,d\mu = \lim_{n \to \infty} \int_{\R_+} \chi_A f_n^\sV \,d\mu = \int_{\R_+} \chi_A f^\sV \,d\mu.
\end{equation*}
Choosing
\[A_\pm \coloneqq \left\lbrace x \in \left[a,b\right]: \pm\left(F-f^\sV\right)\left(x\right) \geq 0\right\rbrace,\] we get
\begin{equation*}
0 = \int_{\R_+} \chi_{A_\pm} \left(F-f^\sV\right) \,d\mu
\end{equation*}
and therefore
\begin{equation*}
0 = \int_{\R_+} \chi_{A_+} \left(F-f^\sV\right) - \chi_{A_-} \left(F-f^\sV\right) \,d\mu = \int_{\R_+} \left|F-f^\sV\right|\,d\mu,
\end{equation*}
so \(\mu\)-almost everywhere we have
\[F\big|_{\left[a,b\right]}= f^\sV\big|_{\left[a,b\right]}.\]
Since this is true for all \(a,b \in \R\) with \({0<a<b<\infty}\), we \(\mu\)-almost everywhere have \(F=f^\sV\), which shows that \(f^\sV \in L^2\left(\R_+,\C,\mu\right)\). On the other hand, this argument shows that the graph of the map
\[(\cdot)^\sV: H^2(\C_+) \to L^2\left(\R_+,\C,\mu\right)\]
is closed, and therefore, by the Closed Graph Theorem, this map is continuous. Using this continuity, for two functions \(f,g \in H^2\left(\C_+\right)\), we can select sequences \(\left(f_n\right)_{n \in \N}\) and \(\left(g_n\right)_{n \in \N}\) in \(V\) such that \(f_n \xrightarrow{n \to \infty} f\) and \(g_n \xrightarrow{n \to \infty} g\) and get
\begin{align*}
\braket*{f}{A g} &= \lim_{n \to \infty}\lim_{m \to \infty} \braket*{f_n}{A g_m} = \lim_{n \to \infty}\lim_{m \to \infty} \int_{\R_+} \overline{f_n^\sV(\lambda)} \,g_m^\sV(\lambda)\,d\mu(\lambda)
\\&= \int_{\R_+} \overline{f^\sV(\lambda)} \,g^\sV(\lambda)\,d\mu(\lambda) = \braket*{f^\sV}{g^\sV}_{L^2(\R_+,\C,\mu)}. \qedhere
\end{align*}
\end{proof}
\begin{cor}\label{cor:twistedScalar}
Let \(\mu \in \mathcal{BM}^{\mathrm{fin}}\left(\R_+\right) \setminus \{0\}\). Then, for every \(f,g \in H^2\left(\C_+\right)\), we have
\[\braket*{f}{\theta_{h_{W(\mu)}}g} = \int_{\R_+} \overline{f(i\lambda)}\,g(i\lambda) \,d(\cT W(\mu))(\lambda).\]
\end{cor}
\begin{proof}
For every \(s,t \in \R_+\), we have
\begin{align*}
\braket*{S_s F_{W(\mu)}}{\theta_{h_{W(\mu)}}S_t F_{W(\mu)}} &= \braket*{S_s F_{W(\mu)}}{S_{-t} \theta_{h_{W(\mu)}} F_{W(\mu)}} = \braket*{S_s F_{W(\mu)}}{S_{-t} F_{W(\mu)}}
\\&= \braket*{S_{s+t} F_{W(\mu)}}{F_{W(\mu)}} = \int_\R e^{-i(s+t)x} |F_{W(\mu)}(x)|^2 \,dx
\\&=\left(\cF_1|F_{W(\mu)}|^2\right)(s+t) = \left(\cF_1\Psi_{W(\mu)}\right)(s+t)
\\&= \left(\cF_1\psi_{\mu}\right)(s+t) = \phi_\mu\left(s+t\right),
\end{align*}
using \fref{lem:BigSmallPsi} and \fref{prop:PhiPsiFourier} for the last two steps. This yields
\begin{align*}
\braket*{S_s F_{W(\mu)}}{\theta_{h_{W(\mu)}}S_t F_{W(\mu)}} &= \phi_\mu\left(s+t\right) = \int_{\R_+} e^{-\lambda\left(s+t\right)} d\mu\left(\lambda\right)
\\&= \int_{\R_+} \overline{e^{-\lambda s}F_{W(\mu)}\left(i\lambda\right)}\,e^{-\lambda t}F_{W(\mu)}\left(i\lambda\right) \left|F_{W(\mu)}\left(i\lambda\right)\right|^{-2} d\mu\left(\lambda\right)
\\&= \int_{\R_+} \overline{\left(S_s F_{W(\mu)}\right)\left(i\lambda\right)}\,\left(S_t F_{W(\mu)}\right)\left(i\lambda\right) \left|F_{W(\mu)}\left(i\lambda\right)\right|^{-2} \frac{1+\lambda^2}{\lambda}\,d(W(\mu))\left(\lambda\right)
\\&= \int_{\R_+} \overline{\left(S_s F_{W(\mu)}\right)\left(i\lambda\right)}\,\left(S_t F_{W(\mu)}\right)\left(i\lambda\right) \,d(\cT W(\mu))\left(\lambda\right).
\end{align*}
Then, by the sesquilinearity of the scalar product, we get
\begin{equation*}
\braket*{f}{\theta_{h_{W(\mu)}} g} = \int_{\R_+} \overline{f\left(i\lambda\right)}\,g\left(i\lambda\right) \,d(\cT W(\mu))\left(\lambda\right)
\end{equation*}
for every \(f,g \in \Spann S(\R_+) F_{W(\mu)}\). Since \(F_{W(\mu)}\) is an outer function, by \fref{thm:outerSpan}, the subspace \(\Spann S(\R_+) F_{W(\mu)}\) is dense in \(H^2\left(\C_+\right)\), so the statement follows by \fref{prop:denseL2Measure}.
\end{proof}
We now want to show that this corollary is not only true for measures of the form \(W(\mu)\) with \(\mu \in \mathcal{BM}^{\mathrm{fin}}\left(\R_+\right) \setminus \{0\}\), but also for general measures \(\nu \in \mathcal{BM}^{\mathrm{fin}}\left([0,\infty]\right) \setminus \{0\}\). For this we will make use of the fact that, in a certain sense, we can approximate measures \(\nu \in \mathcal{BM}^{\mathrm{fin}}\left([0,\infty]\right)\) by measures of the form \(W(\mu)\) with \(\mu \in \mathcal{BM}^{\mathrm{fin}}\left(\R_+\right)\). The topology we will use for this is the weak-\(*\)-topology on the Banach space of complex finite Borel measures:
\begin{definition}
Let \(X\) be a locally compact Hausdorff space. We denote the set of complex Borel
measures on \(X\) by
\[\gls*{BMXC}.\]
Further, we set
\[\mathcal{BM}_\C^{\mathrm{fin}}\left(X\right) \coloneqq \{\mu \in \mathcal{BM}_\C\left(X\right): |\mu|(X)< \infty\}.\]
\end{definition}
\begin{definition}
Let \(V\) be a Banach space over a field \(\cK = \R,\C\) and \(V'\) be its dual space. For \(v \in V'\) and a sequence \((v_n)_{n \in \N}\) in \(V'\) we write
\[v_n \underset{w}{\xrightarrow{n \to \infty}} v,\]
if \((v_n)_{n \in \N}\) converges to \(v\) with respect to the weak-\(*\)-topology, i.e. the topology induced by the linear functionals
\[\ev_x: V' \to \K, \quad w \mapsto w(x), \qquad x \in V.\]
\end{definition}
\begin{remark}
For an interval \(I \subeq \R\), the space \((C_0(I),\left\lVert \cdot \right\rVert_\infty)\) of complex-valued functions on \(I\) that vanish at infinity with the supremum norm is a Banach space. By the Riesz Representation Theorem (cf. \cite[Thm. 6.19]{Ru86}) its dual space is given by \(\mathcal{BM}_\C^{\mathrm{fin}}\left(I\right)\). In particular \(\mathcal{BM}^{\mathrm{fin}}_\C\left([0,\infty]\right)\) is the dual space of the separable Banach space \(C([0,\infty],\C)\) and therefore carries a weak-\(*\)-topology given by the linear functionals
\[I_f: \mathcal{BM}^{\mathrm{fin}}_\C\left([0,\infty]\right) \to \R, \quad \nu \mapsto \int_{[0,\infty]} f \,d\nu\]
with \(f \in C([0,\infty],\C)\). It is easy to see that \(\mathcal{BM}^{\mathrm{fin}}\left([0,\infty]\right)\) is closed in this topology.
\end{remark}
The following lemmata will be useful later:
\begin{lem}\label{lem:PsiPointwise}
Let \(\nu \in \mathcal{BM}^{\mathrm{fin}}\left([0,\infty]\right)\) and let \(\left(\nu_n\right)_{n \in \N}\) be a sequence in \(\mathcal{BM}^{\mathrm{fin}}\left([0,\infty]\right)\) with \(\nu_n \underset{w}{\xrightarrow{n \to \infty}} \nu\). Then
\[\Psi_{\nu_n}(p) \xrightarrow{n \to \infty} \Psi_\nu(p)\]
for every \(p \in \R^\times\).
\end{lem}
\begin{proof}
This follows immediately by the fact that, for every \(p \in \R^\times\), the function
\[f_p: [0,\infty] \ni \lambda \mapsto \frac{1+\lambda^2}{p^2+\lambda^2}\]
is continuous, and therefore, we have
\[\Psi_{\nu_n}(p) = \int_{[0,\infty]} f_p \,d\nu_n \xrightarrow{n \to \infty} \int_{[0,\infty]} f_p \,d\nu = \Psi_\nu(p). \qedhere\]
\end{proof}
\begin{lem}\label{lem:longIntegralApprox}
Let \(\nu \in \mathcal{BM}^{\mathrm{fin}}\left([0,\infty]\right) \setminus \{0\}\) and let \(\left(\nu_n\right)_{n \in \N}\) be a sequence in \(\mathcal{BM}^{\mathrm{fin}}\left([0,\infty]\right)\) with \(\nu_n \underset{w}{\xrightarrow{n \to \infty}} \nu\).
Then, for every \(s,t \in \R_+\), one has
\[\int_{\R_+} e^{-(s+t)\lambda} |\Lambda(i\lambda)|^2 \frac{1+\lambda^2}{\lambda} \,d\nu_n\left(\lambda\right) \xrightarrow{n \to \infty} \int_{\R_+} e^{-(s+t)\lambda} |\Lambda(i\lambda)|^2 \frac{1+\lambda^2}{\lambda} \,d\nu\left(\lambda\right).\]
\end{lem}
\begin{proof}
We have
\[|\Lambda(i\lambda)|^2 = \left|\frac{i(i\lambda)}{(i\lambda+i)^2}\right|^2 = \frac{\lambda^2}{(1+\lambda)^4} \quad \forall \lambda \in \R_+.\]
Therefore, the function
\[f: [0,\infty] \ni \lambda \mapsto e^{-(s+t)\lambda} |\Lambda(i\lambda)|^2 \frac{1+\lambda^2}{\lambda} = e^{-(s+t)\lambda} \frac{\lambda(1+\lambda^2)}{(1+\lambda)^4}\]
is continuous with
\[f(0) = 0 = f(\infty).\]
This yields
\[\int_{\R_+} f \,d\nu_n = \int_{[0,\infty]} f \,d\nu_n \xrightarrow{n \to \infty} \int_{[0,\infty]} f \,d\nu = \int_{\R_+} f \,d\nu. \qedhere\]
\end{proof}
\begin{lem}\label{lem:longScalarApprox}
Let \(\nu \in \mathcal{BM}^{\mathrm{fin}}\left([0,\infty]\right) \setminus \{0\}\) and let \(\left(\nu_n\right)_{n \in \N}\) be a sequence in \(\mathcal{BM}^{\mathrm{fin}}\left([0,\infty]\right)\) with
\[\nu_n \underset{w}{\xrightarrow{n \to \infty}} \nu \quad \text{and} \quad \sup_{n \in \N} \nu_n([0,\infty]) < \infty.\]
Then, for every \(s,t \in \R\), one has
\[\braket*{S_s \Lambda F_{\nu_n}}{\theta_{h_{\nu_n}} S_t \Lambda F_{\nu_n}} \xrightarrow{n \to \infty} \braket*{S_s \Lambda F_\nu}{\theta_{h_\nu} S_t \Lambda F_\nu}.\]
\end{lem}
\begin{proof}
For \(\rho \in \mathcal{BM}^{\mathrm{fin}}\left([0,\infty]\right) \setminus \{0\}\), by \fref{lem:FNuLambdaOuter}, we have \(F_\rho \in \Out^2_\Lambda(\C_+) \subeq H^2_\Lambda(\C_+)\) and therefore \(\Lambda F_\rho \in H^2(\C_+)\). Therefore, for \(s,t \in \R\), we can define
\[\Phi_\rho(s,t) \coloneqq \braket*{S_s \Lambda F_\rho}{\theta_{h_\rho} S_t \Lambda F_\rho}.\]
Then, using that
\[(R\Lambda)(x) = \frac{-ix}{(-x+i)^2} = \frac{-ix}{(x-i)^2} = \overline{\Lambda(x)}, \quad x \in \R,\]
we get
\begin{align*}
\Phi_\rho(s,t) &= \braket*{S_s \Lambda F_\rho}{\theta_{h_\rho} S_t \Lambda F_\rho} = \braket*{S_s \Lambda F_\rho}{S_{-t} (R\Lambda) \theta_{h_\rho} F_\rho}
\\&= \braket*{S_{(s+t)} \Lambda F_\rho}{\overline{\Lambda} F_\rho} = \int_\R e^{-i(s+t)x} \overline{\Lambda(x)}^2 |F_\rho(x)|^2 \,dx = \int_\R e^{-i(s+t)x} \overline{\Lambda(x)}^2 \Psi_\rho(x) \,dx.
\end{align*}
Now, setting
\[M \coloneqq \sup_{n \in \N} \nu_n([0,\infty]) < \infty,\]
by \fref{lem:FNuLambdaOuter}, we have
\[\left|e^{-i(s+t)x} \overline{\Lambda(x)}^2 \Psi_{\nu_n}(x)\right| = |\Lambda(x)|^2 \cdot \Psi_{\nu_n}(x) \leq \nu_n([0,\infty]) \cdot \frac 1{1+x^2} \leq M \cdot \frac 1{1+x^2}\]
for all \(x \in \R^\times\)
and by \fref{lem:PsiPointwise} we have
\[\Psi_{\nu_n}(x) \xrightarrow{n \to \infty} \Psi_\nu(x)\]
for every \(x \in \R^\times\). This, by the Dominated Convergence Theorem, yields
\begin{equation*}
\Phi_{\nu_n}(s,t) = \int_\R e^{-i(s+t)x} \overline{\Lambda(x)}^2 \Psi_{\nu_n}(x) \,dx \xrightarrow{n \to \infty} \int_\R e^{-i(s+t)x} \overline{\Lambda(x)}^2 \Psi_{\nu}(x) \,dx = \Phi_\nu(s,t). \qedhere
\end{equation*}
\end{proof}
We now prove that we can approximate measures \(\nu \in \mathcal{BM}^{\mathrm{fin}}\left([0,\infty]\right)\) by measures of the form \(W(\mu)\) with \(\mu \in \mathcal{BM}^{\mathrm{fin}}\left(\R_+\right)\):
\begin{lemma}\label{lem:ImcTDense}
Let \(\nu \in \mathcal{BM}^{\mathrm{fin}}\left([0,\infty]\right)\). Then there exists a sequence \(\left(\mu_n\right)_{n \in \N}\) in \(\mathcal{BM}^{\mathrm{fin}}\left(\R_+\right)\) with
\[(W(\mu_n))([0,\infty]) = \nu([0,\infty]) \quad \forall n \in \N\]
and
\[W(\mu_n) \underset{w}{\xrightarrow{n \to \infty}} \nu.\]
\end{lemma}
\begin{proof}
For \(n \in \N\) we define a measure \(\nu_n \in \mathcal{BM}^{\mathrm{fin}}\left([0,\infty]\right)\) by
\[\nu_n(A) \coloneqq \nu\left(\left[\frac 1n,n\right] \cap A\right)\]
for every measurable subset \(A \subeq \left[0,\infty\right]\). Further, for \(n \in \N\) set
\[\tilde \nu_n \coloneqq \nu_n + \nu\left(\left[0,\frac 1n\right)\right) \cdot \delta_{\frac 1n} + \nu\left(\left(n,\infty\right]\right) \cdot \delta_n,\]
where for \(x \in \R\) by \(\delta_x\) we denote the Dirac measure in \(x\). Then, for every \(n \in \N\), one has
\[\tilde \nu_n(\left[0,\infty\right]) = \nu\left(\left[\frac 1n,n\right]\right) + \nu\left(\left[0,\frac 1n\right)\right) + \nu\left(\left(n,\infty\right]\right) = \nu(\left[0,\infty\right]).\]
We now want to show that \(\tilde \nu_n \underset{w}{\xrightarrow{n \to \infty}} \nu\). So let \(f \in C(\left[0,\infty\right],\R)\) and \(\epsilon > 0\). Since \(f \in C(\left[0,\infty\right],\R)\), there exists an \(N \in \N\) such that
\[\left|f(0)-f(x)\right| \leq \epsilon \quad \text{for every } x \in \left[0,\frac 1N\right]\]
and
\[\left|f(\infty)-f(x)\right| \leq \epsilon \quad \text{for every } x \in \left[N,\infty\right].\]
Then, for every \(n \geq N\), one has
\begin{align*}
\left|\int_{[0,\infty]}f \,d\nu - \int_{[0,\infty]}f \,d\tilde \nu_n\right| &=\left|\int_{\left[0,\frac 1n\right)}f \,d\nu + \int_{\left(n,\infty\right]}f \,d\nu - \nu\left(\left[0,\frac 1n\right)\right) \cdot f\left(\frac 1n\right) - \nu\left(\left(n,\infty\right]\right) \cdot f\left(n\right)\right|
\\&= \left|\int_{\left[0,\frac 1n\right)}f(\lambda)-f\left(\frac 1n\right) \,d\nu(\lambda) + \int_{\left(n,\infty\right]}f(\lambda)-f\left(n\right) \,d\nu(\lambda)\right|
\\&\leq \int_{\left[0,\frac 1n\right)}\left|f(\lambda)-f\left(\frac 1n\right)\right| \,d\nu(\lambda) + \int_{\left(n,\infty\right]}\left|f(\lambda)-f\left(n\right)\right| \,d\nu(\lambda)
\\&\leq \int_{\left[0,\frac 1n\right)}\epsilon \,d\nu(\lambda) + \int_{\left(n,\infty\right]}\epsilon \,d\nu(\lambda) \leq \int_{\left[0,\infty\right]}\epsilon \,d\nu(\lambda) = \epsilon \cdot \nu([0,\infty]).
\end{align*}
Now, for \(n \in \N\), we define the measure \(\mu_n\) on \(\R_+\) by
\[d\mu_n(\lambda) \coloneqq \frac{1+\lambda^2}{\lambda} \,d\tilde \nu_n(\lambda), \quad \lambda \in \R_+.\]
Since \(\nu_n\) is finite, \(\tilde \nu_n\left(\left[0,\infty\right] \setminus \left[\frac 1n,n\right]\right) = 0\) and the function
\[\lambda \mapsto \frac{1+\lambda^2}{\lambda}\] is bounded on \(\left[\frac 1n,n\right]\), the measure \(\mu_n\) is finite, so we have \(\mu_n \in \mathcal{BM}^{\mathrm{fin}}\left(\R_+\right)\) and per construction we have \(W(\mu_n) = \tilde \nu_n\) for every \(n \in \N\). Therefore
\[W(\mu_n) \underset{w}{\xrightarrow{n \to \infty}} \nu\]
and for every \(n \in \N\) we have
\[(W(\mu_n))([0,\infty]) = \tilde \nu_n([0,\infty]) = \nu([0,\infty]). \qedhere\]
\end{proof}
We can now prove a generalization of \fref{cor:twistedScalar} for all measures \(\nu \in \mathcal{BM}^{\mathrm{fin}}\left([0,\infty]\right) \setminus \{0\}\):
\begin{thm}\label{thm:cTNuIsCarleson}
Let \(\nu \in \mathcal{BM}^{\mathrm{fin}}\left([0,\infty]\right) \setminus \{0\}\). Then, for all \(f,g \in H^2(\C_+)\), one has
\[\braket*{f}{\theta_{h_\nu}g} = \int_{\R_+} \overline{f(i\lambda)}\,g(i\lambda) \,d(\cT \nu)(\lambda).\]
\end{thm}
\begin{proof}
By \fref{lem:ImcTDense}, there exists a sequence \(\left(\mu_n\right)_{n \in \N}\) in \(\mathcal{BM}^{\mathrm{fin}}\left(\R_+\right)\) with
\[(W(\mu_n))([0,\infty])=\nu([0,\infty]) \quad \forall n \in \N\]
such that
\[W(\mu_n) \underset{w}{\xrightarrow{n \to \infty}} \nu.\]
Now, by \fref{cor:twistedScalar}, for \(n \in \N\) and \(s,t \in \R_+\), we have
\begin{align*}
\braket*{S_s \Lambda F_{W(\mu_n)}}{\theta_{h_{W(\mu_n)}} S_t \Lambda F_{W(\mu_n)}} = \int_{\R_+} \overline{(S_s \Lambda F_{W(\mu_n)})(i\lambda)}\,(S_t \Lambda F_{W(\mu_n)})(i\lambda) \,d(\cT W(\mu_n))(\lambda).
\end{align*}
Since, for every measure \(\nu \in \mathcal{BM}^{\mathrm{fin}}\left(\R_+\right)\setminus \{0\}\), we have
\begin{align*}
\int_{\R_+} \overline{(S_s \Lambda F_\nu)(i\lambda)}\,(S_t \Lambda F_\nu)(i\lambda) \,d(\cT \nu)(\lambda) &= \int_{\R_+} e^{-(s+t)\lambda} |\Lambda(i\lambda)|^2 \cdot \left|F_\nu(i\lambda)\right|^2 \,d(\cT \nu)(\lambda)
\\&= \int_{\R_+} e^{-(s+t)\lambda} |\Lambda(i\lambda)|^2 \frac{1+\lambda^2}{\lambda} \,d\nu(\lambda),
\end{align*}
by \fref{lem:longIntegralApprox}, we get
\begin{align*}
\braket*{S_s \Lambda F_{W(\mu_n)}}{\theta_{h_{W(\mu_n)}} S_t \Lambda F_{W(\mu_n)}} &= \int_{\R_+} e^{-(s+t)\lambda} |\Lambda(i\lambda)|^2 \frac{1+\lambda^2}{\lambda} \,d(W(\mu_n))(\lambda)
\\&\xrightarrow{n \to \infty} \int_{\R_+} e^{-(s+t)\lambda} |\Lambda(i\lambda)|^2 \frac{1+\lambda^2}{\lambda} \,d\nu(\lambda) 
\\&= \int_{\R_+} \overline{(S_s \Lambda F_\nu)(i\lambda)}\,(S_t \Lambda F_\nu)(i\lambda) \,d(\cT \nu)(\lambda).
\end{align*}
On the other hand
\[\sup_{n \in \N} \,(W(\mu_n))([0,\infty]) = \nu([0,\infty]) < \infty,\]
so, by \fref{lem:longScalarApprox}, we have
\[\braket*{S_s \Lambda F_{W(\mu_n)}}{\theta_{h_{W(\mu_n)}} S_t \Lambda F_{W(\nu_n)}} \xrightarrow{n \to \infty} \braket*{S_s \Lambda F_\nu}{\theta_{h_\nu} S_t \Lambda F_\nu}.\]
This yields
\[\braket*{S_s \Lambda F_\nu}{\theta_{h_\nu} S_t \Lambda F_\nu} = \int_{\R_+} \overline{(S_s \Lambda F_\nu)(i\lambda)}\,(S_t \Lambda F_\nu)(i\lambda) \,d(\cT \nu)(\lambda) \quad \forall s,t \in \R_+.\]
Now, by the sesquilinearity of the scalar product, we get
\[\braket*{f}{\theta_{h_\nu} g} = \int_{\R_+} \overline{f(i\lambda)}\,g(i\lambda) \,d(\cT \nu)(\lambda)\]
for every \(f,g \in \spann \,S(\R_+)(\Lambda F_\nu)\). Since, by \fref{lem:FNuLambdaOuter}, we have \(\Lambda F_\nu \in \Out^2(\C_+)\), by \fref{thm:outerSpan}, the subspace \(\spann \,S(\R_+)(\Lambda F_\nu)\) is dense in \(H^2\left(\C_+\right)\) and therefore the statement follows by \fref{prop:denseL2Measure}.
\end{proof}
\begin{cor}\label{cor:hNuMaxPos}
Let \(\nu \in \mathcal{BM}^{\mathrm{fin}}\left([0,\infty]\right) \setminus \{0\}\). Then the triple \(\left(L^2\left(\R,\C\right),H^2\left(\C_+\right),\theta_{h_\nu}\right)\) is a maximal RPHS.
\end{cor}
\begin{proof}
By \fref{thm:cTNuIsCarleson} we get
\[\braket*{f}{\theta_{h_\nu}f} = \int_{\R_+} \overline{f(i\lambda)}\,f(i\lambda) \,d(\cT \nu)(\lambda) = \int_{\R_+} |f(i\lambda)|^2 \,d(\cT \nu)(\lambda) \geq 0 \quad \forall f \in H^2(\C_+),\]
so \(\left(L^2\left(\R,\C\right),H^2\left(\C_+\right),\theta_{h_\nu}\right)\) is a RPHS. The maximality then follows by \fref{prop:maxPos}, since, by \fref{lem:FNuLambdaOuter}, we have
\[F_\nu \in \ker(\Theta_{h_\nu} - \textbf{1}) \cap \Out^2_\Lambda(\C_+). \qedhere\]
\end{proof}
We now know that, starting with a measure \(\nu \in \mathcal{BM}^{\mathrm{fin}}\left([0,\infty]\right) \setminus \{0\}\), we get a maximal RPHS by \(\left(L^2\left(\R,\C\right),H^2\left(\C_+\right),\theta_{h_\nu}\right)\). In the following, we want to see that, in fact, every maximal RPHS is of this form. For this, we need the concept of a Hankel operator:
\begin{definition}\label{def:Hankel}{\rm (cf. \cite{ANS22})}
A linear operator \(H \in B(H^2(\C_+))\) is called \textit{Hankel operator}, if \(S_t^* H = H S_t\) for every \(t \in \R_+\).
\end{definition}
\begin{lemma}\label{lem:RPHSPosHankelEquiv}
Let \(h \in L^\infty(\R,\T)^\flat\). Then the triple \(\left(L^2\left(\R,\C\right),H^2\left(\C_+\right),\theta_h\right)\) is a RPHS, if and only if
\[H_h \coloneqq p_+ \theta_h p_+^*\]
is a positive Hankel operator, where by \(p_+: L^2(\R,\C) \to H^2(\C_+)\) we denote the orthogonal projection from \(L^2(\R,\C)\) onto its subspace \(H^2(\C_+)\).
\end{lemma}
\begin{proof}
For every \(f,g\in H^2(\C_+)\) we have
\[\braket*{f}{H_h g} = \braket*{f}{\theta_h g},\]
which immediately implies that \(H_h\) is a positive operator, if and only if \(\left(L^2\left(\R,\C\right),H^2\left(\C_+\right),\theta_h\right)\) is a RPHS. Further, for \(t \in \R_+\) we get
\[\braket*{f}{H_h S_t g} = \braket*{f}{\theta_h S_t g} = \braket*{f}{S_t^* \theta_h g} = \braket*{S_t f}{\theta_h g} = \braket*{S_t f}{H_h g} = \braket*{f}{S_t^* H_h g}\]
and therefore \(S_t^* H = H S_t\), so \(H_h\) is a Hankel operator.
\end{proof}
We have the following theorem about Hankel operators:
\begin{thm}\label{thm:HankelHasMeasure}{\rm (cf. \cite[Sec. 3.2]{ANS22})}
Let \(H \in B(H^2(\C_+))\) be a positive Hankel operator. Then there exists a measure \(\mu \in \mathcal{BM}\left(\R_+\right)\) such that
\[\braket*{f}{H g} = \int_{\R_+} \overline{f(i\lambda)} g(i\lambda) \,d\mu(\lambda) \quad \forall f,g \in H^2(\C_+).\]
\end{thm}
\begin{cor}\label{cor:thetaHHasMeasure}
Let \(h \in L^\infty(\R,\T)^\flat\) such that the triple \(\left(L^2\left(\R,\C\right),H^2\left(\C_+\right),\theta_h\right)\) is a RPHS. Then there exists a measure \(\mu \in \mathcal{BM}\left(\R_+\right)\) such that
\[\braket*{f}{\theta_h g} = \int_{\R_+} \overline{f(i\lambda)} g(i\lambda) \,d\mu(\lambda) \quad \forall f,g \in H^2(\C_+).\]
\end{cor}
\begin{proof}
This follows immediately by \fref{lem:RPHSPosHankelEquiv} and \fref{thm:HankelHasMeasure}.
\end{proof}
With this corollary, given a function \(h \in L^\infty(\R,\T)^\flat\) such that the triple \(\left(L^2\left(\R,\C\right),H^2\left(\C_+\right),\theta_h\right)\) is a maximal RPHS, we can relate the fixed points \(F \in \ker(\Theta_h - \textbf{1}) \cap \Out^2_\Lambda(\C_+)\) (cf. \fref{cor:kerNonTrivOuterFixExtended}) to measures:
\begin{prop}\label{prop:FixedPointFromConstruction}
Let \(h \in L^\infty(\R,\T)^\flat\) such that the triple \(\left(L^2\left(\R,\C\right),H^2\left(\C_+\right),\theta_h\right)\) is a maximal RPHS and let
\[F \in \ker(\Theta_h - \textbf{1}) \cap \Out^2_\Lambda(\C_+).\]
Then there exists a measure \(\nu \in \mathcal{BM}^{\mathrm{fin}}\left([0,\infty]\right) \setminus \{0\}\) and \(C \in \T\) such that \(F = C \cdot F_\nu\).
\end{prop}
\begin{proof}
We will show that there exists a measure \(\nu \in \mathcal{BM}^{\mathrm{fin}}\left([0,\infty]\right) \setminus \{0\}\) such that \(|F|^2 = \Psi_\nu\), since then by \fref{thm:OuterBetrag} there exists a constant \(C \in \T\) such that
\[F = \Out(C,\sqrt{\Psi_\nu}) = C \cdot F_\nu.\]
For this, for \(n \in \N\) and \(p \in \R^\times\), we define the function
\[F_{p,n}: \C_+ \to \C^\times, \quad z \mapsto \frac{\left(z+i\right)^4n^2}{\pi\left(\left(z+\frac in\right)^2-p^2\right)\left(zn+i\right)\left(z+in\right)}.\]
Since there exists a continuous extension of \(F_{p,n}\) to \(\overline{\C_+}\) and
\[\lim_{|z| \to \infty} F_{p,n}(z) = \frac n\pi,\]
the function \(F_{p,n}\) is bounded and therefore we have \(F_{p,n} \in H^\infty(\C_+)\). Further, for \(\lambda \in \R_+\), we have
\[F_{p,n}(i\lambda) = \frac{\left(i\lambda+i\right)^4n^2}{\pi\left(\left(i\lambda+\frac in\right)^2-p^2\right)\left(i\lambda n+i\right)\left(i\lambda+in\right)} = \frac 1\pi \cdot \frac{(1+\lambda)^4}{\left(p^2+\left(\lambda+\frac 1n\right)^2\right)\left(\frac 1n+\lambda\right)\left(1+\frac \lambda n\right)} > 0.\]
Then also \(\sqrt{F_{p,n}} \in H^\infty(\C_+)\), where by \(\sqrt{F_{p,n}}\) we denote the one of the two square roots of \(F_{p,n}\) on the simply connected domain \(\C_+\) with \(\sqrt{F_{p,n}}(i\lambda)>0\) for every \(\lambda \in \R_+\). Then
\[\sqrt{F_{p,n}} \cdot \Lambda F \in H^\infty(\C_+) \cdot H^2(\C_+) \subeq H^2(\C_+)\]
for every \(n \in \N\). For \(n \in \N\), we now set
\[g_n(x) \coloneqq \frac{x^2}{n\left(\frac 1{n^2}+x^2\right)\left(1+\frac{x^2}{n^2}\right)}.\]
Since \(\Lambda F \in H^2(\C_+)\), we have
\[\infty > \int_\R |\Lambda(x)|^2|F(x)|^2 \,dx = \int_\R \frac{x^2}{\left(1+x^2\right)^2}|F(x)|^2 \,dx\]
and therefore also
\[\int_\R g_n(x) |F(x)|^2 \,dx < \infty\]
for every \(n \in \N\). Now, for \(n \in \N\) and \(p \in \R^\times\), we define
\[a_n \coloneqq \int_{(-1,1)} g_n(x) |F(x)|^2 \,dx, \quad b_n \coloneqq \int_{(-\infty,-1) \cup (1,\infty)} g_n(x) |F(x)|^2 \,dx\]
and
\[c_{p,n} \coloneqq \braket*{\sqrt{F_{p,n}}\Lambda F}{\theta_h\sqrt{F_{p,n}}\Lambda F}.\]
Note that \(a_n,b_n \geq 0\), since \(g_n \geq 0\) and \(c_{p,n} \geq 0\) since \(\left(L^2\left(\R,\C\right),H^2\left(\C_+\right),\theta_h\right)\) is a RPHS.

Since the function
\[\R \ni x \mapsto |\Lambda(x)|\]
is symmetric and for every \(p \in \R^\times\) and \(n \in \N\) the function \(\R \ni x \mapsto |F_{p,n}(x)|\) is symmetric, by \fref{lem:OuterSymmetric}, we have
\[(R\Lambda)(x) = \overline{\Lambda(x)} \quad \text{and} \quad \left(R\sqrt{F_{p,n}}\right)(x) = \overline{\sqrt{F_{p,n}}(x)}, \qquad x \in \R.\]
Using \(\theta_h F = F\), this yields
\begin{align*}
c_{p,n} &= \braket*{\sqrt{F_{p,n}}\Lambda F}{\theta_h\sqrt{F_{p,n}}\Lambda F} = \braket*{\theta_h\sqrt{F_{p,n}}\Lambda F}{\sqrt{F_{p,n}}\Lambda F}
\\& = \braket*{\left(R\sqrt{F_{p,n}}\right) \left(R\Lambda\right) \theta_h F}{\sqrt{F_{p,n}}\Lambda F} = \braket*{\overline{\sqrt{F_{p,n}}\Lambda}F}{\sqrt{F_{p,n}}\Lambda F}
\\&= \int_\R F_{p,n}(x)\Lambda(x)^2 |F(x)|^2 \,dx
\\&= \int_\R \frac{\left(x+i\right)^4n^2}{\pi\left(\left(x+\frac in\right)^2-p^2\right)\left(xn+i\right)\left(x+in\right)} \cdot \left(\frac{ix}{(x+i)^2}\right)^2 |F(x)|^2 \,dx
\\&=\frac 1\pi \int_\R \frac{-x^2n^2}{\left(\left(x+\frac in\right)^2-p^2\right)\left(xn+i\right)\left(x+in\right)} |F(x)|^2 \,dx
\\&=\frac 1\pi \int_\R \frac{x^2}{\left(\left(x^2 - p^2 - \frac 1{n^2}\right) + i\frac{2x}{n}\right)\left(\frac 1n - ix\right)\left(1-i\frac xn\right)} |F(x)|^2 \,dx
\\&=\frac 1\pi \int_\R \frac{x^2\left(\left(x^2 - p^2 - \frac 1{n^2}\right) - i\frac{2x}{n}\right)\left(\frac 1n + ix\right)\left(1+i\frac xn\right)}{\left(\left(x^2 - p^2 - \frac 1{n^2}\right)^2 + 4x^2\frac{1}{n^2}\right)\left(\frac 1{n^2}+x^2\right)\left(1+\frac{x^2}{n^2}\right)} |F(x)|^2 \,dx
\\&=\frac 1\pi \int_\R g_n(x) \cdot \frac{n\left(\left(x^2 - p^2 - \frac 1{n^2}\right) - i\frac{2x}{n}\right)\left(\frac 1n + ix\right)\left(1+i\frac xn\right)}{\left(x^2 - p^2 - \frac 1{n^2}\right)^2 + 4x^2\frac{1}{n^2}} |F(x)|^2 \,dx.
\end{align*}
Since \(c_{p,n} \in \R\), this implies
\begin{align*}
c_{p,n} &= \Re(c_{p,n}) = \Re\left(\frac 1\pi \int_\R g_n(x) \cdot \frac{n\left(\left(x^2 - p^2 - \frac 1{n^2}\right) - i\frac{2x}{n}\right)\left(\frac 1n + ix\right)\left(1+i\frac xn\right)}{\left(x^2 - p^2 - \frac 1{n^2}\right)^2 + 4x^2\frac{1}{n^2}} |F(x)|^2 \,dx\right)
\\&= \frac 1\pi \int_\R g_n(x) \cdot \frac{\Re\left(n\left(\left(x^2 - p^2 - \frac 1{n^2}\right) - i\frac{2x}{n}\right)\left(\frac 1n + ix\right)\left(1+i\frac xn\right)\right)}{\left(x^2 - p^2 - \frac 1{n^2}\right)^2 + 4x^2\frac{1}{n^2}} |F(x)|^2 \,dx
\\&= \frac 1\pi \int_\R g_n(x) \cdot \frac{\left(x^2 - p^2 - \frac 1{n^2}\right)(1-x^2) + 2x^2\left(\frac 1{n^2}  + 1\right)}{\left(x^2 - p^2 - \frac 1{n^2}\right)^2 + 4x^2\frac{1}{n^2}} |F(x)|^2 \,dx,
\end{align*}
so defining the function
\begin{align*}
d_{p,n}(x) \coloneqq g_n(x) \cdot \frac{\left(x^2-p^2-\frac 1{n^2}\right)\left(1-x^2\right)+2x^2\left(\frac 1{n^2} + 1\right)}{\left(x^2-p^2-\frac 1{n^2}\right)^2+4x^2 \frac 1{n^2}}, \quad x \in \R,
\end{align*}
we have
\[c_{p,n} = \frac 1\pi \int_\R d_{p,n}(x) |F(x)|^2 \,dx.\]
Then defining
\begin{align*}
f_{p,n}(x) \coloneqq d_{p,n}(x) + g_n(x) \cdot \left(\chi_{\left(-1,1\right)}(x) \cdot \frac 1{p^2} + \chi_{\left(-\infty,-1\right) \cup \left(1,\infty\right)}(x)\right), \quad x \in \R,
\end{align*}
we have
\[\frac 1\pi \int_\R f_{p,n}(x) |F(x)|^2 \,dx = c_{p,n} + \frac 1\pi a_n \cdot \frac 1{p^2} + \frac 1\pi b_n \quad \forall n \in \N.\]
This, by \fref{thm:bigIntegralApproximation}, yields
\begin{align*}
|F(p)|^2 &= \frac{|F(p)|^2 + |F(-p)|^2}2 = \lim_{n \to \infty} \frac 1\pi \int_\R f_{p,n}(x) |F(x)|^2 \,dx
\\&= \lim_{n \to \infty} \left(c_{p,n} + \frac 1\pi a_n \cdot \frac 1{p^2} + \frac 1\pi b_n\right)
\end{align*}
for almost every \(p \in \R\). Since \(a_n,b_n,c_{p,n} \geq 0\) this implies that the sequences \((a_n)_{n \in \N}\), \((b_n)_{n \in \N}\) and \((c_{p,n})_{n \in \N}\) all are bounded. Therefore there exists a subsequence \((n_k)_{k \in \N}\) such that \((a_{n_k})_{k \in \N}\) and \((b_{n_k})_{k \in \N}\) are convergent. Setting
\[A \coloneqq \lim_{k \to \infty} a_{n_k} \qquad \text{and} \qquad B \coloneqq \lim_{k \to \infty} b_{n_k},\]
we then get
\[|F(p)|^2 = \lim_{k \to \infty} \left(c_{p,{n_k}} + \frac 1\pi a_{n_k} \cdot \frac 1{p^2} + \frac 1\pi b_{n_k} \cdot 1\right) = \left(\lim_{k \to \infty} c_{p,{n_k}}\right) + \frac 1\pi A \cdot \frac 1{p^2} + \frac 1\pi B\]
for almost every \(p \in \R\). Therefore, for almost every \(p \in \R\), the sequence \((c_{p,n_k})_{k \in \N}\) is convergent. In this case, we set
\[C_p \coloneqq \lim_{k \to \infty} c_{p,{n_k}}.\]
Now, by \fref{cor:thetaHHasMeasure}, there exists a Borel measure \(\mu \in \mathcal{BM}\left(\R_+\right)\) such that
\[\braket*{f}{\theta_h g} = \int_{\R_+} \overline{f(i\lambda)} g(i\lambda) \,d\mu(\lambda) \quad \forall f,g \in H^2(\C_+).\]
This, using that
\[\Lambda(i\lambda) = \frac{ii\lambda}{(i\lambda+i)^2} = \frac \lambda{(1+\lambda)^2}, \quad \lambda \in \R_+\]
and that \(\sqrt{F_{p,n}}(i\lambda) > 0\), yields
\begin{align*}
c_{p,n} &=\braket*{\sqrt{F_{p,n}}\Lambda F}{\theta_h\sqrt{F_{p,n}}\Lambda F}
\\&= \int_{\R_+} \overline{\sqrt{F_{p,n}}(i\lambda)\Lambda(i\lambda)F(i\lambda)} \cdot \sqrt{F_{p,n}}(i\lambda)\Lambda(i\lambda)F(i\lambda) \,d\mu(\lambda)
\\&= \int_{\R_+} F_{p,n}(i\lambda)\Lambda(i\lambda)^2\left|F(i\lambda)\right|^2 \,d\mu(\lambda)
\\&= \int_{\R_+} \frac 1\pi \cdot \frac{(1+\lambda)^4}{\left(p^2+\left(\lambda+\frac 1n\right)^2\right)\left(\frac 1n+\lambda\right)\left(1+\frac \lambda n\right)}\left(\frac \lambda{(1+\lambda)^2}\right)^2\left|F(i\lambda)\right|^2 \,d\mu(\lambda)
\\&= \frac 1\pi \int_{\R_+} \frac{\lambda^2}{\left(p^2+\left(\lambda+\frac 1n\right)^2\right)\left(\frac 1n+\lambda\right)\left(1+\frac \lambda n\right)}\left|F(i\lambda)\right|^2 \,d\mu(\lambda).
\end{align*}
Then, by the Monotone Convergence Theorem, we get
\begin{align*}
C_p = \lim_{k \to \infty} c_{p,n_k} &= \lim_{k \to \infty} \frac 1\pi \int_{\R_+} \frac{\lambda^2}{\left(p^2+\left(\lambda+\frac 1{n_k}\right)^2\right)\left(\frac 1{n_k}+\lambda\right)\left(1+\frac \lambda {n_k}\right)}\left|F(i\lambda)\right|^2 \,d\mu(\lambda)
\\&= \frac 1\pi \int_{\R_+} \frac{\lambda^2}{\left(p^2+\lambda^2\right) \cdot \lambda \cdot 1}\left|F(i\lambda)\right|^2 \,d\mu(\lambda)
\\&= \frac 1\pi \int_{\R_+} \frac{1+\lambda^2}{p^2+\lambda^2}\left|F(i\lambda)\right|^2 \frac{\lambda}{1+\lambda^2}\,d\mu(\lambda).
\end{align*}
This in particular implies that the measure \(\tilde \nu\) on \(\R_+\) defined by
\[d\tilde \nu(\lambda) = \left|F(i\lambda)\right|^2 \frac{\lambda}{1+\lambda^2}\,d\mu(\lambda)\]
is finite. We then get a finite measure \(\nu\) on \([0,\infty]\) by
\[\nu\big|_{\R_+} = \tilde \nu, \quad \nu(\{0\}) = A \quad \text{and} \quad \nu(\{\infty\}) = B.\]
Finally, for almost every \(p \in \R\), we have
\begin{align*}
|F(p)|^2 &= C_p + \frac 1\pi A \cdot \frac 1{p^2} + \frac 1\pi B
\\&= \frac 1\pi \int_{\R_+} \frac{1+\lambda^2}{p^2+\lambda^2}\,d\nu(\lambda) + \frac 1\pi \nu(\{0\}) \cdot \frac 1{p^2} + \frac 1\pi \nu(\{\infty\}) \cdot 1
\\&= \frac 1\pi \int_{[0,\infty]} \frac{1+\lambda^2}{p^2+\lambda^2}\,d\nu(\lambda) = \Psi_\nu(p). \qedhere
\end{align*}
\end{proof}
\begin{thm}\label{thm:GeneralizedMaxPosForm}
Let \(h \in L^\infty(\R,\T)^\flat\).
Then the following are equivalent:
\begin{enumerate}[\rm (a)]
\item \(\left(L^2\left(\R,\C\right),H^2\left(\C_+\right),\theta_h\right)\) is a maximal RPHS.
\item \(h = h_{\nu}\) for some measure \(\nu \in \mathcal{BM}^{\mathrm{fin}}\left([0,\infty]\right) \setminus \{0\}\).
\end{enumerate}
\end{thm}
\begin{proof}
That (b) implies (a) is precisely the statement of \fref{cor:hNuMaxPos}. We now show that also the converse is true. So let \(\left(L^2\left(\R,\C\right),H^2\left(\C_+\right),\theta_h\right)\) be a maximal RPHS. Then, by \fref{cor:kerNonTrivOuterFixExtended}, there exists an outer function \(F \in \Out^2_\Lambda(\C_+)\) such that
\[h = \frac F{RF}.\]
By \fref{prop:FixedPointFromConstruction} this implies that there exists a measure \(\nu \in \mathcal{BM}^{\mathrm{fin}}\left([0,\infty]\right) \setminus \{0\}\) and \(C \in \T\) such that
\[F = C \cdot F_\nu.\] 
We get
\[h = \frac F{RF} = \frac{C \cdot F_\nu}{C \cdot RF_\nu} = \frac{F_\nu}{RF_\nu} = h_\nu. \qedhere\]
\end{proof}

\subsection{Uniqueness of the measure}\label{sec:UniqueSymbol}
In \fref{thm:GeneralizedMaxPosForm} we have seen that, for every function \(h \in L^\infty(\R,\T)^\flat\) for which the triple \(\left(L^2\left(\R,\C\right),H^2\left(\C_+\right),\theta_h\right)\) is a maximal RPHS, there exists measure \(\nu \in \mathcal{BM}^{\mathrm{fin}}\left([0,\infty]\right) \setminus \{0\}\) such that \(h = h_{\nu}\). In this subsection we want to discuss to which extent this measure is unique. For this we first notice that, for a measure \(\nu \in \mathcal{BM}^{\mathrm{fin}}\left([0,\infty]\right) \setminus \{0\}\) and \(c \in \R_+\), by \fref{lem:OuterHomo}, we have
\[F_{c \cdot \nu} = \Out\left(\sqrt{\Psi_{c \cdot \nu}}\right) = \Out\left(\sqrt{c} \cdot \sqrt{\Psi_\nu}\right) = \sqrt{c} \cdot \Out\left(\sqrt{\Psi_\nu}\right) = \sqrt{c} \cdot F_\nu\]
and therefore
\[h_{c \cdot \nu} = \frac{F_{c \cdot \nu}}{RF_{c \cdot \nu}} = \frac{\sqrt{c} \cdot F_\nu}{\sqrt{c} \cdot RF_\nu} = \frac{F_\nu}{RF_\nu} = h_\nu.\]
This shows that the map
\[\mathcal{BM}^{\mathrm{fin}}\left([0,\infty]\right) \setminus \{0\} \ni \nu \mapsto h_\nu\]
is not injective. On the other hand, defining the equivalence relation
\[\nu \gls*{ZZZtilde} \nu' \quad :\Leftrightarrow \quad (\exists c \in \R_+)\, \nu = c \cdot \nu'\]
and denoting the corresponding equivalence classes by \([\nu]\), it shows that the map
\[\left(\mathcal{BM}^{\mathrm{fin}}\left([0,\infty]\right) \setminus \{0\}\right) \slash \sim \,\ni [\nu] \mapsto h_\nu\]
is well-defined. We want to show that this map is injective. For this, we need the following proposition:
\begin{prop}\label{prop:H2Dense0Infty}
For \(z \in \C_r\) we set
\begin{equation*}
q_z: [0,\infty] \to \C, \quad \lambda \mapsto \frac{\lambda+1}{\lambda+z}.
\end{equation*}
Then
\[\overline{\spann_\C \left\lbrace q_z: z \in \C_r\right\rbrace} = C([0,\infty],\C) \quad \text{and} \quad \overline{\spann_\R \left\lbrace q_z: z \in \R_+\right\rbrace} = C([0,\infty],\R)\]
where the closure is taken with respect to the supremum norm on \([0,\infty]\).
\end{prop}
\begin{proof}
We set
\[A \coloneqq \overline{\spann_\C \left\lbrace q_z: z \in \C_r\right\rbrace}\subeq C([0,\infty],\C) \quad \text{and} \quad B \coloneqq \overline{\spann_\R \left\lbrace q_z: z \in \R_+\right\rbrace}\subeq C([0,\infty],\R).\]
We now show that \(A\) and \(B\) are algebras. For \(z \in \C_r\) one has
\[q_z(\lambda) = \frac{\lambda+1}{\lambda+z} = \frac{1-z}{\lambda+z} + 1.\]
Further, for \(z,w \in \C_r\), we get
\begin{align*}
(1-z)(1-w) \cdot \frac{w-z}{(\lambda + z)(\lambda + w)} &= (1-z)(1-w) \cdot \left(\frac{1}{\lambda+z} - \frac{1}{\lambda+w}\right)
\\&= (1-w) \cdot \frac{1-z}{\lambda+z} - (1-z) \cdot \frac{1-w}{\lambda+w}
\\&= (1-w) \cdot (q_z(\lambda)-1) - (1-z) \cdot (q_w(\lambda)-1).
\end{align*}
If \(z \neq w\) this yields
\begin{align*}
q_z(\lambda) \cdot q_w(\lambda) &= \left(\frac{1-z}{\lambda+z} + 1\right) \cdot \left(\frac{1-w}{\lambda+w} + 1\right)
\\&= \frac{(1-z)(1-w)}{(\lambda+z)(\lambda+w)} + \frac{1-z}{\lambda+z} + \frac{1-w}{\lambda+w} + 1
\\&= \frac{1}{w-z} \cdot \left((1-w) \cdot (q_z(\lambda)-1) - (1-z) \cdot (q_w(\lambda)-1)\right) + q_z(\lambda) + q_w(\lambda) - 1
\\&= \frac{1-z}{w-z} q_z(\lambda) + \frac{1-w}{z-w} q_w(\lambda).
\end{align*}
Therefore, for \(z,w \in \C_r\) with \(z \neq w\) one has \(q_z \cdot q_w \in A\) and \(q_z \cdot q_w \in B\) if \(z,w \in \R_+\). Further, for \(z,w \in \C_r\), we have
\begin{align*}
\left|q_z(\lambda) - q_w(\lambda)\right| &= \left|\frac{\lambda+1}{\lambda+z} - \frac{\lambda+1}{\lambda+w}\right| = \frac{|\lambda+1| \cdot |w-z|}{|\lambda+z| \cdot |\lambda+w|}
\\&= |q_z(\lambda)| \cdot \frac{|w-z|}{|\lambda+w|} \leq |q_z(\lambda)| \cdot \frac{|w-z|}{|w|} = |q_z(\lambda)| \cdot \left|1-\frac{z}{w}\right|,
\end{align*}
so
\[\left\lVert q_z - q_w\right\rVert_\infty \leq \left\lVert q_z\right\rVert_\infty \cdot \left|1-\frac{z}{w}\right| \xrightarrow{w \to z} 0.\]
This shows that, for \(z \in \C_r\), one has
\[q_z \cdot q_z = \lim_{\substack{w \to z \\ w \neq z}} q_z \cdot q_w \in A.\]
By the same argument we get \(q_z \cdot q_z \in B\) for all \(z \in \R_+\).
We therefore have
\[q_z \cdot q_w \in A \quad \forall z,w \in \C_r \qquad \text{and} \qquad q_z \cdot q_w \in B \quad \forall z,w \in \R_+\]
This implies that
\[A \cdot A \subeq A \quad \text{and} \quad B \cdot B \subeq B,\]
so that \(A\) and \(B\) are algebras. We have
\[\textbf{1} = q_1 \in B \subeq A.\]
Further, for \(\lambda,y \in [0,\infty]\) with \(\lambda \neq y\), we have
\[q_2(\lambda) = \frac{\lambda+1}{\lambda+2} \neq \frac{y+1}{y+2} = q_2(y),\]
so the algebras \(A\) and \(B\) separate points. Since
\[\overline{q_z} = q_{\overline{z}}, \quad \forall z \in \C_r\]
the algebra \(A\) is invariant under complex conjugation. Therefore, by the Stone--Weierstraß Theorem, we have
\[A = \overline{A} = C([0,\infty],\C) \quad \text{and} \quad B = \overline{B} = C([0,\infty],\R). \qedhere\]
\end{proof}
\newpage
\begin{cor}\label{cor:PsiUnique}
Let \(\nu,\nu' \in \mathcal{BM}^{\mathrm{fin}}\left([0,\infty]\right)\). Then \(\Psi_\nu = \Psi_{\nu'}\), if and only if \(\nu = \nu'\).
\end{cor}
\begin{proof}
Since the other implication is trivial, we assume that \(\Psi_\nu = \Psi_{\nu'}\) and show that then \(\nu = \nu'\). Setting
\[f_p: [0,\infty] \to \R, \quad \lambda \mapsto \frac{1+\lambda^2}{p^2+\lambda^2}, \qquad p \in \R^\times,\]
by assumption, we have
\[\int_{[0,\infty]} f_p \,d\nu = \Psi_\nu(p) = \Psi_{\nu'}(p) = \int_{[0,\infty]} f_p \,d\nu' \qquad \forall p \in \R^\times.\]
This yields that
\[\int_{[0,\infty]} f \,d\nu = \int_{[0,\infty]} f \,d\nu' \qquad \forall f \in \overline{\spann \{f_p: p \in \R^\times\}}.\]
Now, for \(z \in \R_+\), we consider the function
\begin{equation*}
q_z: [0,\infty] \to \R, \quad \lambda \mapsto \frac{\lambda+1}{\lambda+z}.
\end{equation*}
Then, defining the map
\[\cW: C([0,\infty],\R) \to C([0,\infty],\R), \quad f \mapsto f \circ \sqrt{\cdot},\]
for \(p \in \R^\times\), we get
\[(\cW f_p)(\lambda) = f_p\big(\sqrt{\lambda}\big) = \frac{1+\lambda}{p^2+\lambda} = q_{p^2}(\lambda), \quad \lambda \in \R_+,\]
so \(\cW f_p = q_{p^2}\). Using that \(\cW\) is an isometric isomorphism,
by \fref{prop:H2Dense0Infty}, we get
\begin{align*}
C([0,\infty],\R) &= \cW^{-1} C([0,\infty],\R) = \cW^{-1} \overline{\spann_\R \left\lbrace q_z: z \in \R_+\right\rbrace}
\\&= \cW^{-1} \overline{\spann_\R \left\lbrace q_{p^2}: p \in \R^\times\right\rbrace} = \overline{\spann_\R \left\lbrace \cW^{-1} q_{p^2}: p \in \R^\times\right\rbrace} = \overline{\spann \{f_p: p \in \R^\times\}}.
\end{align*}
Therefore, we have
\[\int_{[0,\infty]} f \,d\nu = \int_{[0,\infty]} f \,d\nu' \quad \forall f \in C([0,\infty],\R),\]
which implies \(\nu = \nu'\).
\end{proof}
Another ingredient we need to prove the injectivity of the map
\[\left(\mathcal{BM}^{\mathrm{fin}}\left([0,\infty]\right) \setminus \{0\}\right) \slash \sim \,\ni [\nu] \mapsto h_\nu\]
is a better understanding of the kernel
\[\ker\left(\Theta_{h_\nu}-{\bf 1}\right) \cap H^2_\Lambda\left(\C_+\right)\]
appearing in \fref{prop:GeneralizedFixedPoint}. This will be provided by the following proposition:
\begin{prop}\label{prop:ker1Dim}
For every measure \(\nu \in \mathcal{BM}^{\mathrm{fin}}\left([0,\infty]\right) \setminus \{0\}\), we have
\begin{equation*}
\ker\left(\Theta_{h_\nu}-{\bf 1}\right) \cap H^2_\Lambda\left(\C_+\right) = \C F_\nu.
\end{equation*}
\end{prop}
\begin{proof}
For every \(C \in \C\) we have
\[\Theta_{h_\nu} (C \cdot F_\nu) = \frac{F_\nu}{RF_\nu} R(C \cdot F_\nu) = C \cdot \frac{F_\nu}{RF_\nu} RF_\nu = C \cdot F_\nu.\]
This yields
\[\C F_\nu \subeq \ker\left(\Theta_{h_\nu}-{\bf 1}\right) \cap H^2_\Lambda\left(\C_+\right).\]
We now show the other inclusion. For this we take a function \(F \in \ker\left(\Theta_{h_\nu}-{\bf 1}\right) \cap H^2_\Lambda\left(\C_+\right)\). We now want to show that \(F \in \C F_\nu\). If \(F=0\), the statement is trivial, so we can assume that \(F \neq 0\). By \fref{cor:kerNonTrivOuterFixExtended}, we have
\[F \in \ker(\Theta_h - \textbf{1}) \cap \Out^2_\Lambda(\C_+)\]
and therefore, by \fref{prop:FixedPointFromConstruction} there exists a measure \(\nu' \in \mathcal{BM}^{\mathrm{fin}}\left([0,\infty]\right) \setminus \{0\}\) and \(C \in \T\) such that
\[F = C \cdot F_{\nu'}.\]
We notice that for the function
\[K(x) \coloneqq \frac 1{|F_\nu(x)| \cdot \sqrt{1+x^2}}, \quad x \in \R,\]
by \fref{lem:OuterExamples} and \fref{lem:OuterHomo}, we have
\[\Out(K)(z) = \frac 1{F_\nu(z)} \cdot \frac i{z+i}, \quad z \in \C_+.\]
Further, by \fref{lem:PsiEstimate}, one has
\[K(x) = \frac 1{|F_\nu(x)| \cdot \sqrt{1+x^2}} = \frac 1{\sqrt{\Psi_\nu(x)} \cdot \sqrt{1+x^2}} \leq 1, \quad x \in \R,\]
so, by \fref{thm:OuterBetrag}, we have \(\Out(K) \in H^\infty(\C_+)\). By \fref{lem:FNuLambdaOuter}, we get
\[\Lambda F = C \cdot \Lambda F_{\nu'} \in H^2(\C_+)\]
and therefore, by \fref{prop:H2InclusionFunctions}, we have
\[g \coloneqq \Out(K) \cdot \Lambda F \in H^2(\C_+).\]
Further
\[g(z) = \frac 1{F_\nu(z)} \cdot \frac i{z+i} \cdot \frac{iz}{(z+i)^2} \cdot F(z) = -\frac z{(z+i)^3} \cdot \frac{F(z)}{F_\nu(z)}, \quad z \in \C_+.\]
Then, using that \(F,F_\nu \in \ker\left(\Theta_{h_\nu}-{\bf 1}\right)\), we get
\begin{align*}
\frac{F(x)}{F_\nu(x)} = \frac{h_\nu(x) \cdot (RF)(x)}{h_\nu(x) \cdot (RF_\nu)(x)} = \frac{F(-x)}{F_\nu(-x)}, \quad x \in \R
\end{align*}
and therefore
\begin{align*}
g(x) &= -\frac x{(x+i)^3} \cdot \frac{F(x)}{F_\nu(x)} = -\frac x{(x+i)^3} \cdot \frac{F(-x)}{F_\nu(-x)}
\\&= - \frac{(-x+i)^3}{(x+i)^3} \cdot g(-x) = \frac{(x-i)^3}{(x+i)^3} \cdot (Rg)(x) = {\phi_i(x)}^3 (Rg)(x), \quad x \in \R.
\end{align*}
Then, since \(Rg \in RH^2(\C_+) = H^2(\C_-)\), we get
\[g \in H^2(\C_+) \cap M_{\phi_i}^3 H^2(\C_-).\]
Now, we have
\begin{align*}
H^2(\C_+) \cap M_{\phi_i} H^2(\C_-) &= H^2(\C_+) \cap M_{\phi_i} \left(H^2(\C_+)^\perp\right) = \C Q_i,
\end{align*}
using \fref{lem:BlaschkeOrthogonal} in the last step. This implies
\begin{align*}
H^2(\C_+) \cap M_{\phi_i}^3 H^2(\C_-) &= H^2(\C_+) \cap M_{\phi_i}^3 \left(H^2(\C_+)^\perp\right) = \bigoplus_{k=1}^3 M_{\phi_i}^{k-1}H^2(\C_+) \cap M_{\phi_i}^k \left(H^2(\C_+)^\perp\right)
\\&= \bigoplus_{k=1}^3 M_{\phi_i}^{k-1}\left(H^2(\C_+) \cap M_{\phi_i} \left(H^2(\C_+)^\perp\right)\right) = \bigoplus_{k=1}^3 M_{\phi_i}^{k-1}\C Q_i
\\&= \C Q_i + \C M_{\phi_i} Q_i + \C M_{\phi_i}^2 Q_i = (\C M_\textbf{1} + \C M_{\phi_i} + \C M_{\phi_i}^2) \,Q_i.
\end{align*}
Therefore, there are \(a,b,c \in \C\) such that
\[g(z) = (a + b \phi_i(z) + c \phi_i(z)^2) \cdot Q_i(z) = \left(a + b \left(\frac{z-i}{z+i}\right) + c \left(\frac{z-i}{z+i}\right)^2\right) \cdot \frac 1{2\pi} \cdot \frac i{z+i}, \quad z \in \C_+.\]
This is equivalent to
\[\frac{F(z)}{F_\nu(z)} = -\frac {(z+i)^3}z \cdot g(z) = \frac 1z \cdot \frac{-i}{2\pi} \left(a(z+i)^2 + b(z+i)(z-i) + c(z-i)^2\right), \quad z \in \C_+.\]
This shows that there are \(\alpha,\beta,\gamma \in \C\) such that
\[\frac{F(z)}{F_\nu(z)} = \frac 1z \cdot \left(\alpha z^2 + \beta z + \gamma\right) = \alpha z + \beta + \gamma\frac 1z, \quad z \in \C_+.\]
Since
\[\alpha x + \beta + \gamma\frac 1x = \frac{F(x)}{F_\nu(x)} = \frac{F(-x)}{F_\nu(-x)} = -\alpha x + \beta -\gamma\frac 1x, \quad x \in \R^\times,\]
we have \(\alpha = \gamma = 0\) and therefore
\[\frac{F(z)}{F_\nu(z)} = \beta, \quad z \in \C_+,\]
so \(F = \beta \cdot F_\nu\).
\end{proof}
\begin{cor}\label{cor:hNuInjektiveSim}
The map
\[\left(\mathcal{BM}^{\mathrm{fin}}\left([0,\infty]\right) \setminus \{0\}\right) \slash \sim \,\ni [\nu] \mapsto h_\nu\]
is injective.
\end{cor}
\begin{proof}
Let \(\nu,\nu' \in \mathcal{BM}^{\mathrm{fin}}\left([0,\infty]\right) \setminus \{0\}\) with \(h_\nu = h_{\nu'}\). Then, by \fref{prop:ker1Dim}, we have
\[F_{\nu'} \in \ker\left(\Theta_{h_{\nu'}}-{\bf 1}\right) \cap H^2_\Lambda\left(\C_+\right) = \ker\left(\Theta_{h_\nu}-{\bf 1}\right) \cap H^2_\Lambda\left(\C_+\right) = \C F_\nu.\]
Since \(F_{\nu'} \neq 0\) this implies that there exists \(c \in \C^\times\) such that
\[F_{\nu'} = c \cdot F_\nu.\]
This yields
\[\Psi_{\nu'}(p) = |F_{\nu'}(p)|^2 = |c|^2 \cdot |F_\nu(p)|^2 = |c|^2 \cdot \Psi_\nu(p) = \Psi_{|c|^2 \cdot \nu}(p), \quad p \in \R.\]
This, by \fref{cor:PsiUnique}, implies \(\nu' = |c|^2 \cdot \nu\) and therefore \(\nu' \sim \nu\), so \([\nu'] = [\nu]\).
\end{proof}

\subsection{The Osterwalder--Schrader transform of the Hardy space}
In this section, given a function \(h \in L^\infty(\R,\T)^\flat\) such that the triple \(\left(L^2\left(\R,\C\right),H^2\left(\C_+\right),\theta_h\right)\) is a RPHS, we want to give an explicit form of the Osterwalder--Schrader transform (OS transform) of this RPHS, which is defined as follows:
\begin{definition}\label{def:OSTrafo}{(cf. \cite[Def. 3.1.1]{NO18})}
Let \(\left(\mathcal{E},\mathcal{E}_+,\theta\right)\) be a RPHS. We set
\begin{equation*}
\gls*{Ntheta} \coloneqq \left\lbrace \eta \in \mathcal{E}_+ : \braket*{\eta}{\theta\eta}=0\right\rbrace,
\end{equation*}
and consider the quotient map
\begin{equation*}
\gls*{qtheta}: \mathcal{E}_+ \to \mathcal{E}_+ \slash \,\mathcal{N}_\theta, \quad \eta \mapsto \eta + \cN_\theta.
\end{equation*}
On \(\mathcal{E}_+ \slash \,\mathcal{N}_\theta\) we define a scalar product by
\[\braket*{\eta + \cN_\theta}{\xi + \cN_\theta}_{\hat \cE} \coloneqq \braket*{\eta}{\theta \xi}, \quad \eta,\xi \in \cE_+\]
and write \(\gls*{Ehat}\) for the Hilbert space completion of \(\mathcal{E}_+ \slash \,\mathcal{N}_\theta\) with respect to the norm
\[\left\lVert \eta + \cN_\theta\right\rVert_{\hat{\cE}} \coloneqq \sqrt{\braket*{\eta + \cN_\theta}{\eta + \cN_\theta}_{\hat \cE}} = \sqrt{\braket*{\eta}{\theta\eta}}, \quad \eta \in \cE_+.\]
We call \(\hat{\mathcal{E}}\) the \textit{Osterwalder--Schrader transform of \(\cE_+\)}.
\end{definition}
\begin{prop}\label{prop:OsterwalderSchraderHardy}
Let \(h \in L^\infty(\R,\T)^\flat\) such that the triple \(\left(L^2\left(\R,\C\right),H^2\left(\C_+\right),\theta_h\right) \eqqcolon (\cE,\cE_+,\theta)\) is a RPHS and let \(\mu \in \mathcal{BM}\left(\R_+\right)\) such that
\[\braket*{f}{\theta_h g} = \int_{\R_+} \overline{f(i\lambda)} g(i\lambda) \,d\mu(\lambda) \quad \forall f,g \in H^2(\C_+).\]
Further, for a function \(f \in H^2(\C_+)\), we consider the function \(f^\sV \in L^2\left(\R_+,\C,\mu\right)\) defined by
\[f^\sV(\lambda) = f(i\lambda), \quad \lambda \in \R_+.\]
Then the map
\begin{equation*}
\Phi: \cE_+ \big \slash \,\cN_\theta \to L^2\left(\R_+,\C,\mu\right), \quad f+\cN_\theta \mapsto f^\sV
\end{equation*}
extends to an isometric isomorphism of Hilbert spaces
\[\tilde \Phi: \hat \cE \to L^2\left(\R_+,\C,\mu\right).\]
\end{prop}
\begin{proof}
For \(f \in H^2\left(\C_+\right)\) we have
\begin{equation*}
\left\lVert f^\sV\right\rVert^2 = \int_{\R_+} \overline{f(i\lambda)} f(i\lambda) \,d\mu(\lambda) = \braket*{f}{\theta_h f} = \braket*{f+\cN_\theta}{f+\cN_\theta}_{\hat \cE},
\end{equation*}
so the map \(\Phi\) is isometric and therefore well defined and injective. This also shows that \(\Phi\) maps every Cauchy sequence in \(\cE_+ \big \slash \,\cN_\theta\) to a Cauchy sequence in \(L^2\left(\R_+,\C,\mu\right)\) and therefore we can continue \(\Phi\) to a map \(\tilde \Phi: \hat \cE \to L^2\left(\R_+,\C,\mu\right)\) in a continuous way. It remains to show that this map \(\tilde \Phi\) is surjective. So let \(F \in \tilde \Phi(\hat \cE)^\perp\). Then, for every \(z \in \C_+\), we have
\begin{equation*}
0 = \braket{\Phi(Q_z + \cN_\theta)}{F} = \int_{\R_+} \overline{Q_z(i\lambda)} F(\lambda) \,d\mu(\lambda).
\end{equation*}
Since
\[Q_z(i\lambda) = \frac 1{2\pi} \frac i{i\lambda-\overline{z}} = \frac 1{2\pi} \frac i{i\lambda+i} \cdot \frac{i\lambda + i}{i\lambda-\overline{z}} = Q_i(i\lambda) \cdot \frac{i\lambda + i}{i\lambda-\overline{z}} = Q_i(i\lambda) \cdot \frac{\lambda + 1}{\lambda+i\overline{z}}, \quad \lambda \in \R_+,\]
we have
\[0 = \int_{\R_+} \overline{Q_z(i\lambda)} F(\lambda) \,d\mu(\lambda) = \int_{\R_+} \frac{\lambda + 1}{\lambda-iz} \cdot \overline{Q_i(i\lambda)} F(\lambda) \,d\mu(\lambda) \quad \forall z \in \C_+.\]
Further, since \(z \in \C_+\) is equivalent to \((-iz) \in \C_r\), this is equivalent to
\[0 = \int_{\R_+} \frac{\lambda + 1}{\lambda+z} \cdot \overline{Q_i(i\lambda)} F(\lambda) \,d\mu(\lambda) \quad \forall z \in \C_r.\]
Then \fref{prop:H2Dense0Infty} implies that
\[0 = \int_{\R_+} f(\lambda) \cdot \overline{Q_i(i\lambda)} F(\lambda) \,d\mu(\lambda) \quad \forall f \in C([0,\infty],\C),\]
so especially
\[0 = \int_{\R_+} f(\lambda) \cdot \overline{Q_i(i\lambda)} F(\lambda) \,d\mu(\lambda) \quad \forall f \in C_c(\R_+,\C).\]
Since the function \(\R_+ \ni \lambda \mapsto \overline{Q_i(i\lambda)}\) is continuous and bounded from below on compact intervals, this yields
\[0 = \int_{\R_+} f(\lambda) F(\lambda) \,d\mu(\lambda) \quad \forall f \in C_c(\R_+,\C),\]
Therefore, by \cite[Thm. 3.14]{Ru86}, we have \(F = 0 \in L^2\left(\R_+,\C,\mu\right)\). This shows that \(\tilde \Phi(\hat \cE)\) is dense in \(L^2\left(\R_+,\C,\mu\right)\), but since \(\tilde \Phi\) is isometric, we know that \(\tilde \Phi(\hat \cE)\) is a closed subspace of \(L^2\left(\R_+,\C,\mu\right)\) and therefore is equal to \(L^2\left(\R_+,\C,\mu\right)\).
\end{proof}
We now want to apply this proposition to maximal RPHS. By \fref{thm:GeneralizedMaxPosForm}, we know that the functions \(h \in L^\infty(\R,\T)^\flat\) for which \(\left(L^2\left(\R,\C\right),H^2\left(\C_+\right),\theta_h\right)\) is a maximal RPHS, are precisely the functions of the form \(h = h_{\nu}\) for some measure \({\nu \in \mathcal{BM}^{\mathrm{fin}}\left([0,\infty]\right) \setminus \{0\}}\). So, we are interested in how the Osterwalder--Schrader transform looks in these cases:
\begin{cor}\label{cor:OSTrafoFromMeasures}
Let \({\nu \in \mathcal{BM}^{\mathrm{fin}}\left([0,\infty]\right) \setminus \{0\}}\) and \((\cE,\cE_+,\theta) \coloneqq \left(L^2\left(\R,\C\right),H^2\left(\C_+\right),\theta_{h_\nu}\right)\). Then
\[\hat \cE \cong L^2\left(\R_+,\C,\cT\nu\right),\]
where the isometric isomorphism between these Hilbert spaces is given by the map
\[\hat \cE \to L^2\left(\R_+,\C,\cT\nu\right), \quad f + \cN_\theta \mapsto f^\sV\]
with
\begin{equation*}
f^\sV(\lambda) \coloneqq f(i\lambda), \quad \lambda \in \R_+, f \in H^2(\C_+).
\end{equation*}
\end{cor}
\begin{proof}
This follows immediately by \fref{thm:cTNuIsCarleson} and \fref{prop:OsterwalderSchraderHardy}.
\end{proof}

\newpage
\section{Application to Hankel operators}
In the last chapter, we answered the question for which functions \(h \in L^\infty(\R,\T)^\flat\) the triple \(\left(L^2\left(\R,\C\right),H^2\left(\C_+\right),\theta_h\right)\) is a maximal RPHS. The triple \(\left(L^2\left(\R,\C\right),H^2\left(\C_+\right),\theta_h\right)\) being a RPHS is, by \fref{lem:RPHSPosHankelEquiv}, equivalent to the operator
\[H_h \coloneqq p_+ \theta_h p_+^* = p_+ M_h R \,p_+^*\]
being a positive operator, where by \(\gls*{p+}: L^2(\R,\C) \to H^2(\C_+)\) we denote the orthogonal projection from \(L^2(\R,\C)\) onto its subspace \(H^2(\C_+)\). We have already seen that this operator \(H_h\) is an example of a Hankel operator (cf. \fref{def:Hankel}). These Hankel operators will be the object of investigation in this chapter. We start with the following definition:
\begin{definition}
\begin{enumerate}[\rm (a)]
\item We write \(\gls*{Han}\) for the set of Hankel operators \(H \in B(H^2(\C_+))\) and write
\[\gls*{HanP} \coloneqq \{H \in \Han: H \geq 0\}\]
for the set of positive Hankel operators.
\item For a function \({h \in L^\infty(\R,\C)}\) we define
\[\gls*{Hh} \coloneqq p_+ M_h R \,p_+^* \in B(H^2(\C_+)).\]
\item A function \({h \in L^\infty(\R,\C)}\) is called a \textit{symbol} of an Hankel operator \(H \in B(H^2(\C_+))\) if \(H = H_h\).
\end{enumerate}
\end{definition}
We have the following result about symbols of Hankel operators:
\begin{theorem}\label{thm:Nehari}{\rm\textbf{(Nehari's Theorem)} (cf. \cite[Cor. 4.7]{Pa88})}
For every \({h \in L^\infty(\R,\C)}\) the operator \(H_h\) is a Hankel operator. Conversely, for every Hankel operator \(H\) there exists \({h \in L^\infty(\R,\C)}\) such that \(H = H_h\). Further, this symbol \(h\) can be choosen in a way such that \(\left\lVert h\right\rVert_\infty = \left\lVert H\right\rVert\).
\end{theorem}
Since the map
\[L^\infty(\R,\C) \ni h \mapsto H_h\]
is linear, the question of whether a given Hankel operator has a unique symbol comes down to the question, for which functions \(h \in L^\infty(\R,\C)\) one has \(H_h = 0\). This question is answered in the following proposition:
\begin{prop}\label{prop:kernelLowerHalfPlane}{\rm (cf. \cite[Cor. 4.8]{Pa88})}
For a function \(h \in L^\infty(\R,\C)\) one has \(H_h = 0\), if and only if \(h \in H^\infty\left(\C_-\right)\).
\end{prop}
\begin{proof}
One has \(H_h = 0\), if and only if
\[\{0\} = p_+ M_h R p_+^* H^2(\C_+) = p_+ M_h H^2(\C_-),\]
which is equivalent to
\[M_h H^2(\C_-) \subeq H^2(\C_-).\]
This, by \fref{prop:H2InclusionFunctions}, is equivalent to \(h \in H^\infty\left(\C_-\right)\).
\end{proof}
We already know that, given a function \(h \in L^\infty(\R,\T)^\flat\), for which \(\left(L^2\left(\R,\C\right),H^2\left(\C_+\right),\theta_h\right)\) is a maximal RPHS, the Hankel operator \(H_h\) is positive and satisfies
\[\left\lVert H_h\right\rVert = \left\lVert p_+ M_h R p_+^*\right\rVert \leq \left\lVert M_h R \right\rVert = \left\lVert h \right\rVert_\infty = 1.\]
By \fref{thm:GeneralizedMaxPosForm}, we know that these functions are exactly the functions of the form \(h = h_{\nu}\) for some measure \(\nu \in \mathcal{BM}^{\mathrm{fin}}\left([0,\infty]\right) \setminus \{0\}\). So we get a map
\[\mathcal{BM}^{\mathrm{fin}}\left([0,\infty]\right) \setminus \{0\} \mapsto \{H \in \HanP : \left\lVert H\right\rVert \leq 1\}, \quad \nu \mapsto H_{h_\nu}.\]
The goal of this chapter is to show that this map is surjective or in other words: Not only does every maximal RPHS of the form \(\left(L^2\left(\R,\C\right),H^2\left(\C_+\right),\theta_h\right)\) give rise to a positive contractive Hankel operator, but in fact every positive contractive Hankel operator can be obtained in this manner. To prove this, we need to link Hankel operators to measures, which we will do in the following section.

\subsection{Hankel operators, Carleson measures and the WOT}
In this section, we want to see that there is a one-to-one correspondence between positive Hankel operators and a certain class of measures, which are called Carleson measures and defined as follows:
\begin{definition}
\begin{enumerate}[\rm (a)]
\item A measure \(\mu \in \mathcal{BM}^{\mathrm{fin}}\left(\R_+\right)\) is called \textit{Carleson measure}, if the map 
\[H^2(\C_+)^2 \to \C, \quad (f,g) \mapsto \int_{\R_+} \overline{f\left(i\lambda\right)}g\left(i\lambda\right) \,d\mu\left(\lambda\right)\]
defines a continuous sesquilinear form on \(H^2(\C_+)\).
\item We write \(\CM\) for the set of Carleson measures on \(\R_+\).
\item Given a Carleson measure \(\mu \in \CM\), we write \(\gls*{Hmu} \in B(H^2(\C_+))\) for the operator defined by
\[\braket*{f}{H_\mu g} \coloneqq \int_{\R_+} \overline{f\left(i\lambda\right)}g\left(i\lambda\right) \,d\mu\left(\lambda\right) \quad \forall f,g \in H^2(\C_+).\]
\end{enumerate}
\end{definition}
We start by showing that the so obtained operators \(H_\mu \in B(H^2(\C_+))\) with \(\mu \in \CM\) are Hankel operators:
\begin{lemma}
Let \(\mu \in \CM\). Then \(H_\mu \in \HanP\).
\end{lemma}
\begin{proof}
For every \(t \in \R_+\) and \(f,g \in H^2(\Omega)\) we have
\begin{align*}
\braket*{f}{H_\mu S_t g} &= \int_{\R_+} \overline{f\left(i\lambda\right)}e^{it \cdot i\lambda} g\left(i\lambda\right) \,d\mu\left(\lambda\right) = \int_{\R_+} \overline{f\left(i\lambda\right)}e^{-t\lambda} g\left(i\lambda\right) \,d\mu\left(\lambda\right)
\\&= \int_{\R_+} \overline{e^{-t\lambda}f\left(i\lambda\right)} g\left(i\lambda\right) \,d\mu\left(\lambda\right) = \int_{\R_+} \overline{e^{it \cdot i\lambda}f\left(i\lambda\right)} g\left(i\lambda\right) \,d\mu\left(\lambda\right)
\\&= \braket*{S_t f}{H_\mu g} = \braket*{f}{S_t^* H_\mu g}
\end{align*}
so \(H S_t = S_t^* H\) and therefore \(H_\mu\) is a Hankel operator. That \(H_\mu\) is positive follows from the fact that, for \(f \in H^2(\Omega)\), we have
\[\braket*{f}{H_\mu f} = \int_{\R_+} \overline{f\left(i\lambda\right)}f\left(i\lambda\right) \,d\mu\left(\lambda\right) = \int_{\R_+} \left|f\left(i\lambda\right)\right|^2 \,d\mu\left(\lambda\right) \geq 0. \qedhere\]
\end{proof}
The following result shows that, in fact, every positive Hankel operator can be obtained in this manner from a unique measure \(\mu \in \CM\):
\begin{prop}\label{prop:HankelOneToOne}
The map
\[\CM \to \HanP, \quad \mu \mapsto H_\mu\]
is a bijection.
\end{prop}
\begin{proof}
The surjectivity follows immediately by \fref{thm:HankelHasMeasure}. We now show the injectivity. So let \(\mu,\mu' \in \CM\) with \(H_\mu = H_{\mu'}\). This is equivalent to
\[\int_{\R_+} \overline{f\left(i\lambda\right)}g\left(i\lambda\right) \,d\mu\left(\lambda\right) = \int_{\R_+} \overline{f\left(i\lambda\right)}g\left(i\lambda\right) \,d\mu'\left(\lambda\right) \quad \forall f,g \in H^2(\C_+),\]
which especially implies
\[\int_{\R_+} \overline{Q_i\left(i\lambda\right)}Q_z\left(i\lambda\right) \,d\mu\left(\lambda\right) = \int_{\R_+} \overline{Q_i\left(i\lambda\right)}Q_z\left(i\lambda\right) \,d\mu'\left(\lambda\right) \quad \forall z \in \C_+.\]
For \(z \in \C_+\), we have
\[Q_z\left(i\lambda\right) = \frac 1{2\pi} \cdot \frac i{i\lambda-\overline{z}} = \frac 1{2\pi} \cdot \frac 1{\lambda+i\overline{z}},\]
so
\[4\pi^2 \overline{Q_i\left(i\lambda\right)}Q_z\left(i\lambda\right) = \frac 1{\lambda +1} \cdot \frac 1{\lambda+i\overline{z}} = \frac {\lambda+1}{\lambda+i\overline{z}} \cdot \frac 1{(\lambda +1)^2}.\]
Therefore, since \(z \in \C_+\), if and only if \(i\overline{z} \in \C_r\), we have
\[\int_{\R_+} \frac {\lambda+1}{\lambda+z} \cdot \frac 1{(\lambda +1)^2} \,d\mu\left(\lambda\right) = \int_{\R_+} \frac {\lambda+1}{\lambda+z} \cdot \frac 1{(\lambda +1)^2} \,d\mu'\left(\lambda\right) \quad \forall z \in \C_r.\]
By \fref{prop:H2Dense0Infty} this implies
\[\int_{\R_+} f(\lambda) \cdot \frac 1{(\lambda +1)^2} \,d\mu\left(\lambda\right) = \int_{\R_+} f(\lambda) \cdot \frac 1{(\lambda +1)^2} \,d\mu'\left(\lambda\right) \quad \forall f \in C([0,\infty],\C)\]
and therefore \(\mu = \mu'\).
\end{proof}
This theorem shows that there is a one-to-one correspondence between positive Hankel operators and Carleson measures. Hence, we can define:
\begin{definition}
Let \(H\) be a positive Hankel operator. Then we write \(\gls*{muH} \in \CM\) for the unique Carleson measure with \(H = H_{\mu_H}\).
\end{definition}
Since we are interested in positive Hankel operators \(H\) of the form \(H = H_{h_\nu}\) with measures \({\nu \in \mathcal{BM}^{\mathrm{fin}}\left([0,\infty]\right) \setminus \{0\}}\), we want to know what their corresponding Carleson measure is. The following proposition will answer this question:
\begin{prop}\label{prop:cTNuIsCarlesonOfHNu}
Let \(\nu \in \mathcal{BM}^{\mathrm{fin}}\left([0,\infty]\right) \setminus \{0\}\). Then \(H_{h_\nu} \in \HanP\) with \(\left\lVert H_{\nu_n}\right\rVert \leq 1\). Further, \(\cT\nu \in \CM\) (cf. \fref{def:defCT}) and
\[H_{\cT\nu} = H_{h_\nu}.\]
\end{prop}
\begin{proof}
By \fref{thm:cTNuIsCarleson}, for all \(f,g \in H^2(\C_+)\), one has
\[\int_{\R_+} \overline{f(i\lambda)}\,g(i\lambda) \,d(\cT \nu)(\lambda) = \braket*{f}{\theta_{h_\nu}g} = \braket*{f}{H_{h_\nu}g},\]
which immediately implies the statement.
\end{proof}
Since later we will work with approximations in the weak operator topology, we now want to see how this topology behaves on the set of positive Hankel operators. For this, we make the following definition:
\begin{definition}
Let \(\cH\) be a Hilbert space over a field \(\K = \R,\C\).
\begin{enumerate}[\rm (a)]
\item Given a subset \(M \subeq B(\cH)\), we write
\[\overline{M}^{\gls*{WOT}}\]
for the closure of \(M\) with respect to the weak operator topology, i.e. the topology induced by the functions
\[f_{v,w}: B(\cH) \to \K, \quad A \mapsto \braket*{v}{Aw}, \quad v,w \in \cH.\]
\item Given \(A \in B(\cH)\) and a sequence \((A_n)_{n \in \N}\) in \(B(\cH)\), we write
\[A_n \underset{\mathrm{WOT}}{\xrightarrow{n \to \infty}} A\]
if \(A_n\) converges to \(A\) with respect to the weak operator topology, i.e. if
\[\braket*{v}{A_n w} \xrightarrow{n \to \infty} \braket*{v}{Aw} \quad \forall v,w \in \cH.\]
\end{enumerate}
\end{definition}
\begin{lem}\label{lem:NormSubContinuous}
Let \(\cH\) be a Hilbert space. Let \(A \in B(\cH)\) and \((A_n)_{n \in \N}\) be a sequence in \(B(\cH)\) with
\[A_n \underset{\mathrm{WOT}}{\xrightarrow{n \to \infty}} A.\]
Then
\[\left\lVert A\right\rVert \leq \limsup_{n \to \infty} \left\lVert A_n \right\rVert.\]
\end{lem}
\begin{proof}
For \(v,w \in \cH\) we have
\[|\braket*{v}{Aw}| = \limsup_{n \to \infty} |\braket*{v}{A_nw}| \leq \limsup_{n \to \infty} \left\lVert A_n \right\rVert \cdot \left\lVert v \right\rVert \cdot \left\lVert w \right\rVert,\]
which immediately implies the statement.
\end{proof}
\begin{lem}\label{lem:ContractiveHankelWOTClosed}
The set
\[\{H \in \HanP : \left\lVert H\right\rVert \leq 1\}\]
is closed in the weak operator topology.
\end{lem}
\begin{proof}
Let \(H \in \overline{\{H \in \HanP : \left\lVert H\right\rVert \leq 1\}}\). Since \(H^2(\C_+)\) is separable, the weak operator topology restricted to the closed unit ball is metrizable, and therefore there exists a sequence \((H_n)_{n \in \N}\) in \(\{H \in \HanP : \left\lVert H\right\rVert \leq 1\}\) with
\[H_n \underset{\mathrm{WOT}}{\xrightarrow{n \to \infty}} W.\]
By \fref{lem:NormSubContinuous} we then get
\[\left\lVert H\right\rVert \leq \limsup_{n \to \infty} \left\lVert H_n \right\rVert \leq 1.\]
Further, for \(t \in \R_+\) and \(f,g \in H^2(\C_+)\), we have
\[\braket*{f}{HS_t g} = \lim_{n \to \infty} \braket*{f}{H_nS_t g} = \lim_{n \to \infty} \braket*{f}{S_t^*H_n g} = \lim_{n \to \infty} \braket*{S_t f}{H_n g} = \braket*{S_t f}{Hg} = \braket*{f}{S_t^* Hg},\]
so \(HS_t = S_t^* H\) and therefore \(H \in \Han\). Also, for every \(f \in H^2(\C_+)\), one has
\[\braket*{f}{Hf} = \lim_{n \to \infty} \braket*{f}{H_nf} \geq 0,\]
so \(H \in \HanP\).
\end{proof}
We now want to see how convergence of positive Hankel operators in the weak operator topology translates to convergence of their corresponding Carleson measures. For this, we need the following lemmata:
\begin{lemma}\label{lem:WOTDenseFunctionals}
Let \(\cH\) be a Hilbert space and \(V \subeq \cH\) be a dense subspace. Then, for \(A \in B(\cH)\) and a bounded sequence \((A_n)_{n \in \N}\) in \(B(\cH)\), the following are equivalent:
\begin{enumerate}[\rm (a)]
\item \(A_n \underset{\mathrm{WOT}}{\xrightarrow{n \to \infty}} A\).
\item \(\braket*{v}{A_n w} \xrightarrow{n \to \infty} \braket*{v}{Aw} \quad \forall v,w \in V\).
\end{enumerate}
\end{lemma}
\begin{proof}
That (a) implies (b) is trivial. We now assume that (b) holds and let \(v,w \in \cH\) and \(\epsilon > 0\). Further, let
\[M \coloneqq \sup_{n \in \N} \left\lVert A_n-A\right\rVert < \infty.\]
Then there exist \(\tilde v, \tilde w \in V\) with
 \[\left\lVert w-\tilde w\right\rVert < \frac{\epsilon}{3(M \left\lVert v\right\rVert + 1)} \quad \text{and} \quad \left\lVert v-\tilde v\right\rVert < \frac{\epsilon}{3(M \left\lVert \tilde w\right\rVert + 1)}.\]
Further, since
\[\braket*{\tilde v}{A_n \tilde w} \xrightarrow{n \to \infty} \braket*{\tilde v}{A\tilde w},\]
there exists \(N \in \N\) such that
\[\left|\braket*{\tilde v}{(A_n - A) \tilde w}\right| < \frac \epsilon 3 \quad \forall n \geq N.\]
Then, for \(n \geq N\), one has
\begin{align*}
\left|\braket*{v}{(A_n - A) w}\right| &= \left|\braket*{\tilde v}{(A_n - A) \tilde w} + \braket*{v-\tilde v}{(A_n - A) \tilde w}+ \braket*{v}{(A_n - A) (w-\tilde w)}\right|
\\&\leq \left|\braket*{\tilde v}{(A_n - A) \tilde w}\right| + \left|\braket*{v-\tilde v}{(A_n - A) \tilde w}\right| + \left|\braket*{v}{(A_n - A) (w-\tilde w)}\right|
\\&\leq \left|\braket*{\tilde v}{(A_n - A) \tilde w}\right| + \left\lVert v - \tilde v\right\rVert \cdot M \cdot \left\lVert \tilde w\right\rVert + \left\lVert v\right\rVert \cdot M \cdot \left\lVert w - \tilde w\right\rVert
\\&< \frac \epsilon 3 + \frac \epsilon 3 + \frac \epsilon 3 = \epsilon. \qedhere
\end{align*}
\end{proof}
\begin{lemma}\label{lem:WeakDenseFunctionals}
Let \(B\) be a Banach space and \(V \subeq B\) be a dense subspace. Then, for \(\phi \in B'\) and a bounded sequence \((\phi_n)_{n \in \N}\) in \(B'\), the following are equivalent:
\begin{enumerate}[\rm (a)]
\item \(\phi_n \underset{w}{\xrightarrow{n \to \infty}} \phi\).
\item \(\phi_n(v) \xrightarrow{n \to \infty} \phi(v) \quad \forall v \in V\).
\end{enumerate}
\end{lemma}
\begin{proof}
That (a) implies (b) is trivial. We now assume that (b) holds and let \(v \in B\) and \(\epsilon > 0\). Further, let
\[M \coloneqq \sup_{n \in \N} \left\lVert \phi_n\right\rVert < \infty.\]
Then there exists \(\tilde v \in V\) with
 \[\left\lVert v-\tilde v\right\rVert < \frac{\epsilon}{2(M + 1)}.\]
Further, since
\[\phi_n(\tilde v) \xrightarrow{n \to \infty} \phi(\tilde v),\]
there exists \(N \in \N\) such that
\[\left|(\phi_n - \phi)(\tilde v)\right| < \frac \epsilon 2 \quad \forall n \geq N.\]
Then, for \(n \geq N\), one has
\begin{align*}
\left|(\phi_n - \phi)(v)\right| &= \left|(\phi_n - \phi)(\tilde v) + (\phi_n - \phi)(v-\tilde v)\right|
\\&\leq \left|(\phi_n - \phi)(\tilde v)\right| + \left|(\phi_n - \phi)(v-\tilde v)\right|
\\&\leq \left|(\phi_n - \phi)(\tilde v)\right| + M \cdot \left\lVert v - \tilde v\right\rVert < \frac \epsilon 2 + \frac \epsilon 2 = \epsilon. \qedhere
\end{align*}
\end{proof}
\begin{prop}\label{prop:topologiesEquiv}
For \(H \in \HanP\) and a bounded sequence \((H_n)_{n \in \N}\) in \(\HanP\) the following are equivalent:
\begin{enumerate}[\rm (a)]
\item \(H_n \underset{\mathrm{WOT}}{\xrightarrow{n \to \infty}} H\).
\item \(\braket*{Q_w}{H_n Q_z} \xrightarrow{n \to \infty} \braket*{Q_w}{H Q_z} \quad \forall w,z \in \C_+\).
\item One has
\[\int_{\R_+} f(\lambda) \frac 1{(1+\lambda)^2}\,d\mu_{H_n}(\lambda) \xrightarrow{n \to \infty} \int_{\R_+} f(\lambda) \frac 1{(1+\lambda)^2}\,d\mu_H(\lambda)\]
for every function \(f \in C([0,\infty],\C)\).
\item One has
\[\int_{\R_+} f(\lambda) \,d\mu_n(\lambda) \xrightarrow{n \to \infty} \int_{\R_+} f(\lambda) \,d\mu(\lambda)\]
for every function \(f \in C_c(\R_+,\C)\) and
\[\int_{\R_+} \frac 1{(1+\lambda)^2}\,d\mu_n(\lambda) \xrightarrow{n \to \infty} \int_{\R_+} \frac 1{(1+\lambda)^2}\,d\mu(\lambda).\]
\end{enumerate}
\end{prop}
\begin{proof}
\begin{itemize}
\item[(a) \(\Leftrightarrow\) (b):] This follows immediately by the sesquilinearity of the scalar product, the fact that, by \fref{lem: evDense}, the set \(\spann \{Q_z: z \in \C_+\}\) is dense in \(H^2(\C_+)\) and \fref{lem:WOTDenseFunctionals}.
\item[(b) \(\Leftrightarrow\) (c):]
For \(z \in \C_+\), we have
\[Q_z\left(i\lambda\right) = \frac 1{2\pi} \cdot \frac i{i\lambda-\overline{z}} = \frac 1{2\pi} \cdot \frac 1{\lambda+i\overline{z}}, \quad \lambda \in \R_+.\]
For \(w,z \in \C_+\) this yields
\[\overline{Q_w\left(i\lambda\right)}Q_z\left(i\lambda\right) = \frac 1{4\pi^2 } \cdot \frac 1{\lambda-iw} \cdot \frac 1{\lambda+i\overline{z}} = \frac 1{4\pi^2 } \cdot \frac {\lambda + 1}{\lambda-iw} \cdot \frac {\lambda+1}{\lambda+i\overline{z}} \cdot \frac 1{(\lambda +1)^2}, \quad \lambda \in \R_+,\]
so setting
\[f_{w,z}: [0,\infty] \to \C, \quad \lambda \mapsto \frac 1{4\pi^2 } \cdot \frac {\lambda + 1}{\lambda-iw} \cdot \frac {\lambda+1}{\lambda+i\overline{z}},\]
we have \(f_{w,z} \in C([0,\infty],\C)\) and
\[\overline{Q_w\left(i\lambda\right)}Q_z\left(i\lambda\right) = f_{w,z}(\lambda) \cdot \frac 1{(\lambda +1)^2}, \quad \lambda \in \R_+.\]
Since
\[f_{i,z}(\lambda) = \frac 1{4\pi^2 } \cdot \frac {\lambda + 1}{\lambda-i \cdot i} \cdot \frac {\lambda+1}{\lambda+i\overline{z}} = \frac 1{4\pi^2 } \cdot \frac {\lambda+1}{\lambda+i\overline{z}},\]
by \fref{prop:H2Dense0Infty}, we have
\[\overline{\spann \{f_{w,z}: w,z \in \C_+\}} = C([0,\infty],\C),\]
where the closure is taken with respect to the supremum norm. The equivalence then follows by \fref{lem:WeakDenseFunctionals}, since we have
\[\overline{Q_i\left(i\lambda\right)}Q_i\left(i\lambda\right) = f_{i,i}(\lambda) \cdot \frac 1{(\lambda +1)^2} = \frac 1{4\pi^2 } \cdot \frac 1{(\lambda +1)^2}\]
and therefore
\begin{align*}
\sup_{n \in \N} \,\int_{\R_+} \frac 1{(1+\lambda)^2}\,d\mu_n(\lambda) &= 4\pi^2 \cdot \sup_{n \in \N}\, \int_{\R_+} \overline{Q_i\left(i\lambda\right)}Q_i\left(i\lambda\right)\,d\mu_n(\lambda)
\\&= 4\pi^2 \cdot \sup_{n \in \N}\,\braket*{Q_i}{H_n Q_i} \leq 4\pi^2 \cdot \sup_{n \in \N} \left\lVert H_n\right\rVert \cdot \left\lVert Q_i\right\rVert^2 < \infty.
\end{align*}
\item[(c) \(\Leftrightarrow\) (d):] By the Stone--Weierstraß Theorem we have
\[\overline{C_c(\R_+,\C) \oplus (\C \cdot \textbf{1})} = C([0,\infty],\C),\]
where the closure is taken with respect to the supremum norm. This, by \fref{lem:WeakDenseFunctionals}, implies that (c) holds, if and only if
\[\int_{\R_+} \frac 1{(1+\lambda)^2}\,d\mu_n(\lambda) \xrightarrow{n \to \infty} \int_{\R_+} \frac 1{(1+\lambda)^2}\,d\mu(\lambda)\]
and
\[\int_{\R_+} f(\lambda) \frac 1{(1+\lambda)^2} \,d\mu_n(\lambda) \xrightarrow{n \to \infty} \int_{\R_+} f(\lambda) \frac 1{(1+\lambda)^2} \,d\mu(\lambda) \quad \forall f \in C_c(\R_+,\C).\]
The latter is equivalent to
\[\int_{\R_+} f(\lambda) \,d\mu_n(\lambda) \xrightarrow{n \to \infty} \int_{\R_+} f(\lambda) \,d\mu(\lambda) \quad \forall f \in C_c(\R_+,\C),\]
since the function
\[\R_+ \ni \lambda \mapsto \frac 1{(1+\lambda)^2}\]
is continuous and bounded from below on compact sets. \qedhere
\end{itemize}
\end{proof}
As mentioned before, our goal is to show that the map
\[\mathcal{BM}^{\mathrm{fin}}\left([0,\infty]\right) \setminus \{0\} \mapsto \{H \in \HanP : \left\lVert H\right\rVert \leq 1\}, \quad \nu \mapsto H_{h_\nu}\]
is surjective. We will do this in two steps. In the next section, we will show that
\[\{H \in \HanP : \left\lVert H\right\rVert \leq 1\} = \overline{\left\{H_{h_{W(\mu)}}: \mu \in \mathcal{BM}^{\mathrm{fin}}\left(\R_+\right) \setminus \{0\}\right\}}^{\mathrm{WOT}}\]
and in the following section, we then will use this to show that
\[\{H \in \HanP : \left\lVert H\right\rVert \leq 1\} = \left\{H_{h_{\nu}}: \nu \in \mathcal{BM}^{\mathrm{fin}}\left([0,\infty]\right) \setminus \{0\}\right\}.\]

\subsection{Hankel operators with fixed points and a density theorem}
In this section, we start by considering positive Hankel operators \(H \in \HanP\) with \(\left\lVert H\right\rVert \leq 1\) and with the additional assumption that \(H\) has a fixed point. A characterization of these operators is given by the following proposition:
\begin{prop}\label{prop:HankelFixedPoint}
Let \(H \in \HanP\) with \(\left\lVert H\right\rVert \leq 1\). Then the following are equivalent:
\begin{enumerate}[\rm (a)]
\item \(\ker(H-\textbf{1}) \neq \{0\}\).
\item There exists a function \(h \in L^\infty(\R,\T)^\flat\) such that \(H_h = H\) and \(\ker\left(\theta_h-{\bf 1}\right) \cap H^2(\C_+) \neq \left\lbrace 0\right\rbrace\).
\item There exists a measure \(\mu \in \mathcal{BM}^{\mathrm{fin}}\left(\R_+\right) \setminus \{0\}\) such that \(H = H_{h_{W(\mu)}}\).
\end{enumerate}
\end{prop}
\begin{proof}
\begin{itemize}
\item[(a) \(\Rightarrow\) (b):] Let \(f \in \ker(H-\textbf{1}) \setminus \{0\}\). By Nehari's theorem (\fref{thm:Nehari}) there exists a function \(h \in L^\infty(\R,\C)\) with \(\left\lVert h \right\rVert_\infty = \left\lVert H\right\rVert \leq 1\) such that \(H_h = H\). Then, denoting by \(P_+\) the orthogonal projection onto \(H^2(\C_+)\), we have
\[\left\lVert f\right\rVert^2 = \left\lVert H_hf\right\rVert^2 = \left\lVert P_+hRf\right\rVert^2 = \left\lVert hRf\right\rVert^2 - \left\lVert (\textbf{1}-P_+)hRf\right\rVert^2 \leq \left\lVert f\right\rVert^2 - \left\lVert (\textbf{1}-P_+)hRf\right\rVert^2,\]
which yields \(\left\lVert (\textbf{1}-P_+)hRf\right\rVert^2 = 0\) and therefore \((\textbf{1}-P_+)hRf = 0\), so we get
\[\theta_h f = hRf = P_+hRf = H_hf = f.\]
This yields \(f \in \ker\left(\theta_h-{\bf 1}\right) \cap H^2(\C_+)\) and therefore
\[\ker\left(\theta_h-{\bf 1}\right) \cap H^2(\C_+) \neq \left\lbrace 0\right\rbrace.\]
Further, using \fref{thm:HardyAlmostEverywhereNonZero}, we get
\[h = \frac{f}{Rf}.\]
This yields
\[h = \frac{f}{Rf} = R\left(\frac{Rf}{f}\right) = R\left(\frac{1}{h}\right) = \frac 1{Rh}.\]
Then, for almost every \(x \in \R\), we get
\[1 \geq \left\lVert h\right\rVert_\infty \geq \left|h(x)\right| = \frac{1}{|h(-x)|} \geq \frac 1{\left\lVert h\right\rVert_\infty} \geq 1\]
and therefore \(h(x) \in \T\), so we have \(h \in L^\infty(\R,\T)\). This yields \(\frac 1h = \overline{h}\) and therefore
\[h^\flat = \overline{Rh} = \overline{\left(\frac 1h\right)} = h.\]
\item[(b) \(\Rightarrow\) (a):] Let \(h \in L^\infty(\R,\T)^\flat\) and \(f \in \left(\ker\left(\theta_h-{\bf 1}\right) \cap H^2(\C_+)\right) \setminus \{0\}\). Then
\begin{equation*}
H f = H_h f = p_+ \theta_h p_+^* f = p_+ \theta_h f = p_+ f = f
\end{equation*}
and therefore \(f \in \ker(H-\textbf{1}) \setminus \{0\}\).
\item[(b) \(\Leftrightarrow\) (c):] By \fref{lem:RPHSPosHankelEquiv}, for a function \(h \in L^\infty(\R,\T)^\flat\), one has \(H_h \in \HanP\), if and only if \((L^2(\R,\C),H^2(\C_+),\theta_h)\) is a RPHS. Further, by \fref{thm:thetaFixedPointCharac}, the functions \(h \in L^\infty(\R,\T)^\flat\) for which \((L^2(\R,\C),H^2(\C_+),\theta_h)\) is a RPHS and \(\ker\left(\theta_h-{\bf 1}\right) \cap H^2(\C_+) \neq \left\lbrace 0\right\rbrace\) are precisely the functions of the form \(h_{W(\mu)}\) for some measure \(\mu \in \mathcal{BM}^{\mathrm{fin}}\left(\R_+\right) \setminus \{0\}\). \qedhere
\end{itemize}
\end{proof}
We now want to find a sufficient condition for a Hankel operator to have a fixed point, that is easy to check. One of these conditions is that the Hankel operator is compact:
\begin{lemma}\label{lem:compactFixpoint}
Let \(H\) be a compact positive Hankel operator with \(\left\lVert H \right\rVert = 1\). Then
\[\ker(H-\textbf{1}) \neq \{0\}.\]
\end{lemma}
\begin{proof}
By the spectral theorem for compact self-adjoint operators, there is a null sequence \(\left(\lambda_k\right)_{k \in \N}\) in \(\R_{\geq 0}\) with
\[\lambda_{k+1} \leq \lambda_k \quad \forall k \in \N\]
and a sequence of orthogonal finite-dimensional projections \(\left(P_k\right)_{k \in \N}\) in \(B(H^2(\C_+))\) such that
\[H = \sum_{k = 1}^\infty \lambda_k P_k.\]
Since \(\left\lVert H \right\rVert = 1\), we have \(\lambda_1 = 1\) and \(P_1 \neq 0\). Therefore we can choose a function \(f \in P_1 H^2(\C_+) \setminus \{0\}\) and get
\[Hf = \lambda_1P_1 f = f. \qedhere\]
\end{proof}
A powerful tool for checking if a Hankel operator is compact is given by Hartman's Theorem:
\begin{thm}\label{thm:Hartman}{\rm \textbf{(Hartman's Theorem for the upper half-plane)}}
A Hankel operator \(H\) is compact, if and only if \(H\) has a continuous symbol \(h \in C(\R)\) with
\[\lim_{p \to -\infty}h(p) = \lim_{p \to \infty}h(p).\] 
\end{thm}
\begin{proof}
This follows immediately by transforming Hartman's Theorem (\cite[Thm. 3.20]{Pa88}) from the unit disc to the upper half-plane.
\end{proof}
In general, it is quite difficult to find a concrete symbol for a given Hankel operator, even more so to find a continuous one as it is required in Hartman's Theorem, but for positive Hankel operators, the following theorem provides concrete symbols with good analytic properties:
\begin{theorem}\label{thm:CarlesonRepresentant}{\rm (cf. \cite[Thm. 4.1]{ANS22})}
Let \(H\) be a positive Hankel operator and set
\[h\left(p\right) \coloneqq \frac i\pi \cdot \int_{\R_+} \frac {p}{\lambda^2+p^2} \,d\mu_H\left(\lambda\right), \quad p \in \R.\]
Then \(h \in L^\infty\left(\R,\C\right)\) and \(H_h = H\).
\end{theorem}
This result, together with Hartman's Theorem, allows us to provide an easy-to-check criterion for a positive Hankel operator to be compact:
\begin{prop}\label{prop:compactCriterion}
Let \(\mu \in \mathcal{BM}\left(\R_+\right)\) with
\[\int_{\R_+} \frac 1\lambda \,d\mu(\lambda)<\infty.\]
Then \(\mu \in \CM\) and the Hankel operator \(H_\mu\) is compact. 
\end{prop}
\begin{proof}
For \(f \in H^2(\C_+)\) and \(\lambda \in \R_+\), by \fref{prop:H2RKHS}, we have
\[|f(i\lambda)| = |\braket*{Q_{i\lambda}}{f}| \leq \left\lVert Q_{i\lambda}\right\rVert \cdot \left\lVert f\right\rVert = \frac 1{2\sqrt{\pi \lambda}} \cdot \left\lVert f\right\rVert.\]
Therefore, for \(f,g \in H^2(\C_+)\), we get
\[\int_{\R_+} |f(i\lambda)| \cdot |g(i\lambda)| \,d\mu(\lambda) \leq \int_{\R_+} \frac 1{2\sqrt{\pi \lambda}} \cdot \left\lVert f\right\rVert \cdot \frac 1{2\sqrt{\pi \lambda}} \cdot \left\lVert g\right\rVert \,d\mu(\lambda) = \frac 1{4\pi} \int_{\R_+} \frac 1\lambda \,d\mu(\lambda) \cdot \left\lVert f\right\rVert \cdot \left\lVert g\right\rVert,\]
which shows that \(\mu \in \CM\). Further, by \fref{thm:CarlesonRepresentant}, we know that, for the function
\[h(p) \coloneqq \frac i\pi \cdot \int_{\R_+} \frac{p}{\lambda^2+p^2} \,d\mu(\lambda),\]
we have \(h \in L^\infty(\R,\C)\) and \(H_h = H_\mu\). We now show that \(h \in C(\R)\) with
\[\lim_{p \to -\infty}h(p) = \lim_{p \to \infty}h(p),\]
since then the statement follows by \fref{thm:Hartman}. For every \(\lambda \in \R_+\) and \(p \in \R\), we have
\[\left|\frac{p}{\lambda^2+p^2}\right| = \frac{\left|p\right|}{\lambda^2+\left|p\right|^2} \leq \frac{\left|p\right|}{2\lambda \left|p\right|} = \frac 1{2\lambda}.\]
Then, by the Dominated Convergence Theorem, for every \(p \in \R\), one has
\[\lim_{p' \to p} h(p') = \lim_{p' \to p} \frac i\pi \cdot \int_{\R_+} \frac{p'}{\lambda^2+p'^2} \,d\mu(\lambda) = \frac i\pi \cdot \int_{\R_+} \frac{p}{\lambda^2+p^2} \,d\mu(\lambda) = h(p),\]
so \(h \in C(\R)\). Also, by the Dominated Convergence Theorem, we get
\[\lim_{p \to \pm\infty} h(p) = \lim_{p \to \pm\infty} \frac i\pi \cdot \int_{\R_+} \frac{p}{\lambda^2+p^2} \,d\mu(\lambda) = \frac i\pi \cdot \int_{\R_+} 0 \,d\mu(\lambda) = 0,\]
which implies the statement.
\end{proof}
\begin{cor}\label{cor:strongDense}
Let \(H\) be a positive Hankel operator with \(\left\lVert H\right\rVert = 1\). Then there exists a sequence \((\mu_n)_{n \in \N}\) in \(\mathcal{BM}^{\mathrm{fin}}\left(\R_+\right) \setminus \{0\}\) such that
\[H_{h_{W(\mu_n)}} \underset{\mathrm{WOT}}{\xrightarrow{n \to \infty}} H.\]
\end{cor}
\begin{proof}
For \(t \in \left[1,\infty\right)\) we define the measure \(\rho_t\) on \(\R_+\) by
\begin{equation*}
\rho_t(A) \coloneqq \mu_H\left(A \cap \left[\frac 1 t,t\right]\right)
\end{equation*}
for every Borel set \(A \subeq \R_+\). Then
\[\int_{\R_+} \frac 1\lambda \,d\rho_t(\lambda) = \int_{\left[\frac 1t,t\right]} \frac 1\lambda \,d\mu_H(\lambda) \leq \int_{\left[\frac 1t,t\right]} t \,d\mu_H(\lambda) = t \cdot \mu_H\left(\left[\frac 1t,t\right]\right) < \infty,\]
so, by \fref{prop:compactCriterion}, we have \(\rho_t \in \CM\) and the Hankel operator \(H_{\rho_t}\) is compact. Further, for \(f,g \in H^2(\C_+)\) and \(t \in [1,\infty)\), one has
\begin{align*}
\left|\braket*{f}{(H-H_{\rho_t})g}\right| &= \left|\int_{\R \setminus \left[\frac 1 t,t\right]} \overline{f(i\lambda)}g(i\lambda)\,d\mu_H(\lambda)\right|
\\&\leq \sqrt{\int_{\R \setminus \left[\frac 1 t,t\right]} \left|f(i\lambda)\right|^2\,d\mu_H(\lambda)} \cdot \sqrt{\int_{\R \setminus \left[\frac 1 t,t\right]} \left|g(i\lambda)\right|^2\,d\mu_H(\lambda)} \xrightarrow{t \to \infty} 0.
\end{align*}
This implies that \(H_{\rho_t} \underset{\mathrm{WOT}}{\xrightarrow{t \to \infty}} H\). We have
\[\braket*{f}{H_{\rho_t}f} = \int_{\left[\frac 1 t,t\right]} \left|f(i\lambda)\right|^2\,d\mu_H(\lambda) \leq \int_{\R_+} \left|f(i\lambda)\right|^2\,d\mu_H(\lambda) = \braket*{f}{Hf},\]
and therefore \(\left\lVert H_{\rho_t} \right\rVert \leq \left\lVert H \right\rVert\). This, by \fref{lem:NormSubContinuous}, yields
\[1 = \left\lVert H\right\rVert \leq \limsup_{t \to \infty}\left\lVert H_{\rho_t}\right\rVert \leq \limsup_{t \to \infty}\left\lVert H\right\rVert = 1.\]
Therefore, we can choose a sequence \((t_n)_{n \in \N}\) such that
\[t_n \xrightarrow{n \to \infty} 0 \quad \text{and} \quad \left\lVert H_{\rho_{t_n}}\right\rVert \xrightarrow{n \to \infty} 1\]
and
\[H_{\rho_{t_n}} \neq 0 \quad \forall n \in \N.\]
We now set
\[H_n \coloneqq \frac 1{\left\lVert H_{\rho_{t_n}}\right\rVert}H_{\rho_{t_n}}.\]
Then
\[H_n \underset{\mathrm{WOT}}{\xrightarrow{n \to \infty}} H.\]
Further, for every \(n \in \N\), the operator \(H_n\) is a positive compact Hankel operator with \(\left\lVert H_n\right\rVert = 1\), so, by \fref{lem:compactFixpoint} and \fref{prop:HankelFixedPoint}, there exists a measure \(\mu_n \in \mathcal{BM}^{\mathrm{fin}}\left(\R_+\right) \setminus \{0\}\) such that \(H_n = H_{h_{W(\mu_n)}}\). We then have
\[H_{h_{W(\mu_n)}} = H_n \underset{\mathrm{WOT}}{\xrightarrow{n \to \infty}} H. \qedhere\]
\end{proof}
\begin{prop}\label{prop:dDenseNormLeq1}
One has
\[\{H \in \HanP : \left\lVert H\right\rVert \leq 1\} = \overline{\{H_{h_{W(\mu)}}: \mu \in \mathcal{BM}^{\mathrm{fin}}\left(\R_+\right) \setminus \{0\}\}}^{\mathrm{WOT}}.\]
\end{prop}
\begin{proof}
By \fref{prop:cTNuIsCarlesonOfHNu} we have
\[\{H_{h_{W(\mu)}}: \mu \in \mathcal{BM}^{\mathrm{fin}}\left(\R_+\right) \setminus \{0\}\} \subeq \{H \in \HanP : \left\lVert H\right\rVert \leq 1\}\]
and therefore, using \fref{lem:ContractiveHankelWOTClosed}, we get
\begin{align*}
\overline{\{H_{h_{W(\mu)}}: \mu \in \mathcal{BM}^{\mathrm{fin}}\left(\R_+\right) \setminus \{0\}\}}^{\mathrm{WOT}} &\subeq \overline{\{H \in \HanP : \left\lVert H\right\rVert \leq 1\}}^{\mathrm{WOT}}
\\&= \{H \in \HanP : \left\lVert H\right\rVert \leq 1\}.
\end{align*}
We now show the other inclusion. For every \(x \in \R_+\) and every \(f,g \in H^2(\C_+)\), we have
\[\braket*{f}{H_{\delta_x} g} = \int_{\R_+}\overline{f(i\lambda)}g(i\lambda)\,d\delta_x(\lambda) = \overline{f(ix)}g(ix) = \braket{f}{Q_{ix}}\braket*{Q_{ix}}{g}.\]
Then, by \fref{prop:H2RKHS}, we have
\[\left\lVert H_{\delta_x}\right\rVert = \left\lVert Q_{ix}\right\rVert^2 = \frac 1{4\pi x}.\]
Therefore, for the measure
\[\rho_x \coloneqq 4\pi x \cdot \delta_x,\]
we have \(\left\lVert H_{\rho_x}\right\rVert = 1\). Now, let \(H \in \HanP\) with \(\left\lVert H\right\rVert \leq 1\). Then, for every \(f \in H^2(\C_+)\), we have
\[\braket*{f}{(H+H_{\rho_x})f} = \braket*{f}{H f} + \braket*{f}{H_{\rho_x}f} \geq \braket*{f}{H_{\rho_x} f},\]
so
\[\left\lVert H+H_{\rho_x}\right\rVert \geq \left\lVert H_{\rho_x}\right\rVert = 1.\]
Therefore, for the function
\[F_x: [0,1] \to \R, \quad t \mapsto \left\lVert H+t \cdot H_{\rho_x}\right\rVert,\]
we have
\[F_x(0) = \left\lVert H\right\rVert \leq 1 \quad \text{and} \quad F_x(1) = \left\lVert H+H_{\rho_x}\right\rVert \geq 1,\]
so, by the Intermediate Value Theorem, there exists \(t_x \in [0,1]\) such that
\[1 = F_x(t_x) = \left\lVert H+t_x \cdot H_{\rho_x}\right\rVert.\]
This, by \fref{cor:strongDense}, implies
\[H+t_x \cdot H_{\rho_x} \in \overline{\{H_{h_{W(\mu)}}: \mu \in \mathcal{BM}^{\mathrm{fin}}\left(\R_+\right) \setminus \{0\}\}}^{\mathrm{WOT}}, \quad \forall x \in \R_+.\]
Further, for \(w,z \in \C_+\), we have
\begin{align*}
\braket*{Q_w}{H_{\rho_x}Q_z} &= 4\pi x \cdot \int_{\R_+}\overline{Q_w(i\lambda)}Q_z(i\lambda)\,d\delta_x(\lambda)
\\&= 4\pi x \cdot \overline{Q_w(ix)}Q_z(ix) = \frac x{\pi} \cdot \overline{\frac 1{ix - \overline{w}}}\frac 1{ix - \overline{z}} \xrightarrow{x \to 0} 0,
\end{align*}
so, by \fref{prop:topologiesEquiv}, we have
\[H_{\rho_x} \underset{\mathrm{WOT}}{\xrightarrow{x \to 0}} 0\]
and therefore
\[H + t_x \cdot H_{\rho_x} \underset{\mathrm{WOT}}{\xrightarrow{x \to 0}} H,\]
which shows that
\[H \in \overline{\{H_{h_{W(\mu)}}: \mu \in \mathcal{BM}^{\mathrm{fin}}\left(\R_+\right) \setminus \{0\}\}}^{\mathrm{WOT}}.\]
We therefore have
\[\{H \in \HanP : \left\lVert H\right\rVert \leq 1\} \subeq \overline{\{H_{h_{W(\mu)}}: \mu \in \mathcal{BM}^{\mathrm{fin}}\left(\R_+\right) \setminus \{0\}\}}^{\mathrm{WOT}}. \qedhere\]
\end{proof}

\subsection{Convergence in the weak-\(*\)-topology}
In this section, we want to better understand the set
\[\left\{H_{h_{\nu}}: \nu \in \mathcal{BM}^{\mathrm{fin}}\left([0,\infty]\right) \setminus \{0\}\right\}.\]
For this we have to know, given a measure \(\nu \in \mathcal{BM}^{\mathrm{fin}}\left([0,\infty]\right) \setminus \{0\}\) and a sequence \((\nu_n)_{n \in \N}\) in \(\mathcal{BM}^{\mathrm{fin}}\left([0,\infty]\right) \setminus \{0\}\) with
\[\nu_n \underset{w}{\xrightarrow{n \to \infty}} \nu,\]
what are the consequences for the measures \(\cT \nu\) and \((\cT \nu_n)_{n \in \N}\). We start by the following definition:
\begin{definition}
For a measure \(\nu \in \mathcal{BM}^{\mathrm{fin}}\left([0,\infty]\right) \setminus \{0\}\) we define the continuous function
\[\gls*{tnu}: \R_+ \to \R_+, \quad t_\nu(\lambda) = \left|F_\nu(i\lambda)\right|^{-2} \frac{1+\lambda^2}{\lambda}\]
with \(F_\nu\) as in \fref{def:FNuHNu}.
\end{definition}
\begin{remark}
By definition, for \(\nu \in \mathcal{BM}^{\mathrm{fin}}\left([0,\infty]\right) \setminus \{0\}\), we have
\[d(\cT\nu)(\lambda) = t_\nu(\lambda)\,d\nu(\lambda)\]
(cf. \fref{def:defCT}).
\end{remark}
\begin{lem}\label{lem:tNuUniformlyCompact}
Let \(\nu \in \mathcal{BM}^{\mathrm{fin}}\left([0,\infty]\right) \setminus \{0\}\) and \(\left(\nu_n\right)_{n \in \N}\) be a sequence in \(\mathcal{BM}^{\mathrm{fin}}\left([0,\infty]\right) \setminus \{0\}\) with
\[\nu_n \underset{w}{\xrightarrow{n \to \infty}} \nu.\]
Then the functions \((t_{\nu_n})_{n \in \N}\) converge to \(t_\nu\) uniformly on compact sets.
\end{lem}
\begin{proof}
By \fref{lem:OuterHomo}, we have
\[F_\nu^2 = \Out\left(\sqrt{\Psi_\nu}\right)^2 = \Out(\Psi_\nu)\]
and analogously \(F_{\nu_n}^2 = \Out(\Psi_{\nu_n})\) for \(n \in \N\), so, by \fref{lem:OuterSymmetric}, for every \(\lambda \in \R_+\), one has
\begin{equation*}
t_\nu(\lambda) = \frac{1+\lambda^2}{\lambda} \exp\left(-\frac{1}{\pi} \int_\R \frac{\lambda}{p^2+\lambda^2} \log\left(\Psi_\nu\left(p\right)\right) dp\right)
\end{equation*}
and
\begin{equation*}
t_{\nu_n}(\lambda) = \frac{1+\lambda^2}{\lambda} \exp\left(-\frac{1}{\pi} \int_\R \frac{\lambda}{p^2+\lambda^2} \log\left(\Psi_{\nu_n}\left(p\right)\right) dp\right).
\end{equation*}
Therefore, it suffices to show that
\[\int_\R \frac{\lambda}{p^2+\lambda^2} \log\left(\Psi_{\nu_n}\left(p\right)\right) dp \xrightarrow{n \to \infty} \int_\R \frac{\lambda}{p^2+\lambda^2} \log\left(\Psi_\nu\left(p\right)\right) dp\]
uniformly on \([a,b]\) with \(0<a<b<\infty\). Since, for every \(\lambda \in [a,b]\), one has
\begin{align*}
&\left|\int_\R \frac{\lambda}{p^2+\lambda^2} \log\left(\Psi_\nu\left(p\right)\right) dp - \int_\R \frac{\lambda}{p^2+\lambda^2} \log\left(\Psi_{\nu_n}\left(p\right)\right) dp\right|
\\&\qquad\qquad\qquad\qquad\leq \int_\R \frac{\lambda}{p^2+\lambda^2} \left|\log\left(\Psi_\nu\left(p\right)\right) - \log\left(\Psi_{\nu_n}\left(p\right)\right)\right| dp
\\&\qquad\qquad\qquad\qquad\leq \int_\R \frac{b}{p^2+a^2} \left|\log\left(\Psi_\nu\left(p\right)\right) - \log\left(\Psi_{\nu_n}\left(p\right)\right)\right| dp
\\&\qquad\qquad\qquad\qquad\leq b \cdot \max\left\{1,\frac 1{a^2}\right\} \cdot \int_\R \frac{1}{1+p^2} \left|\log\left(\Psi_\nu\left(p\right)\right) - \log\left(\Psi_{\nu_n}\left(p\right)\right)\right| dp,
\end{align*}
it remains to show that
\begin{equation*}
\int_\R \frac{1}{1+p^2} \left|\log\left(\Psi_\nu\left(p\right)\right) - \log\left(\Psi_{\nu_n}\left(p\right)\right)\right| dp \xrightarrow{n \to \infty} 0.
\end{equation*}
Since by assumption
\[\nu_n([0,\infty]) = \int_{[0,\infty]} \textbf{1} \,d\nu_n \xrightarrow{n \to \infty} \int_{[0,\infty]} \textbf{1} \,d\nu = \nu([0,\infty]) \neq 0,\]
we have that
\[M \coloneqq \sup_{n \in \N} \,\left|\log\left(\nu_n([0,\infty])\right)\right| < \infty.\]
Then, for \(n \in \N\) and \(p \in \R^\times\), by \fref{lem:PsiEstimate}, we have
\[\left|\log\left(\Psi_{\nu_n}\left(p\right)\right)\right| = \left|\log\left(\nu_n([0,\infty])\right)\right| + \left|\log\left(\frac{\Psi_{\nu_n}\left(p\right)}{\nu_n([0,\infty])}\right)\right| \leq M + \left|\log\left(p^2\right)\right|\]
and analogously
\[\left|\log\left(\Psi_\nu\left(p\right)\right)\right| \leq M + \left|\log\left(p^2\right)\right|.\]
For \(n \in \N\) and \(p \in \R^\times\) this yields
\[\frac{1}{1+p^2} \left|\log\left(\Psi_\nu\left(p\right)\right) - \log\left(\Psi_{\nu_n}\left(p\right)\right)\right| \leq \frac{\left|\log\left(\Psi_\nu\left(p\right)\right)\right| + \left|\log\left(\Psi_{\nu_n}\left(p\right)\right)\right|}{1+p^2} \leq 2 \cdot \frac{M + \left|\log\left(p^2\right)\right|}{1+p^2}\]
Then, since
\[\int_\R 2 \cdot \frac{M + \left|\log\left(p^2\right)\right|}{1+p^2} \,dp < \infty,\]
and since, by \fref{lem:PsiPointwise}, we have
\[\left|\log\left(\Psi_\nu\left(p\right)\right) - \log\left(\Psi_{\nu_n}\left(p\right)\right)\right| \xrightarrow{n \to \infty} 0, \quad \forall p \in \R^\times,\]
by the Dominated Convergence Theorem, we get
\begin{equation*}
\int_\R \frac{1}{1+p^2} \left|\log\left(\Psi_\nu\left(p\right)\right) - \log\left(\Psi_{\nu_n}\left(p\right)\right)\right| dp \xrightarrow{n \to \infty} 0. \qedhere
\end{equation*}
\end{proof}
\begin{prop}\label{prop:weakImpliesCcConvergence}
Let \(\nu \in \mathcal{BM}^{\mathrm{fin}}\left([0,\infty]\right) \setminus \{0\}\) and \(\left(\nu_n\right)_{n \in \N}\) be a sequence in \(\mathcal{BM}^{\mathrm{fin}}\left([0,\infty]\right) \setminus \{0\}\) with
\[\nu_n \underset{w}{\xrightarrow{n \to \infty}} \nu.\]
Then
\[\int_{\R_+} f \,d(\cT\nu_n) \xrightarrow{n \to \infty} \int_{\R_+} f \,d(\cT\nu) \quad \forall f \in C_c(\R_+,\C).\]
\end{prop}
\begin{proof}
Since by assumption
\[\nu_n([0,\infty]) = \int_{[0,\infty]} \textbf{1} \,d\nu_n \xrightarrow{n \to \infty} \int_{[0,\infty]} \textbf{1} \,d\nu = \nu([0,\infty]),\]
we have that
\[M \coloneqq \sup_{n \in \N} \,\nu_n([0,\infty]) < \infty.\]
Further, by \fref{lem:tNuUniformlyCompact}, the functions \((t_{\nu_n})_{n \in \N}\) converge uniformly to \(t_\nu\) on compact sets. Therefore, for every function \(f \in C_c(\R_+,\C)\) one has
\[\left\lVert f \cdot \left(t_{\nu_n} - t_\nu\right)\right\rVert_\infty \xrightarrow{n \to \infty} 0.\]
This yields
\begin{align*}
\left|\int_{\R_+} f \cdot \left(t_\nu - t_{\nu_n}\right) \,d\nu_n\right|\leq \left\lVert f \cdot \left(t_\nu - t_{\nu_n}\right)\right\rVert_\infty \cdot \nu_n([0,\infty]) = \left\lVert f \cdot \left(t_\nu - t_{\nu_n}\right)\right\rVert_\infty \cdot M \xrightarrow{n \to \infty} 0.
\end{align*}
Also we have \(f \cdot t_\nu \in C_c(\R_+,\C) \subeq C([0,\infty],\C)\), so
\[\left|\int_{\R_+} f \cdot t_\nu \,d\nu - \int_{\R_+} f \cdot t_\nu \,d\nu_n\right| \xrightarrow{n \to \infty} 0.\]
This yields
\begin{align*}
\left|\int_{\R_+} f \,d(\cT\nu) - \int_{\R_+} f \,d(\cT\nu_n)\right| &= \left|\int_{\R_+} f \cdot t_\nu \,d\nu - \int_{\R_+} f \cdot t_{\nu_n} \,d{\nu_n}\right|
\\&\leq \left|\int_{\R_+} f \cdot t_\nu \,d\nu - \int_{\R_+} f \cdot t_\nu \,d\nu_n\right| + \left|\int_{\R_+} f \cdot \left(t_\nu - t_{\nu_n}\right) \,d\nu_n\right|
\\&\xrightarrow{n \to \infty} 0 + 0 = 0. \qedhere
\end{align*}
\end{proof}
\begin{theorem}\label{thm:HankelHasUnimodMeasure}
One has
\[\{H \in \HanP : \left\lVert H\right\rVert \leq 1\} = \left\{H_{h_{\nu}}: \nu \in \mathcal{BM}^{\mathrm{fin}}\left([0,\infty]\right) \setminus \{0\}\right\}.\]
\end{theorem}
\begin{proof}
By \fref{prop:cTNuIsCarlesonOfHNu} we have
\[\left\{H_{h_\nu}: \nu \in \mathcal{BM}^{\mathrm{fin}}\left([0,\infty]\right) \setminus \{0\}\right\} \subeq \{H \in \HanP : \left\lVert H\right\rVert \leq 1\}.\]
We now show the other inclusion. So let \(H \in \HanP\) with \(\left\lVert H\right\rVert \leq 1\). Then, by \fref{prop:dDenseNormLeq1}, we have
\[H \in \overline{\{H_{h_{W(\mu)}}: \mu \in \mathcal{BM}^{\mathrm{fin}}\left(\R_+\right) \setminus \{0\}\}}^{\mathrm{WOT}}.\]
Since \(H^2(\C_+)\) is separable and therefore the weak operator topology restricted to the closed unit ball is metrizable, this implies that there exists a sequence of measures \((\mu_n)_{n \in \N}\) in \(\mathcal{BM}^{\mathrm{fin}}\left(\R_+\right) \setminus \{0\}\) such that
\[H_{h_{W(\mu_n)}} \underset{\mathrm{WOT}}{\xrightarrow{n \to \infty}} H.\]
For \(n \in \N\) we now set
\[\nu_n \coloneqq \frac 1{(W(\mu_n))([0,\infty])} W(\mu_n).\]
Since, by \fref{rem:scaleable}, we have
\begin{equation*}
\cT\left(\nu_n\right) = \cT\left(W(\mu_n)\right),
\end{equation*}
by \fref{prop:cTNuIsCarlesonOfHNu}, we get
\[H_{\cT \nu_n} = H_{\cT W(\mu_n)} = H_{h_{W(\mu_n)}} \underset{\mathrm{WOT}}{\xrightarrow{n \to \infty}} H.\]
Further, by construction, we have
\[\nu_n([0,\infty]) = 1 \quad \forall n \in \N.\]
The set
\[\{\nu \in \mathcal{BM}^{\mathrm{fin}}\left([0,\infty]\right): \nu([0,\infty]) \leq 1\}\]
is a closed subset of
\[B_1(0) = \{\nu \in \mathcal{BM}^{\mathrm{fin}}_\C\left([0,\infty]\right): |\nu|([0,\infty]) \leq 1\}\]
with respect to the weak-\(*\)-topology on
\[\mathcal{BM}^{\mathrm{fin}}_\C\left([0,\infty]\right) = C([0,\infty],\C)'\]
and therefore sequentially compact by the Banach--Alaoglu Theorem. Therefore, there exists a weakly convergent subsequence of \((\nu_n)_{n \in \N}\). By replacing \((\nu_n)_{n \in \N}\) with that subsequence, we can assume that
\[\nu_n \underset{w}{\xrightarrow{n \to \infty}} \nu\]
for some measure \(\nu \in \mathcal{BM}^{\mathrm{fin}}\left([0,\infty]\right)\). Then
\[1 = \nu_n([0,\infty]) = \int_{[0,\infty]} \textbf{1} \,d\nu_n \xrightarrow{n \to \infty} \int_{[0,\infty]} \textbf{1} \,d\nu = \nu([0,\infty]),\]
so \(\nu([0,\infty]) = 1\) and therefore \(\nu \neq 0\).
 Then, by \fref{prop:weakImpliesCcConvergence}, we have
\[\int_{\R_+} f \,d(\cT\nu_n) \xrightarrow{n \to \infty} \int_{\R_+} f \,d(\cT\nu) \quad \forall f \in C_c(\R_+,\C).\]
But, since
\[H_{\cT \nu_n} \underset{\mathrm{WOT}}{\xrightarrow{n \to \infty}} H,\]
we also have
\[\int_{\R_+} f \,d(\cT\nu_n) \xrightarrow{n \to \infty} \int_{\R_+} f \,d\mu_H \quad \forall f \in C_c(\R_+,\C)\]
by \fref{prop:topologiesEquiv}. This implies \(\mu_H = \cT \nu\) and therefore, by \fref{prop:cTNuIsCarlesonOfHNu}, we have
\[H = H_{\cT \nu} = H_{h_\nu}. \qedhere\]
\end{proof}
\begin{cor}
One has
\[\{\mu \in \CM : \left\lVert H_\mu\right\rVert \leq 1\} = \left\{\cT \nu: \nu \in \mathcal{BM}^{\mathrm{fin}}\left([0,\infty]\right) \setminus \{0\}\right\}.\]
\end{cor}
\begin{proof}
By \fref{prop:cTNuIsCarlesonOfHNu} we have
\[\left\{H_{h_{\nu}}: \nu \in \mathcal{BM}^{\mathrm{fin}}\left([0,\infty]\right) \setminus \{0\}\right\} = \left\{H_{\cT \nu}: \nu \in \mathcal{BM}^{\mathrm{fin}}\left([0,\infty]\right) \setminus \{0\}\right\}\]
The statement then follows by \fref{thm:HankelHasUnimodMeasure} and \fref{prop:HankelOneToOne}.
\end{proof}
\begin{example}\label{ex:Lebesgue}
For the measure \(\nu \in \mathcal{BM}^{\mathrm{fin}}\left(\left[0,\infty\right]\right)\) defined by
\[d\nu(\lambda) = \frac 2{1+\lambda^2} \,d\lambda\]
we get
\begin{align*}
\Psi_\nu\left(p\right) &= \frac{1}{\pi}\int_{[0,\infty]} \frac{1+\lambda^2}{p^2 + \lambda^2} \cdot \frac 2{1+\lambda^2} \,d\lambda = \frac{2}{\pi |p|}\int_{\R_+} \frac{1}{1 + \big(\frac \lambda{|p|}\big)^2} \cdot \frac 1{|p|} \,d\lambda
\\&= \frac{2}{\pi |p|}\int_{\R_+} \frac{1}{1 + x^2} \,dx = \frac{2}{\pi |p|} \cdot [\arctan(x)]_0^\infty  = \frac 1{\left|p\right|}.
\end{align*}
Then the outer function \(F_\nu = \Out\left(\sqrt{\Psi_\nu}\right)\), by \fref{lem:OuterExamples}, is given by
\[F_\nu(z) = \frac{1+i}{\sqrt{2z}}, \quad z \in \C_+,\]
where, for \(z \in \C_+\), by \(\sqrt{z}\) we denote the square root of \(z\) with \(\mathrm{Re}(\sqrt{z})>0\).
For \(x > 0\) this yields
\[h_\nu(x) = \frac{F_\nu(x)}{RF_\nu(x)} = \frac{\frac{1+i}{\sqrt{2x}}}{\frac{1+i}{\sqrt{-2x}}} = \frac {\sqrt{-x}}{\sqrt{x}} = i\]
and analogously we get \(h_\nu(x) = -i\) for every \(x<0\), so we have
\[h_\nu = i \cdot \sgn.\]
Further, for \(\lambda \in \R_+\), we have
\[F_\nu(i\lambda)^2 = \left(\frac{1+i}{\sqrt{2i\lambda}}\right)^2 = \frac{(1+i)^2}{2i\lambda} = \frac 1\lambda,\]
so
\[d(\cT\nu)(\lambda) = \left|F_\nu(i\lambda)\right|^{-2} \frac{1+\lambda^2}{\lambda} \,d\nu(\lambda) = \lambda \cdot \frac{1+\lambda^2}{\lambda} \cdot \frac 2{1+\lambda^2} \,d\lambda = 2 \,d\lambda\]
and therefore
\[\cT \nu = 2 \lambda_1.\]
\end{example}

\subsection{Uniqueness of the measure}
In \fref{sec:UniqueSymbol} we have seen, that, defining on \(\mathcal{BM}^{\mathrm{fin}}\left([0,\infty]\right) \setminus \{0\}\) the equivalence relation
\[\nu \sim \nu' \quad :\Leftrightarrow \quad (\exists c \in \R_+)\, \nu = c \cdot \nu',\]
one has
\[h_\nu = h_{\nu'} \quad \Leftrightarrow \quad \nu \sim \nu'.\]
This, of course, implies that
\[\nu \sim \nu' \quad \Rightarrow \quad H_{h_\nu} = H_{h_{\nu'}}\]
and therefore, we can define the map
\[\left(\mathcal{BM}^{\mathrm{fin}}\left([0,\infty]\right) \setminus \{0\}\right) \slash \sim \,\to \{H \in \HanP: \left\lVert H\right\rVert \leq 1\}, \quad [\nu] \mapsto H_{h_\nu}.\]
The goal of this section is to answer the question of whether it is possible to prove an analogue of \fref{cor:hNuInjektiveSim}, i.e. to answer the question of whether the above map is injective. The following lemma shows that, in this generality, this is not true:
\begin{lemma}
Let \(\nu \in \mathcal{BM}^{\mathrm{fin}}\left([0,\infty]\right) \setminus \{0\}\). Then the following are equivalent:
\begin{enumerate}[\rm (a)]
\item \(\cT \nu = 0\).
\item \(\nu = s \cdot \delta_0 + t \cdot \delta_\infty\) for some \(s,t \in \R_{\geq 0}\).
\end{enumerate}
\end{lemma}
\begin{proof}
That \(\cT \nu = 0\) is equivalent to
\[0 = d(\cT \nu)(\lambda) = \left|F_\nu(i\lambda)\right|^{-2} \frac{1+\lambda^2}{\lambda} \,d\nu(\lambda), \quad \lambda \in \R_+\]
and therefore to
\[\nu(\R_+) = 0,\]
since
\[\left|F_\nu(i\lambda)\right|^{-2} \frac{1+\lambda^2}{\lambda} > 0 \quad \forall \lambda \in \R_+.\]
This shows that for a measure \(\nu \in \mathcal{BM}^{\mathrm{fin}}\left(\left[0,\infty\right]\right) \setminus \{0\}\) one has
\[\cT \nu = 0 \quad \Leftrightarrow \quad \nu(\R_+)=0 \quad \Leftrightarrow \quad (\exists s,t \in \R_{\geq 0}:  \nu = s \cdot \delta_0 + t \cdot \delta_\infty). \qedhere\]
\end{proof}
\begin{example}
\begin{enumerate}[\rm (a)]
\item For the measure \(\nu = \pi \cdot \delta_0\), one gets
\[\Psi_\nu\left(p\right) = \int_{[0,\infty]} \frac{1+\lambda^2}{p^2 + \lambda^2} \cdot \,d\delta_0(\lambda) = \frac{1+0^2}{p^2 + 0^2} = \frac 1{p^2}.\]
Then, by \fref{lem:OuterExamples}, one has
\[F_\nu(z) = \frac iz, \quad z \in \C_+\]
and therefore
\[h_\nu(x) = \frac{F_\nu(x)}{RF_\nu(x)} = \frac{\frac ix}{\frac i{-x}} = -1, \quad x \in \R^\times\]
and indeed, one has \(H_{-\textbf{1}} = -p_+ R p_+^* = 0\).
\item For the measure \(\nu = \pi \cdot \delta_\infty\), one gets
\[\Psi_\nu\left(p\right) = \int_{[0,\infty]} \frac{1+\lambda^2}{p^2 + \lambda^2} \cdot \,d\delta_\infty(\lambda) = 1.\]
Then, one has \(F_\nu \equiv 1\) and therefore
\[h_\nu = \frac{F_\nu}{RF_\nu} = \frac{\textbf{1}}{R\textbf{1}} = \textbf{1}\]
and indeed, one has \(H_{\textbf{1}} = p_+ R p_+^* = 0\).
\end{enumerate}
\end{example}
This example shows that the map
\[\left(\mathcal{BM}^{\mathrm{fin}}\left([0,\infty]\right) \setminus \{0\}\right) \slash \sim \,\to \{H \in \HanP: \left\lVert H\right\rVert \leq 1\}, \quad [\nu] \mapsto H_{h_\nu}\]
is not injective, since we have \(\pi \cdot \delta_0 \not\sim \pi \cdot \delta_\infty\) but
\[H_{h_{\pi \cdot \delta_0}} = 0 = H_{h_{\pi \cdot \delta_\infty}}.\] However, we will now see that, when one restricts it to positive Hankel operators with a non-zero fixed point, it is, in fact, a bijection:
\begin{prop}
The map
\begin{align*}
W\left(\mathcal{BM}^{\mathrm{fin}}\left(\R_+\right) \setminus \{0\}\right) \slash \sim \,&\to \{H \in \HanP: \left\lVert H\right\rVert \leq 1, \,\ker(H - \textbf{1}) \neq \{0\}\}
\\ [W(\mu)] &\mapsto H_{h_{W(\mu)}}
\end{align*}
is a bijection.
\end{prop}
\begin{proof}
The surjectivity follows immediately by \fref{prop:HankelFixedPoint}. We now show the injectivity. So let \(\mu,\mu' \in \mathcal{BM}^{\mathrm{fin}}\left(\R_+\right) \setminus \{0\}\) with
\[H_{h_{W(\mu)}} = H_{h_{W(\mu')}}.\]
We set
\[f \coloneqq F_{W(\mu)} \in \ker\left(\theta_{h_{W(\mu)}}-{\bf 1}\right) \cap \Out^2(\C_+)\]
and
\[g \coloneqq F_{W(\mu')} \in \ker\left(\theta_{h_{W(\mu')}}-{\bf 1}\right) \cap \Out^2(\C_+)\]
(cf. \fref{def:FNuHNu}). Then, for \(t \geq 0\), we have
\begin{align*}
\braket*{S_t f}{\left(\theta_{h_{W(\mu)}}-{\bf 1}\right)g} &= \braket*{S_t f}{(H_{h_{W(\mu)}}-\textbf{1}) g}
\\&= \braket*{S_t f}{(H_{h_{W(\mu')}}-\textbf{1}) g} = \braket*{S_t f}{\left(\theta_{h_{W(\mu')}}-{\bf 1}\right)g} = 0
\end{align*}
and for \(t \leq 0\) we have
\begin{align*}
\braket*{S_t f}{\theta_{h_{W(\mu)}} g} &= \braket*{\theta_{h_{W(\mu)}} S_t f}{ g} = \braket*{S_{-t} \theta_{h_{W(\mu)}} f}{g} = \braket*{S_{-t} f}{g}
\\&= \braket*{S_{-t} f}{\theta_{h_{W(\mu')}} g} = \braket*{S_{-t} f}{H_{h_{W(\mu')}} g} = \braket*{S_{-t} f}{H_{h_{W(\mu)}} g}
\\&= \braket*{S_{-t} f}{\theta_{h_{W(\mu)}} g} = \braket*{\theta_{h_{W(\mu)}} S_{-t} f}{g} = \braket*{S_t \theta_{h_{W(\mu)}} f}{g} = \braket*{S_t f}{g}
\end{align*}
and therefore
\begin{align*}
\braket*{S_t f}{\left(\theta_{h_{W(\mu)}}-{\bf 1}\right)g} &= \braket*{S_t f}{\theta_{h_{W(\mu)}} g} - \braket*{S_t f}{g} = 0.
\end{align*}
This yields
\[\braket*{S_t f}{\left(\theta_{h_{W(\mu)}}-{\bf 1}\right)g} = 0 \quad \forall t \in \R.\]
Since, by \fref{cor:OuterDense}, we have \(\overline{\spann S(\R)f} = L^2(\R,\C)\), we get \(\left(\theta_{h_{W(\mu)}}-{\bf 1}\right)g = 0\). This implies \(\theta_{h_{W(\mu)}} g = g\) and therefore
\[h_{W(\mu)} = \frac g{Rg} = \frac{F_{W(\mu')}}{RF_{W(\mu')}} = h_{W(\mu')},\]
so, by \fref{cor:hNuInjektiveSim}, we get \([W(\mu)] = [W(\mu')]\).
\end{proof}

\newpage
\section{Application to standard subspaces}
In this chapter, we will investigate so-called standard subspaces and certain unitary one-parameter groups that are compatible with a standard subspace. We will relate these unitary one-parameter groups to real (reflection positive) ROPGs using Borchers' Theorem. To understand Borchers' Theorem, we need some basic theory of standard subspaces as well as basics of the representation theory of the affine group \(\mathrm{Aff}(\R)\).

\subsection{Standard subspaces}
We start with the definition and basic theory of standard subspaces:
\begin{definition}(cf. \cite[Def. 5.1.3]{NO18})
Let \(\cH\) be a complex Hilbert space.
\begin{enumerate}[\rm (a)]
\item A closed real subspace \(\gls*{V} \subeq \cH\) is called a \textit{standard subspace}, if
\[\sV \cap i\sV = \{0\} \qquad \text{and} \qquad \oline{\sV + i\sV} = \cH.\]
\item Given a standard subspace \(\sV\) we consider the densely defined \textit{Tomita operator} given by
\[\gls*{TV} : \sV+i\sV \to \cH, \quad x+iy \mapsto x-iy.\]
\item We write \(\gls*{JV}\) for the anti-unitary involution and \(\gls*{DeltaV}\) for the strictly positive, densely defined operator one gets from the polar decomposition
\[T_\sV = J_\sV \Delta_\sV^{\frac 12}\]
of the Tomita operator. We call \((\Delta_\sV,J_\sV)\) the pair of \textit{modular objects} of \(\sV\).
\end{enumerate}
\end{definition}
\begin{prop}\label{prop:StandardSubspaceFixed}{\rm (cf. \cite[Rem. 5.1.4]{NO18})}
Let \(\cH\) be a complex Hilbert space and \(\sV \subeq \cH\) be a standard subspace. Then the pair \((\Delta_\sV,J_\sV)\) of modular objects of \(\sV\) satisfies the relation
\[J_V \Delta_V = \Delta_V^{-1}J_V\]
and one has
\[\sV = \mathrm{Fix}\big(J_\sV \Delta_\sV^{\frac 12}\big)\]
Conversely, for every pair of a strictly positive selfadjoint operator \(\Delta\) and an anti-unitary involution \(J\) satisfying \(J\Delta = \Delta^{-1}J\), the space \(\sV \coloneqq \mathrm{Fix}\big(J \Delta^{\frac 12}\big)\) is a standard subspace.
\end{prop}
This proposition shows that, starting with a standard subspace \(\sV\) and considering its modular pair \((\Delta_\sV,J_\sV)\), one gets another modular pair by \((\Delta_\sV^{-1},J_\sV)\), which gives rise to another standard subspace. To get a more explicit form of that standard subspace, we make the following definition:
\begin{definition}
Let \(\cH\) be a Hilbert space and \(\sV \subeq \cH\) be a standard subspace. Then we set
\[\gls*{Vprime} \coloneqq (i\sV)^{\perp_\R},\]
where, by \(\perp_\R\), we denote the orthogonal complement with respect to the real scalar product
\[\braket*{\cdot}{\cdot}_\R: \cH \times \cH \to \R, \quad (v,w) \mapsto \Re \,\braket*{v}{w}.\]
\end{definition}
\begin{lemma}\label{lem:JSymplCompl}{\rm (cf. \cite[Lem. 3.7]{NO17})}
Let \(\cH\) be a Hilbert space and \(\sV \subeq \cH\) a standard subspace. Then \(\sV'\) is also a standard subspace and
\[J_{\sV'} = J_{\sV} \qquad \text{and} \qquad \Delta_{\sV'} = \Delta_{\sV}^{-1}.\]
\end{lemma}

\subsection{Positive energy representations of the affine group}
In this section, we provide some basic results of the representation theory of the affine group \(\mathrm{Aff}(\R)\). The affine group is an example of a so-called graded group, defined as follows:
\begin{definition}
A {\it graded group} is a pair \(\left(G,\epsilon_G\right)\) of a group \(G\) and a group homomorphism \(\epsilon_G:G \to \left\lbrace \pm 1 \right\rbrace \). We set \(G_{\pm}\coloneqq \epsilon_G^{-1}\left(\left\lbrace\pm 1\right\rbrace\right)\).
\end{definition}
\begin{example}
\begin{enumerate}[\rm (a)]
\item The affine group \(\gls*{Aff} = \R \rtimes \R^\times\) becomes a graded group with the grading
\[\gls*{Aff+} = \R \rtimes \R_+ \quad \text{and} \quad \mathrm{Aff}(\R)_- = \R \rtimes \R_-.\]
\item Given a complex Hilbert space \(\cH\) we consider the group \(\gls*{AUH}\) of unitary and anti-unitary operators on \(\cH\). This becomes a graded group with the grading
\[\mathrm{AU}\left(\cH\right)_+ = \U(\cH) \quad \text{and} \quad \mathrm{AU}\left(\cH\right)_- = \mathrm{AU}\left(\cH\right) \setminus \U(\cH).\]
\end{enumerate}
\end{example}
Now that we have the concept of a graded group, we want to define what a representation of a graded group is:
\begin{definition}
An {\it anti-unitary representation} of a graded group \((G,\epsilon_G)\) on a Hilbert space \(\cH\) is a morphism of graded groups \({\rho:G \to \mathrm{AU}\left(\cH\right)}\).
\end{definition}
We are especially interested in so-called positive energy representations of the affine group, which are defined as follows:
\begin{definition}\label{def:PosEnergyRep}
\begin{enumerate}[\rm (a)]
\item Given a unitary representation \({\rho:\mathrm{Aff}(\R)_+ \to \mathrm{U}\left(\cH\right)}\) of the \(ax+b\)-group \(\mathrm{Aff}(\R)_+\), we say that \(\rho\) is a \textit{positive energy representation}, if, for the one-parameter group
\[U: \R \to \U(\cH), \quad t \mapsto \rho(t,1),\]
one has \(\partial U \geq 0\). We call the representation \(\rho\) a \textit{strictly positive energy representation} if additionally \({\ker (\partial U) = \{0\}}\).
\item An anti-unitary representation \({\rho:\mathrm{Aff}(\R) \to \mathrm{AU}\left(\cH\right)}\) of the affine group is a (strictly) positive energy representation if the restriction \({\rho\big|_{\mathrm{Aff}(\R)_+}:\mathrm{Aff}(\R)_+ \to \mathrm{AU}\left(\cH\right)}\) is a (strictly) positive energy representation.
\end{enumerate}
\end{definition}
\begin{example}\label{ex:HK}
Let \(\cK\) be a real Hilbert space. We set
\[\gls*{HK} \coloneqq L^2(\R,\cK_\C)^\sharp.\]
Then \(\cH_\cK\) becomes a complex Hilbert space with the complex structure
\[I: \cH_\cK \to \cH_\cK, \quad f \mapsto i \cdot \sgn \cdot f.\]
Then the map \(\gls*{rhoK}: \mathrm{Aff}(\R) \to \mathrm{AU}(\cH_\cK)\) with
\[(\rho_\cK(t,1)f)(x) \coloneqq e^{itx} f(x), \quad (\rho_\cK(0,e^s)f)(x) \coloneqq e^{\frac s2}f(e^s x) \qquad s,t \in \R\]
and
\[(\rho_\cK(0,-1)f)(x) \coloneqq i \cdot \sgn(x) \cdot f(-x)\]
defines a strictly positive energy representation of the affine group \(\mathrm{Aff}(\R)\).
\end{example}
The following proposition shows that, in fact, every strictly positive energy representation of \(\mathrm{Aff}(\R)\) is of this form:
\begin{prop}\label{prop:RepAffineForm}
Let \((\rho,\cH)\) be a strictly positive energy representation of the affine group \(\mathrm{Aff}(\R)\). Then there exists a real Hilbert space \(\cK\) such that \((\rho,\cH)\) is equivalent to the anti-unitary representation \((\rho_\cK,\cH_\cK)\). Further, the representation \((\rho_\cK,\cH_\cK)\) is irreducible, if and only if \(\cK = \R\).
\end{prop}
\begin{proof}
By \cite[Thm. 1.6.1]{Lo08} there exists -- up to equivalence -- exactly one irreducible unitary representation of \(\mathrm{Aff}(\R)_+\) with strictly positive energy. By \cite[Sec. 2.4.1]{NO17} this representation \(U\) can be realized on the Hilbert space \(L^2(\R_+,\C)\) by
\[(U(t,1)f)(x) = e^{itx} f(x) \quad \text{and} \quad (U(0,e^s)f)(x) \coloneqq e^{\frac s2}f(e^s x), \qquad f \in L^2(\R_+,\C), s,t,x \in \R.\]
Via the map
\[\Psi: L^2\left(\R_+,\C\right) \to L^2\left(\R,\C\right)^\sharp, \quad (\Psi f)(x) \coloneqq \begin{cases} f(x) & \text{if } x > 0 \\[0.75ex] \overline{f(-x)} & \text{if } x < 0,\end{cases}\]
this representation is equivalent to the representation \(\rho_\R\big|_{\mathrm{Aff}(\R)_+}\).

Further, by \cite[Thm. 1.6.1]{Lo08}, every unitary representation of \(\mathrm{Aff}(\R)_+\) with strictly positive energy is a multiple of this unique irreducible representation, so every strictly positive energy representation \((\rho,\cH)\) of the group \(\mathrm{Aff}(\R)_+\) is equivalent to
\[\tilde \rho_\cK \coloneqq \rho_\cK\big|_{\mathrm{Aff}(\R)_+}: \mathrm{Aff}(\R)_+ \to \U(\cH)\]
for some real Hilbert space \(\cK\). Now, by \cite[Thm. 1.6.3]{Lo08}, for every unitary representation of \(\mathrm{Aff}(\R)_+\), there exists an -- up to equivalence -- unique extension to an anti-unitary representation of \(\mathrm{Aff}(\R)\) on the same Hilbert space. Since \(\rho_\cK\) is such an extension of \(\tilde \rho_\cK\), every anti-unitary representation of \(\mathrm{Aff}(\R)\) is equivalent to \(\rho_\cK\) for some real Hilbert space \(\cK\). Finally, by \mbox{\cite[Thm. 1.6.3]{Lo08}} the representation \(\rho_\cK\) is irreducible, if and only if \(\tilde \rho_\cK\) is irreducible, which is the case, if and only if \(\cK = \R\).
\end{proof}

\subsection{Extending Borchers' Theorem beyond standard pairs}
To state Borchers' Theorem, which links strictly positive energy representations of the affine group to standard subspaces, we need one more definition:
\begin{definition}
Let \(\cH\) be a Hilbert space, \(\sV \subeq \cH\) a standard subspace and \(U: \R \to \U(\cH)\) a strongly continuous one-parameter group. We call the pair \((\sV,U)\) a \textit{standard pair} on \(\cH\), if
\[U_t \sV \subeq \sV \quad \forall t \in \R_+ \qquad \text{and} \qquad \partial U \geq 0.\]
We say that a standard pair \((\sV,U)\) is non-degenerate if \(\ker (\partial U) = \{0\}\).
\end{definition}
\begin{thm}\label{thm:Borchers}{\rm \textbf{(Borchers' Theorem (one-particle))}(\cite[Thm. II.9]{Bo92}, \cite[Thm. 2.2.1]{Lo08})}
Let \(\cH\) be a Hilbert space and let \((\sV,U)\) be a non-degenerate standard pair on \(\cH\). Then the following commutation relations hold:
\[\Delta_\sV^{is} U_t \Delta_\sV^{-is} = U_{e^{-2\pi s} t}, \qquad J_\sV U_t J_\sV = U_{-t} \qquad \forall s,t \in \R.\]
Therefore
\[\rho(t,1) \coloneqq U_t, \quad \rho(0,e^s) \coloneqq \Delta_\sV^{-\frac{is}{2\pi}} \quad \text{and} \quad \rho(0,-1) \coloneqq J_\sV \qquad s,t \in \R\]
defines a strictly positive energy representation \(V\) of \(\mathrm{Aff}(\R)\) on \(\cH\).
\end{thm}
Borchers' Theorem shows us that for every non-degenerate standard pair, we get a strictly positive energy representation of the affine group \(\mathrm{Aff}(\R)\). The following proposition shows that conversely, starting with a strictly positive energy representation of the affine group, we also obtain a non-degenerate standard pair:
\begin{prop}\label{prop:WhatIsStandard}
For every real Hilbert space \(\cK\), the space \(\gls*{VK} \coloneqq H^2(\C_+,\cK_\C)^\sharp\) is a standard subspace in the Hilbert space \(\cH_\cK\) with
\[J_{\sV_\cK} = \rho_\cK(0,-1) \qquad \text{and} \qquad \Delta_{\sV_\cK}^{-\frac{is}{2\pi}} = \rho_\cK(0,e^s), \quad s \in \R.\]
Further, defining the one-parameter group
\[U_\cK: \R \to \U(\cH), \quad t \mapsto \rho_\cK(t,1),\]
the pair \((\sV_\cK,U_\cK)\) is a non-degenerate standard pair.
\end{prop}
\begin{proof}
We consider the set
\[\cS_\pi \coloneqq \{z \in \C: 0 < \Im(z) < \pi\}\]
and the corresponding Hardy space
\[H^2(\cS_\pi) \coloneqq \left\{f \in \cO(\cS_\pi): \sup_{y \in (0,\pi)} \int_\R |f(x+iy)|^2 \,dx < \infty \right\}.\]
By \cite[Sec. 4]{LL15} for every function \(f \in H^2(\cS_\pi)\) the lower and upper boundary values
\[f(x) \coloneqq \lim_{\epsilon \downarrow 0} f(x+i\epsilon) \quad \text{and} \quad f(x+i\pi) \coloneqq \lim_{\epsilon \downarrow 0} f(x+i(\pi-\epsilon))\]
exist for almost every \(x \in \R\) and the lower boundary values define a function in \(L^2(\R,\C)\). Now, by \cite[Lem. 4.1]{LL15}, the space
\[H \coloneqq \{f \in H^2(\cS_\pi): \overline{f(x+i\pi)} = f(x) \text{ for almost every }x \in \R\}\]
is a standard subspace in \(L^2(\R,\C)\) with
\[J_H f = \overline{f} \qquad \text{and} \qquad \left(\Delta_H^{-\frac{is}{2\pi}}f\right)(x) = f(x+s), \quad f \in L^2(\R,\C), s,x \in \R.\]
We now consider the unitary map
\[\Psi: L^2(\R,\C) \to L^2(\R,\C)^\sharp = \cH_\R, \quad \begin{cases} \frac {(1+i)}{2\sqrt{x}} f(\log(x)) & \text{if } x > 0 \\[1.5ex] \frac {(1+i)}{2\sqrt{-x}} \overline{f(\log(-x))} & \text{if } x < 0.\end{cases}\]
Then
\[\Psi(H) = H^2(\C_+)^\sharp = \sV_\R.\]
This shows that \(\sV_\R\) is a standard subspace in the Hilbert space \(\cH_\R\) with
\[J_{\sV_\R} = \Psi \circ J_H \circ \Psi^{-1} = \rho_\R(0,-1)\]
and
\[\Delta_{\sV_\R}^{-\frac{is}{2\pi}} = \Psi \circ \Delta_H^{-\frac{is}{2\pi}} \circ \Psi^{-1} = \rho_\R(0,e^s), \quad s \in \R.\]
The first statement then follows by taking direct sums. The second statement follows immediately by the fact that
\[\left(U_\cK\right)_t H^2(\C_+,\cK_\C)^\sharp = S_t H^2(\C_+,\cK_\C)^\sharp \subeq H^2(\C_+,\cK_\C)^\sharp \quad \forall t \in \R_+\]
(cf. \fref{thm:LaxPhillipsReal}) and that \(\rho_\cK\) is a strictly positive energy representation.
\end{proof}
\begin{cor}\label{cor:NormalformStandardPair}
Let \(\cH\) be a Hilbert space and let \((\sV,U)\) be a non-degenerate standard pair on \(\cH\). Then there exists a real Hilbert space \(\cK\) and a unitary operator \(\psi: \cH \to \cH_\cK\) such that
\[\psi(\sV) = \sV_\cK, \qquad \psi \circ J_\sV \circ \psi^{-1} = \theta_{i \cdot \sgn} \qquad \text{and} \qquad \psi \circ U_t \circ \psi^{-1} = S_t \quad \forall t \in \R\]
(cf. \fref{def:defTheta}, \fref{thm:LaxPhillipsReal}).
\end{cor}
\begin{proof}
By \fref{prop:RepAffineForm} there exists a real Hilbert space \(\cK\) and a unitary operator \(\psi: \cH \to \cH_\cK\) such that
\[\psi \circ J_\sV \circ \psi^{-1} = \rho_\cK(0,-1) = \theta_{i \cdot \sgn} \qquad \text{and} \qquad \psi \circ U_t \circ \psi^{-1} = \rho_\cK(t,1) = S_t \quad \forall t \in \R\]
and
\[\psi \circ \Delta_\sV^{-\frac{is}{2\pi}} \circ \psi^{-1} = \rho_\cK(0,e^s) \quad \forall s \in \R.\]
This, by \fref{prop:WhatIsStandard}, implies that \(\psi(\sV)\) and \(\sV_\cK\) are both standard subspaces in \(\cH_\cK\) with
\[J_{\psi(\sV)} = J_{\sV_\cK} \qquad \text{and} \qquad \Delta_{\psi(\sV)} = \Delta_{\sV_\cK}\]
and therefore, we have
\[\psi(\sV) = \Fix\left(J_{\psi(\sV)}\Delta_{\psi(\sV)}^{\frac 12}\right) = \Fix\left(J_{\sV_\cK}\Delta_{\sV_\cK}^{\frac 12}\right) = \sV_\cK. \qedhere\]
\end{proof}
This corollary provides a normal form for non-degenerate standard pairs. The goal of this section is to provide similar normal form theorems for pairs \((V,U)\) of a standard subspace \(\sV\) and a unitary one-parameter group \(U\) on \(\cH\) with \(U_t \sV \subeq \sV\) for all \(t \in \R_+\), but without the additional assumption that the infinitesimal generator of \(U\) is positive. We want to replace this assumption by the weaker assumption that -- as it was the case for non-degenerate standard pairs by \fref{cor:NormalformStandardPair} -- there exists an isomorphism between \((\cH,\sV,U)\) and \((\cH_\cK,\sV_\cK,S)\) for some real Hilbert space \(\cK\). This condition, by the Lax--Phillips Theorem (\fref{thm:LaxPhillipsReal}), is equivalent to \((\cH,\sV,U)\) being a real ROPG. This motivates the following definition:
\begin{definition}\label{def:RegPair}
Let \(\cH\) be a complex Hilbert space, \(\sV \subeq \cH\) a standard subspace, and \(U: \R \to \U(\cH)\) a strongly continuous one-parameter group.
\begin{enumerate}[\rm (a)]
\item We call \((\sV,U)\) a regular pair, if \((\cH,\sV,U)\) is a real ROPG.
\item We call \((\sV,U)\) a reflection positive regular pair, if \((\cH,\sV,U,J_\sV)\) is a real reflection positive ROPG.
\end{enumerate}
\end{definition}
\begin{remark}
It follows directly from the definition that every reflection positive regular pair \((\sV,U)\) is, in particular, a regular pair.
\end{remark}
First, we prove that both of these conditions are fulfilled if \((\sV,U)\) is a non-degenerate standard pair to show that we are, in fact, dealing with a generalization of non-degenerate standard pairs. For this, we need the following proposition:
\begin{prop}\label{prop:StandardSubspaceMaxPos}
Let \(\cH\) be a complex Hilbert space and \(\sV \subeq \cH\)  be a standard subspace. Then \((\cH,\sV,J_\sV)\) is a maximal real RPHS.
\end{prop}
\begin{proof}
Given a standard subspace \(\sV \subeq \cH\), by \fref{prop:StandardSubspaceFixed}, for every \(v \in \sV \setminus \{0\}\), one has
\[J_\sV \Delta_\sV^{\frac 12}v = v\]
and therefore \(\Delta_\sV^{\frac 12}v \neq 0\), so
\[\braket*{v}{J_\sV v} = \braket*{v}{J_\sV \big(J_\sV \Delta_\sV^{\frac 12}\big)v} = \braket*{v}{\Delta_\sV^{\frac 12} v} > 0,\]
hence the triple \((\cH,\sV,J_\sV)\) is a RPHS. By the same argument, using that \(J_\sV = J_{\sV'}\) by \fref{lem:JSymplCompl}, we have that
\[\braket*{v}{J_\sV v} = \braket*{v}{J_{\sV'} v} > 0\]
for every \(v \in \sV' \setminus \{0\}\). This yields
\[\braket*{iv}{J_\sV iv} = \braket*{iv}{-iJ_\sV v} = -\braket*{v}{J_\sV v} < 0\]
for every \(v \in \sV' \setminus \{0\}\). Since
\[i\sV' = \sV^{\perp_\R},\]
this yields that
\[\braket*{v}{J_\sV v} < 0 \quad \forall v \in \sV^{\perp_\R}.\]
We now assume that \(w \in \cH\) such that \((\cH,\sV+\R w,J_\sV)\) is a real RPHS. Then \(\sV+\R w = \sV+\R P_{\sV^{\perp_\R}} w\), so we can assume that \(w \in \sV^{\perp_\R}\), which by the above implies that \(w = 0\), so \((\cH,\sV,J_\sV)\) is a maximal RPHS.
\end{proof}
\begin{cor}\label{cor:standardPairIsRPRegPair}
Let \(\cH\) be a complex Hilbert space, \(\sV \subeq \cH\) a standard subspace, and \(U: \R \to \U(\cH)\) a strongly continuous one-parameter group. If \((\sV,U)\) is a non-degenerate standard pair, then \((\sV,U)\) is a reflection positive regular pair.
\end{cor}
\begin{proof}
By \fref{prop:StandardSubspaceMaxPos} the triple \((\cH,\sV,J_\sV)\) is a RPHS and by Borchers' Theorem (\fref{thm:Borchers}) we have \(J_\sV \circ U_t = U_{-t} \circ J_\sV\). Together with the assumption
\[U_t \sV \subeq \sV, \quad \forall t \in \R_+,\]
this implies that \((\cH,\sV,U,J_\sV)\) is a reflection positive unitary one-parameter group. Further, by \fref{cor:NormalformStandardPair} there exists a real Hilbert space \(\cK\) such that the triple \((\cH,\sV,U)\) is isomorphic to \((\cH_\cK,\sV_\cK,S)\), which is a real ROPG by the Lax--Phillips Theorem (\fref{thm:LaxPhillipsReal}). This implies that \((\cH,\sV,U,J_\sV)\) is a reflection positive ROPG.
\end{proof}
By the Lax--Phillips Theorem (\fref{thm:LaxPhillipsReal}) we know that, for every regular pair \((\sV,U)\) on a Hilbert space \(\cH\), the ROPG \((\cH,\sV,U)\) is equivalent to \((\cH_\cK,\sV_\cK,S)\) for some real Hilbert space \(\cK\). We now want to see which ROPGs \((\cH,\sV,U)\) correspond to the multiplicity-free case \(\cK = \R\). For this, we make the following definition:
\begin{definition}
Let \(\cH\) be a complex Hilbert space and let \(U: \R \to \U(\cH)\) be a strongly continuous one-parameter group. We say that \(U\) is \(\R\)-cyclic, if there exists \(v \in \cH\) such that
\[\overline{\spann_\R U(\R)v} = \cH.\]
\end{definition}
\begin{prop}\label{prop:MultiplicityRCyclic}
Let \(\cH\) be a complex Hilbert space and \((\sV,U)\) be a regular pair on \(\cH\). Then the multiplicity space of \((\cH,\sV,U)\) is isomorphic to \(\R\), if and only if \(U\) is \(\R\)-cyclic.
\end{prop}
\begin{proof}
If the multiplicity space of \((\cH,\sV,U)\) is isomorphic to \(\R\), by the Lax--Phillips Theorem (\fref{thm:LaxPhillipsReal}), we can without loss of generality assume that \((\cH,\sV,U) = (L^2(\R,\C)^\sharp,H^2(\C_+)^\sharp,S)\). Then, for the function
\[F(z) \coloneqq \frac{iz}{(z+i)^2}, \quad z \in \C_+,\]
by \fref{lem:OuterExamples}, \fref{lem:OuterSymmetric} and \fref{thm:OuterBetrag}, we have \(F \in \Out^2(\C_+)^\sharp\). This, by \fref{cor:OuterDense}, yields
\[L^2(\R,\C)^\sharp = \overline{\spann_\R S(\R)F}\]
and therefore \(S\) is \(\R\)-cyclic.

Now, for the converse, we assume that \((\cH,\sV,U)\) is a real ROPG such that \(U\) is \(\R\)-cyclic. Then, by the Lax--Phillips Theorem (\fref{thm:LaxPhillipsReal}), we can without loss of generality assume that \((\cH,\sV,U) = (L^2(\R,\cK_\C)^\sharp,H^2(\C_+,\cK_\C)^\sharp,S)\) for some real Hilbert space \(\cK\). The \(\R\)-cyclicity then implies that there exists \(F \in L^2(\R,\cK_\C)^\sharp\) such that
\[L^2(\R,\cK_\C)^\sharp = \overline{\spann_\R S(\R)F}.\]
This in particular implies that, for \(v \in \cK \setminus \{0\}\), setting \(F_v \coloneqq \braket*{v}{F(\cdot)} \in L^2(\R,\C)^\sharp\), we have
\[\overline{\spann_\R S(\R)F_v} \otimes v = \overline{\spann_\R S(\R)P_{\C v}F} = P_{\C v} \,\overline{\spann_\R S(\R)F} = P_{\C v} L^2(\R,\cK_\C)^\sharp = L^2(\R,\C)^\sharp \otimes v\]
and therefore
\[\overline{\spann_\R S(\R)F_v} = L^2(\R,\C)^\sharp,\]
so, by Wiener's Tauberian Theorem (\cite[Thm. I]{Wi32}), we have
\[0 \neq F_v(x) = \braket*{v}{F(x)}\]
for almost every \(x \in \R\). We now define the function \(f \in L^2(\R,\C)^\sharp\) by
\[f(x) = \left\lVert F(x)\right\rVert, \quad x \in \R.\]
Then also \(f(x) \neq 0\) for almost every \(x \in \R\), so, by Wiener's Tauberian Theorem (\cite[Thm. I]{Wi32}), we have
\[\overline{\spann_\R S(\R)f} = L^2(\R,\C)^\sharp.\]
We now consider the following densely defined linear operator
\[\phi: L^2(\R,\cK_\C)^\sharp \supseteq L^\infty(\R,\C)^\sharp \cdot F \to L^2(\R,\C)^\sharp, \quad g \cdot F \mapsto g \cdot f.\]
For \(g \in L^\infty(\R,\C)^\sharp\) we have
\[\left\lVert \phi(g \cdot F)\right\rVert^2 = \left\lVert g \cdot f\right\rVert^2 = \int_\R |g(x)|^2 \cdot |f(x)|^2 \,dx = \int_\R |g(x)|^2 \cdot \left\lVert F(x)\right\rVert^2 \,dx = \left\lVert g \cdot F\right\rVert^2,\]
so \(\phi\) is isometric, and therefore we can extend it to a unitary operator
\[\tilde \phi: L^2(\R,\cK_\C)^\sharp \to \overline{\spann_\R S(\R)f} = L^2(\R,\C)^\sharp.\]
By construction, we have
\[\phi \circ S_t = S_t \circ \phi \quad \forall t \in \R\]
and therefore also
\[\tilde \phi \circ S_t = S_t \circ \tilde \phi \quad \forall t \in \R.\]
So, setting
\[\cE_+ \coloneqq \tilde \phi(H^2(\C_+,\cK_\C)^\sharp),\]
the operator \(\tilde \phi\) is an isomorphism between the ROPG \((\cH,\sV,U) = (L^2(\R,\cK_\C)^\sharp,H^2(\C_+,\cK_\C)^\sharp,S)\) and the ROPG \((L^2(\R,\C)^\sharp,\cE_+,S)\). By \fref{thm:IndependentFromSubspace} the latter is equivalent to the real ROPG \((L^2(\R,\C)^\sharp,H^2(\C_+)^\sharp,S)\). Since, by \fref{thm:kernelAdjointGeneralReal}, the multiplicity space is unique up to isomorphism, this yields that \(\cK = \R\).
\end{proof}
We now summarize our results in the following three theorems that give a normal form for regular pairs, reflection positive regular pairs, and non-degenerate standard pairs respectively:
\begin{thm}\label{thm:StandardSubspaceROPG}
Let \(\cH\) be a complex Hilbert space and \((\sV,U)\) be a regular pair on \(\cH\). Then the following hold:
\begin{enumerate}[\rm (a)]
\item There exists a real Hilbert space \(\cK\) such that the real ROPG \((\cH,\sV,U)\) is equivalent to \((\cH_\cK,\sV_\cK,S)\).
\item Denoting by \(\phi\) the isomorphism between \((\cH,\sV,U)\) and \((\cH_\cK,\sV_\cK,S)\) and setting
\[\theta \coloneqq \phi \circ J_\sV \circ \phi^{-1} \in B(\cH_\cK),\]
the triple \((\cH_\cK,\sV_\cK,\theta)\) is a maximal RPHS.
\item One has \(\cK = \R\), if and only if \(U\) is \(\R\)-cyclic. 
\end{enumerate}
\end{thm}
\begin{proof}
Statement (a) follows immediately by the Lax--Phillips Theorem (\fref{thm:LaxPhillipsReal}), (b) follows by \fref{prop:StandardSubspaceMaxPos} and (c) by \fref{prop:MultiplicityRCyclic}.
\end{proof}
\begin{thm}\label{thm:StandardSubspaceRPROPG}
Let \(\cH\) be a complex Hilbert space and \((\sV,U)\) be a reflection positive regular pair on \(\cH\). Then the following hold:
\begin{enumerate}[\rm (a)]
\item There exists a real Hilbert space \(\cK\) and a function \(h \in L^\infty(\R,\U(\cK_\C))^\flat\) with \(h = h^\sharp\) such that the real reflection positive ROPG \((\cH,\sV,U,J)\) is equivalent to \((\cH_\cK,\sV_\cK,S,\theta_h)\).
\item The triple \((\cH_\cK,\sV_\cK,\theta_h)\) is a maximal RPHS.
\item One has \(\cK = \R\), if and only if \(U\) is \(\R\)-cyclic. In this case there exists an -- up to multiplication with a positive constant -- unique measure \({\nu \in \mathcal{BM}^{\mathrm{fin}}([0,\infty]) \setminus \{0\}}\) such that \(h = h_\nu\).
\end{enumerate}
\end{thm}
\begin{proof}
Statement (a) follows immediately by \fref{thm:NormalformRPROPG} and (b) follows by \fref{prop:StandardSubspaceMaxPos}. We now prove (c). By \fref{prop:MultiplicityRCyclic} we have that \(\cK = \R\), if and only if \(U\) is \(\R\)-cyclic. By (b) we have that \((\cH_\cK,\sV_\cK,\theta_h)\) is a maximal real RPHS, which by \fref{cor:EPlusMaximalReal} is equivalent to \((L^2(\R,\C),H^2(\C_+),S)\) being a maximal complex RPHS. By \fref{thm:GeneralizedMaxPosForm}, this, in turn, is equivalent to the existence of a measure \({\nu \in \mathcal{BM}^{\mathrm{fin}}([0,\infty]) \setminus \{0\}}\) such that \(h = h_\nu\). The uniqueness assertion about this measure then follows by \fref{cor:hNuInjektiveSim}.
\end{proof}
\begin{thm}\label{thm:StandardSubspaceSP}
Let \(\cH\) be a complex Hilbert space and \((\sV,U)\) be a non-degenerate standard pair on \(\cH\). Then the following hold:
\begin{enumerate}[\rm (a)]
\item There exists a real Hilbert space \(\cK\) such that the real reflection positive ROPG \((\cH,\sV,U,J)\) is equivalent to \((\cH_\cK,\sV_\cK,S,\theta_{i \cdot \sgn})\).
\item The triple \((\cH_\cK,\sV_\cK,\theta_{i \cdot \sgn})\) is a maximal RPHS.
\item One has \(\cK = \R\), if and only if \(U\) is \(\R\)-cyclic. In this case one has \(\theta_{i \cdot \sgn} = \theta_\nu\) for the measure \(\nu \in \mathcal{BM}^{\mathrm{fin}}([0,\infty]) \setminus \{0\}\) given by
\[d\nu(\lambda) = \frac 2{1+\lambda^2}, \quad \lambda \in [0,\infty].\]
\end{enumerate}
\end{thm}
\begin{proof}
Statement (a) follows immediately by \fref{cor:NormalformStandardPair}, (b) follows by \fref{prop:StandardSubspaceMaxPos} and (c) by \fref{prop:MultiplicityRCyclic} and \fref{ex:Lebesgue}.
\end{proof}
The results of these three theorems are presented in the following table:
\vspace{0.2cm}
\begin{center}
\addtolength\tabcolsep{1.3pt}
\begin{tabular}{|c||c|c|c|} 
\hline
\raisebox{30pt}{\phantom{M}}\boldmath\((\sV,U)\)\raisebox{-30pt}{\phantom{M}} & \bf regular pair & \(\substack{\text{\normalsize \bf reflection positive} \\ \text{\normalsize \bf regular pair}}\) & \(\substack{\text{\normalsize \bf non-degenerate} \\ \text{\normalsize \bf standard pair}}\) \\
\hline\hline
\raisebox{30pt}{\phantom{M}} \bf Group \boldmath\(G\)\raisebox{-30pt}{\phantom{M}} & \(\R\) & \(\R \rtimes \{-1,1\}\) & \(\mathrm{Aff}(\R) = \R \rtimes \R^\times\) \\
\hline
\raisebox{30pt}{\phantom{M}}\(\substack{\text{\normalsize \bf Representation} \\ \text{\normalsize \boldmath\(\rho: G \to \mathrm{AU}(\cH)\)}}\)\raisebox{-30pt}{\phantom{M}}  & \(\rho(t) = U_t\) & \(\substack{\text{\normalsize \(\rho(t,1) = U_t\)} \\ \text{\normalsize \(\rho(0,-1) = J_\sV\)} \\ \text{\normalsize (\fref{def:RPHS})}}\) & \(\substack{\text{\normalsize \(\rho(t,1) = U_t\)} \\ \text{\normalsize \(\rho(0,e^s) = \Delta^{-\frac{is}{2\pi}}\)} \\ \text{\normalsize \(\rho(0,-1) = J_\sV\)} \\ \text{\normalsize (\fref{thm:Borchers})}}\) \\
\hline
\raisebox{30pt}{\phantom{M}}\(\substack{\text{\normalsize \bf Normal form} \\ \text{\normalsize \bf for \boldmath\((\cH,\sV,U,J_\sV)\)}}\) \raisebox{-30pt}{\phantom{M}} & \(\substack{\text{\normalsize \((\cH_\cK,\sV_\cK,S,\theta)\) with} \\ \text{\normalsize \((\cH_\cK,\sV_\cK,\theta)\) being a} \\ \text{\normalsize maximal RPHS} \\ \text{\normalsize (\fref{thm:StandardSubspaceROPG})}}\) & \(\substack{\text{\normalsize \((\cH_\cK,\sV_\cK,S,\theta_h)\) with} \\ \text{\normalsize \(h = h^\sharp = h^\flat\) and}\\ \text{\normalsize \((\cH_\cK,\sV_\cK,\theta_h)\) being a} \\ \text{\normalsize maximal RPHS} \\ \text{\normalsize (\fref{thm:StandardSubspaceRPROPG})}}\) & \(\substack{\text{\normalsize \((\cH_\cK,\sV_\cK,S,\theta_{i \cdot \sgn})\)} \\ \text{\normalsize (\fref{thm:StandardSubspaceSP})}}\) \\
\hline
\raisebox{30pt}{\phantom{M}}\(\substack{\text{\normalsize \bf Normal form} \\ \text{\normalsize \bf for \boldmath\((\cH,\sV,U,J_\sV)\)}\\ \text{\normalsize \bf if \boldmath\(U\) is \boldmath\(\R\)-cyclic}}\) \raisebox{-30pt}{\phantom{M}} & \(\substack{\text{\normalsize \((\cH_\R,\sV_\R,S,\theta)\) with} \\ \text{\normalsize \((\cH_\R,\sV_\R,\theta)\) being a} \\ \text{\normalsize maximal RPHS} \\ \text{\normalsize (\fref{thm:StandardSubspaceROPG})}}\) & \(\substack{\text{\normalsize \((\cH_\R,\sV_\R,S,\theta_{h_\nu})\) with} \\ \text{\normalsize \(\nu \in \mathcal{BM}^{\mathrm{fin}}([0,\infty]) \setminus \{0\}\)} \\ \text{\normalsize (\fref{thm:StandardSubspaceRPROPG})}}\) & \(\substack{\text{\normalsize \((\cH_\R,\sV_\R,S,\theta_{i \cdot \sgn}) =\)} \\ \text{\normalsize \((\cH_\R,\sV_\R,S,\theta_\nu)\) with} \\ \text{\normalsize \(d\nu(\lambda) = \frac 2{1+\lambda^2}\),} \\ \text{\normalsize (\fref{thm:StandardSubspaceSP})}}\) \\
\hline
\raisebox{30pt}{\phantom{M}}\(\substack{\text{\normalsize \bf OS transform} \\ \text{\normalsize \bf of \boldmath\((\cH,\sV,J_\sV)\)}\\ \text{\normalsize \bf if \boldmath\(U\) is \boldmath\(\R\)-cyclic}}\)\raisebox{-30pt}{\phantom{M}} & \(\substack{\text{\normalsize\(\overline{\sV_\R \slash \cN_\theta}\)} \\ \text{\normalsize (\fref{def:OSTrafo})}}\) & \(\substack{\text{\normalsize\(L^2(\R_+,\C,\cT \nu)\)} \\ \text{\normalsize (\fref{cor:OSTrafoFromMeasures})}}\) & \(\substack{\text{\normalsize\(L^2(\R_+,\C,2\lambda_1)\)} \\ \text{\normalsize (\fref{cor:OSTrafoFromMeasures},} \\ \text{\normalsize \fref{ex:Lebesgue})}}\) \\
\hline
\end{tabular}
\end{center}
\vspace{0.5cm}

\subsection{Extending the Longo--Witten Theorem beyond standard pairs}
In \cite{LW11}, given a non-degenerate standard pair \((\sV,U)\) on a complex Hilbert space \(\cH\), the authors consider the semigroup
\[\cE(\sV,U) \coloneqq \{A \in \U(\cH): A\sV \subeq \sV, \,(\forall t \in \R)\, U_t A = A U_t\}.\]
The authors give the following theorem characterizing this semigroup as the semigroup of symmetric inner functions (cf. \fref{def:Inner}):
\begin{thm}\label{thm:LW}{\rm \textbf{(Longo-Witten)}(\cite[Thm. 2.6]{LW11})}
Let \(\cH\) be a complex Hilbert space and \((\sV,U)\) be a non-degenerate standard pair on \(\cH\). Further, let \(\cK\) be the multiplicity space of the real reflection positive ROPG \((\cH,\sV,U,J_\sV)\) and let \(\psi: \cH \to \cH_\cK\) be an isomorphism between \((\cH,\sV,U,J_\sV)\) and \((\cH_\cK,\sV_\cK,S,\theta_{i \cdot \sgn})\). Then
\[\cE(\sV,U) = \psi^{-1} \circ M(\Inn(\C_+,B(\cK_\C))^\sharp) \circ \psi.\]
\end{thm}
The goal of this section is to provide a generalization of this theorem from non-degenerate standard pairs to regular pairs. So, we make the following definition:
\begin{definition}
Let \(\cH\) be a complex Hilbert space and \((\sV,U)\) be a regular pair. We set
\[\gls*{EVU} \coloneqq \{A \in \U(\cH): A\sV \subeq \sV, \,(\forall t \in \R)\, U_t A = A U_t\}.\]
\end{definition}
\begin{thm}\label{thm:LWExtended}
Let \(\cH\) be a complex Hilbert space and \((\sV,U)\) be a regular pair on \(\cH\). Further, let \(\cK\) be the multiplicity space of the real ROPG \((\cH,\sV,U)\) and let \(\psi: \cH \to \cH_\cK\) be an isomorphism between \((\cH,\sV,U)\) and \((\cH_\cK,\sV_\cK,S)\). Then, considering the complex structure \(I\) on \(\cH_\cK\) defined by
\[If = \psi(i\psi^{-1}f),\]
there exists a function \(h \in L^\infty(\R,\U(\cK_\C))^\sharp\) such that
\[I = M_h\]
and one has
\[\cE(\sV,U) = \psi^{-1} \circ \left[M(\Inn(\C_+,B(\cK_\C))^\sharp) \cap \{M_h\}'\right] \circ \psi.\]
Further, in the case \(\dim \cK < \infty\), let \(\left(W_t\right)_{t \in \R_+} \subeq \cE(\sV,U)\) be a strongly continuous one-parameter semigroup. Then, setting
\[F \coloneqq \psi \circ \partial W \circ \psi^{-1} \in \mathrm{Pick}(\C_+,B(\cK_\C))_\R\]
(cf. \fref{thm:OPGOfInnerFunctions}) and extending \(W\) to a unitary one-parameter group \(\tilde W: \R \to \U(\cH)\), we have that \((\sV,\tilde W)\) is a regular pair, if and only if one of the following equivalent conditions hold:
\begin{enumerate}[\rm (a)]
\item \(\Spec(F(z)) \subeq \C_+\) for all \(z \in \C_+\).
\item \(\ker(x\textbf{1}-M_{F_*}) \cap H^2(\C_+,\cH) = \{0\}\) for all \(x \in \R\).
\item \(\R \setminus D(F,x)\) has non-zero Lebesgue measure for all \(x \in \R\).
\item \(D(F,x)\) is a Lebesgue zero-set for all \(x \in \R\).
\item \(D(F,M)\) is a Lebesgue zero-set for all Lebesgue zero-sets \(M \subeq \R\).
\end{enumerate}
(cf. \fref{def:DFM} for the definition of \(D(F,x)\) and \(D(F,M)\)). In this case, the multiplicity space \(\cM\) of \((\cH,\sV,\tilde W)\) is given by
\[\cM \cong \begin{cases} \R^{\deg F} & \text{if } F \in \mathrm{Rat}(\cK_\C) \\ \ell^2(\N,\R) & \text{if } F \notin \mathrm{Rat}(\cK_\C).\end{cases}\]
(cf. \fref{def:RationalPick}).
\end{thm}
\begin{proof}
By \fref{thm:StandardSubspaceROPG}, we can, without loss of generality, assume that \((\cH,\sV,U) = (\cH_\cK,\sV_\cK,S)\). Since we assumed \(U\) to be unitary, i.e. complex linear, we have \(S_tI=IS_t\), \(t \in \R\). This, by \fref{prop:SCommutant}, implies that \(I = M_h\) for some function \(h \in L^\infty(\R,B(\cK_\C))\). That \(h \in L^\infty(\R,\U(\cK_\C))^\sharp\) then follows by
\[M_h^* = I^* = I^{-1} = M_h^{-1},\]
using that \(I\) is a complex structure and the fact that
\[M_h \cH_\cK = I \cH_\cK = \cH_\cK.\]
To calculate the semigroup \(\cE(\sV,U)\), we notice that, by \fref{prop:SCommutant}, the commutation relation \(S_tA=AS_t\), \(t \in \R\), is equivalent to \(A \in M(L^\infty(\R,B(\cK_\C)))\). Then the condition \(A\sV_\cK \subeq \sV_\cK\), by \fref{prop:H2InclusionFunctions}, is equivalent to \(A \in M(H^\infty(\C_+,B(\cK_\C))^\sharp)\) and then finally \(A \in \U(\cH)\) is equivalent to \(A \in M(\Inn(\C_+,B(\cK_\C))^\sharp) \cap \{M_h\}'\).

The statement about the regularity of the pair \((\sV,\tilde W)\) follows by \fref{thm:SpecNonDeg} and \fref{thm:SpecNonDegZeroSet}. In this case, we get the formula for the multiplicity space \(\cM\) by \fref{thm:BigLWTheoremReal}.
\end{proof}
\begin{remark}
If \((\sV,U)\) is a non-degenerate standard pair, by \fref{ex:HK} and \fref{cor:NormalformStandardPair}, we have \(I = M_{i \cdot \sgn \cdot \textbf{1}}\) and therefore
\[M(\Inn(\C_+,B(\cK_\C))^\sharp) \cap \{M_{i \cdot \sgn \cdot \textbf{1}}\}' = M(\Inn(\C_+,B(\cK_\C))^\sharp).\]
So, in this case, \fref{thm:LWExtended} specializes to \fref{thm:LW}.
\end{remark}

\newpage
\appendix
\section{Holomorphic functions on the upper half-plane}

\subsection{Hardy spaces}\label{sec:HardySpaces}
In this section, we will introduce the concept of Hardy spaces and collect some basic facts about them. We will restrict ourselves mostly to Hardy spaces on the upper half-plane \(\C_+\), but everything can analogously also be done for the lower half-plane \(\C_-\). We start with the definition of a Hardy space:
\begin{definition}
For a function \(f \in \cO(\C_+)\), we set
\[\left\lVert f\right\rVert_{H^p} \coloneqq \sup_{y > 0} \left(\int_\R \left|f(x+iy)\right|^p \,dx\right)^{\frac 1p} \quad \text{for } 1 \leq p < \infty\]
and
\[\left\lVert f\right\rVert_{H^\infty} \coloneqq \sup_{z \in \C_+} \left|f(z)\right|.\]
Now, for \(1 \leq p \leq \infty\), we define the \textit{Hardy space} \(\gls*{Hp}\) by
\[H^p(\C_+) \coloneqq \{f \in \cO(\C_+) : \left\lVert f\right\rVert_{H^p} < \infty\}.\]
\end{definition}
\begin{theorem}{\rm (cf. \cite[Cor. 5.24, Sec. 5.29]{RR94})}
For \(1 \leq p \leq \infty\), the Hardy space \(H^p(\C_+)\) is a Banach space. Further, \(H^2(\C_+)\) becomes a Hilbert space when endowed with the scalar product
\[\braket*{f}{g}_{H^2} \coloneqq \lim_{y \downarrow 0} \int_\R \overline{f(x+iy)} g(x+iy) dx, \quad f,g \in H^2(\C_+,\cH).\]
\end{theorem}
\begin{theorem}\label{thm:HpBoundaryValues}{\rm (cf. \cite[Cor. 5.17, Thm. 5.19, Thm. 5.23, Sec. 5.29]{RR94})}
For \(1 \leq p \leq \infty\), for every \(f \in H^p(\C_+)\), the non-tangential limit
\[f_*(x) \coloneqq \lim_{\epsilon \downarrow 0} f(x+i\epsilon)\]
exists for almost every \(x \in \R\) and defines a function \(f_* \in L^p(\R,\C)\) with \(\left\lVert f_*\right\rVert_p = \left\lVert f\right\rVert_{H^p}\). In particular, for the scalar product on \(H^2(\C_+)\), one has
\[\braket*{f}{g}_{H^2} = \braket*{f_*}{g_*} \quad \forall f,g \in H^2(\C_+).\]
\end{theorem}
This theorem shows that we have an isometric embedding
\begin{equation*}
\iota_p: H^p\left(\C_+\right) \to L^p\left(\R,\C\right), \quad f \mapsto f_*.
\end{equation*}
From now on, we will identify \(H^p\left(\C_+\right)\) with \(\iota_p\left(H^p\left(\C_+\right)\right) \subeq L^p\left(\R,\C\right)\) and we will write \(f(x)\) instead of \(f_*(x)\) for \(f \in H^p\left(\C_+\right)\) and \(x \in \R\).
\begin{theorem}\label{thm:HardyAlmostEverywhereNonZero}{\rm (cf. \cite[Cor. 5.17]{RR94}, \cite[Sec. 5.14(v)]{RR94})}
Let \(1 \leq p \leq \infty\). Then, for every \({f \in H^p(\C_+,\cH) \setminus \{0\}}\), one has 
\[\int_{\R}\frac{|\log |f(x)||}{1+x^2} \,dx < \infty,\]
so especially
\[f(x) \neq 0\]
for almost every \(x \in \R\).
\end{theorem}
After considering Hardy spaces in general, we will now focus on the Hardy space \(H^2(\C_+)\). \(H^2(\C_+)\) as a Hilbert space is an example of a so-called reproducing kernel Hilbert space, defined as follows:
\begin{definition}
Let \(X\) be an arbitrary set and \(H\) be a Hilbert space that is a subset of the complex-valued functions on \(X\). We say that \(H\) is a \textit{reproducing kernel Hilbert space} if the linear functionals
\begin{equation*}
\ev_x: H \to \C, \quad f \mapsto f\left(x\right)
\end{equation*}
are continuous for every \(x \in X\). Then, by the Riesz Representation Theorem, for every \(x \in X\), there exists a function \(Q_x \in H\) such that
\begin{equation*}
f\left(x\right) = \ev_x\left(f\right) = \braket*{Q_x}{f}.
\end{equation*}
\end{definition}
\begin{lemma}\label{lem: evDense}
Let \(H\) be a reproducing kernel Hilbert space. Then
\begin{equation*}
\overline{\mathrm{span}\left\lbrace Q_x: x \in X\right\rbrace} = H.
\end{equation*}
\end{lemma}
\begin{proof}
Let \(f \in \left\lbrace Q_x: x \in X\right\rbrace^\perp\). Then for every \(x \in X\) we have
\begin{equation*}
0 = \braket*{Q_x}{f} = f\left(x\right)
\end{equation*}
and therefore \(f=0\).
\end{proof}
\begin{prop}\label{prop:H2RKHS}{\rm (cf. \cite[Ex. C.2]{ANS22})}
The Hilbert space \(H^2\left(\C_+\right)\) is a reproducing kernel Hilbert space, and its reproducing kernel functions are given by the Szegö kernel
\begin{equation*}
\gls*{Qw}: \C_+ \to \C, \quad z \mapsto \frac{1}{2\pi} \cdot \frac{i}{z-\overline{w}}, \qquad w \in \C_+.
\end{equation*}
Further, for \(w \in \C_+\), we have
\begin{equation*}
\left\lVert Q_w\right\rVert_2 = \frac 1{2\sqrt{\pi \mathrm{Im}(w)}}.
\end{equation*}
\end{prop}
\begin{proof}
Obviously, \(Q_w\) is holomorphic and we have
\begin{align*}
\left\lVert Q_w\right\rVert_2^2 &= \sup_{y>0} \int_\R \left|Q_w(x+iy)\right|^2 dx = \sup_{y>0} \frac 1{(2\pi)^2}\int_\R \frac 1{(x - \mathrm{Re}(w))^2+(y+\mathrm{Im}(w))^2} dx
\\&= \sup_{y>0} \frac 1{(2\pi)^2} \frac{\pi}{y+\mathrm{Im}(w)} = \frac 1{(2\pi)^2} \frac{\pi}{\mathrm{Im}(w)} = \frac{1}{4 \pi \mathrm{Im}(w)} < \infty,
\end{align*}
so \(Q_w \in H^2\left(\C_+\right)\) with
\[\left\lVert Q_w\right\rVert_2 = \frac 1{2\sqrt{\pi \mathrm{Im}(w)}}.\]
Further, by Cauchy's Integral Formula, we have
\begin{equation*}
\braket*{Q_w}{f} = \frac{1}{2 \pi i} \int_{\R} \frac{1}{p-w}f\left(p\right) dp = f\left(w\right),
\end{equation*}
which shows, that \(Q_w\) is the reproducing kernel function.
\end{proof}
\begin{cor}\label{prop:uniform}
Let \(f \in H^2\left(\C_+\right)\) and \(\left(f_n\right)_{n \in \N} \subseteq H^2\left(\C_+\right)\) such that \(\left\lVert f_n-f\right\rVert_2 \xrightarrow{n \to \infty} 0\). Then, for every \(\beta>0\), the sequence \(\left(f_n\right)_{n \in \N}\) converges uniformly to \(f\) on the set
\[I_\beta \coloneqq \left\lbrace z \in \C: \mathrm{Im}\left(z\right) \geq \beta\right\rbrace.\]
\end{cor}
\begin{proof}
By \fref{prop:H2RKHS}, we have
\begin{align*}
\sup_{z \in I_\beta}\left|f_n\left(z\right)-f\left(z\right)\right| &= \sup_{z \in I_\beta}\left|\braket*{Q_z}{f_n-f}\right| \leq \sup_{z \in I_\beta}\left\lVert Q_z\right\rVert_2 \cdot \left\lVert f_n-f\right\rVert_2
\\&= \sup_{z \in I_\beta}\frac{1}{2\sqrt{\pi \mathrm{Im}(z)}} \cdot \left\lVert f_n-f\right\rVert_2 = \frac{1}{2\sqrt{\pi \beta}} \cdot \left\lVert f_n-f\right\rVert_2 \xrightarrow{n \to \infty} 0
\end{align*}
and therefore the convergence is uniformly on \(I_\beta\).
\end{proof}
After considering Hardy spaces of scalar-valued functions, we now want to introduce Hardy spaces with values in a Hilbert space or in the bounded operators of some Hilbert space respectively:
\begin{definition}
Let \(\cH\) be a complex Hilbert space.
\begin{enumerate}[\rm (a)]
\item For an open subset \(U \subset \C\) we write \(\cO(U,\cH)\) for the set of holomorphic functions on \(U\) with values in \(\cH\) and \(\cO(U,B(\cH))\) for the set of holomorphic functions on \(U\) with values in \(B(\cH)\).
\item For a function \(f \in \cO(\C_+,\cH)\), we set
\[\left\lVert f\right\rVert_{H^2} \coloneqq \sup_{y > 0} \left(\int_\R \left\lVert f(x+iy)\right\rVert^2 \,dx\right)^{\frac 12}\]
and define
\[\gls*{H2H} \coloneqq \{f \in \cO(\C_+,\cH) : \left\lVert f\right\rVert_{H^2} < \infty\}.\]
\item For a function \(f \in \cO(\C_+,B(\cH))\), we set
\[\left\lVert f\right\rVert_{H^\infty} \coloneqq \sup_{z \in \C_+} \left\lVert f(z)\right\rVert\]
and define
\[\gls*{H8BH} \coloneqq \{f \in \cO(\C_+,B(\cH)) : \left\lVert f\right\rVert_{H^\infty} < \infty\}.\]
\end{enumerate}
\end{definition}
\begin{remark}
As before in the scalar-valued case, \(H^\infty(\C_+,B(\cH))\) is a Banach space, and \(H^2(\C_+,\cH)\) becomes a Hilbert space when endowed with the scalar product
\[\braket*{f}{g}_{H^2} \coloneqq \lim_{y \downarrow 0} \int_\R \braket*{f(x+iy)}{g(x+iy)}_\cH \,dx, \quad f,g \in H^2(\C_+,\cH)\]
denoting by \(\braket*{\cdot}{\cdot}_\cH\) the scalar product on \(\cH\).
\end{remark}
By \fref{thm:HpBoundaryValues}, for every \(f \in H^2(\C_+,\cH)\) and \(v \in \cH\), the limit
\[f_v(x) \coloneqq \lim_{\epsilon \downarrow 0} \braket*{v}{f(x+i\epsilon)}_\cH\]
exists for almost every \(x \in \R\). Analogously, for every \(f \in H^\infty(\C_+,B(\cH))\) and every trace-class operator \(T \in B_1(\cH)\), the limit
\[f_T(x) \coloneqq \lim_{\epsilon \downarrow 0} \mathrm{Tr}(T f(x+i\epsilon))\]
exists for almost every \(x \in \R\). This shows that, by taking weak limits, we can identify \(H^2(\C_+,\cH)\) with a subspace of \(L^2(\R,\cH)\) and \(H^\infty(\C_+,B(\cH))\) with a subspace of \(L^\infty(\R,B(\cH))\) respectively.
\begin{prop}\label{prop:HInftyRealBoundary}
Let \(f \in H^\infty(\C_+)\) with \(f(x) \in \R\) for almost every \(x \in \R\). Then \(f\) is constant.
\end{prop}
\begin{proof}
Since \(f \in H^\infty(\C_+)\) and \(\exp: \C \to \C\) maps bounded sets to bounded sets, we have
\[g_+ \coloneqq e^{if} \in H^\infty(\C_+) \quad \text{and} \quad g_- \coloneqq e^{-if} \in H^\infty(\C_+)\]
with
\[|g_+(x)| = |g_-(x)| = 1\]
for almost every \(x \in \R\). This implies that
\[|g_+(z)| \leq \left\lVert g_+\right\rVert_\infty = 1 \quad \text{and} \quad |g_-(z)| \leq \left\lVert g_-\right\rVert_\infty = 1 \qquad \forall z \in \C_+.\]
This yields
\[\Im(f(z)) \geq 0 \quad \text{and} \quad \Im(f(z)) \leq 0 \qquad \forall z \in \C_+\]
and therefore \(f(z) \in \R\) for every \(z \in \R\). This, by the Cauchy--Riemann equations, implies that \(f\) is constant.
\end{proof}
\begin{cor}\label{cor:HInftyIntersection}
Let \(\cH\) be a complex Hilbert space. Then
\[H^\infty(\C_+,B(\cH)) \cap H^\infty(\C_-,B(\cH)) = B(\cH) \cdot \textbf{1}.\]
\end{cor}
\begin{proof}
We first consider the scalar-valued case, so let \(f \in H^\infty(\C_+) \cap H^\infty(\C_-)\). Then \(f \in H^\infty(\C_+)\) and \(f \in H^\infty(\C_-)\), where the latter implies that \(\overline{f} \in H^\infty(\C_+)\). This yields
\[\Re(f),\Im(f) \in H^\infty(\C_+),\]
which, by \fref{prop:HInftyRealBoundary} implies that \(\Re(f)\) and \(\Im(f)\) are constant and therefore also \(f\) is constant, so \(f = c \cdot \textbf{1}\) for some \(c \in \C\).

We now consider the operator-valued case, so let \(f \in H^\infty(\C_+,B(\cH)) \cap H^\infty(\C_-,B(\cH))\). Then, for every \(v,w \in \cH\), we have
\[\braket*{v}{f(\cdot)w} \in H^\infty(\C_+) \cap H^\infty(\C_-) = \C \cdot \textbf{1}.\]
This implies that \(f \in B(\cH) \cdot \textbf{1}\).
\end{proof}
\begin{prop}\label{prop:H2InclusionFunctions}
Let \(\cH\) be a complex Hilbert space and let \(f \in L^\infty(\R,B(\cH))\). Then
\[M_fH^2(\C_+,\cH) \subeq H^2(\C_+,\cH),\]
if and only if \(f \in H^\infty(\C_+,B(\cH))\).
\end{prop}
\begin{proof}
Let \(f \in H^\infty(\C_+,B(\cH))\) and \(g \in H^2(\C_+,\cH)\). Then obviously \(M_f g = f \cdot g \in \cO(\C_+,\cH)\). Further
\begin{align*}
\left\lVert M_f g\right\rVert_2^2 &= \left\lVert f \cdot g\right\rVert_2^2 = \sup_{y > 0} \int_\R \left\lVert f(x+iy) g(x+iy)\right\rVert^2 \,dx \leq \sup_{y > 0} \int_\R \left\lVert f(x+iy)\right\rVert^2 \cdot \left\lVert g(x+iy)\right\rVert^2 \,dx
\\&\leq \sup_{y > 0} \int_\R \left\lVert f\right\rVert_\infty^2 \cdot \left\lVert g(x+iy)\right\rVert^2 \,dx = \left\lVert f\right\rVert_\infty^2 \cdot \left\lVert g\right\rVert_2^2,
\end{align*}
so \(M_f g \in H^2(\C_+,\cH)\). This implies that
\[M_fH^2(\C_+,\cH) \subeq H^2(\C_+,\cH).\]
Now, for the converse, we first consider the scalar-valued case. So let \(f \in L^\infty(\R,\C)\) with
\[M_fH^2(\C_+) \subeq H^2(\C_+).\]
This, in particular, implies
\[f \cdot Q_i \in H^2(\C_+).\]
Therefore, by \cite[Thm. 5.19(i)]{RR94}, we have
\[0 = \frac 1{2\pi i} \int_\R \frac{f(x) \cdot Q_i(x)}{x-\overline{z}} \,dx = \frac 1{4\pi^2} \int_\R \frac{f(x)}{(x-\overline{z})(x+i)} \,dx \quad \forall z \in \C_+,\]
which, by \cite[Sec. 3.1]{Ma09} implies that \(f \in H^\infty(\C_+)\).

We now consider the operator-valued case, so let \(f \in L^\infty(\R,B(\cH))\) with
\[M_fH^2(\C_+,\cH) \subeq H^2(\C_+,\cH).\]
Then, for every \(v,w \in \cH\), defining
\[f_{v,w} \coloneqq \braket*{v}{f(\cdot)w} \in L^\infty(\R,\C),\]
we get
\[M_{f_{v,w}}H^2(\C_+) = \braket*{v}{M_f (H^2(\C_+) \otimes w)} \subeq \braket*{v}{M_fH^2(\C_+,\cH)} \subeq \braket*{v}{H^2(\C_+,\cH)} = H^2(\C_+),\]
so
\[f_{v,w} \in H^\infty(\C_+).\]
This means that, for \(z \in \C_+\), we can define the map
\[F_z: \cH \times \cH \to \C, \quad (v,w) \mapsto f_{v,w}(z),\]
which is sesquilinear by construction. Further, for every \(z \in \C_+\), we have
\[|F_z(v,w)| = |f_{v,w}(z)| \leq \left\lVert f_{v,w}\right\rVert_\infty = \left\lVert \braket*{v}{f(\cdot)w}\right\rVert_\infty \leq \left\lVert f\right\rVert_\infty \cdot \left\lVert v\right\rVert \cdot \left\lVert w\right\rVert, \quad \forall v,w \in \cH,\]
so there exists \(A(z) \in B(\cH)\) with \(\left\lVert A(z) \right\rVert \leq \left\lVert f\right\rVert_\infty\) such that
\[F_z(v,w) = \braket*{v}{A(z) w}, \quad \forall v,w \in \cH.\]
Then
\[A: \C_+ \to B(\cH), \quad z \mapsto A(z)\]
is a bounded function for which all the functions
\[\braket*{v}{A(\cdot)w} = f_{v,w}, \quad v,w \in \cH\]
are holomorphic. This, by \cite[Cor. A.III.5]{Ne00}, implies that \(A\) is holomorphic and therefore \(A \in H^\infty(\C_+,B(\cH))\). Now, by construction, for every \(v,w \in \cH\) and almost every \(x \in \R\), we have
\[\braket*{v}{A(x)w} = f_{v,w}(x) = \braket*{v}{f(x)w},\]
which implies that \(f = A \in H^\infty(\C_+,B(\cH))\).
\end{proof}
\begin{cor}\label{cor:H2EqualFunctions}
Let \(\cH\) be a complex Hilbert space and let \(f,g \in L^\infty(\R,B(\cH))\) such that \(f(x),g(x) \in U(\cH)\) for almost all \(x \in \R\). Then, one has
\[M_f H^2(\C_+,\cH) = M_g H^2(\C_+,\cH),\]
if and only if there exists \(u \in \U(\cH)\) such that
\[f = g \cdot u.\]
\end{cor}
\begin{proof}
By multiplying with \(M_{g^{-1}}\) and \(M_{f^{-1}}\) respectively from the left, we get
\[M_{g^{-1}f} H^2(\C_+,\cH) = H^2(\C_+,\cH) \quad \text{and} \quad M_{f^{-1}g} H^2(\C_+,\cH) = H^2(\C_+,\cH).\]
By \fref{prop:H2InclusionFunctions} this implies \(g^{-1}f,f^{-1}g \in H^\infty(\C_+,B(\cH))\). On the other hand
\[g^{-1}f = (f^{-1}g)^{-1} = (f^{-1}g)^* \in H^\infty(\C_-,B(\cH)),\]
so, by \fref{cor:HInftyIntersection}, we get
\[g^{-1}f \in H^\infty(\C_+,B(\cH)) \cap H^\infty(\C_-,B(\cH)) = B(\cH) \cdot \textbf{1}\]
and therefore, there exists \(u \in B(\cH)\) such that
\[f = g \cdot u.\]
That \(u \in \U(\cH)\) follows from the fact that \(f\) and \(g\) have unitary values almost everywhere.
\end{proof}
We finish this section with some density theorems:
\begin{prop}\label{prop:etDense}
For \(t \in \R\) we consider the function \(e_t \in L^\infty(\R,\C)\) with
\[e_t(x) \coloneqq e^{itx}, \quad x \in \R.\]
Then one has
\[\overline{\spann \{e_t: t \in \R\}} = L^\infty(\R,\C) \qquad \text{and} \qquad \overline{\spann \{e_t: t \in \R_+\}} = H^\infty(\C_+),\]
where the closures are taken with respect to the weak-\(*\)-topology on \(L^\infty(\R,\C)\).
\end{prop}
\begin{proof}
Let \(f \in L^1(\R,\C)\) with
\[0 = \int_\R e_t(x) \cdot f(x) \,dx = \hat f(-t) \quad \forall t \in \R.\]
Then \(\hat f = 0\) and therefore \(f = 0\), since the Fourier transform is injective. This shows that
\[\overline{\spann \{e_t: t \in \R\}} = L^\infty(\R,\C).\]
The second statement follows by \cite[Lem. B.6(b)]{ANS22}.
\end{proof}
We need a similar statement to the one of \fref{prop:etDense} also for the Hilbert space \(L^2(\T,\C)\) and its corresponding Hardy space on the unit disc \(\D\). For this, we first have to introduce Hardy spaces on the disc:
\begin{definition}\label{def:HardySpaceDisc}
\begin{enumerate}[\rm (a)]
\item For a function \(f \in \cO(\D)\), we set
\[\left\lVert f\right\rVert_{H^p} \coloneqq \sup_{0 < r < 1} \left(\int_\T \left|f(rz)\right|^p \,dz\right)^{\frac 1p} \quad \text{for } 1 \leq p < \infty\]
and
\[\left\lVert f\right\rVert_{H^\infty} \coloneqq \sup_{z \in \D} \left|f(z)\right|.\]
Now, for \(1 \leq p \leq \infty\), we define the \textit{Hardy space} \(H^p(\D)\) by
\[\gls*{HpD} \coloneqq \{f \in \cO(\D) : \left\lVert f\right\rVert_{H^p} < \infty\}.\]
\item For a function \(f \in \cO(\C_+,\cH)\), we set
\[\left\lVert f\right\rVert_{H^2} \coloneqq \sup_{0 < r < 1} \left(\int_\T \left\lVert f(rz)\right\rVert^2 \,dz\right)^{\frac 12}\]
and define
\[H^2(\D,\cH) \coloneqq \{f \in \cO(\D,\cH) : \left\lVert f\right\rVert_{H^2} < \infty\}.\]
\item For a function \(f \in \cO(\D,B(\cH))\), we set
\[\left\lVert f\right\rVert_{H^\infty} \coloneqq \sup_{z \in \D} \left\lVert f(z)\right\rVert\]
and define
\[H^\infty(\D,B(\cH)) \coloneqq \{f \in \cO(\D,B(\cH)) : \left\lVert f\right\rVert_{H^\infty} < \infty\}.\]
\end{enumerate}
\end{definition}
\begin{rem}
Similar to the Hardy spaces on the upper half-plane \(\C_+\), by taking boundary values
\[f(z) \coloneqq \lim_{r \uparrow 1} f(rz), \quad z \in \T,\]
one gets an isometric embedding of \(H^p(\D)\) into the Banach space \({L^p(\T,\C) \coloneqq L^p(\T,\C,\mu)}\) (cf. \cite[Cor. 4.26]{RR94}), where by \(\mu\) we denote the unique Haar measure on the group \(\T\) normalized by \(\mu(\T) = 1\). In the same fashion one also has embeddings of \(H^2(\D,\cH)\) into \(L^2(\T,\cH)\) and of \(H^\infty(\D,B(\cH))\) into \(L^\infty(\T,B(\cH))\).
\end{rem}
\begin{prop}\label{prop:PolynomialsDense}
One has
\[\overline{\spann \{\Id_\T^n: n \in \Z\}} = L^\infty(\T,\C) \qquad \text{and} \qquad \overline{\spann \{\Id_\T^n: n \in \N_0\}} = H^\infty(\D),\]
where the closure is taken with respect to the weak-\(*\)-topology on \(L^\infty(\T,\C)\).
\end{prop}
\begin{proof}
Let \(f \in L^1(\T,\C)\) with
\[0 = \int_\R x^n \cdot f(x) \,dx = \hat f(-n) \quad \forall n \in \Z.\]
Then \(\hat f = 0\) and therefore \(f = 0\), since the Fourier transform is injective. This shows that
\[\overline{\spann \{\Id_\T^n: n \in \Z\}} = L^\infty(\T,\C).\]
The second statement follows by \cite[Lem. B.6(a)]{ANS22}.
\end{proof}

\subsection{Outer functions}
In this section, we introduce the concept of outer functions and present some basic facts about them, as well as about their interplay with Hardy spaces. We start with the definition of outer functions:
\begin{definition}\label{def:Outer}{(cf. \cite[Thm. 5.13]{RR94})}
\begin{enumerate}[\rm (a)]
\item For \(C \in \T\) and an almost everywhere defined function \(K: \R \to \R_{\geq 0}\) with
\begin{equation*}
\gls*{IK} \coloneqq \int_\R \frac{\left|\log\left(K\left(p\right)\right)\right|}{1 + p^2} \,dp < \infty
\end{equation*}
we define the function \(\Out(C,K) \in \cO(\C_+)\) by
\begin{equation*}
\Out(C,K)\left(z\right) = C \cdot \exp\left(\frac{1}{\pi i} \int_\R \left[\frac{1}{p-z} - \frac{p}{1 + p^2}\right] \log\left(K\left(p\right)\right) dp\right), \quad z \in \C_+
\end{equation*}
and call such a function an {\it outer function}.
\item We set \(\gls*{OutK} \coloneqq \Out(1,K)\).
\item We write \(\gls*{Out}\) for the set of outer functions on \(\C_+\) and set
\[\Out^2(\C_+) = \Out(\C_+) \cap H^2(\C_+).\]
\end{enumerate} 
\end{definition}
\begin{thm}\label{thm:OuterBetrag}{\rm (cf. \cite[Thm. 5.13]{RR94}, \cite[Thm. 17.16]{Ru86})}
Let \(C \in \T\) and \(K: \R \to \R_+\) with \(I(K) < \infty\). Then the limit
\[\Out(C,K)(x) \coloneqq \lim_{\epsilon \downarrow 0} \Out(C,K)(x+i\epsilon)\]
exists for almost all \(x \in \R\) and satisfies
\[|\Out(C,K)(x)| = K(x).\]
Further, one has \(\Out(C,K) \in H^p(\C_+)\), if and only if \(K \in L^p(\R,\C)\).
\end{thm}
\begin{theorem}\label{thm:outerSpan}{\rm (\cite[Thm. 17.23]{Ru86})}
Let \(F \in H^2(\C_+)\) and consider, for \(t \in \R\), the unitary operator \(S_t \coloneqq M_{e_t}\) with \(e_t(x) \coloneqq e^{itx}\), \(x \in \R\).
Then \(F \in \Out^2(\C_+)\), if and only if
\[H^2(\C_+) = \oline{\Spann S(\R_+) F}.\] 
\end{theorem}
\begin{cor}\label{cor:OuterDense}
For every outer function \(F \in \Out^2(\C_+)\) one has
\[\overline{\spann \,S(\R) F} = \overline{\left(H^\infty(\C_+) + H^\infty(\C_-)\right)F} = \overline{L^\infty(\R,\C)F} = L^2(\R,\C).\]
\end{cor}
\begin{proof}
By \fref{thm:outerSpan}, we have
\[\overline{\spann \,S(\R) F} \supeq \overline{\bigcup_{t \in \R} \overline{\spann \,S(\R_{> t}) F}} = \overline{\bigcup_{t \in \R} S_t \overline{\spann \,S(\R_+) F}} = \overline{\bigcup_{t \in \R} S_t H^2(\C_+)} = L^2(\R,\C),\]
using \fref{thm:LaxPhillipsComplex} in the last step.
Since
\[\spann \,\{e_t: t \in \R\} = \spann \,\{e_t: t \in \R_+\} + \spann \,\{e_t: t \in \R_{\leq 0}\} \subeq H^\infty(\C_+) + H^\infty(\C_-) \subeq L^\infty(\R,\C),\]
we get
\[L^2(\R,\C) \subeq \overline{\spann \,S(\R) F} \subeq  \overline{\left(H^\infty(\C_+) + H^\infty(\C_-)\right)F} \subeq \overline{L^\infty(\R,\C)F} \subeq L^2(\R,\C),\]
which implies the statement.
\end{proof}
We now show some basic properties of outer functions:
\begin{lemma}\label{lem:OuterHomo}
Let \(c \in \R_+\) and let \(f,g: \R \to \R_{\geq 0}\) with \(I(f) < \infty\) and \(I(g) < \infty\). Then
\[I(c \cdot f) < \infty, \quad I(f \cdot g) < \infty \quad \text{and} \quad I\left(\frac{f}{g}\right) < \infty\]
and
\[\Out(c \cdot f) = c \cdot \Out(f), \quad \Out(f \cdot g) =  \Out(f) \cdot \Out(g) \quad \text{and} \quad \Out\left(\frac{f}{g}\right) = \frac{\Out(f)}{\Out(g)}.\]
\end{lemma}
\begin{proof}
For \(p \in \R\) one has
\[\left|\log\left((c \cdot f)\left(p\right)\right)\right| = \left|\log\left(c\right) + \log\left(f\left(p\right)\right)\right| \leq \left|\log\left(c\right)\right| + \left|\log\left(f\left(p\right)\right)\right|\]
and
\[\left|\log\left((f \cdot g)\left(p\right)\right)\right| = \left|\log\left(f\left(p\right)\right) + \log\left(g\left(p\right)\right)\right| \leq \left|\log\left(f\left(p\right)\right)\right| + \left|\log\left(g\left(p\right)\right)\right|\]
and
\[\left|\log\left(\left(\frac fg\right)\left(p\right)\right)\right| = \left|\log\left(f\left(p\right)\right) - \log\left(g\left(p\right)\right)\right| \leq \left|\log\left(f\left(p\right)\right)\right| + \left|\log\left(g\left(p\right)\right)\right|,\]
which immediately implies
\[I(c \cdot f) < \infty, \quad I(f \cdot g) < \infty \quad \text{and} \quad I\left(\frac{f}{g}\right) < \infty.\]
The second statement then follows immediately from the definition of outer functions.
\end{proof}
\begin{lemma}\label{lem:OuterSymmetric}
Let \(K: \R \to \R_{\geq 0}\) be a symmetric function with \(I(K) < \infty\).
Then
\[\overline{\Out(K)(-\overline{z})} = \Out(K)(z) \quad \forall z \in \C_+.\]
Further
\[\Out(K)(i\lambda) = \exp\left(\frac{1}{\pi} \int_\R \frac{\lambda}{p^2+\lambda^2} \log\left(K\left(p\right)\right) dp\right) \in \R_+ \quad \forall \lambda \in \R_+.\]
\end{lemma}
\begin{proof}
For every \(z \in \C_+\), we have
\begin{align*}
\overline{\Out(K)(-\overline{z})} &= \overline{\exp\left(\frac{1}{\pi i} \int_\R \left[\frac{1}{p+\overline{z}} - \frac{p}{1 + p^2}\right] \log\left(K\left(p\right)\right) dp\right)}
\\&= \exp\left(\frac{-1}{\pi i} \int_\R \left[\frac{1}{p+z} - \frac{p}{1 + p^2}\right] \log\left(K\left(p\right)\right) dp\right)
\\&= \exp\left(\frac{-1}{\pi i} \int_\R \left[\frac{1}{-p+z} - \frac{-p}{1 + (-p)^2}\right] \log\left(K\left(-p\right)\right) dp\right)
\\&=\exp\left(\frac{1}{\pi i} \int_\R \left[\frac{1}{p-z} - \frac{p}{1 + p^2}\right] \log\left(K\left(p\right)\right) dp\right) = F_K(z).
\end{align*}
Further, since \(K\) is symmetric, also \(\sqrt{K}\) is symmetric and therefore also
\[\overline{\Out\big(\sqrt{K}\big)(-\overline{z})} = \Out\big(\sqrt{K}\big)(z) \quad \forall z \in \C_+.\]
This in particular yields
\[\Out\big(\sqrt{K}\big)(i\lambda) = \overline{\Out\big(\sqrt{K}\big)(-\overline{i\lambda})} = \overline{\Out\big(\sqrt{K}\big)(i\lambda)} \quad \forall \lambda \in \R_+.\]
Then, by \fref{lem:OuterHomo}, for \(\lambda \in \R_+\), we have
\begin{align*}
\Out(K)(i\lambda) &= \Out\big(\sqrt{K}\big)(i\lambda) \cdot \Out\big(\sqrt{K}\big)(i\lambda)
\\&=  \Out\big(\sqrt{K}\big)(i\lambda) \cdot \overline{\Out\big(\sqrt{K}\big)(i\lambda)} = \left|\Out\big(\sqrt{K}\big)(i\lambda)\right|^2 \geq 0,
\end{align*}
so \(\Out(K)(i\lambda) \in \R_{\geq 0}\). Further, for every \(\lambda \in \R_+\), we have \(\Out(K)(i\lambda) \in \exp(\C) = \C^\times\), so
\[\Out(K)(i\lambda) \in \R_+.\]
Therefore, for \(\lambda \in \R_+\), we get
\begin{align*}
\Out(K)(i\lambda) &= \left|\Out(K)(i\lambda)\right| = \left|\exp\left(\frac{1}{\pi i} \int_\R \left[\frac{1}{p-i\lambda} - \frac{p}{1 + p^2}\right] \log\left(K\left(p\right)\right) dp\right)\right|
\\&= \exp\left(\mathrm{Re}\left(\frac{1}{\pi i} \int_\R \left[\frac{1}{p-i\lambda} - \frac{p}{1 + p^2}\right] \log\left(K\left(p\right)\right) dp\right)\right)
\\&= \exp\left(\mathrm{Re}\left(\frac{1}{\pi i} \int_\R \left[\frac{p+i\lambda}{p^2+\lambda^2} - \frac{p}{1 + p^2}\right] \log\left(K\left(p\right)\right) dp\right)\right)
\\&= \exp\left(\frac{1}{\pi} \int_\R \frac{\lambda}{p^2+\lambda^2} \log\left(K\left(p\right)\right) dp\right). \qedhere
\end{align*}
\end{proof}
Finally, we provide some concrete examples of outer functions:
\begin{lemma}\label{lem:OuterExamples}
Let \(\sqrt{\cdot}: \C_+ \to \C\) denote the square root on \(\C_+\) with \(\sqrt{\C_+} \subeq \C_+\). Then, for every \(z \in \C_+\), given the function \(K: \R \to \R_{\geq 0}\) with
\begin{enumerate}[\rm (a)]
\item \(K(p) = |p|\), one has \(\Out(K)(z) = -iz\).
\item \(K(p) = \frac 1{|p|}\), one has \(\Out(K)(z) = \frac i{z}\).
\item \(K(p) = \frac 1{\sqrt{|p|}}\), one has \(\Out(K)(z) = \frac {i+1}{\sqrt{2z}}\).
\item \(K(p) = 1+p^2\), one has \(\Out(K)(z) = -(z+i)^2\).
\item \(K(p) = \frac {\left|p\right|}{1+p^2}\), one has \(\Out(K)(z) = \frac {iz}{(z+i)^2}\).
\item \(K(p) = \frac 1{\sqrt{1+p^2}}\), one has \(\Out(K)(z) = \frac i{z+i}\).
\end{enumerate}
\end{lemma}
\begin{proof}
\begin{enumerate}[\rm (a)]
\item We have
\begin{align*}
\int_{\R_+} \frac{\log\left(p\right)}{p^2+1} \,dp = \int_\R \frac{\log\left(e^x\right)}{e^{2x}+1} \cdot e^x \,dx = \int_\R \frac{x}{e^x+e^{-x}} \,dx = 0,
\end{align*}
where the last equality follows from the fact that the integrand is skew-symmetric.

Then, for \(\lambda \in \R_+\), by \fref{lem:OuterSymmetric}, we have
\begin{align*}
\Out(K)(i\lambda) &= \exp\left(\frac{1}{\pi} \int_\R \frac{\lambda}{p^2+\lambda^2} \log\left|p\right| dp\right) = \exp\left(\frac{2}{\pi} \int_{\R_+} \frac{\lambda}{p^2+\lambda^2} \log\left(p\right) dp\right)
\\&= \exp\left(\frac{2}{\pi} \int_{\R_+} \frac{\lambda^2}{(\lambda p)^2+\lambda^2} \log\left(\lambda p\right) dp\right) = \exp\left(\frac{2}{\pi} \int_{\R_+} \frac{1}{p^2+1} \log\left(\lambda p\right) dp\right)
\\&= \exp\left(\frac{2}{\pi} \int_{\R_+} \frac{1}{p^2+1} \log\left(\lambda\right) dp + \frac{2}{\pi} \int_{\R_+} \frac{1}{p^2+1} \log\left(p\right) dp\right)
\\&= \exp\left(\log\left(\lambda\right) + 0\right) = \lambda = -i(i\lambda).
\end{align*}
This shows that the holomorphic functions \(\Out(K)\) and \(z \mapsto -iz\) coincide on \(i\R_+\) and are therefore equal.
\item By \fref{lem:OuterHomo} and (a) we have
\[\Out(K)(z) = \frac 1{-iz} = \frac iz, \quad z \in \C_+.\]
\item By \fref{lem:OuterHomo} and (b), for \(z \in \C_+\), we have
\[\Out(K)(z) \cdot \Out(K)(z) = \frac iz,\]
so
\[\Out(K)(z) = \pm \frac{\frac {i+1}{\sqrt{2}}}{\sqrt{z}} = \pm \frac {i+1}{\sqrt{2z}}.\]
The statement then follows by \fref{lem:OuterSymmetric} since
\[+ \frac {i+1}{\sqrt{2i}} = 1 \in \R_+.\]
\item We define the function
\[g: (-1,\infty) \to \R, \quad \lambda \mapsto \int_{\R_+} \frac{1}{p^2+1} \left[\log\left(1+\lambda^2p^2\right) - 2\log\left(\lambda+1\right)\right] dp.\]
For \(\lambda \neq 1\), we have
\begin{align*}
g'(\lambda) &= \int_{\R_+} \frac{1}{p^2+1} \left[\frac{2\lambda p^2}{1+\lambda^2p^2} - \frac 2{\lambda+1}\right] dp
\\&= \int_{\R_+} \frac{2\lambda}{\lambda^2-1}\left[\frac{1}{p^2+1} - \frac{1}{1+\lambda^2p^2}\right] - \frac 2{\lambda+1}\frac{1}{p^2+1} \,dp
\\&= \left[\frac{2\lambda}{\lambda^2-1} - \frac{2}{\lambda+1}\right]\int_{\R_+} \frac{1}{p^2+1} \,dp - \frac{2}{\lambda^2-1} \int_{\R_+} \frac{\lambda}{1+\lambda^2p^2} \,dp
\\&=\frac{2}{\lambda^2-1} \int_{\R_+} \frac{1}{p^2+1} \,dp - \frac{2}{\lambda^2-1} \int_{\R_+} \frac{1}{1+p^2} \,dp = 0.
\end{align*}
Further
\begin{align*}
g'(1) &= \int_{\R_+} \frac{1}{p^2+1} \left[\frac{2p^2}{1+p^2} - \frac 2{1+1}\right] dp = \int_{\R_+} \frac{p^2-1}{(p^2+1)^2} \,dp
\\&= \int_\R \frac{e^{2x}-1}{(e^{2x}+1)^2} \cdot e^x \,dx = \int_\R \frac{e^x-e^{-x}}{(e^x+e^{-x})^2} \,dx = 0,
\end{align*}
where the last equality follows from the fact that the integrand is skew-symmetric. This shows that \(g' \equiv 0\) and therefore \(g\) is constant. Since
\begin{align*}
g(0) = \int_{\R_+} \frac{1}{p^2+1} \left[\log\left(1+0\right) - 2\log\left(0+1\right)\right] dp = 0
\end{align*}
this implies that \(g \equiv 0\). This yields that, for \(\lambda \in \R_+\), we have
\[0 = g(\lambda) = \int_{\R_+} \frac{1}{p^2+1} \left[\log\left(1+\lambda^2p^2\right) - 2\log\left(\lambda+1\right)\right] dp,\]
so
\[\int_{\R_+} \frac{1}{p^2+1} \log\left(1+\lambda^2p^2\right) dp = 2\log\left(\lambda+1\right) \int_{\R_+} \frac{1}{p^2+1} dp = \pi \log\left(\lambda+1\right)\]
and therefore
\begin{align*}
\int_\R \frac{\lambda}{p^2+\lambda^2} \log\left(1+p^2\right) dp &= 2\int_{\R_+} \frac{\lambda}{p^2+\lambda^2} \log\left(1+p^2\right) dp
\\&= 2\int_{\R_+} \frac{\lambda^2}{(\lambda p)^2+\lambda^2} \log\left(1+(\lambda p)^2\right) dp
\\&= 2 \int_{\R_+} \frac{1}{p^2+1} \log\left(1+\lambda^2p^2\right) dp = 2\pi \log\left(\lambda+1\right).
\end{align*}
This, by \fref{lem:OuterSymmetric}, yields
\begin{align*}
\Out(K)(i\lambda) &= \exp\left(\frac{1}{\pi} \int_\R \frac{\lambda}{p^2+\lambda^2} \log\left(1+p^2\right) dp\right)
\\&= \exp\left(2 \log\left(\lambda+1\right)\right) = (\lambda+1)^2 = -(i\lambda+i)^2, \quad \lambda \in \R_+.
\end{align*}
This shows that the holomorphic functions \(\Out(K)\) and \(z \mapsto -(z+i)^2\) coincide on \(i\R_+\) and are therefore equal.
\item By (a), (d) and \fref{lem:OuterHomo} we have
\[\Out(K)(z) = \frac{-iz}{-(z+i)^2} = \frac {iz}{(z+i)^2}, \quad z \in \C_+.\]
\item By \fref{lem:OuterHomo} and (d), for \(z \in \C_+\), we have
\[\Out(K)(z) \cdot \Out(K)(z) = \frac 1{-(z+i)^2},\]
so
\[\Out(K)(z) = \pm \frac i{z+i}.\]
The statement then follows by \fref{lem:OuterSymmetric} since
\[+ \frac {i}{i+i} = \frac 12 \in \R_+. \qedhere\]
\end{enumerate}
\end{proof}

\subsection{Inner functions}
In this section, we will introduce the concept of inner functions and present some of their basic properties: 
\begin{definition}\label{def:Inner}
Let \(\cH\) be a complex Hilbert space.
\begin{enumerate}[\rm (a)]
\item We call a function \(\phi \in H^\infty(\C_+,B(\cH))\) an \textit{inner fuction}, if \(\phi(x) \in \U(\cH)\) for almost every \(x \in \R\).
\item We write \(\gls*{Inn}\) for the set of inner functions.
\item We set \(\Inn(\C_+) \coloneqq \Inn(\C_+,B(\C))\), i.e. \(\Inn(\C_+)\) is the set of functions \(\phi \in H^\infty(\C_+)\) with \(\left|\phi(x)\right| = 1\) for almost every \(x \in \R\).
\end{enumerate}
\end{definition}
The following theorem links Hardy spaces, outer functions, and inner functions:
\begin{theorem}\label{thm:OuterInnerDecomp}{\rm (\cite[Sec. 5.14(iv), Cor. 5.17, Lem. 5.21]{RR94})}
Let \(f \in H^2(\C_+) \setminus \{0\}\). Then there exists an inner function \(\phi \in \Inn(\C_+)\) and an outer function \(F \in \Out^2(\C_+)\) such that
\[f = \phi \cdot F.\]
\end{theorem}
We now consider a special class of inner functions, namely the so-called Blaschke products:
\begin{definition}\label{def:Blaschke}
For \(\omega \in \C_+\), we define the \textit{Blaschke factor} \(\gls*{phio} \in \Inn(\C_+)\) by
\[\phi_\omega(z) \coloneqq \frac{z-\omega}{z-\overline{\omega}}, \quad z \in \C_+.\]
A \textit{finite Blaschke product} is a function \(\phi \in \Inn(\C_+)\) of the form
\[\phi = c \cdot \prod_{n = 1}^N \phi_{\omega_n}\]
with \(c \in \T\), \(N \in \N_0\) and \(\omega_1,\dots,\omega_N \in \C_+\).
\end{definition}
\begin{remark}
By construction the zeros of a finite Blaschke product \(\phi = c \cdot \prod_{n = 1}^N \phi_{\omega_n}\) are precisely the numbers \(\omega_1,\dots,\omega_N\) counted with multiplicity.
\end{remark}
\begin{prop}\label{prop:BlaschkeFactorization}
Let \(\omega \in \C_+\) and \(f \in H^2(\C_+)\) with \(f(\omega) = 0\). Then there exists \(g \in H^2(\C_+)\) such that
\[f = \phi_\omega \cdot g.\]
\end{prop}
\begin{proof}
If \(f = 0\) the statement follows by choosing \(g = 0\). Therefore, we can assume that \({f \neq 0}\). Then, by \fref{thm:OuterInnerDecomp}, there exists an inner function \(\phi \in \Inn(\C_+)\) and an outer function \({F \in \Out^2(\C_+)}\) such that
\[f = \phi \cdot F.\]
Since \(F(\C_+) \subeq \exp(\C) = \C^\times\), we have \(F(\omega) \neq 0\) and therefore \(\phi(\omega) = 0\). This implies that there exists a function \(m \in \cO(\C_+)\) such that \(\phi(z) = (z-\omega) \cdot m(z)\), \(z \in \C_+\). Then, defining \(h \in \cO(\C_+)\) by \(h(z) \coloneqq (z-\overline{\omega}) \cdot m(z)\), \(z \in \C_+\), we have
\[\phi = \phi_\omega \cdot h.\]
Since \(h\) is holomorphic and therefore continuous, for \(\epsilon \in (0,\Im(\omega))\), we have
\[\sup_{z \in U_{\leq \epsilon}(\omega)} |h(z)| < \infty.\]
Further, for every \(z \in \C_+ \setminus U_{\leq \epsilon}(\omega)\), one has
\begin{align*}
|h(z)| &= \left|\frac{z-\overline{\omega}}{z-\omega} \cdot \phi(z)\right| \leq \left|\frac{z-\overline{\omega}}{z-\omega}\right| \cdot \left\lVert \phi\right\rVert_\infty = \left|1 + \frac{\omega-\overline{\omega}}{z-\omega}\right| \cdot 1 \leq 1 + \frac{|\omega-\overline{\omega}|}{|z-\omega|} \leq 1 + \frac{2 \Im(\omega)}{\epsilon},
\end{align*}
so also
\[\sup_{z \in \C_+ \setminus U_{\leq \epsilon}(\omega)} |h(z)| < \infty.\]
This shows that \(h \in H^\infty(\C_+)\). Then, setting
\[g \coloneqq h \cdot F \in H^\infty(\C_+) \cdot \Out^2(\C_+) \subeq H^2(\C_+),\]
we get
\[f = \phi \cdot F = \phi_\omega \cdot h \cdot F = \phi_\omega \cdot g. \qedhere\]
\end{proof}

\subsection{Pick functions}
After considering inner functions, in this section, we want to take a closer look at the infinitesimal generators of one-parameter semigroups
\[\phi: \R_{\geq 0} \to \Inn(\C_+,B(\cH)).\]
We will see that these infinitesimal generators are precisely operator-valued Pick functions, defined as follows:
\begin{definition}\label{def:Pick}
Let \(\cH\) be a finite-dimensional complex Hilbert space.
\begin{enumerate}[\rm (a)]
\item We call a function \({F \in \cO(\C_+,B(\cH))}\) an \textit{operator-valued Pick function} if
\[\Im(F(z)) \coloneqq \frac 1{2i} (F(z)-F(z)^*) \geq 0 \quad \forall z \in \C_+.\]
\item We write \(\gls*{Pick}\) for the set of operator-valued Pick functions.
\item We write \(\gls*{PickR}\) for the set of Pick functions \(F \in \mathrm{Pick}(\C_+,B(\cH))\), for which the limit
\[\gls*{Fstar}(x) \coloneqq \lim_{\epsilon \downarrow 0} F(x+i\epsilon)\]
exists for almost every \(x \in \R\) and is self-adjoint. We say such a Pick function has self-adjoint boundary values.
\item We set
\[\mathrm{Pick}(\C_+) \coloneqq \mathrm{Pick}(\C_+,B(\C)) \quad \text{and} \quad \mathrm{Pick}(\C_+)_\R \coloneqq \mathrm{Pick}(\C_+,B(\C))_\R\]
and, in the latter case, we say these functions have real boundary values.
\end{enumerate}
\end{definition}
To better understand this definition of Pick functions, we first want to see some consequences of the condition \(\Im(F(z)) \geq 0\) and what happens if \(0\) is contained in the spectrum of \(\Im(F(z))\):
\begin{prop}\label{prop:SpecImPos}
Let \(\cH\) be a finite-dimensional Hilbert space and let \(A \in B(\cH)\) such that \({\Im(A) \geq 0}\). Then the following hold:
\begin{enumerate}[\rm (a)]
\item \(\Spec(A) \subeq \overline{\C_+}\).
\item If \(x \in \Spec(A) \cap \R\) and \(v \in \cH \setminus \{0\}\) with \(Av = xv\), then \(\Im(A)v=0\) and \(A^*v = xv\).
\end{enumerate}
\end{prop}
\begin{proof}
\begin{enumerate}[(a)]
\item For every eigenvalue \(\lambda \in \C\) of \(F(z)\) and every normalized eigenvector \(v \in \cH\) to this eigenvalue \(\lambda\), we have
\[\lambda = \braket*{v}{\lambda v} = \braket*{v}{A v} = \braket*{v}{\frac 12(A+A^*) v} + i \braket*{v}{\frac 1{2i}(A-A^*) v} \in \overline{\C_+},\]
so \(\Spec(A) \subeq \overline{\C_+}\).
\item Since \(x \in \R\), we have
\[0 = \braket*{v}{xv} - \braket*{xv}{v} = \braket*{v}{Av} - \braket*{Av}{v} = \braket*{v}{Av} - \braket*{v}{A^*v} = 2i \braket*{v}{\Im(A)v}.\]
Since \(\Im(A) \geq 0\), this implies \(\Im(A)v = 0\) and therefore
\[A^*v = Av = xv. \qedhere\]
\end{enumerate}
\end{proof}
\begin{lemma}\label{lem:PickMaxModPrin}
Let \(f \in \mathrm{Pick}(\C_+)\). If there exists \(z_0 \in \C_+\) such that \(\Im(f(z_0)) = 0\), then \(f\) is constant.
\end{lemma}
\begin{proof}
Using the biholomorphic Cayley transform
\[h(z) \coloneqq \frac{z-i}{z+i}\]
with
\[h(\R) = \T \setminus \{1\} \quad \text{and} \quad h(\C_+) = \D,\]
we get
\[h(f(z_0)) \in h(\R) \subeq \T\]
and
\[h(f(z)) \in h(\overline{\C_+}) \subeq \overline{\D} \quad \forall z \in \C_+.\]
Therefore
\[|h(f(z))| \leq 1 = h(f(z_0)) \quad \forall z \in \C_+.\]
This, by the Maximum modulus principle, implies that \(h \circ f\) and therefore also \(f\) is constant.
\end{proof}
We now present some results about Pick functions with self-adjoint/real boundary values:
\begin{theorem}\label{thm:OPGOfInnerFunctions}{\rm (cf. \cite[Prop. 3.4.6]{Sc20})}
Let \(\cH\) be a finite-dimensional complex Hilbert space and let \(\phi: \R_{\geq 0} \to \Inn(\C_+,B(\cH))\) be a strongly continuous one-parameter semigroup of inner functions. Then there exists a Pick function \({F \in \mathrm{Pick}(\C_+,B(\cH))_\R}\) such that
\[\phi_t(z) = e^{itF(z)}, \quad z \in \C_+.\]
\end{theorem}
\begin{lem}\label{lem:BoundaryMeasurable}
Let \(\cH\) be a finite-dimensional complex Hilbert space and let \({F \in \mathrm{Pick}(\C_+,B(\cH))_\R}\). Then, the function \(F_*\) is measurable.
\end{lem}
\begin{proof}
Let \(v,w \in \cH\). For \(n \in \N\), we define
\[g_n: \R \to \C, \quad x \mapsto \braket*{v}{F\left(x+\frac in\right)w}.\]
Then, for every \(n \in \N\), the function \(g_n\) is continuous and therefore measurable. This implies that also the function
\[\braket*{v}{F_*(\cdot)w} = \limsup_{n \to \infty} g_n\]
is measurable. Since \(v,w \in \cH\) were chosen arbitrarily, this implies that \(F_*\) is measurable.
\end{proof}
\begin{theorem}\label{thm:GePick}{\rm \textbf{(Nevanlinna Representation)}(cf. \cite[Thm. 5.3, Thm. 5.13(iii)]{RR94}, \cite[Thm. 7]{Ge17})}
Let \(F \in \cO(\C_+)\). Then \(F \in \mathrm{Pick}(\C_+)\), if and only if there exist \(C \in \R\), \(D \geq 0\) and some Borel measure \(\mu\) with
\[\int_\R \frac{d\mu(\lambda)}{1+\lambda^2} < \infty\]
such that
\[F(z) = C + Dz + \int_\R \frac 1{\lambda - z} - \frac{\lambda}{1+\lambda^2} \,d\mu(\lambda) \quad \forall z \in \C_+.\]
In this case, one has \(F \in \mathrm{Pick}(\C_+)_\R\), if and only if \(\mu\) is singular with respect to the Lebesgue measure on \(\R\). If \(\mu\) is singular, a minimal support for it is given by
\[\gls*{Smu} \coloneqq \left\{x \in \R: \lim_{\epsilon \downarrow 0} \Im(F(x+i\epsilon)) = +\infty\right\}.\]
\end{theorem}
\begin{remark}\label{rem:GePick}
If one takes the measure \(\mu\) from this theorem and defines a measure \(\tilde \nu\) on \(\R\) by
\[d\tilde \nu(\lambda) \coloneqq \frac{d\mu(\lambda)}{1+\lambda^2},\]
one gets that
\[F(z) = C + Dz + \int_\R \frac {1+\lambda z}{\lambda - z} \,d\tilde \nu(\lambda)\]
for a finite measure \(\tilde \nu\) on \(\R\). Notice that
\[\lim_{\lambda \to \pm\infty} \frac {1+\lambda z}{\lambda - z} = z,\]
so, if we extend our measure \(\tilde \nu\) to a measure \(\nu\) on \(\overline{\R}\coloneqq \R \cup \{\infty\}\) by \(\nu(\{\infty\}) \coloneqq D\), we get
\[F(z) = C + \int_{\overline{\R}} \frac {1+\lambda z}{\lambda - z} \,d\nu(\lambda).\]
Then, setting
\[\gls*{SF} \coloneqq \begin{cases} S_\mu & \text{if }D = 0 \\ S_\mu \cup \{\infty\} & \text{if }D > 0,\end{cases}\] 
in the case that \(\mu\) was singular, the set \(S_F\) is a minimal support for the measure \(\nu\).
\end{remark}
Given a Pick function \(F \in \mathrm{Pick}(\C_+)_\R\), we want to understand how the function \(F_*\) behaves on \(\overline{\R} \setminus S_F\). For this, we need the following definition:
\begin{definition}
Let \(\overline{\R}\coloneqq \R \cup \{\infty\}\) be the one-point compactification of \(\R\). For \(a,b \in \overline{\R}\), we set
\[(a,b) \coloneqq \begin{cases} \{x \in \R: a<x<b\} & \text{if }a,b \in \R, a \leq b \\ \{x \in \R: b < x\} \cup \{\infty\} \cup \{x \in \R: x < a\} & \text{if }a,b \in \R, a > b \\ \{x \in \R: a<x\} & \text{if }a \in \R,b = \infty \\ \{x \in \R: x<b\} & \text{if }b \in \R, a = \infty \\ \eset & \text{if }a,b = \infty.\end{cases}\] 
\end{definition}
\begin{lem}\label{lem:countableUnion}
Let \(A \subeq \overline{\R}\) be a closed subset. Then there exists a countable family \(\left(a_j,b_j\right)_{j \in J}\) in \(\overline{\R} \times \overline{\R}\) such that \(\overline{\R} \setminus A = \bigsqcup_{j \in J} \left(a_j,b_j\right)\).
\end{lem}
\begin{proof}
The maximally connected subsets of \(\overline{\R} \setminus A\) are of the form \((a,b)\) for some \(a,b \in \overline{\R}\). Therefore
\[\overline{\R} \setminus A = \bigsqcup_{j \in J} \left(a_j,b_j\right)\]
for some family \(\left(a_j,b_j\right)_{j \in J}\) in \(\overline{\R} \times \overline{\R}\) with \(\left(a_j,b_j\right) \neq \eset\) for every \(j \in J\). Since \(\left(a_j,b_j\right) \neq \eset\) implies that \(\left(a_j,b_j\right) \cap \Q \neq \eset\), for every \(j \in J\), there exists \(q_j \in \Q\) such that \(q_j \in \left(a_j,b_j\right)\). From this follows the existence of a surjective function \(f: \Q \to J\), and therefore \(J\) is countable.
\end{proof}
\begin{prop}\label{prop:PickBoundaryDiffeo}
Let \(F \in \mathrm{Pick}(\C_+)_\R\). Further, let \(a,b \in \overline{\R}\) with \(a \neq b\) such that \((a,b) \cap S_F = \eset\) and \(a,b \in S_F\). Then the map
\[F_*\big|_{\left(a,b\right)}: (a,b) \to \R\]
is monotonically increasing, bijective, and continuously differentiable.
\end{prop}
\begin{proof}
By replacing \(F\) with \(F \circ M\), where \(M: \C_+ \to \C_+\) is a suitable Möbius transformation, we can, without loss of generality, assume that \(a,b \in \R\) with
\[a<0<b.\]
Now, since \((a,b) \cap S_F = \eset\), by \cite[Lem. II.2]{Do74}, there exists a function \(\tilde F \in \cO(\C_+ \cup (a,b) \cup \C_-)\) with
\[\tilde F(z) = \overline{\tilde F(\overline{z})} \quad \forall z \in \C_+ \cup (a,b) \cup \C_-\]
such that
\[\tilde F \big|_{\C_+} = F.\]
Since \(\tilde F\) is holomorphic, we have that \(F_*\big|_{(a,b)}= \tilde F\big|_{a,b}\) is continuously differentiable. Further, we have
\[\Im \tilde F(z) = \Im F(z) \geq 0 \quad \forall z \in \C_+,\]
which implies that \(\Im \tilde F(z) \leq 0\) for all \(z \in \C_-\) and therefore
\[\frac d{dy}\bigg|_{y=0} \Im \tilde F(x+iy) \geq 0 \quad \forall x \in (a,b).\]
This, by the Cauchy--Riemann equations, implies that
\[0 \leq \frac d{dx} \Re \tilde F(x) = \frac d{dx} F_*(x) \quad \forall x \in (a,b),\]
so \(F_*\big|_{(a,b)}\) is monotonically increasing.
\newpage
Now we assume that \(F_*\big|_{(a,b)} = \tilde F\big|_{(a,b)}\) is not injective, i.e. that there exist \(c,d \in (a,b)\) with \(c<d\) such that
\[\tilde F(c) = \tilde F(d).\]
Then by the monotony of \(F_*\big|_{(a,b)} = \tilde F\big|_{(a,b)}\), the function \(\tilde F\) would be constant on \([c,d]\) and therefore constant in contradiction to the assumption that \(a,b \in S_F\). Therefore \(F_*\big|_{\left(a,b\right)}\) is injective.

To show that the function \(F_*\big|_{\left(a,b\right)}\) is surjective, due to its monotony, it suffices to show that
\[\lim_{x \downarrow a} F_*(x) = -\infty \qquad \text{and} \qquad \lim_{x \uparrow b} F_*(x) = +\infty.\]
For \(x \in (a,b)\), \(\epsilon \in (0,1)\) and \(\lambda \in \overline{\R} \setminus (a,b)\), setting \(m \coloneqq \max\{|a|,|b|\}\), we have
\begin{align*}
\left|\frac{1+ \lambda (x+i\epsilon)}{\lambda - (x+i\epsilon)}\right| &= \left|(x+i\epsilon) + \frac{1+(x+i\epsilon)^2}{\lambda - (x+i\epsilon)}\right| \leq |(x+i\epsilon)| + \frac{1+\left|x+i\epsilon\right|^2}{|\lambda - (x+i\epsilon)|}
\\&\leq \sqrt{1+m^2} + \frac{2+m^2}{|\lambda-x|} \leq \sqrt{1+m^2} + \frac{2+m^2}{\min\{|x-a|,|x-b|\}}.
\end{align*}
Therefore, writing
\[F(z) = C + \int_{\overline{\R}} \frac {1+\lambda z}{\lambda - z} \,d\nu(\lambda), \quad z \in \C_+,\]
with \(\nu\) defined as in \fref{rem:GePick}, by the Dominated Convergence Theorem, we get
\begin{align*}
F_*(x) &= \lim_{\epsilon \downarrow 0} F(x+i\epsilon) = C + \lim_{\epsilon \downarrow 0} \int_{\overline{\R}} \frac {1+\lambda (x+i\epsilon)}{\lambda - (x+i\epsilon)} \,d\nu(\lambda)
\\&= C + \lim_{\epsilon \downarrow 0} \int_{\overline{\R} \setminus (a,b)} \frac {1+\lambda (x+i\epsilon)}{\lambda - (x+i\epsilon)} \,d\nu(\lambda) = C + \int_{\overline{\R} \setminus (a,b)} \lim_{\epsilon \downarrow 0} \frac {1+\lambda (x+i\epsilon)}{\lambda - (x+i\epsilon)} \,d\nu(\lambda)
\\&= C + \int_{\overline{\R} \setminus (a,b)} \frac {1+\lambda x}{\lambda - x} \,d\nu(\lambda) = C + \int_{\overline{\R}} \frac {1+\lambda x}{\lambda - x} \,d\nu(\lambda).
\end{align*}
For \(x \in \left(0,b\right)\) and \(\lambda \in \overline{\R} \setminus (a,b+1)\), we have
\begin{align*}
\left|\frac{1+ \lambda x}{\lambda - x}\right| &= \left|x + \frac{1+x^2}{\lambda - x}\right| \leq |x| + \frac{1+x^2}{|\lambda - x|} \leq b + \frac{1+b^2}{\min\left\{|a|,1\right\}},
\end{align*}
so, by the Dominated Convergence Theorem, we get
\begin{align*}
\lim_{x \uparrow b} \int_{\overline{\R} \setminus (a,b+1)} \frac {1+\lambda x}{\lambda - x} \,d\nu(\lambda) = \int_{\overline{\R} \setminus (a,b+1)} \frac {1+\lambda b}{\lambda - b}. \,d\nu(\lambda)
\end{align*}
Since
\[\frac {1+\lambda x}{\lambda - x} \geq 0 \quad \forall x \in (0,b), \lambda \geq b\]
and since for every \(\lambda \geq b\) the function
\[(0,b) \ni x \mapsto \frac {1+\lambda x}{\lambda - x} = \frac{1+\lambda^2}{\lambda - x} - \lambda\]
is monotonically increasing, by the Monotone Convergence Theorem, we get
\[\lim_{x \uparrow b} \int_{[b,b+1)} \frac {1+\lambda x}{\lambda - x} \,d\nu(\lambda) = \int_{[b,b+1)} \frac {1+\lambda b}{\lambda - b} \,d\nu(\lambda).\]
This yields
\begin{align*}
\lim_{x \uparrow b} F_*(x) &= C + \lim_{x \uparrow b} \int_{\overline{\R}} \frac {1+\lambda x}{\lambda - x} \,d\nu(\lambda) = C + \lim_{x \uparrow b} \int_{\overline{\R} \setminus (a,b)} \frac {1+\lambda x}{\lambda - x} \,d\nu(\lambda)
\\&= C + \lim_{x \uparrow b} \int_{\overline{\R} \setminus (a,b+1)} \frac {1+\lambda x}{\lambda - x} \,d\nu(\lambda) + \lim_{x \uparrow b} \int_{[b,b+1)} \frac {1+\lambda x}{\lambda - x} \,d\nu(\lambda)
\\&= C + \int_{\overline{\R} \setminus (a,b+1)} \frac {1+\lambda b}{\lambda - b} \,d\nu(\lambda) + \int_{[b,b+1)} \frac {1+\lambda b}{\lambda - b} \,d\nu(\lambda) = C + \int_{\overline{\R}} \frac {1+\lambda b}{\lambda - b} \,d\nu(\lambda).
\end{align*}
If \(\lim_{x \uparrow b} F_*(x)<\infty\), this would imply that \(\left(\lambda \mapsto \frac {1+\lambda b}{\lambda - b}\right) \in L^1(\overline{\R},\C;\nu)\) and therefore also
\[\left(\lambda \mapsto \frac1{1+b^2} \cdot \left(\frac {1+\lambda b}{\lambda - b} - b\right) = \frac 1{\lambda - b}\right) \in L^1(\overline{\R},\C;\nu).\]
Then, for \(\epsilon \in (0,1)\), we have
\begin{align*}
\left|\frac{1+ \lambda (b+i\epsilon)}{\lambda - (b+i\epsilon)}\right| &= \left|(b+i\epsilon) + \frac{1+\left(b+i\epsilon\right)^2}{\lambda - (b+i\epsilon)}\right| \leq \left|b+i\epsilon\right| + \frac{1+\left|b+i\epsilon\right|^2}{\left|\lambda - (b+i\epsilon)\right|} \leq \sqrt{1+b^2} + \frac{2+b^2}{\left|\lambda - b\right|},
\end{align*}
so, by the Dominated Convergence Theorem, we would get
\begin{align*}
\lim_{\epsilon \downarrow 0} F(b+i\epsilon) &= C + \lim_{\epsilon \downarrow 0} \int_{\overline{\R}} \frac {1+\lambda (b+i\epsilon)}{\lambda - (b+i\epsilon)} \,d\nu(\lambda)
\\&= C + \int_{\overline{\R}} \lim_{\epsilon \downarrow 0} \frac {1+\lambda (b+i\epsilon)}{\lambda - (b+i\epsilon)} \,d\nu(\lambda) = C + \int_{\overline{\R}} \frac {1+\lambda b}{\lambda - b} \,d\nu(\lambda) = \lim_{x \uparrow b} F_*(x)<\infty
\end{align*}
in contradiction to the assumption \(b \in S_F\). Therefore
\[\lim_{x \uparrow b} F_*(x) = \infty.\]
Analogously one shows that
\begin{align*}
\lim_{x \downarrow a} F_*(x) = C + \int_{\overline{\R}} \frac {1+\lambda a}{\lambda - a} \,d\nu(\lambda)
\end{align*}
and sees that \(\lim_{x \downarrow a} F_*(x) > -\infty\) is in contradiction to the assumption \(a \in S_F\), which shows that
\[\lim_{x \downarrow a} F_*(x) = -\infty. \qedhere\]
\end{proof}

\newpage
\section{An approximation theorem}
In this chapter, we will prove the existence of a limit appearing in the proof of \fref{prop:FixedPointFromConstruction} and calculate its value. Since this proof is very technical and lengthy, we decided to remove it from the proof of \fref{prop:FixedPointFromConstruction} and include it here as an appendix. We start with a definition:
\begin{definition}
Let \(n \in \N\) and \(p \in \R_+\). We define the following three functions:
\begin{align*}
\gls*{gn}(x) &\coloneqq \frac{x^2}{n\left(\frac 1{n^2}+x^2\right)\left(1+\frac{x^2}{n^2}\right)}, \quad x \in \R
\\~
\\\gls*{dpn}(x) &\coloneqq g_n(x) \cdot \frac{\left(x^2-p^2-\frac 1{n^2}\right)\left(1-x^2\right)+2x^2\left(\frac 1{n^2} + 1\right)}{\left(x^2-p^2-\frac 1{n^2}\right)^2+4x^2 \frac 1{n^2}}, \quad x \in \R
\\~
\\\gls*{fpn}(x) &\coloneqq d_{p,n}(x) + g_n(x) \cdot \left(\chi_{\left(-1,1\right)}(x) \cdot \frac 1{p^2} + \chi_{\left(-\infty,-1\right) \cup \left(1,\infty\right)}(x)\right), \quad x \in \R
\end{align*}
\end{definition}
\begin{remark}
\begin{enumerate}[\rm (a)]
\item For every \(n \in \N\) and \(x \in \R^\times\), one has
\[g_n(x) = \frac{x^2}{n\left(\frac 1{n^2}+x^2\right)\left(1+\frac{x^2}{n^2}\right)} \leq \frac {x^2}{n \cdot x^2 \cdot 1} = \frac 1n,\]
so
\[\lim_{n \to \infty} g_n(x) = 0.\]
\item For every \(x \in \R^\times\) with \(|x| \neq p\) one has
\[\lim_{n \to \infty} d_{p,n}(x) = \lim_{n \to \infty} g_n(x) \cdot \frac{\left(x^2-p^2\right)\left(1-x^2\right)+2x^2}{\left(x^2-p^2\right)^2} = 0.\]
\end{enumerate}
\end{remark}
The goal of this section is to show that for any measurable function \(\phi: \R \to \C\) with
\[\int_\R \frac{x^2}{\left(1+x^2\right)^2} \left|\phi(x)\right| \,dx < \infty\]
one has
\[\frac{\phi(p) + \phi(-p)}2 = \lim_{n \to \infty} \frac 1\pi \int_\R f_{p,n}(x) \,\phi(x) \,dx\]
for almost every \(p \in \R_+\). We start by the following lemma:
\begin{lem}\label{lem:approximationZero}
For any measurable function \(\phi: \R \to \C\) with
\[\int_\R \frac{x^2}{\left(1+x^2\right)^2} \left|\phi(x)\right| \,dx < \infty\]
and every \(p \in \R_+\) one has
\[\lim_{n \to \infty} \int_{\left(-p+\frac 1{\sqrt n},p-\frac 1{\sqrt n}\right)} \left(d_{p,n}(x)+ \frac{g_n(x)}{p^2}\right)\phi(x) \,dx = 0.\]
\end{lem}
\begin{proof}
For \(n \in \N\) with \(\frac 1{\sqrt n} < p\) and \(x \in \left(-p+\frac 1{\sqrt n},p-\frac 1{\sqrt n}\right)\) one has
\begin{align*}
&\left|d_{p,n}(x)+ \frac{g_n(x)}{p^2}\right| = \frac{g_n(x)}{p^2} \cdot \left|\frac{\left(x^2-p^2-\frac 1{n^2}\right)\left(1-x^2\right)+2x^2\left(\frac 1{n^2} + 1\right)}{\left(x^2-p^2-\frac 1{n^2}\right)^2+4x^2 \frac 1{n^2}} \cdot p^2 + 1\right|
\\&\qquad=\frac{g_n(x)}{p^2} \cdot \left|\frac{\left(x^2-p^2-\frac 1{n^2}\right)\left(1-x^2\right)p^2+2x^2\left(\frac 1{n^2} + 1\right)p^2 + \left(x^2-p^2-\frac 1{n^2}\right)^2+4x^2 \frac 1{n^2}}{\left(x^2-p^2-\frac 1{n^2}\right)^2+4x^2 \frac 1{n^2}}\right|
\\&\qquad=\frac{g_n(x)}{p^2} \cdot \left|\frac{x^4\left(1-p^2\right) + x^2\left(p^2 + \frac 3{n^2}p^2 + p^4 + \frac 2{n^2}\right) + \left(\frac 1{n^4} + \frac {p^2}{n^2}\right)}{\left(x^2-p^2-\frac 1{n^2}\right)^2+4x^2 \frac 1{n^2}}\right|
\\&\qquad\leq\frac{g_n(x)}{p^2} \cdot \frac{x^2\left|x^2\left(1-p^2\right) + \left(p^2 + \frac 3{n^2}p^2 + p^4 + \frac 2{n^2}\right)\right| + \left(\frac 1{n^4} + \frac {p^2}{n^2}\right)}{\left(p^2+\frac 1{n^2}-x^2\right)^2+4x^2 \frac 1{n^2}}
\\&\qquad\leq\frac{g_n(x)}{p^2} \cdot \frac{x^2\left(p^2\left(1+p^2\right) + \left(4p^2 + p^4 + 2\right)\right) + \frac 1{n^2}\left(\frac 1{n^2} + p^2\right)}{\left(p^2+\frac 1{n^2}-\left(p-\frac 1{\sqrt n}\right)^2\right)^2}
\\&\qquad\leq\frac{g_n(x)}{p^2} \cdot \frac{x^2\left(2p^4 + 5p^2 + 2\right) + \frac 1{n^2}\left(1 + p^2\right)}{\left(\frac 1{n^2}-\frac 1n + 2p \frac 1{\sqrt n}\right)^2}
\\&\qquad\leq\frac{x^2}{p^2n\left(\frac 1{n^2}+x^2\right)\left(1 + \frac{x^2}{n^2}\right)} \cdot \frac{x^2\left(2p^4 + 5p^2 + 2\right) + \frac 1{n^2}\left(1 + p^2\right)}{\left(0 - \frac p{\sqrt n} + 2p \frac 1{\sqrt n}\right)^2}
\\&\qquad\leq\frac{x^2}{p^2n\left(\frac 1{n^2}+x^2\right)} \cdot \frac{x^2\left(2p^4 + 5p^2 + 2\right) + \frac 1{n^2}\left(1 + p^2\right)}{\frac {p^2} n}
\\&\qquad=\frac{x^2}{p^2n\left(\frac 1{n^2}+x^2\right)} \cdot \frac{x^2\left(2p^4 + 5p^2 + 2\right)}{\frac {p^2} n} + \frac{x^2}{p^2n\left(\frac 1{n^2}+x^2\right)} \cdot \frac{\frac 1{n^2}\left(1 + p^2\right)}{\frac {p^2} n}
\\&\qquad=\frac{x^2}{\frac 1{n^2}+x^2} \cdot \frac{x^2\left(2p^4 + 5p^2 + 2\right)}{p^4} + \frac{x^2}{1+n^2x^2} \cdot \frac{\left(1 + p^2\right)}{p^4}
\\&\qquad\leq 1 \cdot \frac{x^2\left(2p^4 + 5p^2 + 2\right)}{p^4} + \frac{x^2}{1} \cdot \frac{\left(1 + p^2\right)}{p^4} = x^2 \cdot \frac{2p^4 + 6p^2 + 3}{p^4}
\\&\qquad\leq \frac{x^2}{\left(1+x^2\right)^2} \cdot \frac{\left(2p^4 + 6p^2 + 3\right)\left(1+p^2\right)^2}{p^4}.
\end{align*}
By assumption, we have
\[\int_{(-p,p)}\frac{x^2}{\left(1+x^2\right)^2} \cdot \frac{\left(2p^4 + 6p^2 + 3\right)\left(1+p^2\right)^2}{p^4} \left|\phi(x)\right| \,dx < \infty\]
and therefore, by the Dominated Convergence Theorem, we have
\begin{align*}
&\lim_{n \to \infty} \int_{\left(-p+\frac 1{\sqrt n},p-\frac 1{\sqrt n}\right)} \left(d_{p,n}(x)+ \frac{g_n(x)}{p^2}\right)\phi(x) \,dx
\\&\qquad= \lim_{n \to \infty} \int_{\left(-p,p\right)} \chi_{\left(-p+\frac 1{\sqrt n},p-\frac 1{\sqrt n}\right)}(x) \left(d_{p,n}(x)+ \frac{g_n(x)}{p^2}\right)\phi(x) \,dx
\\&\qquad= \int_{\left(-p,p\right)} \lim_{n \to \infty}\chi_{\left(-p+\frac 1{\sqrt n},p-\frac 1{\sqrt n}\right)}(x) \left(d_{p,n}(x)+ \frac{g_n(x)}{p^2}\right)\phi(x) \,dx
\\&\qquad= \int_{\left(-p,p\right)} 0 \cdot \phi(x) \,dx = 0. \qedhere
\end{align*}
\end{proof}
\begin{lem}\label{lem:approximationInfinity}
For any measurable function \(\phi: \R \to \C\) with
\[\int_\R \frac{x^2}{\left(1+x^2\right)^2} \left|\phi(x)\right| \,dx < \infty\]
and every \(p \in \R_+\), one has
\[\lim_{n \to \infty} \int_{\left(-\infty,-p-\frac 1{\sqrt n}\right) \cup \left(p+\frac 1{\sqrt n},\infty\right)} \left(d_{p,n}(x)+g_n(x)\right)\phi(x) \,dx = 0.\]
\end{lem}
\begin{proof}
For \(n \in \N\) with \(\frac 1{\sqrt n} < p\) and \(x \in \left(-\infty,-p-\frac 1{\sqrt n}\right) \cup \left(p+\frac 1{\sqrt n},\infty\right)\), one has
\begin{align*}
&\left|d_{p,n}(x)+g_n(x)\right| = g_n(p) \cdot \left|\frac{\left(x^2-p^2-\frac 1{n^2}\right)\left(1-x^2\right)+2x^2\left(\frac 1{n^2} + 1\right)}{\left(x^2-p^2-\frac 1{n^2}\right)^2+4x^2 \frac 1{n^2}} + 1\right|
\\&\quad\leq \frac 1n \cdot \left|\frac{\left(x^2-p^2-\frac 1{n^2}\right)\left(1-x^2\right)+2x^2\left(\frac 1{n^2} + 1\right) + \left(x^2-p^2-\frac 1{n^2}\right)^2+4x^2 \frac 1{n^2}}{\left(x^2-p^2-\frac 1{n^2}\right)^2+4x^2 \frac 1{n^2}}\right|
\\&\quad=\frac 1n \cdot \left|\frac{x^2\left(\frac 5{n^2} + 3 - p^2\right) - \left(p^2 + \frac 1{n^2}\right) + \left(p^2 + \frac 1{n^2}\right)^2}{\left(x^2-p^2-\frac 1{n^2}\right)^2+4x^2 \frac 1{n^2}}\right|
\\&\quad\leq \frac 1n \cdot \frac{x^2\left(\frac 5{n^2} + 3 + p^2\right) + \left(p^2 + \frac 1{n^2}\right)\left(p^2 + \frac 1{n^2} + 1\right)}{\left(x^2-p^2-\frac 1{n^2}\right)^2}
\\&\quad\leq \frac 1n \cdot \frac{x^2\left(8 + p^2\right) + x^2\left(p^2+2\right)}{\left(x^2 \cdot \frac{p^2 + \frac 1{n^2}}{\left(p + \frac 1{\sqrt n}\right)^2}-p^2-\frac 1{n^2} + x^2 \left(1-\frac{p^2 + \frac 1{n^2}}{\left(p + \frac 1{\sqrt n}\right)^2}\right)\right)^2} \leq \frac 1n \cdot \frac{2x^2\left(p^2+5\right)}{\left(x^2 \left(1-\frac{p^2 + \frac 1{n^2}}{\left(p + \frac 1{\sqrt n}\right)^2}\right)\right)^2}
\\&\quad= \frac 1n \cdot \frac{2x^2\left(p^2+5\right)}{\left(x^2 \cdot \frac{2p \frac 1{\sqrt n} + \frac 1n-\frac 1{n^2}}{\left(p + \frac 1{\sqrt n}\right)^2}\right)^2} = \frac 1{nx^2} \cdot \frac{2\left(p^2+5\right)\left(p + \frac 1{\sqrt n}\right)^2}{\left(2p \frac 1{\sqrt n} + \frac 1n-\frac 1{n^2}\right)^2} \leq \frac 1{nx^2} \cdot \frac{2\left(p^2+5\right)\left(p + 1\right)^2}{\left(2p \frac 1{\sqrt n}\right)^2}
\\&\quad= \frac 1{x^2} \cdot \frac{\left(p^2+5\right)\left(p + 1\right)^2}{2p^2} = \frac {x^2\left(\frac 1{x^2} + 1\right)^2}{\left(1+x^2\right)^2} \cdot \frac{\left(p^2+5\right)\left(p + 1\right)^2}{2p^2}
\\&\quad\leq \frac {x^2\left(\frac 1{p^2} + 1\right)^2}{\left(1+x^2\right)^2} \cdot \frac{\left(p^2+5\right)\left(p + 1\right)^2}{2p^2} = \frac {x^2}{\left(1+x^2\right)^2} \cdot \frac{\left(1 + p^2\right)^2\left(p^2+5\right)\left(p + 1\right)^2}{2p^6}.
\end{align*}
By assumption, we have
\[\int_{(-\infty,-p) \cup (p,\infty)}\frac {x^2}{\left(1+x^2\right)^2} \cdot \frac{\left(1 + p^2\right)^2\left(p^2+5\right)\left(p + 1\right)^2}{2p^6} \left|\phi(x)\right| \,dx < \infty\]
and therefore, by the Dominated Convergence Theorem, we have
\begin{align*}
&\lim_{n \to \infty} \int_{\left(-\infty,-p-\frac 1{\sqrt n}\right) \cup \left(p+\frac 1{\sqrt n},\infty\right)} \left(d_{p,n}(x)+g_n(x)\right)\phi(x) \,dx
\\&\qquad= \lim_{n \to \infty} \int_{\left(-\infty,-p\right) \cup \left(p,\infty\right)} \chi_{\left(-\infty,-p-\frac 1{\sqrt n}\right) \cup \left(p+\frac 1{\sqrt n},\infty\right)}(x) \left(d_{p,n}(x)+g_n(x)\right)\phi(x) \,dx
\\&\qquad= \int_{\left(-p,p\right)} \lim_{n \to \infty}\chi_{\left(-\infty,-p-\frac 1{\sqrt n}\right) \cup \left(p+\frac 1{\sqrt n},\infty\right)}(x) \left(d_{p,n}(x)+g_n(x)\right)\phi(x) \,dx
\\&\qquad= \int_{\left(-p,p\right)} 0 \cdot \phi(x) \,dx = 0. \qedhere
\end{align*}
\end{proof}
\begin{definition}
Let \(n \in \N\) and \(p \in \R_+\). We define the following functions:
\begin{align*}
\tilde g_{p,n}(x) \coloneqq \frac{px}{n\left(\frac 1{n^2}+p^2\right)\left(1+\frac{p^2}{n^2}\right)}, \quad x \in \R
\\~
\\\gls*{dpntilde}(x) \coloneqq \tilde g_{p,n}(x) \cdot \frac{2p^2\left(\frac 1{n^2} + 1\right)}{\left(x^2-p^2-\frac 1{n^2}\right)^2+4p^2 \frac 1{n^2}}, \quad x \in \R
\end{align*}
\end{definition}
\begin{lem}\label{lem:dSlightChange}
For every \(p \in \R_+\), there exist sequences \(\left(b_n\right)_{n \in \N}\) and \(\left(B_n\right)_{n \in \N}\) in \(\R\) with
\[\lim_{n \to \infty} b_{p,n} = 1 = \lim_{n \to \infty} B_{p,n}\]
such that for sufficiently large \(n \in \N\), one has
\[b_{p,n} \cdot \tilde d_{p,n}(x) \leq d_{p,n}(x) \leq B_{p,n} \cdot \tilde d_{p,n}(x)\]
for every \(x \in \left(p-\frac 1{\sqrt n},p+\frac 1{\sqrt n}\right)\).
\end{lem}
\begin{proof}
We have
\[\frac{g_n(x)}{\tilde g_{p,n}(x)} = \frac{\frac{x^2}{n\left(\frac 1{n^2}+x^2\right)\left(1+\frac{x^2}{n^2}\right)}}{\frac{px}{n\left(\frac 1{n^2}+p^2\right)\left(1+\frac{p^2}{n^2}\right)}} = \frac{x\left(\frac 1{n^2}+p^2\right)\left(1+\frac{p^2}{n^2}\right)}{p\left(\frac 1{n^2}+x^2\right)\left(1+\frac{x^2}{n^2}\right)},\]
so for \(x \in \left(p-\frac 1{\sqrt n},p+\frac 1{\sqrt n}\right)\) one has
\[\frac{g_n(x)}{\tilde g_{p,n}(x)} \leq \frac{\left(p+\frac 1{\sqrt n}\right)\left(\frac 1{n^2}+p^2\right)\left(1+\frac{p^2}{n^2}\right)}{p\left(\frac 1{n^2}+\left(p-\frac 1{\sqrt n}\right)^2\right)\left(1+\frac{\left(p-\frac 1{\sqrt n}\right)^2}{n^2}\right)} \eqqcolon A_{p,n} \xrightarrow{n \to \infty} 1\]
and
\[\frac{g_n(x)}{\tilde g_{p,n}(x)} \geq \frac{\left(p-\frac 1{\sqrt n}\right)\left(\frac 1{n^2}+p^2\right)\left(1+\frac{p^2}{n^2}\right)}{p\left(\frac 1{n^2}+\left(p+\frac 1{\sqrt n}\right)^2\right)\left(1+\frac{\left(p+\frac 1{\sqrt n}\right)^2}{n^2}\right)} \eqqcolon a_{p,n} \xrightarrow{n \to \infty} 1.\]
Further, we have
\begin{align*}
\frac{d_{p,n}(x)}{\tilde d_{p,n}(x)} &= \frac{g_n(x)}{\tilde g_{p,n}(x)} \cdot \frac{\frac{\left(x^2-p^2-\frac 1{n^2}\right)\left(1-x^2\right)+2x^2\left(\frac 1{n^2} + 1\right)}{\left(x^2-p^2-\frac 1{n^2}\right)^2+4x^2 \frac 1{n^2}}}{\frac{2p^2\left(\frac 1{n^2} + 1\right)}{\left(x^2-p^2-\frac 1{n^2}\right)^2+4p^2 \frac 1{n^2}}}
\\&= \frac{g_n(x)}{\tilde g_{p,n}(x)} \cdot \frac{\left(x^2-p^2-\frac 1{n^2}\right)\left(1-x^2\right)+2x^2\left(\frac 1{n^2} + 1\right)}{2p^2\left(\frac 1{n^2} + 1\right)} \cdot \left(\frac{\left(x^2-p^2-\frac 1{n^2}\right)^2+4x^2 \frac 1{n^2}}{\left(x^2-p^2-\frac 1{n^2}\right)^2+4p^2 \frac 1{n^2}}\right)^{-1}.
\end{align*}
For \(n \in \N\) with \(\frac 1{\sqrt n} \leq p\) and \(x \in \left(p-\frac 1{\sqrt n},p+\frac 1{\sqrt n}\right)\) one has
\begin{align*}
&\left(x^2-p^2-\frac 1{n^2}\right)\left(1-x^2\right)+2x^2\left(\frac 1{n^2} + 1\right) = -x^4 + x^2\left(\frac 3{n^2} + 3 + p^2\right) - \left(p^2 + \frac 1{n^2}\right)
\\&\qquad\geq -\left(p + \frac 1{\sqrt n}\right)^4 + \left(p - \frac 1{\sqrt n}\right)^2\left(\frac 3{n^2} + 3 + p^2\right) - \left(p^2 + \frac 1{n^2}\right)
\\&\qquad\eqqcolon v_{p,n} \xrightarrow{n \to \infty} -p^4 + p^2\left(3+p^2\right) - p^2 =  2p^2 > 0
\end{align*}
and
\begin{align*}
&\left(x^2-p^2-\frac 1{n^2}\right)\left(1-x^2\right)+2x^2\left(\frac 1{n^2} + 1\right) = -x^4 + x^2\left(\frac 3{n^2} + 3 + p^2\right) - \left(p^2 + \frac 1{n^2}\right)
\\&\qquad\leq -\left(p - \frac 1{\sqrt n}\right)^4 + \left(p + \frac 1{\sqrt n}\right)^2\left(\frac 3{n^2} + 3 + p^2\right) - \left(p^2 + \frac 1{n^2}\right)
\\&\qquad\eqqcolon V_{p,n} \xrightarrow{n \to \infty} -p^4 + p^2\left(3+p^2\right) - p^2 =  2p^2 > 0
\end{align*}
and
\begin{align*}
2p^2\left(\frac 1{n^2} + 1\right) \eqqcolon w_{p,n} \xrightarrow{n \to \infty} 2p^2 > 0
\end{align*}
and
\begin{align*}
\frac{\left(x^2-p^2-\frac 1{n^2}\right)^2+4x^2 \frac 1{n^2}}{\left(x^2-p^2-\frac 1{n^2}\right)^2+4p^2 \frac 1{n^2}} &= 1 + \frac{\frac 4{n^2}\left(x^2-p^2\right)}{\left(x^2-p^2-\frac 1{n^2}\right)^2+4p^2 \frac 1{n^2}} \geq 1 + \frac{\frac 4{n^2}\left(\left(p - \frac 1{\sqrt n}\right)^2-p^2\right)}{\left(x^2-p^2-\frac 1{n^2}\right)^2+4p^2 \frac 1{n^2}}
\\&\geq 1 + \frac{\frac 4{n^2}\left(\left(p - \frac 1{\sqrt n}\right)^2-p^2\right)}{4p^2 \frac 1{n^2}} = 1 + \frac {\frac 1n - \frac 2{\sqrt n}}{p^2} \eqqcolon u_{p,n} \xrightarrow{n \to \infty} 1
\end{align*}
and
\begin{align*}
\frac{\left(x^2-p^2-\frac 1{n^2}\right)^2+4x^2 \frac 1{n^2}}{\left(x^2-p^2-\frac 1{n^2}\right)^2+4p^2 \frac 1{n^2}} &= 1 + \frac{\frac 4{n^2}\left(x^2-p^2\right)}{\left(x^2-p^2-\frac 1{n^2}\right)^2+4p^2 \frac 1{n^2}} \leq 1 + \frac{\frac 4{n^2}\left(\left(p + \frac 1{\sqrt n}\right)^2-p^2\right)}{\left(x^2-p^2-\frac 1{n^2}\right)^2+4p^2 \frac 1{n^2}}
\\&\leq 1 + \frac{\frac 4{n^2}\left(\left(p + \frac 1{\sqrt n}\right)^2-p^2\right)}{4p^2 \frac 1{n^2}} = 1 + \frac {\frac 1n + \frac 2{\sqrt n}}{p^2} \eqqcolon U_{p,n} \xrightarrow{n \to \infty} 1.
\end{align*}
Therefore, for sufficiently large \(n \in \N\) and \(x \in \left(p-\frac 1{\sqrt n},p+\frac 1{\sqrt n}\right)\), one has
\[\frac{d_{p,n}(x)}{\tilde d_{p,n}(x)} \leq A_{p,n} \cdot \frac{V_{p,n}}{w_{p,n}} \cdot u_{p,n}^{-1} \eqqcolon B_{p,n} \xrightarrow{n \to \infty} 1 \cdot \frac{2p^2}{2p^2} \cdot 1 = 1\]
and
\[\frac{d_{p,n}(x)}{\tilde d_{p,n}(x)} \geq a_{p,n} \cdot \frac{v_{p,n}}{w_{p,n}} \cdot U_{p,n}^{-1} \eqqcolon b_{p,n} \xrightarrow{n \to \infty} 1 \cdot \frac{2p^2}{2p^2} \cdot 1 = 1.  \qedhere\]
\end{proof}
\begin{definition}
Let \(f \in L^1_{\mathrm{loc}}(\R,\C)\). Then \(p \in \R\) is called a \textit{Lebesgue point of \(f\)}, if
\[\lim_{\epsilon \downarrow 0} \frac 1{2\epsilon} \int_{(p-\epsilon,p+\epsilon)} |f(x)-f(p)| \,dx = 0.\]
\end{definition}
\begin{lem}\label{lem:LebesguePointApprox}
Let \(f \in L^1_{\mathrm{loc}}(\R,\C)\) and \(p \in \R\) be a Lebesgue point of \(f\). Further, let \(\left(k_n\right)_{n \in \N}\) be a sequence of non-negative functions with
\[\int_{\left(p-\frac 1{\sqrt n},p+\frac 1{\sqrt n}\right)} k_n(x) \,dx = 1,\]
such that there exist constants \(r,R>0\) such that for every \(n \in \N\), one has
\[\max\left\{k_n\left(p-\frac 1{\sqrt n}\right),k_n\left(p+\frac 1{\sqrt n}\right)\right\} \leq R\sqrt{n}\]
and \(k_n \in C^1\left(\left(p-\frac 1{\sqrt n},p+\frac 1{\sqrt n}\right)\right)\) with
\begin{align*}
k_n'(x) &\geq 0 &&\forall x \in \left(p-\frac 1{\sqrt n},p-\frac 1n\right)
\\\left|k_n'(x)\right| &\leq rn^2 &&\forall x \in \left(p-\frac 1n,p+\frac 1n\right)
\\k_n'(x) &\leq 0 &&\forall x \in \left(p+\frac 1n,p+\frac 1{\sqrt n}\right).
\end{align*}
Then
\[f(p) = \lim_{n \to \infty} \int_{\left(p-\frac 1{\sqrt n},p+\frac 1{\sqrt n}\right)}k_n(x)f(x) \,dx.\]
\end{lem}
\begin{proof}
For \(\epsilon>0\) we set
\[c_\epsilon \coloneqq \frac 1{2\epsilon} \int_{(p-\epsilon,p+\epsilon)} |f(x)-f(p)| \,dx.\]
Then, since \(p\) is a Lebesgue point, we have \(\lim_{\epsilon \downarrow 0}c_\epsilon = 0\). We set
\[C_\epsilon \coloneqq \sup_{t \in (0,\epsilon]} c_t.\]
Then also \(\lim_{\epsilon \downarrow 0}C_\epsilon = 0\). We now define the function
\[m_n(x) \coloneqq \begin{cases} rn - rn^2|x-p| & \text{if } x \in \left(p-\frac 1n,p+\frac 1n\right) \\ 0 & \text{else}\end{cases}.\]
Then, for every \(x \in \left(p,p+\frac 1n\right)\), we have
\[k_n'(x) + m_n'(x) = k_n'(x) - rn^2 \leq rn^2 - rn^2 = 0\]
and for every \(x \in \left(p-\frac 1n,p\right)\), we have
\[k_n'(x) + m_n'(x) = k_n'(x) + rn^2 \geq -rn^2 + rn^2 = 0,\]
so \(k_n'(x)+m_n'(x) \leq 0\) for \(x \in \left(p,p+\frac 1{\sqrt n}\right)\) and \(k_n'(x)+m_n'(x) \geq 0\) for \(x \in \left(p-\frac 1{\sqrt n},p\right)\).
This yields
\begin{align*}
&\int_p^{p+\frac 1{\sqrt n}} \left(k_n(x) + m_n(x)\right) |f(x)-f(p)| \,dx
\\&\qquad= \int_p^{p+\frac 1{\sqrt n}} \left(k_n\left(p+\frac 1{\sqrt n}\right) - \int_x^{p+\frac 1{\sqrt n}} k_n'(t) + m_n'(t) \,dt\right) |f(x)-f(p)| \,dx
\\&\qquad= k_n\left(p+\frac 1{\sqrt n}\right) \cdot \int_p^{p+\frac 1{\sqrt n}} |f(x)-f(p)| \,dx
\\&\qquad\qquad + \int_p^{p+\frac 1{\sqrt n}} \int_x^{p+\frac 1{\sqrt n}} -\left(k_n'(t) + m_n'(t)\right) |f(x)-f(p)| \,dt \,dx
\\&\qquad\leq R\sqrt{n} \cdot 2\frac 1{\sqrt n}c_{\frac 1{\sqrt n}} + \int_p^{p+\frac 1{\sqrt n}} \int_p^t -\left(k_n'(t) + m_n'(t)\right) |f(x)-f(p)| \,dx \,dt
\\&\qquad= 2R \cdot c_{\frac 1{\sqrt n}} + \int_p^{p+\frac 1{\sqrt n}} -\left(k_n'(t) + m_n'(t)\right)\int_p^t  |f(x)-f(p)| \,dx \,dt
\\&\qquad\leq 2R \cdot c_{\frac 1{\sqrt n}} + \int_p^{p+\frac 1{\sqrt n}} -\left(k_n'(t) + m_n'(t)\right) 2\left(t-p\right)c_{t-p} \,dt
\\&\qquad\leq 2R \cdot c_{\frac 1{\sqrt n}} + 2C_{\frac 1{\sqrt n}} \int_p^{p+\frac 1{\sqrt n}} -\left(k_n'(t) + m_n'(t)\right) \left(t-p\right) \,dt
\\&\qquad= 2R \cdot c_{\frac 1{\sqrt n}} + 2C_{\frac 1{\sqrt n}} \left(\left[-\left(k_n(t) + m_n(t)\right) \left(t-p\right)\right]_p^{p+\frac 1{\sqrt n}} + \int_p^{p+\frac 1{\sqrt n}} \left(k_n(t) + m_n(t)\right) \,dt\right)
\\&\qquad= 2R \cdot c_{\frac 1{\sqrt n}} + 2C_{\frac 1{\sqrt n}} \left(-k_n\left(p+\frac 1{\sqrt n}\right)\frac 1{\sqrt n} + \int_p^{p+\frac 1{\sqrt n}} k_n(t) \,dt + \frac r2\right)
\\&\qquad\leq 2R \cdot c_{\frac 1{\sqrt n}} + 2C_{\frac 1{\sqrt n}} \left(0 + 1 + \frac r2\right) \xrightarrow{n \to \infty} 0
\end{align*}
and analogously
\[\int_{p-\frac 1{\sqrt n}}^p \left(k_n(x) + m_n(x)\right) |f(x)-f(p)| \,dx \xrightarrow{n \to \infty} 0.\]
Finally, we have
\begin{align*}
\left|\int_{\left(p-\frac 1{\sqrt n},p+\frac 1{\sqrt n}\right)}k_n(x)f(x) \,dx - f(p)\right| &= \left|\int_{\left(p-\frac 1{\sqrt n},p+\frac 1{\sqrt n}\right)}k_n(x)\left(f(x)-f(p)\right) \,dx\right|
\\&\leq \int_{\left(p-\frac 1{\sqrt n},p+\frac 1{\sqrt n}\right)}k_n(x)\left|f(x)-f(p)\right| \,dx
\\&\leq \int_{\left(p-\frac 1{\sqrt n},p+\frac 1{\sqrt n}\right)}\left(k_n(x) + m_n(x)\right)\left|f(x)-f(p)\right| \,dx \xrightarrow{n \to \infty} 0,
\end{align*}
so
\[f(p) = \lim_{n \to \infty} \int_{\left(p-\frac 1{\sqrt n},p+\frac 1{\sqrt n}\right)}k_n(x)f(x) \,dx. \qedhere\]
\end{proof}
\begin{lem}\label{lem:dIntegral1}
For every \(p \in \R_+\), one has
\[\lim_{n \to \infty} \frac 1\pi \int_{\left(p-\frac 1{\sqrt n},p+\frac 1{\sqrt n}\right)} \tilde d_{p,n}(x) \,dx = \frac 12.\]
\end{lem}
\begin{proof}
We define the function
\[D_{p,n}(x) \coloneqq \frac{p^2\left(1+\frac 1{n^2}\right)}{2\left(\frac 1{n^2}+p^2\right)\left(1+\frac{p^2}{n^2}\right)} \cdot \arctan\left(\frac{x^2-p^2-\frac 1{n^2}}{\frac{2p}n}\right).\]
Then
\begin{align*}
D_{p,n}'(x) &= \frac{p^2\left(1+\frac 1{n^2}\right)}{2\left(\frac 1{n^2}+p^2\right)\left(1+\frac{p^2}{n^2}\right)} \cdot \frac{\frac{nx}p}{1 + \left(\frac{x^2-p^2-\frac 1{n^2}}{\frac{2p}n}\right)^2}
\\&= \frac{px}{n\left(\frac 1{n^2}+p^2\right)\left(1+\frac{p^2}{n^2}\right)} \cdot \frac{2p^2\left(1+\frac 1{n^2}\right)}{\left(x^2-p^2-\frac 1{n^2}\right)^2+4p^2 \frac 1{n^2}} = \tilde d_{p,n}(x).
\end{align*}
Since
\[\arctan\left(\frac{\left(p\pm\frac 1{\sqrt n}\right)^2-p^2-\frac 1{n^2}}{\frac{2p}n}\right) = \arctan\left(\frac{\frac 1n \pm \frac {2p}{\sqrt n} - \frac 1{n^2}}{\frac{2p}n}\right) = \arctan\left(\frac 1{2p} \pm \sqrt{n} - \frac 1{2pn}\right),\]
we get
\begin{align*}
&\int_{\left(p-\frac 1{\sqrt n},p+\frac 1{\sqrt n}\right)} \tilde d_{p,n}(x) \,dx = \left[D_{p,n}(x)\right]_{p-\frac 1{\sqrt n}}^{p+\frac 1{\sqrt n}}
\\&\qquad= \frac{p^2\left(1+\frac 1{n^2}\right)}{2\left(\frac 1{n^2}+p^2\right)\left(1+\frac{p^2}{n^2}\right)} \cdot  \left(\arctan\left(\frac 1{2p} + \sqrt{n} - \frac 1{2pn}\right) - \arctan\left(\frac 1{2p} - \sqrt{n} - \frac 1{2pn}\right)\right)
\\&\qquad \xrightarrow{n \to \infty} \frac 12 \cdot \frac \pi 2 - \frac 12 \cdot \left(-\frac \pi 2\right) = \frac \pi 2. \qedhere
\end{align*}
\end{proof}
\begin{lem}\label{lem:localPhiIntegral}
Let \(\phi: \R \to \C\) be a measurable function with
\[\int_\R \frac{x^2}{\left(1+x^2\right)^2} \left|\phi(x)\right| \,dx < \infty.\]
Further, let \(p \in \R_+\) be a Lebesgue point of \(\phi\). Then
\[\lim_{n \to \infty} \frac 1\pi \int_{\left(p-\frac 1{\sqrt n},p+\frac 1{\sqrt n}\right)} \tilde d_{p,n}(x) \phi(x) \,dx = \frac 12 \phi(p).\]
\end{lem}
\begin{proof}
For \(n \in \N\) with \(\frac 1{\sqrt n} < p\) set
\[D_n \coloneqq \frac 1\pi \int_{\left(p-\frac 1{\sqrt n},p+\frac 1{\sqrt n}\right)} \tilde d_{p,n}(x) \,dx.\]
Then, by \fref{lem:dIntegral1}, we have \(\lim_{n \to \infty}D_n = \frac 12\). Further, we define
\[k_n(x) \coloneqq \frac{\tilde d_{p,n}(x)}{\pi D_n}.\]
Then
\[\int_{\left(p-\frac 1{\sqrt n},p+\frac 1{\sqrt n}\right)} k_n(x) \,dx = 1\]
and \(k_n(x) \geq 0\) for \(x \in \left(p-\frac 1{\sqrt n},p+\frac 1{\sqrt n}\right)\).
Then, for sufficiently large \(n \in \N\), one has
\begin{align*}
k_n\left(p\pm\frac 1{\sqrt n}\right) &= \frac 1{\pi D_n} \frac{p\left(p\pm\frac 1{\sqrt n}\right)}{n\left(\frac 1{n^2}+p^2\right)\left(1+\frac{p^2}{n^2}\right)} \cdot \frac{2p^2\left(\frac 1{n^2} + 1\right)}{\left(\left(p\pm\frac 1{\sqrt n}\right)^2-p^2-\frac 1{n^2}\right)^2+4p^2 \frac 1{n^2}}
\\&\leq \frac 1{\pi D_n} \frac{p\left(p+1\right)}{n \cdot p^2 \cdot 1} \cdot \frac{4p^2}{\left(\frac 1n \pm \frac{2p}{\sqrt n}-\frac 1{n^2}\right)^2 + 0}
\\&= \frac 1{\pi D_n} \frac{4p\left(p+1\right)}{\left(\frac 1{\sqrt n} \pm 2p -\frac 1{\sqrt{n^3}}\right)^2} \xrightarrow{n \to \infty} \frac 2\pi \frac{4p\left(p+1\right)}{4p^2} = \frac{2(p+1)}{\pi p},
\end{align*}
so for sufficiently large \(n \in \N\), we have
\[k_n\left(p\pm\frac 1{\sqrt n}\right) \leq \frac{p+1}{p} \leq \frac{p+1}{p} \sqrt{n}.\]
Further, we have
\begin{align*}
k_n'(x) &= \frac 1{\pi D_n} \frac{2p^3\left(\frac 1{n^2} + 1\right)}{n\left(\frac 1{n^2}+p^2\right)\left(1+\frac{p^2}{n^2}\right)} \cdot \left(\frac{x}{\left(x^2-p^2-\frac 1{n^2}\right)^2+4p^2 \frac 1{n^2}}\right)'
\\&= \frac 1{\pi D_n} \frac{2p^3\left(\frac 1{n^2} + 1\right)}{n\left(\frac 1{n^2}+p^2\right)\left(1+\frac{p^2}{n^2}\right)} \cdot \left(\frac{1}{\left(x^2-p^2-\frac 1{n^2}\right)^2+4p^2 \frac 1{n^2}} - \frac{2x^2\left(x^2-p^2-\frac 1{n^2}\right)}{\left(\left(x^2-p^2-\frac 1{n^2}\right)^2+4p^2 \frac 1{n^2}\right)^2}\right)
\end{align*}
For \(x \in \left(p-\frac 1n,p+\frac 1n\right)\), using
\[\frac{2p^3\left(\frac 1{n^2} + 1\right)}{n\left(\frac 1{n^2}+p^2\right)\left(1+\frac{p^2}{n^2}\right)} \leq \frac{4p^3}{n \cdot p^2 \cdot 1} = \frac{4p}{n}\]
and
\begin{align*}
\left|x^2-p^2-\frac 1{n^2}\right| &= \left|\left(p+(x-p)\right)^2-p^2-\frac 1{n^2}\right| = \left|2p(x-p) + (x-p)^2-\frac 1{n^2}\right|
\\&\leq 2p\frac 1n + \left(\frac 1n\right)^2 + \frac 1{n^2} \leq \frac{2(p+1)}n,
\end{align*}
this yields
\begin{align*}
\left|k_n'(x)\right| &\leq\frac 1{\pi D_n} \frac{4p}{n} \cdot \left(\frac{1}{\left(x^2-p^2-\frac 1{n^2}\right)^2+4p^2 \frac 1{n^2}} + \frac{2x^2\left|x^2-p^2-\frac 1{n^2}\right|}{\left(\left(x^2-p^2-\frac 1{n^2}\right)^2+4p^2 \frac 1{n^2}\right)^2}\right)
\\&\leq\frac 1{\pi D_n} \frac{4p}{n} \cdot \left(\frac{1}{4p^2 \frac 1{n^2}} + \frac{2(p+1)^2 \cdot \frac{2(p+1)}n}{\left(4p^2 \frac 1{n^2}\right)^2}\right)
\\&=\frac 1{\pi D_n} \left(\frac{n}{p} + \frac{(p+1)^3n^2}{p^3}\right) \leq \frac 1{\pi D_n} \frac{2(p+1)^3}{p^3} \cdot n^2.
\end{align*}
Then the fact that \(\lim_{n \to \infty}D_n = \frac 12\) implies that there is \(r>0\) such that \(\left|k_n'(x)\right| \leq rn^2\) for every \(x \in \left(p-\frac 1n,p+\frac 1n\right)\).

Further, for \(x \in \left(p-\frac 1{\sqrt n},p-\frac 1n\right)\), we have
\[\left(x^2-p^2-\frac 1{n^2}\right) \leq 0,\]
so
\begin{align*}
k_n'(x) &= \frac 1{\pi D_n} \frac{2p^3\left(\frac 1{n^2} + 1\right)}{n\left(\frac 1{n^2}+p^2\right)\left(1+\frac{p^2}{n^2}\right)} \cdot \left(\frac{1}{\left(x^2-p^2-\frac 1{n^2}\right)^2+4p^2 \frac 1{n^2}} - \frac{2x^2\left(x^2-p^2-\frac 1{n^2}\right)}{\left(\left(x^2-p^2-\frac 1{n^2}\right)^2+4p^2 \frac 1{n^2}\right)^2}\right)
\\&\geq \frac 1{\pi D_n} \frac{2p^3\left(\frac 1{n^2} + 1\right)}{n\left(\frac 1{n^2}+p^2\right)\left(1+\frac{p^2}{n^2}\right)} \cdot \left(0 - 0\right) = 0.
\end{align*}
For \(x \in \left(p+\frac 1n,p+\frac 1{\sqrt n}\right)\), we get
\begin{align*}
&\left(x^2-p^2-\frac 1{n^2}\right)^2+4p^2 \frac 1{n^2} - 2x^2\left(x^2-p^2-\frac 1{n^2}\right) = \left(-x^2-p^2-\frac 1{n^2}\right)\left(x^2-p^2-\frac 1{n^2}\right)+4p^2 \frac 1{n^2}
\\&\qquad\qquad\qquad= \left(p^2+\frac 1{n^2}\right)^2-x^4+4p^2 \frac 1{n^2} \leq \left(p^2+\frac 1{n^2}\right)^2-\left(p+\frac 1n\right)^4+4p^2 \frac 1{n^2}
\\&\qquad\qquad\qquad= \left(p^2+\frac 1{n^2}\right)^2-\left(p^2+\frac 1{n^2} + 2\frac pn\right)^2+\left(2\frac pn\right)^2 \leq 0,
\end{align*}
so
\begin{align*}
k_n'(x) &= \frac 1{\pi D_n} \frac{2p^3\left(\frac 1{n^2} + 1\right)}{n\left(\frac 1{n^2}+p^2\right)\left(1+\frac{p^2}{n^2}\right)} \cdot \left(\frac{1}{\left(x^2-p^2-\frac 1{n^2}\right)^2+4p^2 \frac 1{n^2}} - \frac{2x^2\left(x^2-p^2-\frac 1{n^2}\right)}{\left(\left(x^2-p^2-\frac 1{n^2}\right)^2+4p^2 \frac 1{n^2}\right)^2}\right)
\\&= \frac 1{\pi D_n} \frac{2p^3\left(\frac 1{n^2} + 1\right)}{n\left(\frac 1{n^2}+p^2\right)\left(1+\frac{p^2}{n^2}\right)} \cdot \left(\frac{\left(x^2-p^2-\frac 1{n^2}\right)^2+4p^2 \frac 1{n^2} - 2x^2\left(x^2-p^2-\frac 1{n^2}\right)}{\left(\left(x^2-p^2-\frac 1{n^2}\right)^2+4p^2 \frac 1{n^2}\right)^2}\right)
\\&\leq \frac 1{\pi D_n} \frac{2p^3\left(\frac 1{n^2} + 1\right)}{n\left(\frac 1{n^2}+p^2\right)\left(1+\frac{p^2}{n^2}\right)} \cdot \left(\frac{0}{\left(\left(x^2-p^2-\frac 1{n^2}\right)^2+4p^2 \frac 1{n^2}\right)^2}\right) = 0.
\end{align*}
These estimates show that we can apply \fref{lem:LebesguePointApprox} and get
\[\frac 1\pi \int_{\left(p-\frac 1{\sqrt n},p+\frac 1{\sqrt n}\right)} \tilde d_{p,n}(x) \phi(x) \,dx = D_n \cdot \int_{\left(p-\frac 1{\sqrt n},p+\frac 1{\sqrt n}\right)} k_n(x) \phi(x) \,dx \xrightarrow{n \to \infty} \frac 12 \cdot \phi(p). \qedhere\]
\end{proof}
\begin{lem}\label{lem:LebesguePointKonvergence}
Let \(\phi: \R \to \C\) be a measurable function with
\[\int_\R \frac{x^2}{\left(1+x^2\right)^2} \left|\phi(x)\right| \,dx < \infty.\]
Further, let \(p \in \R_+\) such that \(p\) and \(-p\) are Lebesgue points of \(\phi\). Then
\[\frac{\phi(p) + \phi(-p)}2 = \lim_{n \to \infty} \frac 1\pi \int_{\left(-p-\frac 1{\sqrt n},-p+\frac 1{\sqrt n}\right) \cup \left(p-\frac 1{\sqrt n},p+\frac 1{\sqrt n}\right)} d_{p,n}(x) \,\phi(x) \,dx.\]
\end{lem}
\begin{proof}
By \fref{lem:localPhiIntegral} and \fref{lem:dSlightChange}, we have
\[\lim_{n \to \infty} \frac 1\pi \int_{\left(p-\frac 1{\sqrt n},p+\frac 1{\sqrt n}\right)} d_{p,n}(x)\phi(x) \,dx = \frac 12 \phi(p).\]
Since for every \(p \in \R_+\) and every \(n \in \N\) the function \(d_{p,n}\) is symmetric, we then also get
\[\lim_{n \to \infty} \frac 1\pi \int_{\left(-p-\frac 1{\sqrt n},-p+\frac 1{\sqrt n}\right)} d_{p,n}(x)\phi(x) \,dx = \frac 12 \phi(-p),\]
so
\[\frac{\phi(p) + \phi(-p)}2 = \lim_{n \to \infty} \frac 1\pi \int_{\left(-p-\frac 1{\sqrt n},-p+\frac 1{\sqrt n}\right) \cup \left(p-\frac 1{\sqrt n},p+\frac 1{\sqrt n}\right)} d_{p,n}(x) \,\phi(x) \,dx. \qedhere\]
\end{proof}
\begin{theorem}\label{thm:LebesguePointAe}{\rm\textbf{(Lebesgue Differentiation Theorem)}(cf. \cite{Le10})}
Let \(f \in L^1_{\mathrm{loc}}(\R,\C)\). Then almost every \(p \in \R\) is a Lebesgue point of \(f\).
\end{theorem}
\begin{cor}\label{cor:LebesguePointPhi}
For any measurable function \(\phi: \R \to \C\) with
\[\int_\R \frac{x^2}{\left(1+x^2\right)^2} \left|\phi(x)\right| \,dx < \infty\]
almost every \(p \in \R\) is a Lebesgue point of \(\phi\).
\end{cor}
\begin{proof}
For \(0<a<b<\infty\), one has
\begin{align*}
\int_{[-b,-a] \cup [a,b]} |\phi(x)| \,dx &\leq \int_{[-b,-a] \cup [a,b]} \frac{\left(1+b^2\right)^2}{a^2} \cdot \frac{x^2}{\left(1+x^2\right)^2} |\phi(x)| \,dx
\\&\leq \frac{\left(1+b^2\right)^2}{a^2} \cdot \int_\R \frac{x^2}{\left(1+x^2\right)^2} \left|\phi(x)\right| \,dx < \infty.
\end{align*}
The statement then follows by \fref{thm:LebesguePointAe}.
\end{proof}
\begin{theorem}\label{thm:bigIntegralApproximation}
For any measurable function \(\phi: \R \to \C\) with
\[\int_\R \frac{x^2}{\left(1+x^2\right)^2} \left|\phi(x)\right| \,dx < \infty\]
one has
\[\frac{\phi(p) + \phi(-p)}2 = \lim_{n \to \infty} \frac 1\pi \int_\R f_{p,n}(x) \,\phi(x) \,dx\]
for almost every \(p \in \R^\times\).
\end{theorem}
\begin{proof}
Let \(p \in \R_+\) such that \(p\) and \(-p\) are Lebesgue points of \(\phi\). For \(n \in \N\), we set
\[c_n \coloneqq \int_{\left[-p - \frac 1{\sqrt n},-1\right) \cup \left(1,p + \frac 1{\sqrt n}\right]} g_n(x) \phi(x) \,dx \quad \text{if } p + \frac 1{\sqrt n} \geq 1\]
and
\[c_n \coloneqq - \int_{\left[-1,-p - \frac 1{\sqrt n}\right) \cup \left(p + \frac 1{\sqrt n},1\right]} g_n(x) \phi(x) \,dx \quad \text{if } p + \frac 1{\sqrt n} < 1.\]
Then, for \(p + \frac 1{\sqrt n} \geq 1\), one has
\begin{align*}
|c_n| \leq \int_{\left[-p - 1,-1\right) \cup \left(1,p + 1\right]} \frac 1n |\phi(x)| \,dx &\leq \int_{\left[-p - 1,-1\right) \cup \left(1,p + 1\right]} \frac 1n \frac{\left(1+\left(p+1\right)^2\right)^2}{1^2} \cdot \frac{x^2}{\left(1+x^2\right)^2} |\phi(x)| \,dx
\\&\leq \frac{\left(1+\left(p+1\right)^2\right)^2}{n} \cdot \int_\R \frac{x^2}{\left(1+x^2\right)^2} |\phi(x)| \,dx \xrightarrow{n \to \infty} 0
\end{align*}
and for \(p + \frac 1{\sqrt n} < 1\), one has
\begin{align*}
|c_n| \leq \int_{\left[-1,-p - \frac 1{\sqrt n}\right) \cup \left(p + \frac 1{\sqrt n},1\right]} \frac 1n |\phi(x)| \,dx &\leq \int_{\left[-1,-p\right) \cup \left(p,1\right]} \frac 1n \frac{\left(1+\left(1\right)^2\right)^2}{p^2} \cdot \frac{x^2}{\left(1+x^2\right)^2} |\phi(x)| \,dx
\\&\leq \frac 4{np^2} \cdot \int_\R \frac{x^2}{\left(1+x^2\right)^2} |\phi(x)| \,dx \xrightarrow{n \to \infty} 0,
\end{align*}
so in every case
\[\lim_{n \to \infty} c_n = 0.\]
Further, for \(n \in \N\) with \(\frac 1{\sqrt n} \leq \frac p2\), we set
\[k_n \coloneqq \int_{\left[-p + \frac 1{\sqrt n},-1\right) \cup \left(1,p - \frac 1{\sqrt n}\right]} g_n(x) \phi(x) \,dx \quad \text{if } p - \frac 1{\sqrt n} \geq 1\]
and
\[k_n \coloneqq - \int_{\left[-1,-p + \frac 1{\sqrt n}\right) \cup \left(p - \frac 1{\sqrt n},1\right]} g_n(x) \phi(x) \,dx \quad \text{if } p - \frac 1{\sqrt n} < 1.\]
Then, for \(p - \frac 1{\sqrt n} \geq 1\), one has
\begin{align*}
|k_n| \leq \int_{\left[-p,-1\right) \cup \left(1,p\right]} \frac 1n |\phi(x)| \,dx &\leq \int_{\left[-p,-1\right) \cup \left(1,p\right]} \frac 1n \frac{\left(1+p^2\right)^2}{1^2} \cdot \frac{x^2}{\left(1+x^2\right)^2} |\phi(x)| \,dx
\\&\leq \frac{\left(1+p^2\right)^2}{n} \cdot \int_\R \frac{x^2}{\left(1+x^2\right)^2} |\phi(x)| \,dx \xrightarrow{n \to \infty} 0
\end{align*}
and for \(p - \frac 1{\sqrt n} < 1\), one has
\begin{align*}
|k_n| \leq \int_{\left[-1,-\frac p2\right) \cup \left(\frac p2,1\right]} \frac 1n |\phi(x)| \,dx &\leq \int_{\left[-1,-\frac p2\right) \cup \left(\frac p2,1\right]} \frac 1n \frac{\left(1+\left(1\right)^2\right)^2}{\left(\frac p2\right)^2} \cdot \frac{x^2}{\left(1+x^2\right)^2} |\phi(x)| \,dx
\\&\leq \frac {16}{np^2} \cdot \int_\R \frac{x^2}{\left(1+x^2\right)^2} |\phi(x)| \,dx \xrightarrow{n \to \infty} 0,
\end{align*}
so, in every case
\[\lim_{n \to \infty} k_n = 0.\]
Then, by \fref{lem:approximationZero}, \fref{lem:approximationInfinity} and \fref{lem:LebesguePointKonvergence}, one has
\begin{align*}
\lim_{n \to \infty} \frac 1\pi \int_\R f_{p,n}(x) \,\phi(x) \,dx &= \lim_{n \to \infty} \frac 1\pi \int_{\left(-p-\frac 1{\sqrt n},-p+\frac 1{\sqrt n}\right) \cup \left(p-\frac 1{\sqrt n},p+\frac 1{\sqrt n}\right)} d_{p,n}(x) \,\phi(x) \,dx
\\&\qquad + \lim_{n \to \infty} \frac 1\pi \int_{\left(-p+\frac 1{\sqrt n},p-\frac 1{\sqrt n}\right)} \left(d_{p,n}(x)+ \frac{g_n(x)}{p^2}\right)\phi(x) \,dx
\\&\qquad + \lim_{n \to \infty} \frac 1\pi \int_{\left(-\infty,-p-\frac 1{\sqrt n}\right) \cup \left(p+\frac 1{\sqrt n},\infty\right)} \left(d_{p,n}(x)+g_n(x)\right)\phi(x) \,dx
\\&\qquad + \lim_{n \to \infty} \frac 1\pi c_n + \lim_{n \to \infty} \frac 1\pi \frac{k_n}{p^2}
\\&= \frac{\phi(p) + \phi(-p)}2 + 0 + 0 + 0 + 0
\\&= \frac{\phi(p) + \phi(-p)}2.
\end{align*}
The assertion then follows, since by \fref{cor:LebesguePointPhi} almost every \(p \in \R^\times\) is a Lebesgue point of the function \(\phi\).
\end{proof}

\newpage
\printnoidxglossary[type=symbols,style=long,title={List of Symbols}]

\end{document}